\documentclass[10pt, a4paper]{article}

\usepackage[german,english]{babel} 
\usepackage{booktabs} 
\usepackage{comment} 
\usepackage[utf8]{inputenc} 
\usepackage[T1]{fontenc} 
\usepackage{caption, soul}
\usepackage{authblk}

\usepackage{bigints,bbm,slashed,mathtools,amssymb,amsmath,amsfonts,amsthm} 
\usepackage{physics, upgreek}
\usepackage{cancel}
\usepackage{url}
\usepackage{mathrsfs}
\usepackage{enumerate}
\usepackage{tikz} 
\usepackage{tikz-cd}
\usepackage{pgfplots}
\pgfplotsset{compat=1.16}
\usetikzlibrary{babel}
\usetikzlibrary{positioning,arrows} 
\usetikzlibrary{decorations.pathreplacing} 
\usetikzlibrary{patterns}
\usepackage{tikzsymbols} 
\usepackage{blindtext} 

\usepackage{geometry} 

\usepackage{setspace} 
\usepackage[T1]{fontenc} 
\usepackage[utf8]{inputenc} 

\usepackage{fancyhdr} 
\usepackage[hang,flushmargin]{footmisc}

\usepackage[square, numbers]{natbib} 

\usepackage{har2nat} 
\usepackage[colorlinks=true]{hyperref}

\numberwithin{equation}{section}
\theoremstyle{definition}
\newtheorem{definition}{Definition}[section]
\theoremstyle{plain}
\newtheorem{theorem}{Theorem}[section]
\newtheorem{proposition}[theorem]{Proposition}
\newtheorem{lemma}[theorem]{Lemma}
\newtheorem{corollary}[theorem]{Corollary}
\newtheorem{conjecture}[theorem]{Conjecture}

\newtheorem{innercustomgeneric}{\customgenericname}
\providecommand{\customgenericname}{}
\newcommand{\newcustomtheorem}[2]{%
  \newenvironment{#1}[1]
  {%
   \renewcommand\customgenericname{#2}%
   \renewcommand\theinnercustomgeneric{##1}%
   \innercustomgeneric
  }
  {\endinnercustomgeneric}
}

\newcustomtheorem{customthm}{Theorem}
\newcustomtheorem{customlemma}{Lemma}
\newcustomtheorem{customcorollary}{Corollary}

\theoremstyle{remark}
\newtheorem*{remark}{Remark}

\newcommand{\R}{\mathbb{R}}
\newcommand{\Z}{\mathbb{Z}}

\newcommand{\matr}[3]{{#1}_{{#2}}^{\phantom{{#2}}{#3}}}
\newcommand{\CC}{\mathbb{C}}

\DeclareMathOperator{\Skew}{Skew}

\DeclareFontFamily{U}{mathx}{}
\DeclareFontShape{U}{mathx}{m}{n}{<-> mathx10}{}
\DeclareSymbolFont{mathx}{U}{mathx}{m}{n}
\DeclareMathAccent{\widehat}{0}{mathx}{"70}
\DeclareMathAccent{\widecheck}{0}{mathx}{"71}

\title{Scattering towards the singularity for the wave equation and \\ the linearized Einstein--scalar field system in Kasner spacetimes}
\author{Warren Li}
\affil{\small Princeton University, Department of Mathematics, Fine Hall, Washington Road, Princeton, NJ 08544, USA}

\begin{document}

\maketitle

\begin{abstract}
    We consider the scalar wave equation $\square_g \phi = 0$ and the linearized Einstein--scalar field system around generalized Kasner spacetimes with spatial topology $\mathbb{T}^D$. In suitable regimes for the Kasner exponents, it is known that solutions to such equations arising from regular Cauchy data (e.g.~at $t=1$) have certain quantitative blow-up asymptotics near the initial time (i.e~$t=0$) singularity of Kasner. For instance, solutions to the wave equation behave as $\phi(t, x) \approx \psi_{\infty} (x) \log t + \varphi_{\infty}(x)$ near $t = 0$.

    This article provides a description, and proof, of a scattering theory for the above equations, linking Cauchy data at $t=1$ and suitable asymptotic data at $t = 0$ in Kasner. For the scalar wave equation, this means a Hilbert space isomorphism between $(\phi, \partial_t \phi)$ at $t = 1$ and the functions $(\psi_{\infty}, \varphi_{\infty})$. A curious detail is that certain quantities e.g.~$\psi_{\infty}$, feature a gain of $1/2$ a derivative when compared to $\partial_t \phi$ at $t = 1$. 

    The study of the linearized Einstein--scalar field system reveals further interesting phenomena, including differences between diagonal and off-diagonal components of certain tensors in the scattering theory, and that the losses of derivatives feature a sensitive dependence on the anisotropy of the background Kasner spacetime. In fact, though our result holds for the entire subcritical regime of background Kasner exponents, the number of derivatives lost and gained in the scattering theory can become unbounded as one nears the boundary of this regime.
%
%
\end{abstract}

\setcounter{tocdepth}{2}
\tableofcontents

\section{Introduction} \label{sec:intro}

\subsection{Background} \label{sub:intro_background}
The (generalized) Kasner spacetimes \cite{Kasner} are a finite-dimensional family of spacetimes, defined on the manifold $\mathcal{M}_{Kas} = (0, + \infty) \times \mathbb{T}^D$ (we view $\mathbb{T}^D$ as $[- \pi, \pi]^D$ with endpoints identified) and equipped with the following Lorentzian metric $g_{Kas}$:
\begin{equation} \label{eq:kasner}
    g_{Kas} = - dt^2 + \sum_{i = 1}^D t^{2p_i} (dx^i)^2.
\end{equation}
The parameters $p_i \in \R$ are called the \emph{Kasner exponents}. 

We consider the Kasner metric \eqref{eq:kasner} as a family of cosmological\footnote{Following the relativistic literature, the word cosmological refers to a gloally hyperbolic spacetimes having a closed spatial topology.} solutions to the Einstein field equations coupled to a massless scalar field $\phi: \mathcal{M}_{Kas} \to \R$:
\begin{equation} \label{eq:einstein}
    \mathbf{G}_{\mu\nu}[g] \coloneqq \mathbf{Ric}_{\mu\nu}[g] - \frac{1}{2} \mathbf{R}[g] \, g_{\mu\nu} = 2 \, \mathbf{T}_{\mu\nu} [\phi],
\end{equation}
\begin{equation} \label{eq:energymomentum}
    \mathbf{T}_{\mu\nu}[\phi] \coloneqq \mathbf{D}_{\mu} \phi \, \mathbf{D}_{\nu} \phi - \frac{1}{2} \mathbf{D}_{\sigma} \phi \, \mathbf{D}^{\sigma}\phi \, g_{\mu\nu}.
\end{equation}
Here $\mathbf{D}$, $\mathbf{Ric}_{\mu\nu}[g]$ and $\mathbf{R}[g]$ denote the Levi-Civita connection, the Ricci curvature and the scalar curvature of the spacetime metric $g$, while $\mathbf{T}_{\mu\nu}[\phi]$ denotes the \emph{energy-momentum} tensor of the scalar field $\phi$. It follows from the contracted Bianchi equation and \eqref{eq:einstein}--\eqref{eq:energymomentum} that the scalar wave equation holds:
\begin{equation} \label{eq:wave0}
    \square_g \phi \coloneqq \left( g^{-1} \right)^{\mu\nu} \mathbf{D}_{\mu} \mathbf{D}_{\nu} \phi = 0.
\end{equation}

The Kasner metrics \eqref{eq:kasner}, complemented by the following expression for $\phi$:
\begin{equation} \label{eq:kasner_phi}
    \phi_{Kas} = p_{\phi} \log t,
\end{equation}
are solutions to the Einstein--scalar field system \eqref{eq:einstein}--\eqref{eq:wave0} if and only if the following two algebraic conditions, known as the (generalized) \emph{Kasner relations}, between the Kasner exponents $p_i \in \R$ and the scalar field coefficient $p_{\phi} \in \R$, hold:
\begin{equation} \label{eq:kasner_relations}
    \sum_{i=1}^D p_i = 1, \qquad \sum_{i=1}^D p_i^2 + 2 p_{\phi}^2 = 1. 
\end{equation}

The Kasner spacetimes feature a spacelike past curvature singularity at $t = 0$, and play an important role in the physics literature regarding singularity formation for Einstein's equations, see \cite{kl63, bkl71, bk72, DHRW}. In particular, at least in certain regimes, it is expected there exists a reasonably large class of spacetimes containing spacelike singularities which are locally described by a Kasner spacetime, though whose Kasner exponents have spatially dependence. We return to these heuristics in Section~\ref{sub:bkl}. 

In this article, we determine a \emph{scattering theory} (between Cauchy data at $t=1$ and suitable asymptotic data living at the $t=0$ singularity) for two model \emph{linear} hyperbolic problems in the Kasner spacetimes $(\mathcal{M}_{Kas}, g_{Kas})$, as an initial step towards understanding the analogous scattering problem for the nonlinear Einstein--scalar field system \eqref{eq:einstein}--\eqref{eq:energymomentum}. By scattering, we mean a Hilbert space isomorphism between Cauchy data for our hyperbolic evolution equations at the $t=1$ Cauchy hypersurface $\Sigma_1 = \{1\} \times \mathbb{T}^D \subset \mathcal{M}_{Kas}$ of Kasner, and suitable asymptotic data which can be thought to live at the $t=0$ boundary of Kasner, between Hilbert spaces which must be defined as part of the theory.

The first of our model problems will be the scalar wave equation in a fixed Kasner background:
\begin{equation} \label{eq:wave}
    \square_{g_{Kas}} \phi \coloneqq (g_{Kas}^{-1})^{\mu \nu} \mathbf{D}_{\alpha} \mathbf{D}_{\beta} \phi = 0.
\end{equation}
In the $(t, x^i)$ coordinates of \eqref{eq:kasner}, this can be written as the following PDE:
\begin{equation} \label{eq:wave_coord}
    - \partial_t^2 \phi - \frac{\partial_t \phi}{t} + \sum_{i=1}^D t^{- 2 p_i} \partial_{x^i}^2 \phi = 0.
\end{equation}

The standard theory of linear hyperbolic PDE allows one to study the following initial value problem: if one poses regular Cauchy data $(\phi, \partial_t \phi) \in H^{s+1}(\mathbb{T}^D) \times H^s(\mathbb{T}^D)$ on the spacelike Cauchy hypersurface $\Sigma_1$, then there exists a unique regular solution to \eqref{eq:wave} on the entire spacetime $\mathcal{M}_{Kas}$ attaining the initial data. Standard hyperbolic regularity theory yields that the solution $(\phi, \partial_t \phi)$ obeys:
\[
    \left( \phi, \partial_t \phi \right) \in C^0((0, + \infty), H^{s+1} \times H^s) \cap C^1((0, + \infty), H^s \times H^{s-1}).
\]
However, since the metric $g_{Kas}$ degenerates towards $t=0$, the scalar field $\phi$ may blow up as $t \to 0$, even for smooth Cauchy data. In fact, one obtains e.g.~from \cite{AlhoFournodavlosFranzen, PetersenMode, RingstromAsterisque} the following asymptotics for $\phi$ as $t \to 0$:

\begin{theorem}[Asymptotics for the wave equation in non-degenerate Kasner spacetimes] \label{thm:asymp_wave}
    Let $\phi: \mathcal{M}_{Kas} \to \mathbb{R}$ be a smooth solution to the wave equation \eqref{eq:wave} on a non-degenerate\footnote{Non-degenerate means that the Kasner exponents are not such that all but one of the Kasner exponents $p_i$ vanish. The degenerate case with e.g.~$p_1 = 1$ and the remaining exponents vanishing is special in that the $t = 0$ boundary becomes a null Cauchy horizon rather than a spacelike singularity.} Kasner spacetime $(\mathcal{M}_{Kas}, g_{Kas})$, arising from smooth initial data $(\phi, \partial_t \phi)|_{\Sigma_1} = (\phi_C, \psi_C) \in C^{\infty}(\mathbb{T}^D)^2$. Then there exist smooth functions $\psi_{\infty}$ and $\varphi_{\infty}$ on $\mathbb{T}^D$ such that as $t \to 0$, $\phi$ obeys the following asympotics:
    \begin{equation} \label{eq:phi_asymp}
        \phi(t, x) = \psi_{\infty}(x) \log t + \varphi_{\infty}(x) + o(1).
    \end{equation}
\end{theorem}

\begin{remark}
    Theorem~\ref{thm:asymp_wave} implies that solutions to \eqref{eq:wave} asymptotically behave like solutions to the velocity dominated equation, meaning the equation \eqref{eq:wave_coord} with the spatial $\partial_{x^i}$-terms ignored, i.e.
    \(
        - \partial_t^2 \phi - \frac{\partial_t \phi}{t} = 0.
        \)
    We make further comments on asymptotically velocity dominated behaviour in Section~\ref{sub:bkl}.
\end{remark}

The asymptotics \eqref{eq:phi_asymp} obtained in Theorem~\ref{thm:asymp_wave} suggest a notion of asymptotic data at the $t=0$ singularity for the problem under consideration, namely the functions $\psi_{\infty}, \varphi_{\infty}: \mathbb{T}^D \to \R$. The full scattering problem for the wave equation may thus be formulated as follows:
\begin{enumerate}[(i)]
    \item (\emph{Existence of the scattering operator}):
        For what class of Cauchy data $(\phi_C, \psi_C)$ does there exist corresponding scattering data $(\psi_{\infty}, \varphi_{\infty})$ such that the solution $\phi$ to the wave equation attains the asymptotics \eqref{eq:phi_asymp} in a suitable sense? What can be said about the map $\mathcal{S}_{\downarrow}: (\phi_C, \psi_C) \mapsto (\psi_{\infty}, \varphi_{\infty})$?
    \item (\emph{Asymptotic completeness}):
        If instead we are given $(\psi_{\infty}, \phi_{\infty})$, when is it possible to find Cauchy data $(\phi_C, \psi_C)$ such that the solution $\phi$ to the wave equation attaining the Cauchy data obeys the asymptotics \eqref{eq:phi_asymp} with the given $\psi_{\infty}$ and $\phi_{\infty}$? What about the map $\mathcal{S}_{\uparrow} = (\mathcal{S}_{\downarrow})^{-1}: (\psi_{\infty}, \phi_{\infty}) \mapsto (\phi_C, \psi_C)$.
    \item (\emph{Scattering isomorphism}):
        Can we find Hilbert spaces $\mathcal{H}_C$ and $\mathcal{H}_{\infty}$ such that the scattering operator $\mathcal{S}_{\downarrow}$ can be extended to a Hilbert space isomorphism $\mathcal{S}_{\downarrow}: \mathcal{H}_C \to \mathcal{H}_{\infty}$ with inverse $\mathcal{S}_{\uparrow}$?
\end{enumerate}
This article provides a complete resolution to the scattering problem for the wave equation in non-degenerate Kasner spacetimes, see Theorem~\ref{thm:wave_scat} below.

The second model system addressed in this article is the linearized Einstein--scalar field system. This is given by the Fr\'echet derivative of the Einstein--scalar field system \eqref{eq:einstein}--\eqref{eq:wave0} with respect to the metric $g = g_{Kas} + \varepsilon h$ and scalar field $\phi = \phi_{Kas} + \varepsilon \upphi$. Inserting these into \eqref{eq:einstein}--\eqref{eq:wave0} and considering the resulting $O(\varepsilon)$ term, one finds the formal linearization to be
\begin{equation} \label{eq:einsteinlinear}
    \left( g_{Kas}^{-1} \right)^{\alpha \beta} \mathbf{D}_{\alpha} \mathbf{D}_{\beta} h_{\mu\nu} + \mathbf{D}_{\mu} \mathbf{D}_{\nu} \left( \tr h \right) - 2 \mathbf{D}_{(\mu} \left( \div h \right)_{\nu)} + 2 \, \mathbf{Riem}^{\alpha \phantom{\mu} \beta \phantom{\nu}}_{\phantom{\alpha} \mu \phantom{\beta} \nu} \, h_{\alpha \beta} = 4 \mathbf{D}_{(\mu} \upphi \mathbf{D}_{\nu)} \phi_{Kas}.
\end{equation}

Here and in the sequel the Levi-Civita connection $\mathbf{D}$, the Riemann tensor $\mathbf{Riem}$, and the raising and lowering of indices, will all be associated to the background (non-degenerate) Kasner spacetime $g_{Kas}$. We also have the following linearization of the wave equation\footnote{Just as the wave equation \eqref{eq:wave} follows from \eqref{eq:einstein}--\eqref{eq:energymomentum} and the (twice contracted) Bianchi equation, one can derive \eqref{eq:wavelinear} from \eqref{eq:einsteinlinear} and the linearized Bianchi equations.}:
\begin{equation} \label{eq:wavelinear}
    \square_{g_{Kas}} \upphi = h^{\alpha \beta} \mathbf{D}_{\alpha} \mathbf{D}_{\beta} \phi_{Kas} + \left( \div h \right)^{\alpha} \mathbf{D}_{\alpha} \phi_{Kas} - \frac{1}{2} \mathbf{D}^{\alpha} \! \left( \tr h \right) \mathbf{D}_{\alpha} \phi_{Kas}.
\end{equation}

To view the system \eqref{eq:einsteinlinear}--\eqref{eq:wavelinear} in the context of a well-posed initial value problem, we must impose a choice of (linearized) gauge, most often inherited from a suitable gauge for the nonlinear Einstein--scalar field system. 
In this article we to adopt a $(1+D)$-dimensional ADM-type gauge with Constant Mean Curvature (CMC), where the spacetime is foliated using a time function $t$ whose level sets $\Sigma_t = \{ t \} \times \mathbb{T}^D$ have constant mean curvature equal to $\tr k = - t^{-1}$. 

In this gauge, the key evolutionary variables of the linearized problem will be $h_{ij}$, $\kappa_{ij}$, $\upphi$, $\uppsi$, which represent linearized versions of the first fundamental form, the second fundamental form, the scalar field, and the normal derivative of the scalar field with respect to the foliation.  
%
With this choice of CMC foliation, together with a choice of spatial gauge, one may cast \eqref{eq:einsteinlinear}--\eqref{eq:wavelinear} as a well-posed initial value problem of elliptic-hyperbolic type. See \cite{AnderssonMoncrief} for an introduction to elliptic-hyperbolic problems in general relativity.

In particular, if one poses Cauchy data $((h_C)_{ij}, (\kappa_C)_{ij}, \upphi_C, \uppsi_C) \in H^{s+1} \times H^s \times H^{s+1} \times H^s$ on the Cauchy hypersurface $\Sigma_1$, which moreover satisfies various (linearized) constraints, then this data launches a solution to \eqref{eq:einsteinlinear}--\eqref{eq:wavelinear} in $\mathcal{M}_{Kas}$, with
\begin{gather*}
    \left( h_{ij}, \kappa_{ij} \right) \in C^0((0, + \infty), H^{s+1} \times H^s) \cap C^1((0, + \infty), H^s \times H^{s-1}), \\[0.5em]
    \left( \upphi, \uppsi \right) \in C^0((0, + \infty), H^{s+1} \times H^s) \cap C^1((0, + \infty), H^s \times H^{s-1}).
\end{gather*}

Now consider the asymptotics of these quantities as $t \to 0$. 
It turns out the analogue of Theorem~\ref{thm:asymp_wave} will only hold if we assume the Kasner exponents of $g_{Kas}$ to satisfy the following inequality, related to the subcritical regime of \cite{DHRW}:
\begin{equation} \label{eq:subcritical}
    \max_{i \neq j} \left( p_i + p_j - p_k \right) < 1.
\end{equation}
We defer the further discussion of subcritical exponents to Section~\ref{sub:bkl}. Assuming that this relation between the exponents holds on the background Kasner spacetime, we have the following result, where for now we define $\kappa_{ij}$ and $\uppsi$ as logarithmic time derivatives in the following sense:
\begin{equation*}
    \kappa_{ij} = - \frac{1}{2} t \partial_t h_{ij} + l.o.t., \qquad \uppsi = t \partial_t \upphi + l.o.t.
\end{equation*}

\begin{theorem}[Asymptotics for the linearized Einstein--scalar field system in subcritical Kasner spacetimes] \label{thm:asymp_einstein}
    Let $(\mathcal{M}_{Kas}, g_{Kas}, \phi_{Kas})$ be a Kasner spacetime whose Kasner exponents $p_i \in \R$ obey the subcriticality relation \eqref{eq:subcritical}. Suppose $(h_{ij}, \kappa_{ij}, \upphi, \uppsi)$ is a \underline{smooth} solution to the linearized Einstein--scalar field system \eqref{eq:einsteinlinear}--\eqref{eq:wavelinear} around $(g_{Kas}, \phi_{Kas})$ in a CMC ADM-type gauge, arising from smooth Cauchy data $((h_C)_{ij}, (\kappa_C)_{ij}, \upphi_C, \uppsi_C)$ on $\Sigma_1$ satisfying suitable constraints. 

    Then there exist smooth $(1, 1)$ tensors $(\Upupsilon_{\infty})_i^{\phantom{i}j}$, $(\upkappa_{\infty})_{i}^{\phantom{i}j}$ and smooth functions $\uppsi_{\infty}$ and $\upvarphi_{\infty}$ on $\mathbb{T}^D$ such that the solution obeys the following asymptotics\footnote{
    The expressions \eqref{eq:einstein_asymp_1}--\eqref{eq:einstein_asymp_2} do not, at first, appear symmetric in $i$ and $j$; the fact that they are follow from the symmetry constraints in Lemma~\ref{lem:scattering_constraints}.}as $t \to 0$,
    \begin{gather} \label{eq:einstein_asymp_1}
        \kappa_{ij} = \begin{cases}
            \left( \matr{(\kappa_{\infty})}{{i}}{j} + 2 p_{{i}} \matr{(\Upupsilon_{\infty})}{{i}}{j} \right) t^{2p_{{i}}} + 2 p_{{i}} \matr{(\upkappa_{\infty})}{{i}}{j} t^{2 p_{{i}}} \log t + o(t^{2p_i})  & \text{ if } p_i = p_j, \\
            - \frac{p_j}{p_i - p_i} \matr{(\upkappa_{\infty})}{i}{j} t^{2p_j} + \left( \frac{p_i}{p_i -p_j} \matr{(\upkappa_{\infty})}{i}{j} + 2 p_i \matr{(\Upupsilon_{\infty})}{j}{i} \right) t^{2 p_i} + o(t^{\min\{2p_i, 2p_j\}})  & \text{ if } p_i \neq p_j,
        \end{cases} 
        \\[0.5em] \label{eq:einstein_asymp_2}
        h_{ij} = \begin{cases}
            - 2 \left (\left( \upkappa_{\infty} \right) _i^{\phantom{i}j} \log t + \left( \Upupsilon_{\infty} \right)_i^{\phantom{i}j} \right) t^{2 p_j} + o(t^{2p_j}), & \text{ if }p_i = p_j, \\
            - \left ( \frac{\left (\upkappa_{\infty} \right)_i^{\phantom{i}j}}{p_i - p_j} + 2 \left( \Upupsilon_{\infty} \right)_i^{\phantom{i}j} \right) t^{2p_j} + \frac{\left( \upkappa_{\infty} \right)_i^{\phantom{i}j}}{p_i - p_j} \, t^{2p_i} + o(t^{\min\{2p_i, 2p_j\}}),
                                                                                                                                                                       & \text{ if }p_i \neq p_j,
        \end{cases} 
        \\[0.5em] \label{eq:einstein_asymp_3}
        \uppsi = \uppsi_{\infty} + o(1),
        \\[0.5em] \label{eq:einstein_asymp_4}
        \upphi = \uppsi_{\infty} \log t + \upvarphi_{\infty} + o(1).
    \end{gather}
\end{theorem}

Theorem~\ref{thm:asymp_einstein} is a corollary of \cite[Theorem 7.1]{RodnianskiSpeck0}, at least for the mildly anisotropic range of exponents included in that result. For the full range of Kasner exponents satisfying the subcriticality condition \eqref{eq:subcritical}, Theorem~\ref{thm:asymp_einstein} will follow from our scattering theory, and the proof is presented in Section~\ref{sub:scat_conclusion}. Note an analogous result for the fully nonlinear Einstein--scalar field equations (for the full subcritical range of exponents) can be deduced from \cite{FournodavlosRodnianskiSpeck}.

Given Theorem~\ref{thm:asymp_einstein}, we can make sense of the associated scattering problem as follows: what are the properties of the map $((h_C)_{ij}, (\kappa_C)_{ij}, \upphi_C, \uppsi_C) \mapsto ((\upkappa_{\infty})_i^{\phantom{i}j}, (\Upupsilon_{\infty})_i^{\phantom{i}j}, \uppsi_{\infty}, \upvarphi_{\infty})$ and, if well-defined, its inverse? Just as in the case of the wave equation, we may ask about (i) existence of the scattering operator, (ii) asymptotic completeness and (iii) the scattering isomorphism. 

In the scattering problem for linearized gravity, one encounters additional subtleties due to the constraint equations and the choice of spatial gauge. As part of our scattering theory, we must understand how the constraint equations are realized on the asymptotic data $((\upkappa_{\infty})_i^{\phantom{i}j}, (\Upupsilon_{\infty})_i^{\phantom{i}j}, \uppsi_{\infty}, \upvarphi_{\infty})$, and also understand how one fixes the spatial gauge throughout the scattering process. We return to these issues in Section~\ref{sub:intro_scattering2}.

\subsection{Scattering for the wave equation in Kasner} \label{sub:intro_scattering1}

We first return to the simpler setting of the wave equation \eqref{eq:wave} on $(\mathcal{M}_{Kas}, g_{Kas})$; our goal in this section will be to state the scattering theorem in (non-degenerate) Kasner spacetimes. The precise scattering theorem will involve the symbol (or pseudodifferential operator) $\mathcal{T}_*$ which we define as follows; see Appendix~\ref{app:fourier} for a review of the Fourier decomposition and symbol calculus on $\mathbb{T}^D$.

\begin{definition} \label{def:tstar}
    Firstly, for $\lambda \in \Z^d \setminus \{0 \}$, let $t_{\lambda*} > 0$ be the unique solution of  
    $\tau_{\lambda}^2 (t_{\lambda*}) \coloneqq \sum_{i=1}^D t_{\lambda*}^{2 - 2p_i} \lambda_i^2 = 1$, and let $t_{0*} = 1$.
    Then set $\mathcal{T}_*$ be the symbol associated (in the sense of Definition~\ref{def:symbolcalc}) to the function $T_*: \Z^D \to \R$ which has $T_*(\lambda) = t_{\lambda*}$ for $\lambda \in \Z^D$. For any exponent $\alpha \in \R$, we also define $\mathcal{T}_*^{\alpha}$ and $\log(\mathcal{T}_*)$ following the standard symbol calculus.
\end{definition}

\begin{remark}
    In the language of \cite[Chapter XVIII]{HormanderPseudo}, we show in Lemma~\ref{lem:tstar} that $\mathcal{T}_*^{\alpha}$ lies in the symbol classes $S^{\max\{0, \alpha (1-p)^{-1}\}}$, where $p = \max_i p_i$, while $\log(\mathcal{T}_*)$ lies in $S^{\varepsilon}$ for any $\varepsilon > 0$. In particular, for suitably chosen $s_1, s_2 \in \R$ the operators $\mathcal{T}_*^{\alpha}$ and $\log(\mathcal{T}_*)$ are bounded operators from $H^{s_1}$ to $H^{s_2}$.
\end{remark}


\subsubsection{The wave equation as a first-order system} \label{sub:intro_scattering1_wave}

It will be advantageous to rewrite the wave equation \eqref{eq:wave} as a first-order (hyperbolic) system of evolution equations for the variables $\phi$ and $\psi \coloneqq t \partial_t \phi$. Multiplying \eqref{eq:wave_coord} by $t^2$, one gets the system:
\begin{gather}
    t \partial_t \phi = \psi, \label{eq:wave_phi_evol} \\[0.5em]
    t \partial_t \psi = \sum_{i=1}^D t^{2 - 2 p_i} \partial_{x^i}^2 \phi. \label{eq:wave_psi_evol}
\end{gather}


Next, recall from Theorem~\ref{thm:asymp_wave} that although $\psi$ remains bounded as one approaches $t = 0$, $\phi$ itself will blow up as $O(\log t^{-1})$. Thus as $t \to 0$, instead of considering $\phi$ itself we consider the renormalized variable $\varphi$, defined as:
\begin{equation} \label{eq:varphi}
    \varphi(t, \cdot) = \phi(t, \cdot) - \log t \cdot \psi(t, \cdot).
\end{equation}
Due to Theorem~\ref{thm:asymp_wave} one expects $\varphi(t, \cdot)$ to remain bounded as $t \to 0$, in fact it will converge to $\varphi_{\infty}$.

\subsubsection{The scattering theorem}

Having set up our system, we now state Theorem~\ref{thm:wave_scat}, our scattering result for the scalar wave equation \eqref{eq:wave} in non-degenerate Kasner spacetimes. The proof of Theorem~\ref{thm:wave_scat} will be in Section~\ref{sec:wavescat}.

\begin{theorem}[Scattering for the wave equation in non-degenerate Kasner spacetimes] \label{thm:wave_scat}
    Let $(\mathcal{M}_{Kas}, g_{Kas})$ be a non-degenerate Kasner spacetime. We consider solutions $\phi: \mathcal{M}_{Kas} \to \R$ to the wave equation \eqref{eq:wave} on $(\mathcal{M}_{Kas}, g_{Kas})$. Then we have the following scattering results:
    \begin{enumerate}[(i)]
        \item \label{item:wave_scatop} (Existence of the scattering operator) 
            Let $s > \frac{D}{2}+1$. Given Cauchy data $(\phi_C, \psi_C) \in H^{s+1} \times H^{s}$, there exists a unique (classical) solution $(\phi(t), \psi(t))$ to the system \eqref{eq:wave_phi_evol}--\eqref{eq:wave_psi_evol} in $\mathcal{M}$ attaining the Cauchy data $\phi(1, \cdot) = \phi_C(\cdot)$ and $\partial_t \phi(1, \cdot)  = \psi_C(\cdot)$. Let $\varphi(t)$ be defined by \eqref{eq:varphi}.

            Then there exist unique functions $\psi_{\infty}(x)$ and $\varphi_{\infty}(x)$ in $H^s$ such that
            \begin{equation} \label{eq:asymp_psivarphi}
                \psi(t, x) \to \psi_{\infty}(x) \text{ in } H^s, \qquad \varphi(t, x) \to \varphi_{\infty}(x) \text{ in } H^s \quad \text{ strongly as } t \to 0.
            \end{equation}
            Furthermore, one has that
            \begin{equation} \label{eq:asymp_reg}
                \psi_{\infty} \in H^{s + \frac{1}{2}}, \qquad \varphi_{\infty} + \log(\mathcal{T}_*) \psi_{\infty} \in H^{s + \frac{1}{2}}.
            \end{equation}
            
        \item \label{item:wave_asympcomp} (Asymptotic completeness) 
            Let $s > \frac{D}{2} + 1$, and let $\psi_{\infty}$ and $\varphi_{\infty}$ be functions in $\mathbb{T}^D$ with Sobolev regularity as in \eqref{eq:asymp_reg}. Then there exists a unique (classical) solution $(\phi(t), \psi(t))$ to the system \eqref{eq:wave_phi_evol}--\eqref{eq:wave_psi_evol} in $\mathcal{M}_{Kas}$, with regularity
            \[
                \left( \phi, \psi \right) \in C^0((0, + \infty), H^{s+1} \times H^s) \cap C^1((0, + \infty), H^s \times H^{s-1}),
            \]
            such that if $\varphi(t)$ is defined as in \eqref{eq:varphi}, then the asymptotics \eqref{eq:asymp_psivarphi} hold, strongly as $t \to 0$.

        \item \label{item:wave_scatiso} (Scattering isomorphism)
            For $s \in \R$, define the following Hilbert spaces together with their norms:
            \begin{gather} \label{eq:wave_hilbert1}
                \mathcal{H}^s_C = \{ \left( \phi, \psi \right): \phi \in H^{s+1}, \psi \in H^s \}, \quad \| \left( \phi, \psi \right) \|_{\mathcal{H}^s_C}^2 \coloneqq \| \phi \|_{H^{s+1}}^2 + \| \psi \|_{H^s}^2, 
                \\[0.5em] \label{eq:wave_hilbert2}
                \mathcal{H}^s_{\infty} = \{ \left( \psi, \varphi \right): \psi \in H^{s+\frac{1}{2}}, \varphi + \log(\mathcal{T}_*) \psi \in H^{s + \frac{1}{2}} \}, \quad \| \left( \psi, \varphi \right) \|_{\mathcal{H}^s_C}^2 \coloneqq \| \psi \|_{H^{s+\frac{1}{2}}}^2 + \| \varphi + \log(\mathcal{T}_*) \psi \|_{H^{s+ \frac{1}{2}}}^2.
            \end{gather}
            Then the map $\mathcal{S}_{\downarrow}$, which is defined to map $(\phi_C, \psi_C)$ from (\ref{item:wave_scatop}) to $(\psi_{\infty}, \varphi_{\infty})$ in \eqref{eq:asymp_psivarphi}, may be extended to a Hilbert space isomorphism $\mathcal{S}_{\downarrow}: \mathcal{H}^s_C \to \mathcal{H}^s_{\infty}$, whose inverse $\mathcal{S}_{\uparrow} = \mathcal{S}_{\downarrow}^{-1}$ is exactly given by (a suitable extension of) the map taking $(\psi_{\infty}, \varphi_{\infty})$ in (\ref{item:wave_asympcomp}) to $\left( \phi(1, \cdot), \psi(1, \cdot) \right)$.
    \end{enumerate}
\end{theorem}

\begin{remark}
    A remarkable feature of Theorem~\ref{thm:wave_scat} is the gain or loss of half a derivative from Cauchy data to scattering data. In particular, the renormalized time derivative $\psi = t \partial_t \phi$ lies in $H^s$ at any fixed time $t > 0$, but its limit $\psi_{\infty}$ at $t = 0$ has its regularity increased to $H^{s + \frac{1}{2}}$. This as a sign that spatial derivatives not only become negligible compared to temporal derivatives near $t = 0$, but actually become better behaved as one approaches the singular boundary at $t = 0$.

    Note than an almost-sharp correspondence between Cauchy data and asymptotic data was already found by Ringstr\"om in \cite[Propositions 8.1 and 8.9]{RingstromAsterisque}, however in this work the derivative gain is from $H^s$ to $H^{s + \frac{1}{2}-\varepsilon}$ rather than the sharp $H^{s + \frac{1}{2}}$. In order to achieve this sharp result, it is essential that we modify our Hilbert spaces using the symbol $\log(\mathcal{T}_*)$, a feature not present in \cite{RingstromAsterisque}. This final modification was partly inspired by a scattering result for the wave equation in higher-dimensional de Sitter spacetimes \cite{CicortasScattering}.
\end{remark}

\begin{remark}
    One could also entertain scattering between Cauchy data at $t = 1$ and a suitable notion of asymptotic data at $t = + \infty$. The issue is that although one may obtain upper bounds for $\phi(t)$ and $\psi(t)$ as $t \to +\infty$, it is difficult to upgrade these to asymptotics since both $\phi(t)$ and $\psi(t)$ oscillate at late times. Moreover different Fourier modes will oscillate at different frequencies, see e.g.~the Bessel function type behaviour suggested in Section~\ref{sub:intro_proof} or the partial results of \cite{PetersenMode}.

    Another reason we do not study $t \to + \infty$ scattering is that the Kasner spacetimes are expected to be unstable in the $t \to + \infty$ direction as solutions to the Einstein--scalar field system, see \cite{KofmanUzanPitrou}. One could conjecture that generic perturbations lead to either isotropic expansion or recollapse into a singularity. This is related to the ``cosmic no-hair conjecture'' in the presence of a positive cosmological constant, see \cite{AndreassonRingstrom, BrauerRendallReula}. 
\end{remark}

\subsection{Scattering for linearized gravity in Kasner} \label{sub:intro_scattering2}

The second goal of the article will be to describe and prove a similar scattering theory for the linearized Einstein--scalar field system \eqref{eq:einstein}--\eqref{eq:energymomentum}, where the dynamical variables $(\matr{\upeta}{i}{j}, \matr{\upkappa}{i}{j}, \upphi, \uppsi)$ are linearized perturbations around the background Kasner solution \eqref{eq:kasner} and \eqref{eq:kasner_phi}. We first introduce more precisely the linearized system, and formulate it as a first-order system in the sense of Section~\ref{sub:intro_scattering1_wave}.

\subsubsection{Linearized Einstein--scalar field as a first-order system} \label{sub:intro_scattering2_eq}

As stated previously, we employ a (linearized) CMC ADM-type gauge. More precisely, our scattering theorem, Theorem~\ref{thm:einstein_scat}, is with respect to a (linearized) \emph{Constant Mean Curvature Transported Coordinate} (CMCTC) gauge. The nonlinear version of this gauge has coordinates $(t, x^i) \in (0, T) \times \mathbb{T}^D$ and a Lorentzian metric
\begin{equation} \label{eq:adm1}
    g = - n^2 dt^2 + \bar{g}_{ij} dx^i dx^j.
\end{equation}

Here $n: \mathcal{M} \to \R^+$ is known as the \textit{lapse}, while $\bar{g}_{ij} = \bar{g}_{ij}(t)$ are a time-dependent family of Riemannian metrics on $\mathbb{T}^D$.
Our time function $t$, and thus the lapse, is fixed since we choose our foliation by hypersurfaces $\Sigma_t = \{t\} \times \mathbb{T}^D$ to be one of constant mean curvature equal to $\tr k |_{\Sigma_t} = - t^{-1}$. In this gauge
the Einstein--scalar field equations \eqref{eq:einstein}--\eqref{eq:energymomentum} may be written as a system of second-order evolution equations for the $\bar{g}_{ij}$ and $\phi$, coupled to an elliptic equation for the lapse $n$ that arises from the CMC condition. There are also the usual Hamiltonian and momentum constraints.

This system of evolution equations plus elliptic equations plus constraints can be rewritten as a first-order system of evolution equations for the variables $(\bar{g}_{ij}, k_{ij}, \phi, \mathbf{N} \phi)$, where $k_{ij}$ is the second fundamental form and $\mathbf{N} \phi$ is the derivative of $\phi$ normal to the foliation, coupled also to an elliptic equation for $n$ and with $(\bar{g}_{ij}, k_{ij}, \phi, \mathbf{N}\phi, n)$ satisfying several constraints. The resulting system will be written in full detail in Section~\ref{sub:adm_adm}.

For now, we first present the first-order system for the \emph{linearized} Einstein--scalar field system, where the variables are linearized perturbations of a fixed generalized Kasner solution of the form \eqref{eq:kasner}, \eqref{eq:kasner_phi} satisfying the Kasner relations \eqref{eq:kasner_relations}. Quantities associated to this background Kasner spacetime $(\mathcal{M}_{Kas}, g_{Kas}, \phi_{Kas})$ are denoted with an overhead circle, for instance\footnote{Here, and in the sequel, underlined indices mean that there is no summation. Also, we often write $\mathring{g}_{ij}$ in place of $\mathring{\bar{g}}_{ij}$.}
\[
    \mathring{g}_{ij} = \mathring{\bar{g}}_{ij} = t^{2 p_{\underline{i}}} \delta_{\underline{i}j}, \quad \mathring{n} = 1, \quad \mathring{\phi} = p_{\phi} \log t, \quad \text{ etc.}
\]

The variables in our linearized Einstein--scalar field system are chosen to be the following (see already Definition~\ref{def:linsmall} in Section~\ref{sub:lineinstein}):
\begin{gather}
    \matr{\upeta}{i}{j} \coloneqq - \frac{1}{2} (\mathring{g}^{-1})^{jk} \left( g_{ik} - \mathring{g}_{ik} \right), \quad
    \matr{\upkappa}{i}{j} \coloneqq t (\bar{g}^{-1})^{jk} k_{ik} - t ( \mathring{\bar{g}}^{-1} )^{jk} \mathring{k}_{ik}, 
    \label{eq:hkappa} \\[0.5em]
    \upnu \coloneqq n - \mathring{n} = n - 1, \label{eq:upnu} \\[0.5em]
    \upphi \coloneqq \phi - \mathring{\phi}, \quad \uppsi \coloneqq t \left( \mathbf{N} \phi - \partial_t \mathring{\phi} \right) \label{eq:upphipsi}. 
\end{gather}
The linearized Einstein--scalar field system \eqref{eq:einsteinlinear}--\eqref{eq:wavelinear}, with respect to these variables, will be derived in Proposition~\ref{prop:adm_linear}, and are roughly categorised as follows. Firstly there are the evolution equations for the linearized metric and matter variables $(\matr{\upeta}{i}{j}, \matr{\upkappa}{i}{j}, \upphi, \uppsi )$:

\begin{equation} \label{eq:upeta_evol_intro}
    t \partial_t \matr{\upeta}{i}{j} = \matr{\upkappa}{i}{j} + (2t \matr{\mathring{k}}{p}{j}) \matr{\upeta}{i}{p} - (2 t \matr{\mathring{k}}{i}{p}) \matr{\upeta}{p}{j} + (t \matr{\mathring{k}}{i}{j}) \upnu, 
\end{equation}
\begin{equation} \label{eq:upkappa_evol_intro}
    t \partial_t \matr{\upkappa}{i}{j} = t^2 \mathring{g}^{ab} \partial_a \partial_b \matr{\upeta}{i}{j} + t^2 \mathring{g}^{il} \partial_j \partial_l (\tr \upeta) - t^2 \mathring{g}^{ab} \partial_i \partial_b \matr{\upeta}{a}{j} - t^2 \mathring{g}^{aj} \partial_a \partial_b \matr{\upeta}{i}{b} - t^2 \mathring{g}^{jk} \partial_i \partial_k \upnu - ( t \matr{\mathring{k}}{i}{j} ) \upnu,
\end{equation}
\begin{equation} \label{eq:upphi_evol_intro}
    t \partial_t \upphi = \uppsi + p_{\phi} \upnu,
\end{equation}
\begin{equation} \label{eq:uppsi_evol_intro} 
    t \partial_t \uppsi = t^2 \mathring{g}^{ab} \partial_a \partial_b \upphi - p_{\phi} \upnu. 
\end{equation}

Next, we have the elliptic equation for $\upnu$, 
\begin{equation} \label{eq:upnu_elliptic_intro}
    \left( 1 - t^2 \mathring{g}^{ab} \partial_a \partial_b \right) \upnu = - 2 t^2 \mathring{g}^{ab} \partial_a \partial_b (\tr \upeta) + 2 t^2 \mathring{g}^{ab} \partial_b \partial_c \matr{\upeta}{a}{c}.
\end{equation}
Finally, we have the constraint equations and gauge conditions, which are propagated by \eqref{eq:upeta_evol_intro}--\eqref{eq:upnu_elliptic_intro}. To keep the exposition brief, we do not write these here.

\begin{remark}
    The system \eqref{eq:upeta_evol_intro}--\eqref{eq:upnu_elliptic_intro} is written in a tensorial form in order to apply tensorial estimates, and in the hope that one could generalize to spatial topologies other than $\mathbb{T}^D$. The downside is that the equations appear a little mysterious, due to terms featuring $t \matr{\mathring{k}}{i}{j}$. To demystify, note that in the usual coordinates $t \matr{\mathring{k}}{i}{j} = - p_{\underline{i}} \matr{\delta}{\underline{i}}{j}$, where $p_i$ are the Kasner exponents. 
    
    For instance, \eqref{eq:upeta_evol_intro} may be rewritten as:
    \[
        t \partial_t \matr{\upeta}{i}{j} = \matr{\upkappa}{i}{j} + (2 p_{\underline{i}} - 2 p_{\underline{j}}) \matr{\upeta}{\underline{i}}{\underline{j}} - p_{\underline{i}} \matr{\delta}{\underline{i}}{j} \upnu. 
    \]
    In particular, even ignoring the $\matr{\upkappa}{i}{j}$ term, one expects a term of order $t^{2p_i - 2p_j}$ in the asymptotics of $\matr{\upeta}{i}{j}$. 
\end{remark}

\subsubsection{Renormalized variables} \label{sub:intro_scattering2_renorm}

In light of this remark, not all components of $\matr{\upeta}{i}{j}$ will generically remain bounded as $t \to 0$, hence $\matr{\upeta}{i}{j}$ is not a suitable choice of asymptotic data at $t = 0$. This is true also for $\upphi$, which is $O(\log t^{-1})$ near $t = 0$. 
As a result, we make the following definitions which do have suitable limits as $t \to 0$:
\begin{definition} \label{def:upupsilonupvarphi}
    For a solution $(\matr{\upeta}{i}{j}, \matr{\upkappa}{i}{j}, \upphi, \uppsi)$ to the linearized Einstein--scalar field equations \eqref{eq:upeta_evol_intro}--\eqref{eq:upnu_elliptic_intro}, we define the renormalized quantities $\matr{\Upupsilon}{i}{j}$ and $\upvarphi$ as follows:
    \begin{equation} \label{eq:upupsilon}
        \matr{\Upupsilon}{i}{j} \coloneqq \matr{\upeta}{i}{j} + \int^1_t \mathring{g}_{ip}(s) \mathring{g}^{jq}(s) \frac{ds}{s} \cdot \matr{\upkappa}{q}{p}
        = \begin{cases}
            \matr{\upeta}{i}{j} - \matr{\upkappa}{j}{i} \log t, & \text{ if } p_i = p_j \\[0.6em]
            \matr{\upeta}{i}{j} - \matr{\upkappa}{\underline{j}}{\underline{i}} \left( \displaystyle{\frac{t^{2p_{\underline{i}} - 2 p_{\underline{j}}} - 1}{2 p_{\underline{i}} - 2 p_{\underline{j}}}} \right)  & \text{ if } p_i \neq p_j.
        \end{cases} 
    \end{equation}
    \begin{equation} \label{eq:upvarphi}
        \upvarphi \coloneqq \upphi - \uppsi \log t.
    \end{equation}
\end{definition}


\begin{remark}
    Though \eqref{eq:upupsilon} and \eqref{eq:upvarphi} are convenient in that they agree with the natural Cauchy data $\matr{\upeta}{i}{j}$ and $\upphi$ at $t=1$, it will be useful to consider more general versions where $t = 1$ is replaced by $t = T$, i.e.~
    \begin{equation} \label{eq:upupsilonupvarphi_T}
        \matr{\tilde{\Upupsilon}}{i}{j} = \matr{\upeta}{i}{j} + \int^T_t \mathring{g}_{ip}(s) \mathring{g}^{jq}(s) \frac{ds}{s} \cdot \matr{\upkappa}{q}{p}, \qquad
        \tilde{\upvarphi} = \upphi - \uppsi \log( \frac{t}{T} ).
    \end{equation}
    In our energy estimates, the natural choice of $T$ turns out to be the frequency dependent $T = t_{\lambda *}$, with $t_{\lambda*}$ as in Definition~\ref{def:tstar}. This indicates that a sharp scattering statement requires the use of the symbol $\mathcal{T}_*$, as we will see in Theorem~\ref{thm:einstein_scat}.
\end{remark}

\subsubsection{Infinitesimal gauge transformations} \label{sub:intro_scattering2_gauge}

Before proceeding to state the scattering theorem for the linearized Einstein--scalar field system \eqref{eq:upeta_evol_intro}--\eqref{eq:upnu_elliptic_intro}, we make one further observation regarding the choice of spatial gauge. The observation is that though our time function $t$ is fixed by the CMC condition, there is remaining gauge freedom in the ADM-type gauge \eqref{eq:adm1}, due to diffeomorphisms from $\mathbb{T}^D$ to itself.

At the infinitesimal level, such diffeomorphisms are associated to a vector field $\upxi^i$ on $\mathbb{T}^D$, since such vector fields generate a one-parameter family of diffeomorphisms via their flow. Such a vector field (which we often denote a \emph{change of gauge}) causes the linearized quantities $\matr{\upeta}{i}{j}$ and $\matr{\upkappa}{i}{j}$ to transform as follows:
\begin{gather} \label{eq:upeta_gauge}
    \matr{\upeta}{i}{j} \mapsto \matr{\upeta}{i}{j} + \frac{1}{2} \partial_i \upxi^j + \frac{1}{2} \mathring{g}_{ip} \mathring{g}^{jq} \partial_q \upxi^p,
    \\[0.4em] \label{eq:upkappa_gauge}
    \matr{\upkappa}{i}{j} \mapsto \matr{\upkappa}{i}{j} + (t \matr{\mathring{k}}{i}{p}) \partial_p \upxi^j - (t \matr{\mathring{k}}{p}{j}) \partial_i \upxi^p.
\end{gather}
The renormalized quantity $\matr{\Upupsilon}{i}{j}$ defined in \eqref{eq:upupsilon} transforms as follows:
\begin{equation} \label{eq:upupsilon_gauge}
    \matr{\Upupsilon}{i}{j} \mapsto \matr{\Upupsilon}{i}{j} + \frac{1}{2} \partial_i \upxi^j + \frac{1}{2} \mathring{g}_{ip}(1) \mathring{g}^{jq}(1) \partial_q \upxi^p.
\end{equation}

For now, let $\upxi^i$ be a regular vector field on $\mathbb{T}^D$ that does not vary in time. Under the change of gauge \eqref{eq:upeta_gauge}--\eqref{eq:upupsilon_gauge}, the system of equations \eqref{eq:upeta_evol_intro}--\eqref{eq:upnu_elliptic_intro} is invariant. (Such gauge invariance holds in a broader context in Section~\ref{sub:adm_gauge}, where we allow the vector field $\upxi^i$ to also depend on time.)

A complete scattering theory for the linearized Einstein--scalar field system must be armed with a canonical choice of gauge. It turns out that the choice of gauge that leads to sharp estimates will be different for Cauchy data at $t = 1$ and for asymptotic data at $t=0$, as we explain now.

Our canonical gauge for Cauchy data is inherited from the requirement that the coordinates $x^i$ are harmonic with respect to the Riemannian metric $\bar{g}_{ij}(1)$ at $t = 1$. Upon linearization, this translates to the following condition for $\matr{\upeta}{i}{j}$:
\begin{equation} \label{eq:spatiallyharmonic}
    2 \partial_j \matr{\upeta}{i}{j} = \partial_i ( \tr \upeta ) \quad \text{ at } t = 1.
\end{equation}
This is related to the linearized \emph{Constant Mean Curvature Spatially Harmonic} (CMCSH) gauge that appears in Section~\ref{sub:lineinstein}, the difference being that in CMCSH gauge, \eqref{eq:spatiallyharmonic} is instead true at all times $t > 0$ (and for that we must allow $\upxi^j$ to be time-dependent).
It will be shown in Lemma~\ref{lem:diffeo} that any Cauchy data for the linearized Einstein--scalar field system can be transformed into this gauge. 

One may now understand one direction of the scattering problem as follows. Firstly, one prescribes Cauchy data for $(\matr{\upeta}{i}{j}, \matr{\upkappa}{i}{j}, \upphi, \uppsi)$ at $t = 1$ which satisfies suitable constraints as well as the spatially harmonic condition \eqref{eq:spatiallyharmonic}.
One then evolves this data using \eqref{eq:upeta_evol_intro}--\eqref{eq:upnu_elliptic_intro}, as a regular solution in $\mathcal{M}_{Kas}$. 
For (i) the existence of scattering states, we show that the solution obeys certain asymptotics as $t \to 0$, namely that the renormalized quantities $(\matr{\upkappa}{i}{j}, \matr{\Upupsilon}{i}{j}, \uppsi, \upvarphi)$ attain limits as $t \to 0$:
\[
    (\matr{\upkappa}{i}{j}, \matr{\Upupsilon}{i}{j}, \uppsi, \upvarphi) \to (\matr{(\upkappa_{\infty})}{i}{j}, \matr{(\Upupsilon_{\infty})}{i}{j}, \uppsi_{\infty}, \upvarphi_{\infty}).
\]

The opposite direction involves prescribing the asymptotic data $(\matr{(\upkappa_{\infty})}{i}{j}, \matr{(\Upupsilon_{\infty})}{i}{j}, \uppsi_{\infty}, \upvarphi_{\infty})$ at $t = 0$. These also satisfy suitable constraint equations, derived the usual constraint equations, see Lemma~\ref{lem:scattering_constraints}. More importantly, one must impose a canonical gauge for the asymptotic data. In light of our energy estimates, we make a slightly peculiar choice, which will be the following identity involving\footnote{We make sense of the integral in \eqref{eq:asymptoticallyharmonic} in terms of the symbol calculus; indeed upon evaluation of the integral it is just a combination of constants and $\mathcal{T}_*^{2p_i - 2p_j}$ or $\log(\mathcal{T}_*)$.} the symbol $\mathcal{T}_*$.
\begin{equation} \label{eq:asymptoticallyharmonic}
    2 \partial_j \left( \matr{(\Upupsilon_{\infty})}{i}{j} + \int^1_{\mathcal{T}_*} \mathring{g}_{ip}(s) \mathring{g}^{jq}(s) \frac{ds}{s} \cdot \matr{(\upkappa_{\infty})}{q}{p} \right) = 2 \partial_i \left( \tr \Upupsilon_{\infty} \right).
\end{equation}
We later show that we can always transform suitable asymptotic data into this gauge, which we often call the \emph{canonical asymptotic gauge condition}. 

\sloppy For (ii) asymptotic completeness, we show that for any sufficiently regular asymptotic data $(\matr{(\upkappa_{\infty})}{i}{j}, \matr{(\Upupsilon_{\infty})}{i}{j}, \uppsi_{\infty}, \upvarphi_{\infty})$ satisfying the constraints as well as the gauge condition \eqref{eq:asymptoticallyharmonic}, there exists a unique regular solution $(\matr{\upeta}{i}{j}, \matr{\upkappa}{i}{j}, \upphi, \uppsi)$ to the system \eqref{eq:upeta_evol_intro}--\eqref{eq:upnu_elliptic_intro} in the whole spacetime $\mathcal{M}_{Kas}$, which achieves the prescribed asymptotic data at $t = 0$.

As warned, the gauge choice corresponding to linearly spatially harmonic data at $t=1$ (i.e.~data satisfying \eqref{eq:spatiallyharmonic}), will not necessarily match in evolution with asymptotic data which satisfies the canonical asymptotic gauge condition \eqref{eq:asymptoticallyharmonic}, and vice versa. Therefore, in the scattering theory we allow for a change of gauge vector field $\upxi^j$ that moves from one gauge to the other.
%
%
See Section~\ref{sub:adm_gauge} for further comments regarding the spatial gauge.

\subsubsection{The scattering theorem} \label{sub:intro_scattering2_thm}

We now state a slightly imprecise, but still rather detailed, version of the scattering theorem for the linearized Einstein--scalar field system around a Kasner spacetime obeying the subcriticality relation \eqref{eq:subcritical}. We prove Theorem~\ref{thm:einstein_scat} in Section~\ref{sec:einstein_scat}. The proof will actually be via a scattering theorem in another gauge, Theorem~\ref{thm:einstein_scat_v2}.

\begin{theorem}[Scattering for the linearized Einstein--scalar field system in subcritical Kasner spacetimes] \label{thm:einstein_scat}
    Let $(\mathcal{M}_{Kas}, g_{Kas}, \phi_{Kas})$ be a (generalized) Kasner spacetime whose exponents satisfy the subcriticality condition \eqref{eq:subcritical}. Consider solutions $(\matr{\upeta}{i}{j}, \matr{\upkappa}{i}{j}, \upphi, \uppsi)$ to the linearized Einstein--scalar field equations around $(\mathcal{M}_{Kas},g_{Kas}, \phi_{Kas})$, written as the first-order elliptic-hyperbolic system \eqref{eq:upeta_evol_intro}--\eqref{eq:upnu_elliptic_intro} and satisfying suitable constraints. The scattering theory is as follows:
    \begin{enumerate}[(i)]
        \item \label{item:einstein_scatop} (Existence of the scattering operator) 
            Let $s$ be sufficiently large. Then given Cauchy data \linebreak$(\matr{(\upeta_C)}{i}{j}, \matr{(\upkappa_C)}{i}{j}, \upphi_C, \uppsi_C) \in H^{s+1} \times H^{s} \times H^{s+1} \times H^s$ at $t=1$ satisfying suitable constraints as well as the spatially harmonic gauge condition \eqref{eq:spatiallyharmonic}, there exists a unique (classical) solution $(\matr{\upeta}{i}{j}, \matr{\upkappa}{i}{j}, \upphi, \uppsi)$ to the system \eqref{eq:upeta_evol_intro}--\eqref{eq:upnu_elliptic_intro} in $\mathcal{M}_{Kas}$ attaining the Cauchy data 
            \[
                \matr{\upeta}{i}{j}(1, \cdot) = \matr{(\upeta_C)}{i}{j}(\cdot), \quad
                \matr{\upkappa}{i}{j}(1, \cdot) = \matr{(\upkappa_C)}{i}{j}(\cdot), \quad
                \upphi(1, \cdot) = \upphi_C (\cdot), \quad
                \uppsi(1, \cdot) = \uppsi_C (\cdot).
            \]

            Let $\matr{\Upupsilon}{i}{j}$ and $\upvarphi$ be as in \eqref{eq:upupsilon}--\eqref{eq:upvarphi}. 
            Then there exist unique $(1,1)$ tensors $\matr{(\upkappa_{\infty})}{i}{j}(x)$ and $\matr{(\Upupsilon_{\infty})}{i}{j}(x)$ and unique functions $\uppsi_{\infty}(x)$ and $\upvarphi_{\infty}(x)$ such that
            \begin{equation} \label{eq:asymp_upkappaupupsilon}
                \matr{\upkappa}{i}{j}(t, \cdot) \to \matr{(\upkappa_{\infty})}{i}{j}(\cdot) \text{ in } C^3, \qquad \matr{\Upupsilon}{i}{j}(t, \cdot) \to \matr{(\Upupsilon_{\infty})}{i}{j}(\cdot) \text{ in } C^3 \quad \text{ strongly as } t \to 0,
            \end{equation}
            \begin{equation} \label{eq:asymp_uppsiupvarphi}
                \uppsi(t, \cdot) \to \uppsi_{\infty}(\cdot) \text{ in } H^s, \qquad \upvarphi(t, \cdot) \to \upvarphi_{\infty}(\cdot) \text{ in } H^s \quad \text{ strongly as } t \to 0.
            \end{equation}
            Furthermore, there exists a change of gauge vector field $\upxi_{\infty}^j \in H^{s+\frac{1}{2}}$ such that after applying the transformations \eqref{eq:upeta_gauge}--\eqref{eq:upupsilon_gauge}, the transformed versions of $(\matr{(\upkappa_{\infty})}{i}{j}, \matr{(\Upupsilon_{\infty})}{i}{j}, \uppsi_{\infty}, \upvarphi_{\infty})$ satisfy suitable asymptotic constraints as well as the canonical asymptotic gauge condition \eqref{eq:asymptoticallyharmonic}, and have the following regularity:
            \begin{equation} \label{eq:asymp_metric_reg}
                \mathcal{T}_*^{- p_{\underline{i}} + p_{\underline{j}}} \matr{(\upkappa_{\infty})}{\underline{i}}{\underline{j}} \in H^{s + \frac{1}{2}},  \qquad
                \mathcal{T}_*^{ - p_{\underline{i}} + p_{\underline{j}}} \left( \matr{(\Upupsilon_{\infty})}{\underline{i}}{\underline{j}} + \int^1_{\mathcal{T}_*} \mathring{g}_{\underline{i}p}(s) \mathring{g}^{\underline{j}q}(s) \frac{ds}{s} \cdot \matr{(\upkappa_{\infty})}{q}{p} \right) \in H^{s + \frac{1}{2}},
            \end{equation}
            \begin{equation} \label{eq:asymp_scalar_reg}
                \uppsi_{\infty} \in H^{s + \frac{1}{2}},
                \qquad \upvarphi_{\infty} + \log(\mathcal{T}_*) \uppsi_{\infty} \in H^{s + \frac{1}{2}}.
            \end{equation}
            
        \item \label{item:einstein_asympcomp} (Asymptotic completeness) 
            Let $s$ be sufficiently large, and let $\matr{(\upkappa_{\infty})}{i}{j}$, $\matr{(\Upupsilon_{\infty})}{i}{j}$ be $(1, 1)$ tensors in $\mathbb{T}^D$, and $\uppsi_{\infty}$, $\upvarphi_{\infty}$ be functions in $\mathbb{T}^D$ with Sobolev regularity as in \eqref{eq:asymp_metric_reg}--\eqref{eq:asymp_scalar_reg}, and satisfying suitable constraints as well as the canonical asymptotic gauge condition \eqref{eq:asymptoticallyharmonic}. Then there exists a unique (classical) solution $(\matr{\upeta}{i}{j}, \matr{\upkappa}{i}{j}, \upphi, \uppsi)$ to the system \eqref{eq:upeta_evol_intro}--\eqref{eq:upnu_elliptic_intro} in $\mathcal{M}_{Kas}$, with regularity
            \[
                \left( \matr{\upeta}{i}{j}, \matr{\upkappa}{i}{j} \right) \in C^0((0, + \infty), H^{s+1} \times H^s) \cap C^1((0, + \infty), H^s \times H^{s-1}),
            \]
            \[
                \left( \upphi, \uppsi \right) \in C^0((0, + \infty), H^{s+1} \times H^s) \cap C^1((0, + \infty), H^s \times H^{s-1}),
            \]
            as well as a change of gauge vector field $\upxi^j \in H^{s-1}$, such that:
            \begin{enumerate}[1.]
                \item
                    The quantity $\matr{\upeta}{i}{j}$ satisfies the spatially harmonic condiion \eqref{eq:spatiallyharmonic} at $t = 1$.
                \item
                    If $\matr{\Upupsilon}{i}{j}$ and $\upvarphi$ are defined as in \eqref{eq:upupsilon}--\eqref{eq:upvarphi}, then after applying the transformation \eqref{eq:upeta_gauge}--\eqref{eq:upupsilon_gauge} associated to the vector field $\upxi^j$, the resulting transformed quantities obey the asymptotics \eqref{eq:asymp_upkappaupupsilon} and \eqref{eq:asymp_uppsiupvarphi} hold, strongly as $t \to 0$.
            \end{enumerate}

        \item \label{item:einstein_scatiso} (Scattering isomorphism)
            For $s \in \R$, define the following Hilbert spaces together with their norms:
            \begin{equation} 
                \mathcal{H}^s_C = \left \{ \left( \matr{\upeta}{i}{j}, \matr{\upkappa}{i}{j}, \upphi, \uppsi \right): \matr{\upeta}{i}{j} \in H^{s+1}, \matr{\upkappa}{i}{j} \in H^s, \upphi \in H^{s+1}, \uppsi \in H^s \right \},
            \end{equation}
            \begin{equation}
                \left \| \left( \matr{\upeta}{i}{j}, \matr{\upkappa}{i}{j}, \upphi, \uppsi \right) \right \|_{\mathcal{H}^s_C}^2 \coloneqq \sum_{i, j = 1}^D  \left( \| \matr{\upeta}{i}{j} \|_{H^{s+1}}^2 + \| \matr{\upkappa}{i}{j} \|_{H^s}^2 \right) + \| \upphi \|_{H^{s+1}}^2 + \| \uppsi \|_{H^s}^2,
            \end{equation}
            \begin{multline}
                \mathcal{H}^s_{\infty} = \left \{ \left( \matr{\upkappa}{i}{j}, \matr{\Upupsilon}{i}{j}, \uppsi, \upvarphi \right): \mathcal{T}_*^{-p_i + p_j} \matr{\upkappa}{i}{j} \in H^{s + \frac{1}{2}}, \mathcal{T}_*^{-p_i + p_j} \left( \matr{\Upupsilon}{i}{j} - \int^1_{\mathcal{T}_*} \mathring{g}_{ip}(s) \mathring{g}^{jq}(s) \frac{ds}{s} \cdot \matr{\upkappa}{q}{p}\right) \in H^{s+ \frac{1}{2}} \right . , \\[0.3em] 
                \left . \uppsi \in H^{s+\frac{1}{2}}, \upvarphi + \log(\mathcal{T}_*) \uppsi \in H^{s + \frac{1}{2}} \right \},
            \end{multline}
            \begin{multline}
                \left \| \left( \matr{\upkappa}{i}{j}, \matr{\Upupsilon}{i}{j}, \uppsi, \upvarphi \right) \right \|_{\mathcal{H}^s_C}^2 \coloneqq \sum_{i, j = 1}^D \left( \left \| \mathcal{T}_*^{-p_i  + p_j} \matr{\upkappa}{i}{j} \right \|_{H^{s + \frac{1}{2}}} + \left \| \mathcal{T}_*^{-p_i + p_j} \left( \matr{\Upupsilon}{i}{j} - \int^1_{\mathcal{T}_*} \mathring{g}_{ip}(s) \mathring{g}^{jq} (s) \frac{ds}{s} \cdot \matr{\upkappa}{q}{p} \right) \right \|_{H^{s+\frac{1}{2}}} \right) \\[0.3em] 
                + \| \uppsi \|_{H^{s+\frac{1}{2}}}^2 + \| \upvarphi + \log(\mathcal{T}_*) \uppsi \|_{H^{s+ \frac{1}{2}}}^2.
            \end{multline}

            Further let $\mathcal{H}_{C, c}^s$ be the subspace of $\mathcal{H}_C^s$ obeying suitable constraints as well as the spatially harmonic gauge condition \eqref{eq:spatiallyharmonic}, and similarly let $\mathcal{H}_{\infty,c}^s$ be the subspace of $\mathcal{H}_{\infty}^s$ obeying the appropriate asymptotic constraints as well as the canonical asymptotic gauge condition \eqref{eq:asymptoticallyharmonic}.

            From (\ref{item:einstein_scatop}), let the operator $\mathcal{S}_{\downarrow}$ map the Cauchy data $(\matr{(\upeta_C)}{i}{j}, \matr{(\upkappa_C)}{i}{j}, \upphi_C, \uppsi_C)$ to the asymptotic quantities $(\matr{(\upkappa_{\infty})}{i}{j}, \matr{(\Upupsilon_{\infty})}{i}{j}, \uppsi_{\infty}, \upvarphi_{\infty})$, where the asymptotic quantities are taken \underline{after} modification by the gauge transformation associated to the vector field $\upxi^j$ in (\ref{item:einstein_scatop}). Then this map can be extended to a bounded, injective homomorphism from $\mathcal{H}_{C, c}^s$ to $\mathcal{H}_{\infty,c}^s$.

            Conversely, from (\ref{item:einstein_asympcomp}), let the operator $\mathcal{S}_{\uparrow}$ map the asymptotic data $(\matr{(\upkappa_{\infty})}{i}{j}, \matr{(\Upupsilon_{\infty})}{i}{j}, \uppsi_{\infty}, \upvarphi_{\infty})$ satisfying \eqref{eq:asymptoticallyharmonic} to $(\matr{\upeta}{i}{j}(1, \cdot), \matr{\upkappa}{i}{j}(1, \cdot), \upphi(1, \cdot), \uppsi(1, \cdot))$, noting that $\matr{\upeta}{i}{j}(1, \cdot)$ is chosen to satisfy \eqref{eq:spatiallyharmonic} after modification by the change of gauge vector field $\xi^j$. Then this map $\mathcal{S}_{\uparrow}$ may be extended to a bounded, injective homomorphism from $\mathcal{H}_{\infty, c}^s$ to $\mathcal{H}_{C, c}^s$.

            Finally, $\mathcal{H}_{\infty, c}^s$ is the image of $\mathcal{S}_{\downarrow}$ acting on $\mathcal{H}_{C, c}^s$, and we may view $\mathcal{S}_{\downarrow}: \mathcal{H}_{C, c}^s \to \mathcal{H}_{\infty, c}^s$ as a Hilbert space isomorphism with inverse $\mathcal{S}_{\uparrow}$.
    \end{enumerate}
\end{theorem}

\begin{remark}
    Aside from the $H^s \leftrightarrow H^{s+\frac{1}{2}}$ phenomenon already observed in Theorem~\ref{thm:wave_scat}, there is a new phenomenon that arises in the scattering theory upon introduction of the $(1, 1)$-tensors $\matr{\upeta}{i}{j}$, $\matr{\upkappa}{i}{j}$ and $\matr{\Upupsilon}{i}{j}$. This is that due to the symbol $\mathcal{T}_*^{-p_i + p_j}$, the anisotropy of the Kasner spacetime also affects the derivative gain and loss incurred in the scattering process, particularly in \emph{off-diagonal} components of the tensors $\matr{\upkappa}{i}{j}$ and $\matr{\Upupsilon}{i}{j}$.

    In fact, as one varies across the whole range of Kasner exponents satisfying the subcriticality condition \eqref{eq:subcritical}, the symbol $\mathcal{T}_*^{-p_i + p_j}$ can be of arbitrarily high order. 
    Thus the degeneration of top order energy estimates -- without modification by such a symbol -- as observed in the nonlinear setting of say \cite{FournodavlosRodnianskiSpeck} or \cite{FournodavlosLuk}, is a true phenomenon, and one expects that in any nonlinear scattering result one must allow for a large loss of derivatives in either direction.
\end{remark}

\begin{remark}
    The subcriticality condition \eqref{eq:subcritical} is necessary in order to conclude asymptotics of the type given in Theorem~\ref{thm:einstein_scat}(i), due to heuristics of \cite{kl63}. However, even if one violates \eqref{eq:subcritical}, it may be possible to prove a version of Theorem~\ref{thm:einstein_scat} if we suppress the linearly unstable directions. 

    For instance, if one makes the additional assumption that for any $i, j, k$ such that $- p_i + p_j + p_k \geq 1$,
    \begin{equation} \label{eq:subcrit_linear}
        \frac{\partial_k \matr{\upkappa}{\underline{j}}{\underline{i}}(t)}{p_{\underline{i}} - p_{\underline{j}}} - 
        \frac{\partial_j \matr{\upkappa}{\underline{k}}{\underline{i}}(t)}{p_{\underline{i}} - p_{\underline{k}}} = 0
    \end{equation}
    for all times $t > 0$, then our proof shows that a version of Theorem~\ref{thm:einstein_scat} still holds. \eqref{eq:subcrit_linear} is true in certain symmetry classes, for instance $\mathbb{T}^2$ symmetry or polarized $U(1)$ symmetry in $1+3$ dimensions. This complements known nonlinear stability results in these classes \cite{AmesT2Stability, FournodavlosRodnianskiSpeck}. In fact one would expect a version of Theorem~\ref{thm:einstein_scat} to still hold if \eqref{eq:subcrit_linear} is only valid \emph{asymptotically as $t \to 0$}. We return to this issue in Section~\ref{sub:bkl}.
\end{remark}

\subsection{Features of the proof} \label{sub:intro_proof}

We illustrate several ideas involved in the proofs of Theorem~\ref{thm:wave_scat} and Theorem~\ref{thm:einstein_scat}. For simplicity, we mainly focus on the isotropic Kasner spacetime with $p_I = \frac{1}{D}$ for $I = 1, \ldots, D$, and mention only briefly the differences in the anisotropic case. 
We also rescale time in such a way that the wave equation \eqref{eq:wave} can be rewritten as the following first order hyperbolic system, where $\Delta = (g_{Euc}^{-1})^{ij} \partial_i \partial_j$ is the standard Laplacian on $\mathbb{T}^D$.
\begin{equation} \label{eq:intro_sketch_wave}
    t \partial_t \phi = \psi, \qquad t \partial_t \psi = t^2 \Delta \phi.
\end{equation}

Our linear scattering results will always follow upon decomposing our PDEs in frequency space using Fourier series, and performing ODE analysis for each Fourier mode. Taking the Fourier decomposition of \eqref{eq:intro_sketch_wave}, we have for $|\lambda|^2 = \sum_{i = 1}^D \lambda_i^2$ the system
\begin{equation} \label{eq:intro_sketch_wave_l}
    t \partial_t \phi_{\lambda} = \psi_{\lambda}, \qquad t \partial_t \psi_{\lambda} = - |\lambda|^2 \phi_{\lambda}.
\end{equation}

Remarkably, \eqref{eq:intro_sketch_wave_l} can be solved explicitly. This is because \eqref{eq:intro_sketch_wave_l} is Bessel's equation (of order 0) written in first order form. Using \cite{NIST:DLMF} as a reference regarding Bessel's equation and the Bessel functions $J_{\nu}$ and $Y_{\nu}$, solutions to \eqref{eq:intro_sketch_wave_l} can be written as the following linear combinations of Bessel functions:
\begin{equation} \label{eq:intro_sketch_wave_bessel}
    \phi_{\lambda} = c_J(\lambda) J_0( |\lambda| t ) + c_Y(\lambda) Y_0(|\lambda| t), \qquad \psi_{\lambda} = - c_J(\lambda) (|\lambda| t) J_1 ( |\lambda| t ) - c_Y(\lambda) (|\lambda| t) Y_1( |\lambda| t).
\end{equation}
In particular, the asymptotics of Bessel functions in \cite[Chapter 10.7]{NIST:DLMF} imply the following behaviour for $\phi_{\lambda}$ and $\psi_{\lambda}$ in the regimes $|\lambda| t \to 0$ and $|\lambda| t \to +\infty$.
\begin{gather*}
    \phi_{\lambda}(t) = \begin{cases}
        \frac{2}{\pi} c_Y(\lambda) \log(|\lambda| t) + c_J(\lambda) + \cdots & \text{ as } |\lambda| t \to 0,\\
        \sqrt{ \frac{2}{\pi |\lambda| t} } \left( c_J(\lambda) \cos(|\lambda|t - \frac{\pi}{4}) + c_Y(\lambda) \sin(|\lambda|t - \frac{\pi}{4}) \right) + \cdots & \text{ as } |\lambda| t \to + \infty,
    \end{cases}
    \\[0.4em]
    \psi_{\lambda}(t) = \begin{cases}
        \frac{2}{\pi} c_Y(\lambda) + \cdots & \text{ as } |\lambda| t \to 0,\\
        \sqrt{ \frac{2 |\lambda| t}{\pi} }\left( - c_J(\lambda) \sin(|\lambda|t - \frac{\pi}{4}) + c_Y(\lambda) \cos(|\lambda|t - \frac{\pi}{4}) \right) + \cdots & \text{ as } |\lambda| t \to + \infty.
    \end{cases}
\end{gather*}

One may read off from these asymptotics that as $t \to 0$, we expect an expansion $\phi = \psi_{\infty} \log t + \varphi_{\infty} + \cdots$ as in Theorem~\ref{thm:asymp_wave}, and moreover that the Fourier coefficients of $\psi_{\infty}$ and $\varphi_{\infty}$ can be read off from the coefficients $c_J(\lambda)$ and $c_Y(\lambda)$. 
Using the characterization of Sobolev spaces $H^s$ in Fourier space, we may then translate these asymptotics into statements about Sobolev regularity:
\begin{enumerate}[1.]
    \item
        Having $\phi(1, \cdot) \in H^{s+1}$ and $\psi(1, \cdot) \in H^s$ is equivalent to the sum $\sum_{\lambda} |\lambda|^{2s + 1} (|c_J(\lambda)|^2 + |c_Y(\lambda)|^2)$ being finite. This is due to the powers of $|\lambda| t$ arising in the $|\lambda|t \to \infty$ regime.
    \item
        At $t = 0$, we observe that $(\psi_{\infty})_{\lambda} = \frac{2}{\pi} c_Y(\lambda)$ and $(\varphi_{\infty})_{\lambda} = c_J(\lambda) - \frac{2}{\pi} c_Y(\lambda) \log |\lambda| = c_J(\lambda) - (\psi_{\infty})_{\lambda} \log |\lambda|$. Therefore, the sum $\sum_{\lambda} |\lambda|^{2s + 1} (|c_J(\lambda)|^2 + |c_Y(\lambda)|^2)$ being finite is equivalent to $\psi_{\infty}$ and the modified $\varphi_{\infty} + (\log |\nabla|) \psi_{\infty}$ being in $H^{s + \frac{1}{2}}$.
\end{enumerate}
Combining these two observations yields the scattering result of Theorem~\ref{thm:wave_scat}, at least after detailed analysis of the error terms in the Bessel asymptotics.

However, since we wish to generalize to anisotropic Kasner spacetimes, and also to the study of the linearized Einstein--scalar field system, it is desirable to find a method more robust than using the explicit Bessel functions. That is, we wish to capture Bessel-like asymptotics in both the $|\lambda| t \to + \infty$ and the $|\lambda| t \to 0$ regimes, which we denote the high-frequency and low-frequency regimes respectively.

It turns out we can achieve this via energy estimates for each Fourier mode $(\phi_{\lambda}, \psi_{\lambda})$. In our two regimes, we shall require two different energy estimates, which we explain as follows:
\begin{enumerate}[1.]
    \item
        In the high-frequency regime $|\lambda| t \gg 1$, we use a high-frequency energy which is given by
        \[
            \mathcal{E}_{\lambda, high}(t) = \frac{\psi_{\lambda}^2}{|\lambda| t} + \frac{\psi_{\lambda} \phi_{\lambda}}{|\lambda|t} + \frac{\phi_{\lambda}^2}{2 |\lambda| t} + |\lambda| t \phi_{\lambda}^2.
        \]
        Using \eqref{eq:intro_sketch_wave_l}, a computation involving crucial top order cancellations yields that for some $C > 0$,
        \[
            |t \partial_t \mathcal{E}_{\lambda, high}(t)| \leq C (\lambda t)^{-2} \mathcal{E}_{\lambda, high}(t) .
        \]
        Therefore, Gr\"onwall's inequality implies that for a fixed Fourier mode $\lambda \in \Z^D \setminus \{0\}$, the high-frequency energy remains well controlled in the region $|\lambda|t \geq 1$.

    \item
        In the low-frequency regime $|\lambda| t \ll 1$, we first define $ \tilde{\varphi}_{\lambda} = \phi_{\lambda} - \psi_{\lambda} \log(|\lambda| t)$ so that $\phi_{\lambda} = \tilde{\varphi}(|\lambda| t)$ exactly when $|\lambda| t = 1$. We then use a low-frequency energy quantity which is simply
        \[
            \mathcal{E}_{\lambda, low}(t) = \psi_{\lambda}^2 + \tilde{\varphi}_{\lambda}^2.
        \]
        This is uniformly equivalent to the high-frequency energy $\mathcal{E}_{\lambda, high}(t)$ at the transitional time $t = |\lambda|^{-1}$, and one may show from the system \eqref{eq:intro_sketch_wave_l} that
        \[
            |t \partial_t \mathcal{E}_{\lambda, low}(t)| \leq C (\lambda t)^2 (1 + |\log (\lambda t)|^2) \mathcal{E}_{\lambda, low}(t).
        \]
        Therefore, Gr\"onwall's inequality implies that for a fixed Fourier mode $\lambda \in \Z^D \setminus \{0\}$, the low-frequency energy remains well controlled in the region $0 < |\lambda| t \leq 1$.
\end{enumerate}

Combining these two energy estimates and translating into a Sobolev regularity statement yields a proof of Theorem~\ref{thm:wave_scat}. Furthermore, this energy method is robust and may be applied more widely.

In particular, it is robust enough to deal with the linearized Einstein equations. Focusing only on the linearized metric perturbations, we write the linearized Einstein equations (in the isotropic Kasner with a rescaled time variable) in the following schematic form, where $\upeta$ and $\upkappa$ are our main dynamical variables:
\begin{equation} \label{eq:intro_sketch_einstein}
    t \partial_t \upeta = \upkappa + \upnu + \partial \upchi, \qquad t \partial_t \upkappa = t^2 \Delta \upeta + t^2 \partial^2 * \upeta + \upnu + \partial \upchi.
\end{equation}
There are three major differences from the case of the scalar wave equation. Firstly, $\upeta$ and $\upkappa$ are $(1, 1)$-tensors rather than scalars. In particular, we see terms of the kind $t^2 \partial^2 * \upeta$, which schematically represent linear combinations of contractions of the derivatives and the tensor indices. 

Secondly, the equations \eqref{eq:intro_sketch_einstein} are coupled to an elliptic equation for $\upnu$, which is schematically
\begin{equation} \label{eq:intro_sketch_einstein_upnu}
    (1 - t^2 \Delta) \upnu = t^2 \partial^2 * \upeta.
\end{equation}
Finally, we see another variable $\upchi$, which corresponds to a gauge choice (see Section~\ref{sub:adm_gauge}). We single out two such gauge choices: one choice is we set $\upchi = 0$, or alternatively we may choose a gauge, called the CMCSH gauge, such that $\upchi$ will also solve an elliptic equation of the schematic form
\begin{equation} \label{eq:intro_sketch_einstein_upchi}
    t^2 \Delta \partial \upchi = t^2 \partial^2 * \upeta.
\end{equation}
The benefit is that in CMCSH gauge, the term $t^2 \partial^2 * \upeta$ in \eqref{eq:intro_sketch_einstein} vanishes. Thus in CMCSH gauge \eqref{eq:intro_sketch_einstein}, coupled to \eqref{eq:intro_sketch_einstein_upnu}--\eqref{eq:intro_sketch_einstein_upchi}, can be viewed as a first-order elliptic-hyperbolic system.

Let us now see how the high- and low-frequency energy estimates apply to this system.
\begin{enumerate}[1.]
    \item
        At high-frequencies $|\lambda|t \gg 1$, we wish to use a similar looking energy of the form
        \[
            \mathcal{E}_{\lambda, high}(t) = \frac{|\upkappa_{\lambda}|_g^2}{|\lambda|t} + \frac{\langle \upkappa_{\lambda}, \upeta_{\lambda} \rangle_g}{|\lambda| t} + \frac{|\upeta_{\lambda}|^2_g}{2 |\lambda|t} + |\lambda| t \cdot |\upeta_{\lambda}|^2_g.
        \]
        Here the norm $|\cdot|_g^2$ and the inner product $\langle \cdot, \cdot \rangle_g$ indicate one needs a metric to define norms on tensors. As in the wave equation case, we require crucial top order cancellations in order to obtain a good high-frequency energy estimate. In order to achieve the same cancellations in this case, it is essential that we use the CMCSH gauge, in order to remove the top order term $t^2 \partial^2 * \upeta$ in \eqref{eq:intro_sketch_einstein}.

        In this gauge, we get the desired top order cancellations from the $\upeta$ and $\upkappa$ terms. However, we still need to control the contributions from $\upnu$ and $\partial \upchi$. The elliptic equation \eqref{eq:intro_sketch_einstein_upnu} suggests that in the $|\lambda t| \geq 1$ regime, $\upnu_{\lambda}$ is of the same size as $\upeta_{\lambda}$, and thus the $\upnu$ appearing in the first equation in \eqref{eq:intro_sketch_einstein} seems to produce a top order contribution not conducive to using Gr\"onwall towards $t = 0$.

        Fortunately, the Hamiltonian constraint means that we may rewrite the elliptic equation in a different schematic form, namely $(1 - t^2 \Delta) \upnu = \upkappa$. For $|\lambda| t \geq 1$, this tells us that $\upnu$ is of the same size as $(|\lambda| t)^{-2} \upkappa$, which in our definition of $\mathcal{E}_{\lambda, high}(t)$ behaves better than $\upeta$, and allows us to Gr\"onwall terms involving $\upnu$.

        Similar considerations apply for $\partial \upchi$. Though the elliptic equation \eqref{eq:intro_sketch_einstein_upchi} suggests $\partial \upchi$ is of the same size as $\upeta$, upon taking the derivative $t \partial_t \mathcal{E}_{\lambda, high}(t)$, we can apply the CMCSH gauge condition and the momentum constraint to close the energy estimate. 

    \item
        At low-frequencies $|\lambda| t \ll 1$, we first define a quantity akin to $\tilde{\varphi}$ from the wave equation case which remain boundeds as $t \to 0$. In the isotropic case, it will suffice to take $\tilde{\Upupsilon} = \upeta - \upkappa \log(\lambda t)$ (in the anisotropic setting renormalized quantities are as in \eqref{eq:upupsilon}). We employ an energy of the type
        \[
            \mathcal{E}_{\lambda, low}(t) = | \upkappa_{\lambda} |_{g_*}^2 + | \tilde{\Upupsilon}_{\lambda} |_{g_*}^2.
        \]
        Note that we use a different norm $|\cdot|_{g_*}^2$ to emphasize that different metrics are used to quantify tensorial objects in the high-frequency and the low-frequency estimates.

        Ignoring this subtlety for now, we consider how the different terms in \eqref{eq:intro_sketch_einstein} affect our low-frequency energy estimates. It seems that $\partial \upchi$ will be extremely problematic. This is since if one uses CMCSH gauge and the elliptic equation \eqref{eq:intro_sketch_einstein_upchi}, we expect $\partial \upchi$ to have the same size as $\upeta$. However, since it is not $\upeta$ but instead $\tilde{\Upupsilon}$ that is bounded, a term of order $\upeta$ really has a term of order $\upkappa \log(|\lambda| t)$, even in the isotropic case. Therefore this term is not even bounded as $|\lambda| t \to 0$.

        We resolve this by using a different gauge, namely the gauge where $\upchi \equiv 0$. Of course, this means that the term $t^2 \partial^2 * \upeta$ in \eqref{eq:intro_sketch_einstein} will return. Fortunately, the low-frequency energy estimate did not require top order cancellations for the main terms involving $\upkappa$ and $\upeta$, and thus the term $t^2 \partial^2 * \upeta$, which goes to $0$ as $|\lambda| t \to 0$, is of no issue. 

        For terms involving $\upnu$, from \eqref{eq:intro_sketch_einstein_upnu}, we see $\upnu$ is also of order $t^2 \partial^2 * \upeta$ and presents no further issue in the energy estimates. Thus at least in the isotropic case, we can derive the same estimate $t \partial_t \mathcal{E}_{\lambda, low}(t) \leq C (|\lambda|t)^2 (1 + |\log(\lambda t)|^2)$, and thereby close the argument.
\end{enumerate}

Translating into statements about Sobolev spaces, we can recover a scattering result for the linearized Einstein equations. Note the different gauges used in the high-frequency and low-frequency regimes, and is exactly the reason that we must allow a change of gauge (associated to a vector field $\upxi^j$) in Theorem~\ref{thm:einstein_scat}.

We briefly comment upon the additional difficulties when the background Kasner is not isotropic. Firstly, in order to obtain the correct top order cancellations at high frequency, we need to tailor the high-frequency energy in a way that sees the anisotropy. This is done by introducing frequency-dependent coefficients $\zeta = \zeta_{\lambda}(t)$ into the definition of $\mathcal{E}_{\lambda, high}$, see already Lemma~\ref{lem:tau} and Definition~\ref{def:wave_high_freq}.

Another difficulty is associated to the tensorial norms $|\cdot|_g$ and $|\cdot|_{g_*}$ introduced in the high-frequency and low-frequency energies above. In order to close the estimates, $|\cdot|_g$ must be associated to the time-dependent metric $\mathring{g}_{ij}(t)$, where $\mathring{g}_{ij}$ represents the background Kasner metric. On the other hand $|\cdot|_{g_*}$ is associated to the metric $\mathring{g}_{ij}(t_{\lambda*})$ where $t_{\lambda*}$ is as in Definition~\ref{def:tstar}. This is the source of the symbols $\mathcal{T}_*^{-p_i + p_j}$ appearing in Theorem~\ref{thm:einstein_scat}.

Yet another difficulty arises from the fact that our renormalized quantity $\tilde{\Upupsilon}$ is not just $\upeta - \upkappa \log(|\lambda| t)$, but a more complicated expression involving powers of $t$. Therefore, the low-frequency energy estimate involves carefully tracking the powers of $t$ that arise, in order to ensure that $\partial_t \mathcal{E}_{\lambda, low}(t)$ is integrable towards $t = 0$; this is exactly where we require the subcriticality condition \eqref{eq:subcritical}. This concludes our expository discussion of the proof.

\subsection{Outline of the paper} \label{sub:intro_outline}

\begin{itemize}
    \item
        In Section~\ref{sec:background}, we provide further background regarding the Einstein--scalar field system in an ADM-type gauge, and discuss the role played by Kasner and Kasner-like spacetimes in this context. We also provide a review of the literature and briefly discuss the nonlinear scattering problem. Readers only interested in the wave equation \eqref{eq:wave} may skip to Section~\ref{sec:wavescat}.

    \item
        In Section~\ref{sec:wavescat}, we provide a complete proof of our scattering result for the scalar wave equation in (non-degenerate) Kasner spacetimes, Theorem~\ref{thm:wave_scat}. As outlined in Section~\ref{sub:intro_proof}, the main ingredients will be a high-frequency energy estimate and a low-frequency energy estimate for each Fourier mode.

    \item
        In Section~\ref{sec:lingrav}, we introduce the linearized Einstein--scalar field system to which our scattering result Theorem~\ref{thm:einstein_scat} applies. We derive the full system (including the linearized shift vector $\upchi^j$) in Proposition~\ref{prop:adm_linear}, then address several key issues including linearized gauge transformations, well-posedness, and appropriate renormalized variables near $t=0$.

    \item
        In Section~\ref{sec:high_freq}, we derive the high-frequency energy estimate which applies to our linearized Einstein--scalar field system. As mentioned in Section~\ref{sub:intro_proof}, this energy estimate uses the CMCSH gauge.

    \item
        In Section~\ref{sec:low_freq}, we derive the low-frequency energy estimate which applies to our linearized Einstein--scalar field system. As mentioned in Section~\ref{sub:intro_proof}, this energy estimate instead uses the CMCTC gauge.

    \item
        In Section~\ref{sec:einstein_scat}, we patch together our high-frequency energy estimate and low-frequency energy estimate to complete the proof of Theorem~\ref{thm:einstein_scat}. As is turns out, Theorem~\ref{thm:einstein_scat} will be derived from an alternative version of our scattering result, Theorem~\ref{thm:einstein_scat_v2}, which follows more naturally from our energy estimates.
\end{itemize}

\subsection*{Notation}

\noindent
\emph{Indices:} Lower case Roman letters $a, b, i, j, p, q,$ etc.~will always denote spatial indices, and vary between $1$ and $D$. We use Einstein summation convention with one index up and one index down, and specify cases where there is no summation by underlining the indices, e.g.~$t \matr{\mathring{k}}{i}{j} = - p_{\underline{i}} \matr{\delta}{\underline{i}}{j}$.

\vspace{1em} \noindent
\emph{Tensors:} From Section~\ref{sec:background} onwards, tensors will always be on $\mathbb{T}^D$, and written with respect to standard coordinates. We generally clarify the rank of tensors by denoting them with indices. For instance, $\matr{\upalpha}{i}{j}$ is a $(1, 1)$ tensor on $\mathbb{T}^D$, while $\upbeta^j$ is a vector field.

\vspace{1em} \noindent
\emph{Math accents:} We often denote variables with accents to signify their features. Expressions with an overhead circle, e.g.~$\mathring{g}^{ij}$ or $t \matr{\mathring{k}}{i}{j}$ will always refer to background Kasner variables. A hat e.g.~on $\matr{\hat{\upeta}}{i}{j}$ or $\matr{\hat{\upkappa}}{i}{j}$ indicates the usage of the CMCSH gauge. A tilde e.g.~on $\tilde{\upvarphi}$ or $\matr{\tilde{\Upupsilon}}{i}{j}$ indicates the usage of renormalized quantities where the ``renormalization time'' is not chosen to be $T = 1$, see for instance \eqref{eq:upupsilon_tilde_0}.

\vspace{1em} \noindent
\emph{Fourier modes:} The letter $\lambda$ is reserved for Fourier modes $\lambda = (\lambda_1, \ldots, \lambda_D) \in \Z^D$. As in Definition~\ref{def:fourier}, the Fourier series coefficients of a function $f$ or a tensor $\matr{\upalpha}{i}{j}$ will be denoted by $f_{\lambda}$ or $\matr{(\upalpha_{\lambda})}{i}{j}$ respectively.

\vspace{1em} \noindent
\emph{Constants:} The letter $C$ will always be used to denote constants which are dependent only on the exponents of the background Kasner spacetime (and thus independent of the data, as well as the Fourier mode $\lambda \in \Z^D$); as per usual we allow $C$ to vary between lines. 
We write $F \lesssim G$ to denote that $F \leq CG$, while $F \asymp G$ means that both $F \lesssim G$ and $G \lesssim F$.

\subsection*{Acknowledgements}

We would like to thank Mihalis Dafermos for his interest in the problem and for valuable advice in the writing of this paper. We also thank Serban Cicortas and Igor Rodnianski for insightful discussions and suggestions.

\section{Kasner spacetimes and the ADM equations} \label{sec:background}

In this second expository section, we further detail the ADM-type gauge used to describe the Einstein--scalar field system, then explain the role of the Kasner metrics \eqref{eq:kasner} and the related Kasner-like metrics.

\subsection{The ADM decomposition} \label{sub:adm_adm}

To study the linearized Einstein--scalar field system in the generality we require, and to simply fix notation, we shall introduce the ADM-type gauge in more detail than was discussed in Section~\ref{sub:intro_background}. In this section we introduce also a \emph{shift vector} $X^i$.

For a spacetime $\mathcal{M} \cong (0, T) \times \mathbb{T}^D$ endowed with a timelike function $t$, the standard $(1 + D)$-decomposition of the spacetime metric is as follows, where $x^i$ are standard coordinates on $\mathbb{T}^D$.
\begin{equation} \label{eq:adm_metric}
    g = - n^2 dt^2 + \bar{g}_{ij} (dx^i + X^i dt)(dx^j + X^j dt).
\end{equation}
We call this an \emph{ADM-type} gauge, with $n: \mathcal{M} \to \R_+$ the \emph{lapse function} and $X = X^i \partial_{x^i}$ a vector field tangent to the surfaces $\Sigma_t$ of constant $t$ called the \emph{shift vector field}. The Riemannian metric $\bar{g}_{ij} = \bar{g}_{ij}(t)$ on $\Sigma_t$ is often called the \textit{first fundamental form}.

While the lapse $n = \left(- g (\nabla t, \nabla t)\right)^{1/2}$ is determined purely by the choice of time function $t$, the shift vector field $X^i$ is determined by the choice of coordinates $x^i$ on each spatial slice $\Sigma_t$. The (future-directed) unit normal to the foliation of $(\mathcal{M}, g)$ by the hypersurfaces $\Sigma_t$ is given by:
\begin{equation} \label{eq:normal}
    \mathbf{N} \coloneqq n^{-1} \left( \frac{\partial}{\partial t} - X^i \frac{\partial}{\partial x^i} \right). 
\end{equation}
The following definitions are standard for the ADM-type gauge, see e,g,~\cite[Chapter 2]{RendallPDE}.

\begin{definition}
    The \emph{second fundamental form} $k_{ij}$ of the foliation of $\mathcal{M}$ by $\Sigma_t$ is a symmetric $(0, 2)$-tensor tangent to $\Sigma_t$, given by the following:
    \[
        k(\xi, \eta) = - g( \xi, \mathbf{D}_{\eta} \mathbf{N} ), \quad \forall \, \xi, \eta \text{ tangent to }\Sigma_t.
    \]
    The \emph{mean curvature} $\tr k = \tr_{\bar{g}} k$ of the foliation is then given by $\tr k = (\bar{g}^{-1})^{ij} k_{ij}$.
\end{definition}

\begin{definition}
    One says that the ADM-type gauge \eqref{eq:adm_metric} is \emph{Constant Mean Curvature} (CMC) if the mean curvature of the foliation is constant on each slice $\Sigma_t$, with mean curvature $\tr k = - t^{-1}$.
\end{definition}

\begin{definition} \label{eq:cmctc_cmcsh}
    We say that the gauge \eqref{eq:adm_metric} is \emph{Constant Mean Curvature Transported Coordinates} (CMCTC) if it is CMC and also satisfies $X^i \equiv 0$. In this case, the coordinates $x^i$ are transported by the normal vector field $\mathbf{N}$, i.e.~$\mathbf{N} x^i \equiv 0$ for $i = 1, \ldots, D$.

    On the other hand, we say that the gauge \eqref{eq:adm_metric} is \emph{Constant Mean Curvature Spatially Harmonic} (CMCSH) if it is CMC and is such that the coordinate map $(\mathbb{T}^D, \bar{g}_{ij}(t)) \to (\mathbb{T}^D, \bar{g}_{Euc})$ is a harmonic map for all $t \in (0, T)$, where $\bar{g}_{Euc}$ represents the standard Euclidean metric on $\mathbb{T}^D$. In other words,
    \begin{equation} \label{eq:spatiallyharmonic_general}
        (\bar{g}^{-1})^{jk} \left( \Gamma^{i}_{jk}[\bar{g}(t)] - \Gamma^{i}_{jk}[\bar{g}_{Euc}] \right) = 0,
    \end{equation}
    where $\Gamma^i_{jk}[\bar{g}]$ are the Christoffel symbols of the Levi-Civita connection associated to $\bar{g}$, and similarly for $\bar{g}_{Euc}$. (With respect to standard coordinates, $\Gamma^i_{jk}[\bar{g}_{Euc}] = 0$, but we choose to write \eqref{eq:spatiallyharmonic_general} in this form to clarify that it is actually a geometric equation.)
\end{definition}

In Proposition~\ref{prop:adm_evol}, we present the (nonlinear) Einstein--scalar field system \eqref{eq:einstein}--\eqref{eq:energymomentum} with respect to the ADM-type gauge. The variables $(\bar{g}_{ij}, k_{ij}, n, X^i)$ are used to describe the metric, while the variables $(\phi, \mathbf{N}\phi)$ are used to describe the scalar field $\phi$. In Proposition~\ref{prop:adm_evol}, the symbol $\nabla_i$ will always represent the Levi-Civita connection with respect to the Riemannian metric $\bar{g}_{ij} = \bar{g}_{ij}(t)$. We also raise and lower indices with respect to this metric. In particular, $\matr{k}{i}{j}$ will always mean $(\bar{g}^{-1})^{jk} k_{ik}$.

\begin{proposition}[Einstein--scalar field in ADM-type gauge] \label{prop:adm_evol}
    In a spacetime described by a CMC ADM-type gauge with $k = - t^{-1}$, the Einstein--scalar field system \eqref{eq:einstein}--\eqref{eq:energymomentum} are as follows:
    \begin{enumerate}[1.]
        \item
            The metric variables $(\bar{g}_{ij}, k_{ij})$ obey the system
            \begin{equation} \label{eq:h_evol}
                \partial_t \bar{g}_{ij} = - 2 n k_{ij} + \bar{g}_{ik} \nabla_j X^k + \bar{g}_{jk} \nabla_i X^k,
            \end{equation}
            \begin{equation} \label{eq:k_evol}
                \partial_t k_{ij} = - \nabla_i \nabla_j n + n \left( \textup{Ric}_{ij}[\bar{g}] - 2 \nabla_i \phi \nabla _j \phi - 2 \matr{k}{i}{\ell} k_{\ell j} - \frac{1}{t}k_{ij} \right) + X^k \nabla_k k_{ij} + k_{jk} \nabla_i X^k + k_{ik} \nabla_j X^k.
            \end{equation}
            Here $\textup{Ric}_{ij}[\bar{g}]$ denotes the Ricci tensor of the spatial metric $\bar{g}$, and may be expressed as
            \begin{equation} \label{eq:ricci_evol_0}
                \textup{Ric}_{ij}[\bar{g}] = - \frac{1}{2} (\bar{g}^{-1})^{ab} \partial_a \partial_b \bar{g}_{ij} - \frac{1}{2} (\bar{g}^{-1})^{ab} \partial_i \partial_j \bar{g}_{ab} + \frac{1}{2} (\bar{g}^{-1})^{ab} \partial_a \partial_i \bar{g}_{bj} + \frac{1}{2} (\bar{g}^{-1})^{ab} \partial_a \partial_j \bar{g}_{bi} + \mathcal{N}( \bar{g}^{-1}, \partial \bar{g}),
            \end{equation}
            where the expression $\mathcal{N}( \bar{g}^{-1}, \partial \bar{g} )$ is a nonlinear term quadratic in the term $\partial \bar{g}$.

            If the gauge is moreover spatially harmonic (i.e.~satisfies \eqref{eq:spatiallyharmonic_general}), then
            \begin{equation} \label{eq:ricci_evol}
                \textup{Ric}_{ij}[\bar{g}] = - \frac{1}{2} (\bar{g}^{-1})^{ab} \partial_a \partial_b \bar{g}_{ij} + \widehat{\mathcal{N}}( \bar{g}^{-1}, \partial \bar{g}),
            \end{equation}
            where $\widehat{\mathcal{N}}( \bar{g}^{-1}, \partial \bar{g} )$ is a modified nonlinear term, still quadratic in $\partial \bar{g}$.
        \item
            The matter variables $(\phi, \mathbf{N} \phi)$ evolve as follows:
            \begin{equation} \label{eq:phi_evol}
                \partial_t \phi = n \left( \mathbf{N} \phi + X^i \nabla_i \phi \right),
            \end{equation}
            \begin{equation} \label{eq:nphi_evol}
                \partial_t ( \mathbf{N} \phi ) = n \left( (\bar{g}^{-1})^{ij} \nabla_i \nabla_j \phi - \frac{1}{t} \mathbf{N} \phi + X^i \nabla_i (\mathbf{N} \phi) \right) + (\bar{g}^{-1})^{ij} \nabla_i n \nabla_j \phi.
            \end{equation}
        \item
            The lapse $n$ obeys the following elliptic equation:
            \begin{equation} \label{eq:n_elliptic}
                - (\bar{g}^{-1})^{ij} \nabla_i \nabla_j n + \frac{1}{t^2} (n - 1) = - n \left( \textup{R}[\bar{g}] - 2 (\bar{g}^{-1})^{ij} \nabla_i \phi \nabla_j \phi \right).
            \end{equation}
            Here $\textup{R}[\bar{g}] = \tr \textup{Ric}_{ij}[\bar{g}]$ is the scalar curvature of $\bar{g}$. If the gauge is moreover chosen to be spatially harmonic, then the shift vector $X^i$ also obeys an elliptic equation as follows.
            \begin{multline} \label{eq:x_elliptic}
                (\bar{g}^{-1})^{ab} \nabla_a \nabla_b X^i + (\bar{g}^{-1})^{ij}\textup{Ric}_{jk}[\bar{g}] X^k 
                =
                (\bar{g}^{-1})^{ij} \left( 2 \matr{k}{j}{a} \nabla_a n + \frac{1}{t} \nabla_j n - 4 n \, \mathbf{N} \phi \, \nabla_j \phi \right)
                \\ - 2 (\bar{g}^{-1})^{ab} \left( n \matr{k}{a}{c} + \nabla_a X^c \right) \left( \Gamma^i_{bc}[\bar{g}] - \Gamma^i_{bc}[g_{Euc}] \right).
            \end{multline}
        \item Finally, the following Hamiltonian and momentum constraint equations are satisfied:
            \begin{equation} \label{eq:hamiltonian}
                \textup{R}[\bar{g}] - \matr{k}{a}{b} \matr{k}{b}{a} + \frac{1}{t^2} = 2 (\mathbf{N} \phi)^2 + 2 (\bar{g}^{-1})^{ij} \nabla_i \phi \nabla_j \phi,
            \end{equation}
            \begin{equation} \label{eq:momentum}
                \nabla_b \, \matr{k}{a}{b} =  - 2 \, \mathbf{N} \phi \, \nabla_a \phi.
            \end{equation}
    \end{enumerate}
\end{proposition}

\begin{proof}
    The ADM evolution equations \eqref{eq:h_evol}--\eqref{eq:k_evol} and the constraints \eqref{eq:hamiltonian}--\eqref{eq:momentum} are standard and follow from a higher dimensional analogue of \cite[Section 2.3]{RendallPDE}, upon substitution of the energy momentum tensor \eqref{eq:energymomentum} and the gauge condition $\tr k = - t^{-1}$. One obtains the elliptic equation \eqref{eq:n_elliptic} by taking the trace of \eqref{eq:k_evol} and again using the CMC condition. The equation \eqref{eq:phi_evol} is a consequence of the expression \eqref{eq:normal} for the normal vector field $\mathbf{N}$, while \eqref{eq:nphi_evol} follows from the wave equation \eqref{eq:wave} written in the gauge \eqref{eq:adm_metric}.

    In the case of CMCSH gauge, the relevant equations can be found in \cite{AnderssonMoncrief}, at least in the vacuum case. See also \cite{AnderssonFajman, FajmanUrban} for the addition of matter described by the scalar field $\phi$ with energy momentum tensor \eqref{eq:energymomentum}.
\end{proof} 

\begin{remark}
    From the PDE viewpoint, it is important to have a local well-posedness theory for the equations of Proposition~\ref{prop:adm_evol}. There are several ways to do this, each of which involves picking a specific spatial gauge. One of these is the CMCSH gauge, where local well-posedness is shown in \cite{AnderssonMoncrief}. For a more general overview we refer the reader to \cite{FriedrichRendall}.
\end{remark}

\subsection{Kasner and Kasner-like spacetimes} \label{sub:adm_kasner}

Let us revisit the Kasner spacetimes in light of the discussion of Section~\ref{sub:adm_adm}. The (generalized) Kasner spacetimes $(\mathcal{M}_{Kas}, g_{Kas}, \phi_{Kas})$ with \eqref{eq:kasner}, \eqref{eq:kasner_phi}, and the Kasner relations \eqref{eq:kasner_relations}, are naturally written in the ADM-type gauge, with:
\begin{gather} \label{eq:kasner_metric}
    \mathring{n} \equiv 1, \quad \mathring{X}^i \equiv 0, \quad \mathring{\bar{g}}_{ij} = t^{2 p_{\underline{i}}} \delta_{\underline{i}j}, \quad \mathring{\phi} = p_{\phi} \log t,
    \\[0.3em] \label{eq:kasner_evol}
    \mathring{k}_{ij} = - \frac{p_{\underline{i}}}{t} \mathring{\bar{g}}_{\underline{i}j}, \quad \textup{Ric}_{ij}[\mathring{g}] \equiv 0, \quad \mathring{\mathbf{N}} \mathring{\phi} = \partial_t \mathring{\phi} = \frac{p_{\phi}}{t}.
\end{gather}
The CMC condition $\tr \mathring{k} = - t^{-1}$ is a consequence of the first Kasner relation in \eqref{eq:kasner_relations}, while the second Kasner relation in \eqref{eq:kasner_relations} is due to the Hamiltonian constraint \eqref{eq:hamiltonian}. 

The Kasner spacetimes are spatially homogeneous, in the sense that the coordinate vector fields $\partial_{x^i}$ are Killing fields of $g_{Kas}$, and solve to the Einstein--scalar field equations as stated in Proposition~\ref{prop:adm_evol}. With $\mathring{n}$ and $\mathring{X}^i$ as above, the gauge is moreover both CMCTC and CMCSH. Thus in studying perturbations of the Kasner spacetimes, it is permissible to use either of these gauges.

One should then ask: what do $O(\varepsilon)$-perturbations of Kasner spacetimes to look like? One possibility is that one perturbs a Kasner spacetime into a nearby exact Kasner spacetime with Kasner exponents $q_i = p_i + O(\varepsilon)$ and scalar field coefficient $q_{\phi} = p_{\phi} + O(\varepsilon)$. 
However, when the perturbation is not exactly spatially homogeneous one should expect the Kasner exponents $q_i$ to also depend on the spatial variable $x \in \mathbb{T}^D$.

Furthermore, one expects spatial variation in the frame that diagonalizes the second fundamental form. That is, one expects \emph{Kasner-like} metrics\footnote{Due to setting $n^2 = 1$ in \eqref{eq:bkl_ansatz1}, this ansatz will not generally be CMC. Nevertheless, it suffices to consider such an ansatz for the heuristic discussion below. One notes, for instance, that heuristically the metric is asymptotically CMC in the sense that $\tr k = - t^{-1} + O(t^{-1})$.} of the following form:
\begin{equation} \label{eq:bkl_ansatz1}
    g = - dt^2 + \sum_{I=1}^D t^{2q_I(x)} \omega_i^I(x) dx^i \otimes \omega^I_j(x) dx^j + l.o.t.
\end{equation}
Here the Kasner exponents $q_I: \mathbb{T}^D \to \R$ are now functions, while $\omega^I(x) = \omega^I_i(x) dx^i$ are now $1$-forms on $\mathbb{T}^D$ which diagonalize $k_{ij}$ near $t = 0$. Finally, $l.o.t.$ represents lower order terms that arise since \eqref{eq:bkl_ansatz1} without the lower order correction cannot be expected to be an exact solution to the Einstein--scalar field system.

Assuming moreover that the scalar field $\phi$ takes the form $\phi = q_{\phi}(x) \log t + r_{\phi}(x) + l.o.t.$ for functions $q_{\phi}, r_{\phi}: \mathbb{T}^D \to \R$, one initially observes that the CMC condition and the Hamiltonian constraint impose the following generalized Kasner relations on the functions $q_I(x)$ and $q_{\phi}(x)$:
\begin{equation} \label{eq:bkl_ansatz_relations}
    \sum_{I = 1}^D q_I(x) \equiv 1, \qquad \sum_{I = 1}^D q_I^2(x) + 2 q_{\phi}^2(x) = 1.
\end{equation}
There are a further $D$ equations relating the quantities $q_I(x)$, $\omega^I_i(x)$, $q_{\phi}(x)$ and $r_{\phi}(x)$ which arise from the momentum constraints.

We return to size $O(\varepsilon)$ perturbations of Kasner spacetimes. Assuming we have stability, one expects the new exponents $q_I(x)$, frame coefficients $\omega^I_i(x)$, and scalar field functions $q_{\phi}(x)$ and $r_{\phi}(x)$, to be $\varepsilon$-close to their analogous values in the background Kasner spacetime, that is,
\begin{equation} \label{eq:bkl_ansatz_first}
    q_I(x) = p_I + O(\varepsilon), \qquad \omega^I_i(x) = \delta^I_i + O(\varepsilon), \qquad q_{\phi}(x) = p_{\phi} + O(\varepsilon), \qquad r_{\phi}(x) = O(\varepsilon).
\end{equation}
In other words, assuming stability, $O(\varepsilon)$-perturbations of the initial data $(\bar{g}_{ij}, k_{ij}, \phi, \partial_t \phi)$ at $t = 1$, should lead to Kasner-like spacetimes which are $O(\varepsilon)$ close to the background in the above sense. 


In such situations, one senses that there might be a one-to-one correspondence between the Kasner-like metrics \eqref{eq:bkl_ansatz1} and initial data $(\bar{g}_{ij}, k_{ij}, \phi, \partial_t \phi)$, at least in a neighborhood of exact Kasner spacetimes. One na\"ive justification is the function counting argument of \cite{kl63}: note that there are $D(D+1) + 2 = D + D^2 + 2$ ``functional degrees of freedom'' for initial data\footnote{Recall that $\bar{g}_{ij}$ and $k_{ij}$ are symmetric and therefore have $\frac{D(D+1)}{2}$ independent components.} $(\bar{g}_{ij}, k_{ij}, \phi, \partial_t \phi)$, while amongst the quantities $(q_I, \omega^I_i, q_{\phi}, r_{\phi})$ appearing in the final Kasner-like state, there are also $D + D^2 + 2$ ``functional degrees of freedom''.

One must also take into account constraint equations and gauge choices in justifying this correspondence. Vaguely speaking, there are $D + 2$ functional constraints in both cases: at $t=1$ one has the Hamiltonian constraint, the $D$ momentum constraints, and the constraint arising from the wish to ``synchronise the singularity'' at $t=0$, while at $t = 0$ we have the $2$ Kasner relations and the $D$ momentum constraints. 

Supposing we believe this heuristic function counting argument, this provides another interpretation of the scattering problem. Using the scattering terminology of Section~\ref{sub:intro_background}, this interpretation is as follows:
\begin{enumerate}[(i)]
    \item (\emph{Existence of the scattering operator}):
        For what class of Cauchy data $(\bar{g}_{ij}, k_{ij}, \phi, \partial_t \phi)$ at $t=1$ does the corresponding solution to the Einstein--scalar field system attain Kasner-like aymptotics near $t = 0$ in the above sense, with corresponding Kasner-like quantities $(q_I, \omega^I_i, q_{\phi}, r_{\phi}$)? 
    \item (\emph{Asymptotic completeness}):
        If instead we are given a Kasner-like end state at $t=0$ with quantities $(q_I, \omega^I_i, q_{\phi}, r_{\phi})$, when is it possible to find Cauchy data $(\bar{g}_{ij}, k_{ij}, \phi, \partial_t \phi)$ at $t=1$ such that the corresponding solution to the Einstein--scalar field system has the required Kasner-like asymptotics?
    \item (\emph{Scattering isomorphism}):
        What are the properties of the maps in (i) and (ii)? Can we make sense of it as a map between Banach spaces? In the linearized setting, do we have a Hilbert space isomorphism?
\end{enumerate}

\begin{remark}
    For linearized gravity, the above description of the Kasner-like end state differs somewhat from the asymptotic scattering data that we refer to in Theorem~\ref{thm:einstein_scat}, i.e.~the renormalized variables $\matr{(\upkappa_{\infty})}{i}{j}$, $\matr{(\Upupsilon_{\infty})}{i}{j}$, $\upphi_{\infty}$, and $\uppsi_{\infty}$. 
    The reason we use the latter quantities $(\matr{(\upkappa_{\infty})}{i}{j}, \matr{(\Upupsilon_{\infty})}{i}{j})$ in place of $(q_I, \omega^I_i)$, is that $q_I$ and $\omega^I$ are secretly eigenvalues and eigenvectors of matrices related to $(\matr{(\upkappa_{\infty})}{i}{j}, \matr{(\Upupsilon_{\infty})}{i}{j})$, and there are losses of regularity\footnote{For instance, the map $E: \mbox{Sym}_{D \times D} \to \CC^D$ sending a symmetric $D \times D$ matrix to its $D$ ordered eigenvalues is known to be only Lipschitz continuous in the neighborhood of matrices with eigenvalues having multiplicity greater than 1.} incurred in this diagonalization process.
\end{remark}

\subsection{Velocity dominated asymptotics and the BKL conjecture} \label{sub:bkl}

In claiming the existence of a Kasner-like end state, the discussion of Section~\ref{sub:adm_kasner} posits that the background Kasner spacetime around which we are perturbing is in some sense stable. In this section, we outline several heuristics from the physics literature detailing when we expect this to be the case.

Before one asks the question of stability, an more fundamental question is to ask whether a metric of form \eqref{eq:bkl_ansatz1}, complemented by a scalar field of the form $\phi = q_{\phi} \log t + r_{\phi} + l.o.t.$, can be considered as a legitimate approximate solution to the Einstein--scalar field system \eqref{eq:einstein}--\eqref{eq:wave0}.

This metric and scalar field, together with the Kasner relations \eqref{eq:bkl_ansatz_relations}, solve the so-called \emph{velocity dominated} equations, at least if we ignore terms denoted $l.o.t.$. We explain briefly what we mean by the velocity dominated equations. These are the ADM evolution equations \eqref{eq:h_evol}--\eqref{eq:x_elliptic}, except where we discard any terms involving a spatial derivative of any sort -- this includes discarding the term involving the spatial Ricci tensor $\mathrm{Ric}[\bar{g}]$.
In other words, inserting also $n = 1$ and $X^i = 0$, we solve the equations
\begin{gather*}
    \partial_t \bar{g}_{ij} = - 2 k_{ij}, \qquad
    \partial_t k_{ij} = - 2 k_i^{\phantom{i}\ell} k_{lj} - \frac{1}{t} k_{ij}, \\[0.4em]
    \partial_t \phi = \mathbf{N} \phi, \qquad
    \partial_t (\mathbf{N} \phi) = - \frac{1}{t} \mathbf{N} \phi.
\end{gather*}
Therefore, a reasonable consistency check for the approximation \eqref{eq:bkl_ansatz1} is that all the discarded terms involving spatial derivatives are indeed suitably small.

This heuristic is reflected in work of Belinskii, Khalatnikov and Lifschitz from the latter half of the 20th century. In \cite{kl63} (see also \cite{bk72, DemaretHenneauxSpindel}), the authors suggest that the required smallness for the most important discarded term, namely the spatial Ricci tensor $\mathrm{Ric}_{ij}[\bar{g}]$, is that for some $\delta > 0$, the following bounds hold as $t \to 0$.
\begin{equation} \label{eq:spatialricci}
    | \mathrm{Ric}_{IJ}[\bar{g}] | \lesssim t^{q_I + q_J - 2 + \delta}
\end{equation}
Here $\mathrm{Ric}_{IJ}$ is the projection of the Ricci tensor onto the frame $e_I$ dual to the $1$-forms $\omega^I$.

On the other hand, in the same paper \cite{kl63}, the authors use the (approximately) orthonormal frame $e_I(x)$ to compute the leading order term in $\mathrm{Ric}_{IJ}$. Denoting $\lambda_{IJ}^K = \langle [ e_I, e_J ], e_K \rangle$ to be structure coefficients of the frame, the resulting computation yields
\begin{equation} \label{eq:bkl_ricci}
    t^{- p_I - p_J + 2} \cdot \mbox{Ric}_{IJ}[\bar{g}] = \sum_{\substack{1 \leq K, L, M \leq D \\ L \neq M}} \pm (\lambda^K_{LM})^2 \cdot t^{2(1 + q_K - q_L - q_M)} \cdot \delta_{IJ} + O(t^{\delta}),
\end{equation}
where we remain imprecise about the values of $\pm$. This computation still requires a weaker notion of \emph{asymptotically velocity dominated behaviour}, in assuming that the functions $q_I(x)$ and $\omega_i^I(x)$ are regular as functions on $\mathbb{T}^D$, and thus taking derivatives in spatial directions does not add any singular behaviour.

Now compare \eqref{eq:bkl_ricci} to \eqref{eq:spatialricci}. The key is that, at least in the expected-to-be-generic case where the structure coefficients $\lambda^K_{LM}$ are non-vanishing, the asymptotics \eqref{eq:bkl_ricci} agree with the smallness in \eqref{eq:spatialricci} only if the following \emph{subcriticality condition} holds for some $\delta > 0$:
\begin{equation} \label{eq:subcritical_delta0}
    \min_{x \in \mathbb{T}^D} \min_{\substack{1\leq K,L,M \leq D \\ L \neq M}} (1 + q_K(x) - q_L(x) - q_M(x)) \geq \delta > 0.
\end{equation}
Therefore, in the absence of any symmetries that would lead to the vanishing of certain structure coefficients, the approximate metric \eqref{eq:bkl_ansatz1} is an admissible Kasner-like metric if and only if \eqref{eq:subcritical_delta0} holds.

We return now to stability of exact Kasner spacetimes with metric \eqref{eq:kasner_metric} and scalar field \eqref{eq:kasner_phi}. In order to expect a form of stability (outside of symmetry) where perturbations give rise to \emph{admissible} Kasner-like metrics near $t=0$, then in light of \eqref{eq:bkl_ansatz_first} and \eqref{eq:subcritical_delta0}, one expects to require that for some $\updelta > 0$, the background Kasner exponents satisfy the following inequality:
\begin{equation} \label{eq:subcritical_delta}
    \min_{\substack{1\leq i,j,k \leq D \\ j \neq k}} (1 + p_i - p_j - p_k) \geq \updelta > 0.
\end{equation}
We call this the \emph{subcriticality condition} on the background Kasner exponents, and say that Kasner exponents satisfying \eqref{eq:subcritical_delta} lie in the \emph{subcritical regime}. At the heuristic level, the subcriticality condition seems to be a necessary and sufficient condition for the stability of the Kasner spacetime.

Indeed, the remarkable work of Fournodavlos, Rodnianski and Speck \cite{FournodavlosRodnianskiSpeck} shows that the subcriticality condition \eqref{eq:subcritical_delta} is sufficient to prove the \emph{nonlinear} past stability of the Kasner spacetimes. To go beyond heuristics to a rigorous proof, a key step of \cite{FournodavlosRodnianskiSpeck} is to show that the quantities akin to the variables $(\matr{\upkappa}{i}{j}, \matr{\Upupsilon}{i}{j}, \uppsi, \upvarphi)$ to be defined in Section~\ref{sec:lingrav}, are regular enough to obey the asymptotically velocity dominated behaviour assumed in the heuristics.

In \cite{FournodavlosRodnianskiSpeck}, this is done via a carefully chosen hierarchy of energy and pointwise estimates, where the hierarchy is such that high--order derivatives of these quantities may blow up as $t \to 0$, but that low-order derivatives remain bounded in $L^{\infty}$, and remain close to their background Kasner values. Such pointwise control at low order turns out to be enough to verify the BKL heuristics and hence close the argument. 

One key new insight of the scattering result of the present article, in comparison to \cite{FournodavlosRodnianskiSpeck}, is that correctly chosen top-order energies need not blow up as $t \to 0$, at least at the level of linearized gravity. Note, however, that our top-order energies depend sensitively on both time and frequency. 

\subsubsection*{Further remarks concerning the subcriticality condition \eqref{eq:subcritical_delta}}

One should comment upon the existence of Kasner spacetimes satisfying \eqref{eq:subcritical_delta} for some $\updelta > 0$. For $D = 3$, it has been known since \cite{kl63} that for \emph{vacuum} Kasner solutions, the subcriticality condition \eqref{eq:subcritical_delta} \emph{never} holds. This remains true for $3 \leq D \leq 9$, see \cite{DemaretHenneauxSpindel}. Thus in these dimensions our result applies only if we introduce also a scalar field $\phi$, as observed in \cite{bk72}. For instance, our result applies to the isotropic spacetimes with $p_I = \frac{1}{D}$ for all $I$ and $p_{\phi} = \sqrt{\frac{D-1}{2D}}$. In high dimension $D \geq 10$, \cite{DemaretHenneauxSpindel} observes there are also vacuum Kasner spacetimes satisfying \eqref{eq:subcritical_delta}, and therefore to which our results indeed apply.

In the case that \eqref{eq:subcritical_delta} fails, i.e.~if there are $i$ and $j \neq k$ such that $1 + p_i - p_j - p_k \leq 0$, then the Kasner spacetime is expected to be unstable to the past. In this case, the BKL heuristics \cite{bkl71} predict that generically the spacetime will transition to another Kasner-like regime with different exponents, and these transitions will cascade either indefinitely, or until an Kasner-like end state satisfying \eqref{eq:subcritical_delta0} is reached. We refer the reader to \cite[Chapter 3]{BelinskiHenneaux} for more details regarding such BKL transitions.

\subsubsection*{Velocity dominated behaviour outside the subcritical regime}

It remains of interest to study backgrounds where \eqref{eq:subcritical_delta} fails, particularly since this is true for $1+3$-dimensional vacuum spacetimes. In light of \eqref{eq:bkl_ricci}, we can only expect asymptotically velocity dominated behaviour in the perturbed spacetime (and in particular no BKL transitions as above), if for any Kasner exponents $q_I, q_J, q_K$ of the perturbed spacetime violating \eqref{eq:subcritical_delta0}, we have that the corresponding connection coefficient $\lambda^I_{JK}$ vanishes. 

At the linearized level, one can show that the vanishing of $\lambda^I_{JK}$ corresponds to saying that for any $p_i, p_j, p_k$ violating \eqref{eq:subcritical_delta}, the following limit holds for the linearized Weingarten map $\matr{\upkappa}{i}{j}$:
\begin{equation} \label{eq:subcrit_lin_1}
    \lim_{t \to 0} \left(
    \frac{\partial_k \matr{\upkappa}{\underline{j}}{\underline{i}}(t)}{p_{\underline{i}} - p_{\underline{j}}} - 
\frac{\partial_j \matr{\upkappa}{\underline{k}}{\underline{i}}(t)}{p_{\underline{i}} - p_{\underline{k}}} \right) = 0.
\end{equation}
This suggests that some version of our scattering result Theorem~\ref{thm:einstein_scat} should hold in cases where the above limit holds. We could interpret this hope as a codimensional\footnote{Since \eqref{eq:subcrit_lin_1} must hold at every spatial point $x \in \mathbb{T}^D$, this is an infinite codimension statement.} stability and scattering statement for linearized gravity around Kasner spacetimes violating \eqref{eq:subcritical_delta}.

In our proof, the only stage where we use \eqref{eq:subcritical} is in the low-frequency energy estimate, in particular the derivative estimate of Propostition~\ref{prop:einstein_low_der}. A more careful study of this proof shows that our result still holds in the case that \eqref{eq:subcrit_lin_1} holds \emph{without the limit} i.e.~if the expression vanishes for all $t > 0$. This is related to known results in polarized $\mathbb{T}^2$ symmetry or polarized $U(1)$ symmetry.

In fact, one could hope to prove a version of Proposition~\ref{prop:einstein_low_der} where one only assumes \eqref{eq:subcrit_lin_1}, which would involve a detailed analyis of the rate at which the expression in \eqref{eq:subcrit_lin_1} tends to $0$. We formulate this resolution of codimensional stability and scattering in the following conjecture.

\begin{conjecture}
    Let $(\mathcal{M}_{Kas}, g_{Kas}, \phi_{Kas})$ be a (generalized) non-degenerate Kasner spacetime. \underline{not necessarily} satisfying the subcriticality relation \eqref{eq:subcritical_delta}. Consider solutions $(\matr{\upeta}{i}{j}, \matr{\upkappa}{i}{j}, \upphi, \uppsi)$ to the linearized Einstein--scalar field equations around $(\mathcal{M}_{Kas}, g_{Kas}, \phi_{Kas})$, written as the first-order elliptic-hyperbolic system \eqref{eq:upeta_evol_intro}--\eqref{eq:upnu_elliptic_intro} and satisfying suitable constraints. Then we have the following:
    \begin{enumerate}[(i)]
        \item (Existence of the scattering operator)
            Let $s$ be sufficiently large. Suppose that the Cauchy data $(\matr{(\upeta_C)}{i}{j}, \matr{(\upkappa_C)}{i}{j}, \upphi_C, \uppsi_C)$, satisfying suitable constraints and the spatially harmonic gauge condition \eqref{eq:spatiallyharmonic}, launches a solution to \eqref{eq:upeta_evol_intro}--\eqref{eq:upnu_elliptic_intro} such that for any Kasner exponents $p_i, p_j, p_k$ violating \eqref{eq:subcritical_delta}, the limit \eqref{eq:subcrit_lin_1} holds.

            Then the conclusions of Theorem~\ref{thm:einstein_scat}(\ref{item:einstein_scatop}) still hold, and the limiting quantity $\matr{(\upkappa_{\infty})}{i}{j}$ is such that for any Kasner exponents $p_i, p_j, p_k$ violating \eqref{eq:subcritical_delta},
            \begin{equation} \label{eq:subcrit_limit}
                \frac{\partial_k \matr{(\upkappa_{\infty})}{\underline{j}}{\underline{i}}(t)}{p_{\underline{i}} - p_{\underline{j}}} - 
                \frac{\partial_j \matr{(\upkappa_{\infty})}{\underline{k}}{\underline{i}}(t)}{p_{\underline{i}} - p_{\underline{k}}} = 0.
            \end{equation}

        \item (Asymptotic completeness)
            Let $s$ be sufficiently large, and let and let $\matr{(\upkappa_{\infty})}{i}{j}$, $\matr{(\Upupsilon_{\infty})}{i}{j}$ be $(1, 1)$ tensors in $\mathbb{T}^D$, and $\uppsi_{\infty}$, $\upvarphi_{\infty}$ be functions in $\mathbb{T}^D$ with Sobolev regularity as in \eqref{eq:asymp_metric_reg}--\eqref{eq:asymp_scalar_reg}, and satisfying suitable constraints as well as the canonical asymptotic gauge condition \eqref{eq:asymptoticallyharmonic}.
            
            Suppose furthermore that $\matr{(\upkappa_{\infty})}{i}{j}$ is such that for any Kasner exponents $p_i, p_j, p_k$ violating \eqref{eq:subcritical_delta}, the identity \eqref{eq:subcrit_limit} holds. Then the conclusions of Theorem~\ref{thm:einstein_scat}(\ref{item:einstein_asympcomp}) still hold.

        \item (Scattering isomorphism)
            For $s \in \R$, recall the Hilbert spaces $\mathcal{H}_{C, c}^{s}$ and $\mathcal{H}_{\infty, c}^s$ from Theorem~\ref{thm:einstein_scat}(\ref{item:einstein_scatiso}). Further define $\mathcal{H}_{C, AVTD}^s \subset \mathcal{H}_{C, c}^s$ be the subspace such that the corresponding solution $(\matr{\upeta}{i}{j}, \matr{\upkappa}{i}{j}, \upphi, \uppsi)$ to the linearized Einstein--scalar field system satisfies \eqref{eq:subcrit_lin_1}\footnote{For $s$ small, the convergence in \eqref{eq:subcrit_lin_1} may hold in a distributional sense.} for Kasner exponents $p_i, p_j, p_k$ violating \eqref{eq:subcritical_delta}. Also define the subspace $\mathcal{H}_{\infty, AVTD}^s \subset \mathcal{H}_{\infty, C}^s$ to be the subspace such that for Kasner exponents $p_i, p_j, p_k$ violating \eqref{eq:subcritical_delta}, the identity \eqref{eq:subcrit_limit} holds.
            
            In light of (i) and (ii), define $\mathcal{S}_{\downarrow}$ and $\mathcal{S}_{\uparrow}$ as in Theorem~\ref{thm:einstein_scat}(\ref{item:einstein_scatiso}). Then $\mathcal{S}_{\downarrow}: \mathcal{H}_{C, AVTD}^s \to \mathcal{H}_{\infty, AVTD}^s$ is a Hilbert space isomorphism with inverse $\mathcal{S}_{\uparrow}$.
    \end{enumerate}
\end{conjecture}

\subsection{Relationship to previous literature}

There is a rich mathematical literature regarding problems in general relativity related to the discussions of Sections~\ref{sec:intro} and \ref{sec:background}. Below, we classify the literature into four broad categories.
\subsubsection*{Model hyperbolic systems in cosmological spacetimes}

The scalar wave equation \eqref{eq:wave} in Kasner spacetimes has been studied in \cite{AlhoFournodavlosFranzen, AllenRendall, PetersenMode, RingstromAsterisque}; most of these works show that $\phi$ blows up towards $t = 0$, as in Theorem~\ref{thm:asymp_wave}. In fact similar conclusions may be reached in a wider class of spacetimes which have Kasner-like asymptotics in the sense of the BKL ansatz. See also \cite{FournodavlosSbierski} for related results in the case of the interior of the Schwarzschild black hole.

The work \cite{RingstromAsterisque} in particular provides a comprehensive study of a wide array of linear systems of hyperbolic equations in such backgrounds. The related \cite{grisales2023asymptotics} provides detailed asymptotics of all orders, and addresses, as an example, Maxwell's equations on a fixed Kasner background. Note that \cite{grisales2023asymptotics, RingstromAsterisque} address solving the equations both towards and away from the singularity, though their results are not sharp in terms of regularity.

There is also a vast literature on linear and nonlinear hyperbolic systems in expanding cosmological spacetimes, for instance (asymptotically) de Sitter spacetimes \cite{CicortasScattering, HintzdS, VasydS} and expanding FLRW cosmologies \cite{CostaFranzenOliver, CostaNatarioPedro, KlainermanSarnak}. As explained in \cite{RingstromAsterisque}, these can be likened to the study of spacelike singularities if the analysis is with respect to a conformal time variable. We mention in particular \cite{CicortasScattering}, which proves a sharp linear scattering theory in de Sitter, and in particular features similar mechanisms as in our proof.


\subsubsection*{The Einstein equations: from Cauchy data to spacelike singularities}

In $D = 3$ dimensions, heuristics of \cite{kl63} suggest solutions the Einstein \emph{vacuum} equations possessing spacelike singularities with convergent Kasner-like asymptotics are expected to be non-generic, as the subcriticality condition \eqref{eq:subcritical_delta} is violated. This remains true for most matter models (although not the scalar field). Therefore, most results regarding the formation of spacelike singularities from Cauchy data occur in symmetry classes where instabilities of \cite{kl63} are suppressed.

This includes work on Gowdy symmetric spacetimes \cite{SCC_PolarizedGowdy, SCC_T3Gowdy}, where various notions of the strong cosmic censorship conjecture may be proved, as well as work addressing polarized $U(1)$ perturbations of the Schwarzschild singularity \cite{AlexakisFournodavlos}. We also mention the study of the Einstein--scalar field system in spherical symmetry, pioneered by Christodoulou \cite{Christodoulou_formation, Christodoulou_BV, Christodoulou_cc}, and continued by many other authors \cite{Dafermos_trapped, kommemi, AnZhang, Me_Kasner}. Our previous result \cite{Me_Kasner} provides a Kasner-like description of spacelike singularities arising in spherically symmetric spacetimes.

More relevant to our result are the bodies of work addressing stability of singular solutions to the Einstein--scalar field system \emph{outside of symmetry}. The breakthrough results were those of Rodnianski--Speck \cite{RodnianskiSpeck0, RodnianskiSpeck1, RodnianskiSpeck2}, who address the linear and nonlinear stability of generalized Kasner spacetimes which are near-isotropic. This was extended to all Kasner spacetimes obeying the subcriticality condition \eqref{eq:subcritical} in \cite{FournodavlosRodnianskiSpeck}; see also the recent \cite{groeniger2023formation}. 
Note that in these works the authors use a gauge involving a Fermi-propagated frame in order to capture the subcriticality condition \eqref{eq:subcritical_delta0}. See also \cite{SpeckS3, FajmanUrban, BeyerOliynyk} for extensions to other topologies and gauges.

The linearized Einstein--scalar field system we study is essentially the same as that studied in \cite{RodnianskiSpeck0}. Even upon restricting our scattering result to the study of evolution of Cauchy data at $t = 1$ towards $t = 0$, our result improves upon \cite{RodnianskiSpeck0} in that our linear stability result applies to the whole subcritical range \eqref{eq:subcritical}, as our energy estimates do not rely on ``monotonicity formulae'' that apply only to their mildly anisotropic regime.

In all of \cite{FajmanUrban, FournodavlosRodnianskiSpeck, RodnianskiSpeck0, RodnianskiSpeck1, RodnianskiSpeck2, SpeckS3}, the top-order energy estimates used blow-up towards $t=0$; we interpret this as a loss of derivatives in going from Cauchy data to the Kasner-like asymptotics. Theorem~\ref{thm:einstein_scat} confirms that this is expected even in the linearized setting, at least when one considers the full range of Kasner exponents obeying the subcriticality condition \eqref{eq:subcritical}.

\subsubsection*{The Einstein equations: Evolving from asymptotic data on spacelike singulatities}

The converse direction of our scattering theory addresses prescribing suitable asymptotic data at singularity and solving a ``singular initial value problem''. There is a history of such constructions in general relativity, which typically evolves solving so-called Fuchsian PDEs with singular coefficients at $t = 0$.

The original results addressed solving from singular initial data in the real analytic class \cite{AnderssonRendall, DHRW, KichenassamyRendall}, where a Fuchsian analogue of the Cauchy-Kovalevskaya theorem may be applied. Here \cite{AnderssonRendall} addresses the (nonlinear) Einstein--scalar field system in $D = 3$ dimensions, while \cite{DHRW} addresses more general subcritical regimes including higher dimensional spacetimes.

To go beyond real analyticity, one must combine the Fuchsian analysis with energy estimates: as before such energy estimates often allow for a mild rate of blow-up at top-order. 
In the context of Kasner-like spacetimes with asymptotically velocity dominated asymptotics, there are many examples of such results in various symmetry classes \cite{AmesT2, BeyerOliynykOlvera, ChoquetIsenbergMoncrief}.

We also mention the remarkable \cite{FournodavlosLuk}, where the authors evolve from asymptotic data for the Einstein vacuum equations in $D = 3$ spatial dimensions. They are able to perform this procedure whenever the structure coefficient $\lambda^K_{LM}$ responsible for the inconsistency in \eqref{eq:bkl_ricci} is made to vanish by a suitable choice of the frame $\omega^I(x)$. We consider this a construction of the ``non-generic'' class of spacetimes in \cite{kl63}. 


\subsubsection*{Scattering constructions in general relativity}

Scattering theory is ubiquitous in the study of linear and nonlinear evolutionary PDE, see the classical text \cite{ReedSimonScattering}, or \cite{TaoNonlinearDispersive} for a modern treatment in the context of nonlinear dispersive PDE. The scattering terminology used throughout the article is taken from these texts. 

In general relativity, there are many results concerning scattering theory for wave equations in the exterior of black holes, see \cite{Chandrasekhar1984, futterman_handler_matzner_1988} for analysis at the level of modes, and \cite{Dimock_scattering, DRS_scattering} for an energy-method approach. There are also scattering results for linearized gravity in the Schwarzschild exterior \cite{Masaood1, Masaood2}. Going into the black hole interior, there are results \cite{KehleYakov} regarding scattering for the scalar wave equation between the event horizon and the Cauchy horizon of subextremal Reissner-Nordstr\"om. All of these results address scattering between data on \emph{null} rather than spacelike hypersurfaces.

To our knowledge, our result is the first to determine a sharp scattering isomorphism (of Sobolev spaces) between Cauchy data and asymptotic data in the context of spacelike singularities. However, there is some relation to scattering in (asymptotically) de Sitter spacetimes, where past and future null infinities $\mathcal{I}^{\pm}$ appear as spacelike boundaries after Penrose compactification, see \cite{CicortasScattering, TaujanskasScattering, VasydS}. The result \cite[Theorem 1.2]{CicortasScattering} in particular features a symbol $\log \nabla$ which plays the same role as $\log \mathcal{T}_*$ in Theorem~\ref{thm:wave_scat}.

\subsection{Applications to the nonlinear Einstein--scalar field system}

We conclude this section with comments on how our ideas could potentially be applied to the study of Kasner-like singularities for the fully nonlinear Einstein--scalar field system \eqref{eq:einstein}--\eqref{eq:wave0}. 
For instance, it would be desirable to formulate a scattering theory relating Cauchy data and asymptotic data in the context of Kasner-like spacetimes. In light of the nonlinear stability result \cite{FournodavlosRodnianskiSpeck}, one could conjecture that there exists a scattering map between Cauchy data that is a small perturbation Kasner data at $t=1$, and suitable asymptotic data close to the associated Kasner data at $t=0$. 

Many additional difficulties arise in addressing the nonlinear problem. In fact, some of these difficulties arise already in \emph{linear} hyperbolic problems where the background Kasner-like spacetime is not exactly $(\mathcal{M}_{Kas}, g_{Kas})$, but a non-homogeneous perturbation thereof, since our proof is reliant on a Fourier decomposition on $\mathbb{T}^D$, and outside of the homogeneous setting one must what consider carefully what operators such as $\log \mathcal{T}_*$ and $\mathcal{T}^{-p_i + p_j}_*$ mean. It is possible we must consider microlocal operators with both physical space and Fourier space dependence.


We now comment single out several features of our proof to discuss how they could be adapted to the nonlinear problem. (Or in some cases, linear problems on non-homogeneous spacetimes.)

\subsubsection*{Renormalized quantities in the nonlinear problem}

In Theorem~\ref{thm:einstein_scat}, the Cauchy data at $t = 1$ uses the variables $(\matr{\upeta}{i}{j}, \matr{\upkappa}{i}{j}, \upphi, \uppsi)$ while our asymptotic data at $t = 0$ uses instead the renormalized variables $(\matr{\upkappa}{i}{j}, \matr{\Upupsilon}{i}{j}, \uppsi, \upvarphi)$, since it is these renormalized variables that remain bounded (and convergent) as $t \to 0$.

In the nonlinear problem, one must determine the correct renormalized variables at $t = 0$, or in other words the variables we prescribe as ``scattering data''. One proposal is the following:
\begin{gather*}
    \matr{K}{i}{j}(t) = (\bar{g}^{-1}(t))^{jk} (t k_{ik}(t)), \quad
    M_{ij}(t) = \bar{g}_{jk}(t) \cdot \matr{\exp( 2 K(t) \ln t )}{i}{k}, \\[0.3em]
    \psi(t) = t \mathbf{N}\phi(t), \quad \varphi(t) = \phi(t) - \psi(t) \log t.
\end{gather*}
Here $\exp$ denotes matrix exponentiation. Note it is a consequence of \cite[Theorem 6.1]{FournodavlosRodnianskiSpeck} that these quantities are bounded and convergent as $t \to 0$ for perturbations of Kasner spacetimes obeying the subcriticality condition \eqref{eq:subcritical_delta}. 
%
Thus a candidate nonlinear scattering map could be a map between Cauchy data $(\bar{g}_{ij}(1), k_{ij}(1), \phi(1), \psi(1))$ and the asymptotic quantities $(\matr{(K_{\infty})}{i}{j}, (M_{\infty})_{ij}, \psi_{\infty}, \varphi_{\infty}) = \lim_{t \to 0} (\matr{K}{i}{j}(t), M_{ij}(t), \psi(t), \varphi(t))$.

Another subtlety regarding how we use renormalized quantities in our energy estimates is that our low-frequency energy estimates really apply to $(\matr{\upkappa}{i}{j}, \matr{\tilde{\Upupsilon}}{i}{j}, \uppsi, \tilde{\upvarphi})$, with the tildes signifying that we use $T = \mathcal{T}_*$ in the definitions \eqref{eq:upupsilonupvarphi_T}. 
Thus in a background which is not exactly Kasner, one must also understand how to adapt the renormalized variables in a manner akin to choosing $T = \mathcal{T}_*$. 

\subsubsection*{Gauge in the nonlinear problem}

In contrast to \cite{FournodavlosRodnianskiSpeck} where the authors use Fermi-propagated frames, our scattering result Theorem~\ref{thm:einstein_scat} employs a transported coordinate ADM-type gauge. In reality, the proof does not really use a transported coordinate gauge but instead a \emph{frequency adapted} gauge, see already Section~\ref{sec:einstein_scat} and Theorem~\ref{thm:einstein_scat_v2}.

A quick way of describing this gauge is that it is CMCSH at high-frequency and essentially CMCTC at low-frequency, see Section~\ref{sub:intro_proof} for why this is the case. 
Both the CMCSH gauge and the CMCTC gauge exist at the nonlinear level, and our method suggests it is reasonable to use a mixture of the two gauges in the nonlinear setting, so long as one makes sense of the frequency-dependent transition between the gauges.

In particular, this has the effect that the spatial Ricci tensor $\mathrm{Ric}_{ij}[g]$ can be viewed as an elliptic operator acting on $\bar{g}_{ij}$ at high-frequency (see \eqref{eq:ricci_evol}), but that there is no shift vector $X^i$ at low-frequency and thus the equation for $\partial_t \matr{K}{i}{j} = \partial_t( (\bar{g}^{-1})^{jk} (t\matr{k}{i}{j}))$ will remain suitably integrable towards $t = 0$.

\subsubsection*{Norms and energies in the nonlinear problem}

In Theorem~\ref{thm:einstein_scat}, the Hilbert space $\mathcal{H}_{\infty}^s$ associated to the scattering data $(\matr{(\upkappa_{\infty})}{i}{j}, \matr{(\Upupsilon_{\infty})}{i}{j}, \uppsi_{\infty}, \upvarphi_{\infty})$ also features the symbol $\mathcal{T}_*$ extensively; this arises due to the way that one compares the high-frequency energy and the low-frequency energy and the relevant metrics used to control the tensorial norms. The idea is that high-frequencies use the metric $\mathring{g}_{ij}(t)$, while low-frequencies use $\mathring{g}_{ij}(t_{\lambda*})$. 

As a result, the symbol $\mathcal{T}_*^{-p_i + p_j}$ appears, especially when acting on off-diagonal components of tensors. 
The dependence of $\mathcal{T}_*^{-p_i + p_j}$ on the Kasner exponents means that if one were instead to consider perturbations of Kasner have spatially dependent exponents, one must be careful about the norms and energies used to control the asymptotic data at $t = 0$.

This issue is perhaps easier to resolve in the direction where one solves ``forwards'' from asymptotic data at $t = 0$, since in this context the perturbed Kasner exponents, and thus potentially the nonlinear analogues of $\mathcal{T}_*^{-p_i + p_j}$, are given. In the other direction, one has to understand how we can generate the correct such operators in evolution. 

\subsubsection*{The nonlinear terms in the evolution equations}

If one can resolve the previous issues (which already play a role in studying linearized gravity around perturbations of Kasner), it remains to control the `semilinear' nonlinear terms in the evolution equations \eqref{eq:h_evol}--\eqref{eq:nphi_evol}. 
For instance, there are the $\mathcal{N}(g^{-1}, \partial g)$ terms appearing in \eqref{eq:ricci_evol}, and terms involving the shift vector $X^i$.

Given our time-dependent frequency decomposition, this would involve an extremely careful analysis of the different types of interaction (i.e.~high-high, high-low, and low-low) both in terms of regularity and in terms of behaviour in $t$. Similarly, one would have to understand the interaction of commutator terms that arise, including commutation with the suggested microlocal operator $\mathcal{T}_*$. 


\section{Scattering for the scalar wave equation} \label{sec:wavescat}

In this section, we prove Theorem~\ref{thm:wave_scat}, concerning scattering between Cauchy data and asymptotic data for solutions of the wave equation \eqref{eq:wave} in a non-degenerate Kasner spacetime $(\mathcal{M}_{Kas}, g_{Kas})$. As outlined in Section~\ref{sub:intro_proof}, the proof relies on a detailed ODE analysis for each Fourier mode $(\phi_{\lambda}, \psi_{\lambda})$, $\lambda \in \Z^d$.

\subsection{Fourier decomposition and preliminary lemmas} \label{sub:wavescat_fourier}

Projecting the first order system \eqref{eq:wave_phi_evol}--\eqref{eq:wave_psi_evol} onto each Fourier mode $\lambda \in \Z^D$, we find the following first order system of ODEs for the Fourier coefficients $\phi_{\lambda}$ and $\psi_{\lambda}$. 
\begin{gather} 
    t \partial_t \phi_{\lambda} = \psi_{\lambda}, \label{eq:wave_phi_l_evol} \\[0.5em]
    t \partial_t \psi_{\lambda} = - \tau^2 \phi_{\lambda}, \label{eq:wave_psi_l_evol}.
\end{gather}

Here the expression $\tau^2 = \tau^2_{\lambda}(t)$ is the following function\footnote{When clear from context, we often denote $\tau$ without its subscript $\tau_{\lambda}$, primarily to shorten the exposition, and also to make it clear that $\tau_{\lambda}$ is not a Fourier coefficient. Similarly for $\zeta$ and $\zeta_{\lambda}$.} of $\lambda \in \Z^D$ and $t > 0$.
\begin{equation} \label{eq:tau}
    \tau^2 = \sum_{i=1}^D t^{2 - 2p_i} \lambda_i^2.
\end{equation}

Our aim is to find energy estimates for $\phi_{\lambda}$, $\psi_{\lambda}$ and $\varphi_{\lambda} = \phi_{\lambda} - \psi_{\lambda} \log t$ which are uniform in frequency $\lambda \in \Z^D$. The energies appearing in such energy estimates will often feature $\tau$ and derivatives thereof, and thus much of our analysis boils down to understanding properties of $\tau$. We often appeal to the following two lemmas:

\begin{lemma} \label{lem:tau}
    For $\lambda \in \Z^D \setminus \{0\}$, let $\zeta = \zeta_\lambda(t)$ be the following logarithmic $\partial_t$-derivative of $\tau$:
    \begin{equation} \label{eq:zeta}
        \zeta = \zeta_{\lambda}(t) \coloneqq \left( t \frac{d}{dt} \log \tau \right)^{-1} = \frac{\tau}{t} \cdot \left( \frac{d \tau}{dt} \right).
    \end{equation}
    Then the following upper and lower bounds for $\zeta = \zeta_{\lambda}(t)$ hold, uniformly in $\lambda \in \Z^D \setminus \{ 0 \}$ and $t > 0$:
    \begin{equation} \label{eq:zeta_upperlower}
        \min_{1 \leq i \leq D} \frac{1}{1-p_i} \leq \zeta_{\lambda}(t) \leq \max_{1 \leq i \leq D} \frac{1}{1 - p_i}.
    \end{equation}

    Furthermore, $\zeta_{\lambda}(t)$ is nonincreasing in $t$, and one has the integral estimate
    \begin{equation} \label{eq:zeta_bv}
        \int_0^1 | \zeta_{\lambda}'(s) | \, ds \leq \max_{1 \leq i \leq D} \frac{1}{1-p_i} - \min_{1 \leq i \leq D} \frac{1}{1-p_i}.
    \end{equation}
\end{lemma}

\begin{proof}
    We differentiate the expression \eqref{eq:tau} for $\tau$ with respect to $t$. This yields
    \[
        2 t \tau \cdot \frac{d \tau}{dt} = \sum_{i=1}^D (2 - 2p_i) t^{2 - 2p_i} \lambda_i^2.
    \]
    We then use the definition \eqref{eq:zeta} and rearrange to find
    \[
        \zeta(t) = \frac{ 2 \tau^2 }{ \sum_{i=1}^D (2 - 2p_i) t^{2 - 2p_i} \lambda_i^2}
        = \frac{ 2 \sum_{i=1}^D t^{2 - 2p_i} \lambda_i^2 }{ \sum_{i=1}^D (2 - 2p_i) t^{2 - 2p_i} \lambda_i^2}.
    \]
    For $\lambda \in \Z^D \setminus \{0\}$, at least one summand appearing in both the numerator and denominator is positive, and from this the bounds \eqref{eq:zeta_upperlower} are immediate.

    We now show that $\zeta(t)$ is nonincreasing. Using the above, we compute the following logarithmic derivative
    \begin{align*}
        t \frac{d}{dt} \log(\zeta(t)) 
        &= \frac{\sum_{i=1}^{D} (2 - 2p_i) t^{2 - 2p_i} \lambda_i^2 }{ \sum_{i=1}^{D} t^{2 - 2p_i} \lambda_i^2 } - \frac{\sum_{i=1}^{D} (2 - 2p_i)^2 t^{2 - 2p_i}\lambda_i^2}{\sum_{i=1}^{D} (2 - 2p_i) t^{2 - 2p_i}\lambda_i^2} \\[0.8em]
        &= \frac{\left(\sum_{i=1}^{D} (2 - 2p_i) t^{2 - 2p_i} \lambda_i^2 \right)^2 - \left(\sum_{i=1}^{D} t^{2 - 2p_i} \lambda_i^2\right) \left(\sum_{i=1}^{D} (2 - 2p_i)^2 t^{2 - 2p_i} \lambda_i^2 \right)  }{ \left( \sum_{i=1}^{D} t^{2 - 2p_i} \lambda_i^2 \right) \left( \sum_{i=1}^{D} (2 - 2p_i) t^{2 - 2p_i} \lambda_i^2 \right) }.
    \end{align*}
    The result now follows from the Cauchy-Schwarz inequality in $\R^D$, i.e.~that
    \[
        \left( \textstyle{\sum_{i=1}^D} x_i^2 \right) \left( \textstyle{\sum_{i=1}^D} y_i \right)^2 \geq \left( \textstyle{\sum_{i = 1}^D} x_i y_i \right)^2,
    \]
    with $x_i = t^{1 - p_i} \lambda_i$ and $y_i = (1 - p_i) t^{1 - p_i} \lambda_i$. This shows that the numerator above is nonpositive, and thus that $\zeta(t)$ is nonincreasing. This monotonicity and the upper and lower bounds \eqref{eq:zeta_upperlower} for $\zeta(t)$ immediately yields \eqref{eq:zeta_bv}.
\end{proof}

We use Lemma~\ref{lem:tau}, particularly \eqref{eq:zeta_upperlower}, to estimate integrals involving powers of $\tau$.

\begin{lemma} \label{lem:tauint}
    As in Definition~\ref{def:tstar}, for each $\lambda \in \Z^D \setminus \{0\}$, let $t_{\lambda*} \in (0, 1]$ be such that $\tau_{\lambda}^2 (t_{\lambda*}) = 1$. Then for $\alpha > 0$, we have the following integral estimate corresponding to $\tau \geq 1$:
    \begin{equation} \label{eq:tauintegralup}
        \int_{t_{\lambda*}}^{1} \tau^{- \alpha}(s) \, \frac{ds}{s} \lesssim \alpha^{-1}.
    \end{equation}
    One also has the following integral estimate corresponding to $\tau \leq 1$. For $0 < t \leq t_{\lambda_*}$,
    \begin{equation} \label{eq:tauintegraldown}
        \int_{0}^{t} \tau^2(s) \left( 1 + \log(\frac{t_{\lambda*}}{s})^2 \right) \, \frac{ds}{s} \lesssim \tau^2(t) \left(1 + \log( \frac{t_{\lambda*}}{t})^2 \right) \lesssim 1.
    \end{equation}
\end{lemma}

\begin{proof}
    The key to these integral estimates is that we may perform a change of variables $t \mapsto \tau$, and that 
    \[
        \frac{dt}{t} = \zeta \frac{d \tau}{\tau}.
    \]

    For instance, for \eqref{eq:tauintegralup}, upon applying this change of variables the upper bound in \eqref{eq:zeta_upperlower} yields
    \[
        \int_{t_{\lambda*}}^{1} \tau^{- \alpha}(s) \, \frac{ds}{s} \leq
        \int_1^{+\infty} \tau^{- \alpha} \, \zeta \frac{d\tau}{\tau} \leq \alpha^{-1} \cdot \max_{1 \leq i \leq D} \frac{1}{1-p_i}.
    \]

    For \eqref{eq:tauintegraldown} we start by expressing $\log(\frac{t_{\lambda*}}{t})$ in terms of $\tau$. Using the same change of variables, we find
    \[
        \log(\frac{t_{\lambda*}}{t}) = \int^{t_{\lambda *}}_{t} \frac{dt}{t} = \int^1_{\tau(t)} \zeta \frac{d \tau}{\tau}.
    \]
    Using again the upper and lower bounds \eqref{eq:zeta_upperlower} on $\zeta(t)$, we thus have that for $0 < t \leq t_{\lambda*}$,
    \begin{equation} \label{eq:log_upperlower}
        \log \left( \frac{t_{\lambda*}}{t} \right) \asymp \log \tau^{-1}.
    \end{equation}

    Therefore, the integral in \eqref{eq:tauintegraldown} is estimated as follows.
    \begin{align*}
        \int_{0}^{t} \tau^2(s) \left( 1 + \log(\frac{t_{\lambda*}}{s})^2 \right) \, \frac{ds}{s} 
        &\lesssim \int_{0}^{\tau(t)} \tau^2 \left( 1 + (\log \tau^{-1})^2 \right) \, \frac{d\tau}{\tau} \\[0.5em]
        &\lesssim \tau^2 \left( 1 + (\log \tau^{-1})^2 \right).
    \end{align*}
    The estimate \eqref{eq:tauintegraldown} then follows from using \eqref{eq:log_upperlower} again and that $\tau^2 \left( 1 + (\log \tau^{-1})^2 \right) \lesssim 1$ for $\tau \leq 1$.
\end{proof}

\subsection{The high-frequency energy estimate} \label{sub:wavescat_high}

We commence the study of the first order ODE system \eqref{eq:wave_phi_l_evol}--\eqref{eq:wave_psi_l_evol} for $\lambda \in \Z^D \setminus \{0\}$. In the high-frequency regime $\tau = \tau_{\lambda}(t) \geq 1$, we make use of the following energy quantity.

\begin{definition} \label{def:wave_high_freq}
    The \emph{high-frequency energy quantity} $\mathcal{E}_{\lambda, high}(t)$, used for $\lambda \in \Z^D \setminus \{ 0 \}$ and $\tau_{\lambda} (t) \geq 1$, is given by
    \begin{equation} \label{eq:wave_energy_high}
        \mathcal{E}_{\lambda, high}(t) = \zeta^2 \frac{\psi_{\lambda}^2}{\tau} + \zeta \frac{\psi_{\lambda} \phi_{\lambda}}{\tau} + \frac{\phi_{\lambda}^2}{2\tau} + \zeta^2 \tau \phi_{\lambda}^2.
    \end{equation}
\end{definition}

The bounds for $\zeta = \zeta_{\lambda}(t)$ obtained in Lemma~\ref{lem:tau} immediately yield the following coercivity of the energy, which will be especially useful to compare $\mathcal{E}_{\lambda, high}(\cdot)$ to Sobolev norms at $\Sigma_1$.

\begin{corollary} \label{cor:wave_high_coercive}
    One has that uniformly in frequency $\lambda \in \mathbb{Z}^D \setminus \{0\}$ and $t \geq t_{\lambda*}$,
    \[
        \mathcal{E}_{\lambda, high}(t) \asymp \frac{\psi_{\lambda}^2(t)}{\langle \tau(t) \rangle} + \langle \tau(t) \rangle \phi_{\lambda}^2(t).
    \]
    In particular, at $t=1$ one has $\mathcal{E}_{\lambda, high}(1) \asymp \langle \lambda \rangle^{-1} \psi_{\lambda}^2 (1) + \langle \lambda \rangle \, \phi_{\lambda}^2 (1)$.
\end{corollary}

The main result of the current section is the following high-frequency energy estimate.

\begin{proposition} \label{prop:wave_high_freq}
    Let $(\phi_{\lambda}(t), \psi_{\lambda}(t))$ be a solution to \eqref{eq:wave_phi_l_evol}--\eqref{eq:wave_psi_l_evol}. Then the $\partial_t$-derivative of the high-frequency energy quantity obeys the following bound, uniformly in $\lambda \in \Z^D \setminus \{0\}$ and $t \geq t_{\lambda*}$:
    \begin{equation} \label{eq:wave_high_freq_der}
        \left| t \frac{d}{dt} \mathcal{E}_{\lambda, high}(t) \right| \lesssim \left( t |\zeta'(t)| + \frac{1}{\tau^2} \right) \mathcal{E}_{\lambda, high}(t).
    \end{equation}
    There exists some $C_{high} > 0$, depending only on the background Kasner spacetime, such that for $\lambda \in \Z^D \setminus \{ 0 \}$ and $t_{\lambda*} \leq t \leq 1$, one has
    \begin{equation} \label{eq:wave_high_freq_fin}
        C_{high}^{-1} \mathcal{E}_{\lambda, high}(t) \leq \mathcal{E}_{\lambda, high}(1) \leq C_{high} \, \mathcal{E}_{\lambda, high}(t).
    \end{equation} 
\end{proposition}

\begin{proof}
    We take the $t \frac{d}{dt}$ derivative of each expression appearing in \eqref{eq:wave_energy_high}, using the equations \eqref{eq:wave_phi_l_evol}--\eqref{eq:wave_psi_l_evol} as well as $t \frac{d}{dt} \tau = \zeta^{-1} \tau$. Isolating the terms where the derivative hits $\zeta(t)$, one has
    \begin{align}
        t \frac{d}{dt} \mathcal{E}_{\lambda, high} \addtocounter{equation}{1}
        &= \zeta^2 \frac{2 \psi_{\lambda} ( - \tau^2 \phi_{\lambda} )}{\tau} - \zeta \frac{\psi_{\lambda}^2}{\tau} 
        \tag{\theequation a} \label{eq:wave_high_freq_a} \\[0.3em]
        &\qquad + \zeta \frac{ \phi_{\lambda}(- \tau^2 \phi_{\lambda}) }{\tau} + \zeta \frac{\psi_{\lambda}^2}{\tau} - \frac{\psi_{\lambda} \phi_{\lambda}}{\tau} 
        \tag{\theequation b} \label{eq:wave_high_freq_b} \\[0.3em]
        &\qquad + \frac{\phi_{\lambda} \psi_{\lambda}}{\tau} - \zeta^{-1} \frac{\phi_{\lambda}^2}{2 \tau}
        \tag{\theequation c} \label{eq:wave_high_freq_c} \\[0.3em]
        &\qquad + 2 \zeta^2 \tau \phi_{\lambda} \psi_{\lambda} + \zeta \tau \phi_{\lambda}^2 
        \tag{\theequation d} \label{eq:wave_high_freq_d} \\[0.3em]
        &\qquad + t \zeta'(t) \left( 2 \zeta \frac{\psi_{\lambda}^2}{\tau} + \frac{\psi_{\lambda} \phi_{\lambda}}{\tau} + 2 \zeta \tau \phi_{\lambda}^2 \right).
        \tag{\theequation e} \label{eq:wave_high_freq_e} 
    \end{align}

    We shall exploit crucial cancellations in the above expression. We see that
    \begin{itemize}
        \item
            The first term in \eqref{eq:wave_high_freq_a} cancels the first term in \eqref{eq:wave_high_freq_d}.
        \item
            The second term in \eqref{eq:wave_high_freq_a} cancels the second term in \eqref{eq:wave_high_freq_b}.
        \item
            The first term in \eqref{eq:wave_high_freq_b} cancels the second term in \eqref{eq:wave_high_freq_d}.
        \item
            The third term in \eqref{eq:wave_high_freq_b} cancels the first term in \eqref{eq:wave_high_freq_c}.
    \end{itemize}

    Upon applying these cancellations, we are left with the following:
    \[
        t \frac{d}{dt} \mathcal{E}_{\lambda, high} = - \zeta^{-1} \frac{\phi_{\lambda}^2}{2 \tau} + \text{\eqref{eq:wave_high_freq_e}}.
    \]
    The proof of the derivative bound \eqref{eq:wave_high_freq_der} then follows straightforwardly, after applying Corollary~\ref{cor:wave_high_coercive} and the upper and lower bounds on $\zeta(t)$ from \eqref{eq:zeta_upperlower}.

    We shall deduce \eqref{eq:wave_high_freq_fin} by applying Gr\"onwall argument to \eqref{eq:wave_high_freq_der}. For any $t_{\lambda_*}\leq t_1 \leq t_2 \leq 1$ and $C$ the implied constant in \eqref{eq:wave_high_freq_der}, Gr\"onwall's inequality yields that
    \[
        \mathcal{E}_{\lambda, high} (t_2) \leq \exp \left( C \int_{t_1}^{t_2} |\zeta'(t)| \, dt + C \int_{t_1}^{t_2} \tau^{-2}(t) \, \frac{dt}{t} \right) \mathcal{E}_{\lambda, high}(t_1).
    \]
    In fact, by applying Gr\"onwall in the opposite direction, the reverse inequality also holds:
    \[
        \mathcal{E}_{\lambda, high} (t_1) \leq \exp \left( C \int_{t_1}^{t_2} |\zeta'(t)| \, dt + C \int_{t_1}^{t_2} \tau^{-2}(t) \, \frac{dt}{t} \right) \mathcal{E}_{\lambda, high}(t_2).
    \]

    Now, from Lemmas~\ref{lem:tau} and \ref{lem:tauint}, the exponent
    \[
        C \int_{t_1}^{t_2} |\zeta'(t)| \, dt + C \int_{t_1}^{t_2} \tau^{-2}(t) \, \frac{dt}{t} \leq
        C \int_{t_{\lambda*}}^{1} |\zeta'(t)| \, dt + C \int_{t_{\lambda*}}^{1} \tau^{-2}(t) \, \frac{dt}{t}
    \]
    is bounded, uniformly in $\lambda \in \Z^D \setminus \{0\}$. The energy estimate \eqref{eq:wave_high_freq_fin} then follows upon taking $t_2 = 1$.
\end{proof}

\subsection{The low-frequency energy estimates} \label{sub:wavescat_low}

In the low-frequency regime where $\tau = \tau_{\lambda}(t) \leq 1$, we need to use a different energy estimate. Since $\phi_{\lambda}(t)$ generally blows up as $t \to 0$, we first introduce a renormalized scalar field:
\begin{equation} \label{eq:varphitilde}
    \tilde{\varphi}_{\lambda}(t) = \phi_{\lambda}(t) - \psi_{\lambda}(t) \log( \frac{t}{t_{\lambda*}} ).
\end{equation}

Note that $\tilde{\varphi}_{\lambda}(t)$ is chosen to coincide with $\phi_{\lambda}(t)$ exactly at $t = t_{\lambda*}$, and differs from the (Fourier coefficients) of the quantity $\varphi(t)$ defined in \eqref{eq:varphi} by a frequency dependent object. Nevertheless, it will be the correct renormalized quantity for our energy estimates. We use the following low-frequency energy quantity. 

\begin{definition}
    The \emph{low-frequency energy quantity} $\mathcal{E}_{\lambda, low}(t)$, used for $\lambda \in \Z^D \setminus \{ 0 \}$ and $0 < \tau_{\lambda} (t) \leq 1$, is given by
    \begin{equation} \label{eq:varphi2}
        \mathcal{E}_{\lambda, low}(t) = \psi_{\lambda}^2 + \tilde{\varphi}_{\lambda}^2.
    \end{equation}
\end{definition}

Using Corollary~\ref{cor:wave_high_coercive}, the low-frequency energy quantity is such that $\mathcal{E}_{\lambda, low}(t_{\lambda*}) \asymp \mathcal{E}_{\lambda, high}(t_{\lambda*})$. 
The main result of this section is the following low-frequency energy estimate.

\begin{proposition} \label{prop:wave_low_freq}
    Let $(\phi_{\lambda}(t), \psi_{\lambda}(t))$ be a solution to \eqref{eq:wave_phi_l_evol}--\eqref{eq:wave_psi_l_evol}. Then the $\partial_t$-derivative of the low-frequency energy quantity obeys the following bound, uniformly in $\lambda \in \Z^D \setminus \{0\}$ and $0 < t \leq t_{\lambda*}$:
    \begin{equation} \label{eq:wave_low_freq_der}
        \left| t \frac{d}{dt} \mathcal{E}_{\lambda, low}(t) \right| \lesssim \tau^2 \left( 1 + \log( \frac{t_{\lambda*}}{t})^2 \right) \mathcal{E}_{\lambda, low}(t).
    \end{equation}
    There exists some $C_{low} > 0$, depending only on the background Kasner spacetime, such that for $\lambda \in \Z^D \setminus \{ 0 \}$ and $0 < t \leq t_{\lambda*}$, one has
    \begin{equation} \label{eq:wave_low_freq_fin}
        C_{low}^{-1} \mathcal{E}_{\lambda, low}(t) \leq \mathcal{E}_{\lambda, low}(t_{\lambda*}) \leq C_{low} \, \mathcal{E}_{\lambda, low}(t).
    \end{equation} 
\end{proposition}

\begin{proof}
    In the low-frequency regime, our fundamental quantities are now $\psi_{\lambda}(t)$ and $\tilde{\varphi}_{\lambda}(t)$. We start by deriving appropriate evolution equations for these quantities. From \eqref{eq:wave_phi_l_evol}--\eqref{eq:wave_psi_l_evol} and \eqref{eq:varphi2} above we have
    \begin{gather}
        t \partial_t \psi_{\lambda} = - \tau^2 \tilde{\varphi}_{\lambda} + \tau^2 \log ( \frac{t_{\lambda*}}{t} ) \psi_{\lambda} \label{eq:wave_psi_evol_2}, \\[0.5em]
        t \partial_t \tilde{\varphi}_{\lambda} = - \tau^2 \log ( \frac{t_{\lambda*}}{t} ) \tilde{\varphi}_{\lambda} + \tau^2 \log ( \frac{t_{\lambda*}}{t} )^2 \psi_{\lambda}. \label{eq:wave_phi_evol_2}
    \end{gather}

    From these equations, we easily compute
    \[
        t \frac{d}{dt} \mathcal{E}_{\lambda, low} (t) = 2 \tau^2 \log ( \frac{t_{\lambda*}}{t} ) \psi_{\lambda}^2 + 2 \tau^2 \left( \log ( \frac{t_{\lambda*}}{t}^2 ) - 1 \right) \psi_{\lambda} \tilde{\varphi}_{\lambda} - 2 \tau^2 \log ( \frac{t_{\lambda*}}{t} ) \tilde{\varphi}_{\lambda}^2.
    \]
    Therefore the derivative bound \eqref{eq:wave_low_freq_der} is immediate.

    The energy estimate \eqref{eq:wave_low_freq_fin} follows as in the proof of Proposition~\ref{prop:wave_high_freq}. Letting $C$ be the implied constant in \eqref{eq:wave_low_freq_der}, Gr\"onwall's inequality tells us that for any $0 < t_1 \leq t_2 \leq t_{\lambda*}$, we have
    \[
        \mathcal{E}_{\lambda, low} (t_2) \leq \exp ( C \int_{t_1}^{t_2} \tau^2(t) \left( 1 + \log ( \frac{t_{\lambda*}}{t} )^2 \right) \frac{dt}{t} ) \, \mathcal{E}_{\lambda, low} (t_1),
    \]
    \[
        \mathcal{E}_{\lambda, low} (t_1) \leq \exp ( C \int_{t_1}^{t_2} \tau^2(t) \left( 1 + \log ( \frac{t_{\lambda*}}{t} )^2 \right) \frac{dt}{t} ) \, \mathcal{E}_{\lambda, low} (t_2).
    \]
    By Lemma~\ref{lem:tauint}, the integral in the exponent is bounded, uniformly in $\lambda \in \Z^D \setminus \{ 0 \}$ and $0 < t_1 \leq t_2 \leq t_{\lambda*}$. This completes the proof of the energy estimate.
\end{proof}

In order to justify the existence of the scattering states $(\psi_{\infty}, \tilde{\varphi}_{\infty})$, one must also show that the functions $\psi_{\lambda}(t)$ and $\tilde{\varphi}_{\lambda}(t)$ attain a limit as $t \to 0$. Furthermore, for the purpose of asymptotic completeness we need to show that we can prescribe the limits $(\psi_{\infty})_{\lambda}$ and $(\tilde{\varphi}_{\infty})_{\lambda}$ to launch a solution to the ODEs. This is achieved in the following proposition.

\begin{proposition} \label{prop:wave_low_freq_scat}
    For some fixed $\lambda \in \Z^D \setminus \{ 0 \}$, and $C_{low}>0$ as in Proposition~\ref{prop:wave_low_freq}, the following hold.
    \begin{enumerate}[(i)]
        \item 
            Let $(\phi_{\lambda}(t), \psi_{\lambda}(t))$ be a solution to the system \eqref{eq:wave_phi_l_evol}--\eqref{eq:wave_psi_l_evol} for $0 < t \leq t_{\lambda*}$. Then there exist real numbers $(\psi_{\infty})_{\lambda}$ and $(\tilde{\varphi}_{\infty})_{\lambda}$ such that $\psi_{\lambda}(t)$ and $\tilde{\varphi}_{\lambda}(t)$ attain these as limits as $t \to 0$, at the following rate
            \begin{equation} \label{eq:wave_low_freq_scat}
                | \psi_{\lambda}(t) - (\psi_{\infty})_{\lambda} | + | \tilde{\varphi}_{\lambda}(t) - (\tilde{\varphi}_{\infty})_{\lambda} | \lesssim \tau^2 \left( 1 + \log(\frac{t_{\lambda*}}{t})^2 \right) \mathcal{E}_{\lambda, low}^{1/2}(t_{\lambda*}) \lesssim \tau \, \mathcal{E}_{\lambda, low}^{1/2}(t_{\lambda*}).
            \end{equation}
            Furthermore, one has that $(\psi_{\infty})_{\lambda}^2 + (\tilde{\varphi}_{\infty})_{\lambda}^2 \leq C_{low} \, \mathcal{E}_{\lambda, low}(t_{\lambda*})$.
        \item
            Conversely, let $(\psi_{\infty})_{\lambda}$ and $(\tilde{\varphi}_{\infty})_{\lambda}$ be any real numbers. Then there exists a unique solution to $(\phi_{\lambda}(t), \psi_{\lambda}(t))$ to the system \eqref{eq:wave_phi_l_evol}--\eqref{eq:wave_psi_l_evol} for $0 < t \leq t_{\lambda*}$ such that $\psi_{\lambda}(t)$ and $\tilde{\varphi}_{\lambda}(t)$ obey the asymptotics \eqref{eq:wave_low_freq_scat}.

            For this solution, one has that $\mathcal{E}_{\lambda, low}(t_{\lambda*}) \leq C_{low} \left( (\psi_{\infty})_{\lambda}^2 + (\tilde{\varphi}_{\infty})_{\lambda}^2 \right)$.
    \end{enumerate}
\end{proposition}

\begin{proof}
    We first prove (i). By Proposition~\ref{prop:wave_low_freq}, we have that the low-frequency energy $\mathcal{E}_{\lambda, low}(t)$ is uniformly bounded for $0 < t \leq t_{\lambda*}$ by its value at $t = t_{\lambda*}$. Therefore, one sees from the equations \eqref{eq:wave_psi_evol_2}--\eqref{eq:wave_phi_evol_2} that the following derivative bounds hold:
    \[
        \left| t \frac{d}{dt} \psi_{\lambda}(t) \right| + \left| t \frac{d}{dt} \tilde{\varphi}_{\lambda} (t) \right | 
        \lesssim \tau^2 \left ( 1 + \log (\frac{t_{\lambda*}}{t} )^2 \right) \mathcal{E}_{\lambda, low}^{1/2}(t_{\lambda *}).  
    \]

    This derivative estimate, along with Lemma~\ref{lem:tauint}, shows that $\partial_t \psi_{\lambda}(t)$ and $\partial_t \tilde{\varphi}_{\lambda}(t)$ are integrable towards $t = 0$. Therefore the limits $(\psi_{\infty})_{\lambda}$ and $(\tilde{\varphi}_{\infty})_{\lambda}$ exist, and moreover satisfy the bounds \eqref{eq:wave_low_freq_scat}. 
    The final estimate $(\psi_{\infty})_{\lambda}^2 + (\tilde{\varphi}_{\infty})_{\lambda}^2 \leq C_{low} \, \mathcal{E}_{\lambda, low}(t_{\lambda*})$ follows from the upper bound in \eqref{eq:wave_low_freq_fin} and taking the limit $t \to 0$.

    For the converse statement (ii), we require a local existence statement for the ODE system \eqref{eq:wave_psi_evol_2}--\eqref{eq:wave_phi_evol_2} given initial data at $t = 0$. Since these ODEs are singular at $t=0$, we cannot apply the usual Picard--Lindel\"of theorem. Nonetheless, local existence follows, using the Fuchsian theory of first-order ODEs in Lemma~\ref{lem:fuchsian}. 

    We apply Lemma~\ref{lem:fuchsian}(ii) with $\mathbf{V} = (\psi_{\lambda}, \tilde{\varphi}_{\lambda})$ and $A = 0$. The derivative bounds above imply that we can apply this lemma, with initial data $\mathbf{v} = ((\psi_{\infty})_{\lambda}, (\tilde{\varphi}_{\infty})_{\lambda})$, and thereby generate a local solution $(\psi_{\lambda}(t), \tilde{\varphi}_{\lambda}(t))$ to these ODEs such that $(\psi_{\lambda}(t), \tilde{\varphi}_{\lambda}(t)) \to ((\psi_{\infty})_{\lambda}, (\tilde{\varphi}_{\infty})_{\lambda})$ as $t \to 0$, and that this solution is the unique solution attaining the desired limit.

    Since \eqref{eq:wave_psi_evol_2}--\eqref{eq:wave_phi_evol_2} are regular for $t > 0$ are linear in $(\psi_{\lambda}, \tilde{\varphi}_{\lambda})$, we may extend the solution to $0 < t \leq t_{\lambda*}$, indeed one may extend to the whole spacetime $t > 0$. The asymptotics \eqref{eq:wave_low_freq_scat} follow as in (i), and the final estimate $\mathcal{E}_{\lambda, low}(t_{\lambda*}) \leq C_{low} \left( (\psi_{\infty})_{\lambda}^2 + (\tilde{\varphi}_{\infty})_{\lambda}^2 \right)$ follows from the lower bound in \eqref{eq:wave_low_freq_fin}
\end{proof}

\subsection{Construction of the scattering isomorphism} \label{sub:wavescat_construction}

We now complete the proof of Theorem~\ref{thm:wave_scat} using our energy estimates in Sections~\ref{sub:wavescat_high} and \ref{sub:wavescat_low}. We shall first assume smoothness for the solution $(\psi(t), \phi(t))$ and derive a priori estimates relating Sobolev spaces defined at $\Sigma_1$ and Sobolev spaces defined at the $t = 0$ singularity. From these a priori estimates, the density of smooth functions in $H^s(\mathbb{T}^D)$ may be used recover the full scattering theory.

There are two directions for our scattering theory. First we show that in Section~\ref{subsub:wavescat_10} that smooth Cauchy data $(\phi_C, \psi_C)$ at $\Sigma_1$ gives rise to a smooth solution to \eqref{eq:wave_phi_evol}--\eqref{eq:wave_psi_evol} for $t > 0$, achieving suitable (smooth) asymptotics at $t=0$ with Sobolev bounds, and conversely that one may prescribe smooth asymptotic data $(\phi_{\infty}, \varphi_{\infty})$ at $t = 0$ which launches a smooth solution to \eqref{eq:wave_phi_evol}--\eqref{eq:wave_psi_evol} in the same region, and whose restriction to $\Sigma_1$ is smooth and satisfies good Sobolev bounds.

Before proceeding, note that our results in Sections~\ref{sub:wavescat_high} and \ref{sub:wavescat_low} do not apply to the zero frequency $\lambda = 0$, since $\tau(t)$ defined in \eqref{eq:tau} vanishes. However, the projection of \eqref{eq:wave_phi_evol}--\eqref{eq:wave_psi_evol} on zero frequency is trivial, as evident from the following lemma.

\begin{lemma} \label{lem:wave_zerofreq}
    Let $(\phi(t), \psi(t))$ be a solution of the equations \eqref{eq:wave_phi_evol}--\eqref{eq:wave_psi_evol}. Then defining $\varphi(t)$ as in \eqref{eq:varphi}, the $0$-frequency Fourier coefficients $\psi_0(t)$ and $\varphi_0(t)$ are constant, in particular for all $t > 0$:
    \begin{equation}
        \psi_0 (t) = \psi_0(1), \qquad \varphi_0 (t) = \varphi_0 (1) = \phi_0(1).
    \end{equation}
\end{lemma}

\begin{proof}
    This is immediate from \eqref{eq:wave_phi_evol} and \eqref{eq:wave_psi_evol} in the case that $\lambda = 0$, since in this case $\tau_0(t) = 0$ also.
\end{proof}

\subsubsection{From Cauchy data to asymptotic data} \label{subsub:wavescat_10}


\begin{proposition} \label{prop:wavescat_10}
    Let $\phi_C$ and $\psi_C$ be smooth functions on $\mathbb{T}^D$. Then there exists a unique smooth solution $(\phi(t), \psi(t))$ to the system \eqref{eq:wave_phi_evol}--\eqref{eq:wave_psi_evol} achieving the initial data in the sense
    \[
        \phi(1, \cdot) = \phi_C(\cdot), \qquad \psi(1, \cdot) = \psi_C(\cdot).
    \]

    Moreover, there exist smooth functions $\psi_{\infty}, \varphi_{\infty} \in C^{\infty}(\mathbb{T}^D)$ such that for any $s \in \mathbb{R}$ and $\varphi(t)$ as in \eqref{eq:varphi}, the following (strong) convergence holds:
    \begin{equation} \label{eq:wavescat_10_conv}
        \psi(t) \to \psi_{\infty} \text{ in } H^s(\mathbb{T}^D), \quad \varphi(t) \to \varphi_{\infty} \text{ in } H^s(\mathbb{T}^D) \quad \text{ as } t \to 0.
    \end{equation}
    Finally, for the symbol $\log(\mathcal{T}_*)$ defined in Definition~\ref{def:tstar}, one has
    \begin{equation} \label{eq:wavescat_10_bounded}
        \| \psi_{\infty} \|_{H^{s+\frac{1}{2}}}^2 + \| \varphi_{\infty} + (\log \mathcal{T}_*) \psi_{\infty} \|_{H^{s + \frac{1}{2}}}^2 \lesssim \| \phi_C \|_{H^{s+1}}^2 + \| \psi_C \|_{H^s}^2.
    \end{equation}
\end{proposition}

\begin{proof}
    The existence and uniqueness of a smooth solution $(\phi(t), \psi(t))$ for $t > 0$, which attains the initial data $\phi(1, \cdot) = \phi_C(\cdot), \psi(1,\cdot) = \psi_C( \cdot )$, follows from standard theory for first-order linear hyperbolic systems. We then take the Fourier decomposition.

    Since $\phi_C$ and $\psi_C$ are smooth, it follows from the definition of the high-frequency energy quantity \eqref{eq:wave_energy_high} and Corollary~\ref{cor:wave_high_coercive} that the energy quantities and Sobolev norms can be related in the following way:
    \begin{equation} \label{eq:wavescat_10_data}
        \phi_0^2 + \psi_0^2 + \sum_{\lambda \in \mathbb{Z}^D \setminus \{0\}} \langle \lambda \rangle^{2s+1} \mathcal{E}_{\lambda, high}(1)
        \lesssim \sum_{\lambda \in \Z^D} \left( \langle \lambda \rangle^{2s+2} \phi_{\lambda}^2(t) + \langle \lambda \rangle^{2s} \psi_{\lambda}^2(t) \right) 
        = \| \phi_C \|_{H^{s+1}}^2 + \| \psi_C \|_{H^s}^2.
    \end{equation}

    The next step is to produce the functions $\psi_{\infty}$ and $\varphi_{\infty}$. For this, we consider the Fourier coefficients $\phi_{\lambda}(t)$ and $\psi_{\lambda}(t)$, which evolve according to \eqref{eq:wave_phi_l_evol}--\eqref{eq:wave_psi_l_evol} when $\lambda \neq 0$. By Proposition~\ref{prop:wave_high_freq} and Corollary~\ref{cor:wave_high_coercive}, we know that for $t_{\lambda*} \leq t < 1$, one has
    \begin{equation} \label{eq:wavescat_10_high}
        \frac{\psi_{\lambda}^2(t)}{\tau_{\lambda}(t)} + \tau_{\lambda}(t) \phi_{\lambda}^2(t) \lesssim \mathcal{E}_{\lambda, high}(1).
    \end{equation}

    Combining this with Proposition~\ref{prop:wave_low_freq} and the fact that $\mathcal{E}_{\lambda, high}(t_{\lambda*})) \asymp \mathcal{E}_{\lambda, low}(t_{\lambda*})$, we find that for $0 < t < t_{\lambda*}$, one has
    \begin{equation} \label{eq:wavescat_10_low}
        \psi_{\lambda}^2(t) + \tilde{\varphi}_{\lambda}^2(t) \lesssim \mathcal{E}_{\lambda, high}(1).
    \end{equation}
    Here $\tilde{\varphi}_{\lambda}(t)$ is defined in \eqref{eq:varphi2}, and is related to Fourier coefficients of $\varphi(t)$ via $\tilde{\varphi}_{\lambda}(t) = \varphi_{\lambda}(t) + \psi_{\lambda}(t) \log t_{\lambda*}$.

    Recall now that Proposition~\ref{prop:wave_low_freq_scat}(i) asserts the existence of limiting coefficients $(\psi_{\infty})_{\lambda}$ and $(\tilde{\varphi}_{\infty})_{\lambda}$ such that $\psi_{\lambda}(t) \to (\psi_{\infty})_{\lambda}$ and $\varphi_{\lambda}(t) + \psi_{\lambda}(t) \log t_{\lambda*} = \tilde{\varphi}_{\lambda}(t) \to (\tilde{\varphi}_{\infty})_{\lambda}$ as $t \to 0$. Recalling also Lemma~\ref{lem:wave_zerofreq}, we thus define the functions $\psi_{\infty}$ and $\varphi_{\infty}$ as follows:
    \begin{equation} \label{eq:wavescat_10_psii}
        \psi_{\infty}(x) = \psi_0(1) + \sum_{\lambda \in \Z^D \setminus \{0\}} (\psi_{\infty})_{\lambda} \, e^{i \lambda \cdot x}, 
        \qquad \varphi_{\infty}(x) = \varphi_0(1) + \sum_{\lambda \in \Z^D \setminus \{ 0 \}} \left( (\tilde{\varphi}_{\infty})_{\lambda} - (\psi_{\infty})_{\lambda} \log t_{\lambda*} \right) e^{i \lambda \cdot x}.
    \end{equation}
    Note that by definition, the function $\tilde{\varphi}_{\infty}(x) = \varphi_{\infty}(x) + (\log \mathcal{T}_*) \psi_{\infty}(x)$ has Fourier decomposition 
    \begin{equation} \label{eq:wavescat_10_phii}
        \tilde{\varphi}_{\infty}(x) = \varphi_0(1) + \sum_{\lambda \in \Z^D \setminus \{ 0 \}} (\tilde{\varphi}_{\infty})_{\lambda} \, e^{i \lambda \cdot x}.
    \end{equation}

    We next prove the regularity of $\psi_{\infty}$ and $\varphi_{\infty}$. For any $s \in \R$, letting $t \to 0$ in \eqref{eq:wavescat_10_low}, multiplying by $\langle \lambda \rangle^{2s+1}$ and summing over $\lambda \in \Z^D$ yields
    \begin{multline*}
        \left( \psi_0^2(1) + \sum_{\lambda \in \Z^D \setminus \{0\}} \langle \lambda \rangle^{2s+1} (\psi_{\infty})_{\lambda}^2 \right) 
        + \left( \varphi_0^2(1) + \sum_{\lambda \in \Z^D \setminus \{0\}} \langle \lambda \rangle^{2s+1} (\tilde{\varphi}_{\infty})_{\lambda}^2 \right) \\[0.3em]
        \lesssim \phi_0^2(1) + \psi_0^2(1) + \sum_{\lambda \in \Z^D \setminus \{0\}} \langle \lambda \rangle^{2s+1} \mathcal{E}_{\lambda, high}(1).
    \end{multline*}
    By \eqref{eq:wavescat_10_data}, this tells us that $\psi_{\infty}$ and $\tilde{\varphi}_{\infty} = \varphi_{\infty} + (\log \mathcal{T}_* ) \psi_{\infty}$ are both in $H^{s+ \frac{1}{2}}$, and in fact \eqref{eq:wavescat_10_bounded} holds. Since this is true for all $s \in \R$, and $\log \mathcal{T}_*: H^{s + \frac{1}{2}} \to H^s$ is bounded for every $s \in \R$ (see Lemma~\ref{lem:tstar}), we have that $\psi_{\infty}$ and $\varphi_{\infty}$ are smooth, as required.

    The final step is to show the convergence \eqref{eq:wavescat_10_conv}. We write the $H^s$-norm of $\psi(t) - \psi_{\infty}$ in the following manner:
    \begin{equation*}
        \| \psi(t) - \psi_{\infty} \|_{H^s}^2
        = \sum_{\substack{\lambda \in \Z^D \setminus \{0\} \\ \tau_\lambda(t) > 1}} \langle \lambda \rangle^{2s} \left| \psi_{\lambda}(t) - (\psi_{\infty})_{\lambda} \right|^2
        + \sum_{\substack{\lambda \in \Z^D \setminus \{0\} \\ \tau_\lambda(t) \leq 1}} \langle \lambda \rangle^{2s} \left| \psi_{\lambda}(t) - (\psi_{\infty})_{\lambda} \right|^2.
    \end{equation*}
    We estimate the first sum on the right hand side by expanding crudely and using \eqref{eq:wavescat_10_high}, while for the second sum we use Proposition~\ref{prop:wave_low_freq_scat}, particularly the estimate \eqref{eq:wave_low_freq_scat}. We find the following:
    \begin{equation*}
        \| \psi(t) - \psi_{\infty} \|_{H^s}^2
        \lesssim \sum_{\substack{\lambda \in \Z^D \setminus \{0\} \\ \tau_\lambda(t) > 1}} \left( \tau_{\lambda}(t) \langle \lambda \rangle^{2s} \mathcal{E}_{\lambda, high}(1)
        + \langle \lambda \rangle^{2s} ( \psi_{\infty} )_{\lambda}^2 \right)
        + \sum_{\substack{\lambda \in \Z^D \setminus \{0\} \\ \tau_\lambda(t) \leq 1}} \tau_{\lambda}^2(t) \langle \lambda \rangle^{2s} \mathcal{E}_{\lambda, low}(t_{\lambda*}).
    \end{equation*}
    
    Using again Propositions~\ref{prop:wave_high_freq} and \ref{prop:wave_low_freq} to deduce $\mathcal{E}_{\lambda, low}(t_{\lambda*}) \lesssim \mathcal{E}_{\lambda, high}(1) \lesssim \langle \lambda \rangle^{-1} (\psi_C)_{\lambda}^2 + \langle \lambda \rangle (\phi_C)_{\lambda}^2$, as well as the fact that $\tau_{\lambda}(t) \leq \langle \lambda \rangle t^{1 - p}$ for $p = \max \{p_i\}$, we therefore have that
    \begin{equation} \label{eq:wavescat_10_psipsi}
        \| \psi(t) - \psi_{\infty} \|_{H^s}^2 \lesssim t^{1-p} \left( \| \phi_C \|_{H^{s+1}}^2 + \| \psi_C \|_{H^{s}}^2 \right) 
        + \sum_{\substack{\lambda \in \Z^D \setminus \{0\} \\ \tau_\lambda(t) > 1}} \langle \lambda \rangle^{2s} ( \psi_{\infty} )_{\lambda}^2.
    \end{equation}
    Since we have $\psi_{\infty} \in H^s$ and $\min \{ |\lambda|: \tau_{\lambda}(t) < 1 \} \to \infty$ as $t \to 0$, we conclude that $\psi_{\infty}(t) \to \psi_{\infty}$ in $H^s$ as $t \to 0$.

    The convergence of $\varphi(t)$ is more complicated due to the additional log factors. Nevertheless, a similar but slightly more involved computation will yield that
    \begin{align*}
        \| \varphi(t) - \varphi_{\infty} \|_{H^s}^2
        &\lesssim \sum_{\substack{\lambda \in \Z^D \setminus \{0\} \\ \tau_\lambda(t) > 1}} \left( \tau_{\lambda}(t) (1 + (\log t^{-1})^2) \langle \lambda \rangle^{2s} \mathcal{E}_{\lambda, high}(1)
        + \langle \lambda \rangle^{2s} ( \varphi_{\infty} )_{\lambda}^2 \right) \\[-1em]
        &\qquad \hspace{2cm}
        + \sum_{\substack{\lambda \in \Z^D \setminus \{0\} \\ \tau_\lambda(t) \leq 1}} \tau_{\lambda}^2(t) (1 + (\log t_{\lambda*}^{-1})^2) \langle \lambda \rangle^{2s} \mathcal{E}_{\lambda, low}(t_{\lambda*}).
    \end{align*}
    Using the additional fact \eqref{eq:log_upperlower}, one may show that when $\tau_{\lambda}(t) \leq 1$, one has $\tau_{\lambda} (1 + (\log t_{\lambda^*}^{-1} )^2 ) \lesssim (1 + (\log t^{-1})^2)$. Therefore the analogue of \eqref{eq:wavescat_10_psipsi} is that
    \begin{equation} \label{eq:wavescat_10_phiphi}
        \| \varphi(t) - \varphi_{\infty} \|_{H^s}^2 \lesssim t^{1-p} \, ( 1 + (\log t^{-1})^2 )  \left( \| \phi_C \|_{H^{s+1}}^2 + \| \psi_C \|_{H^{s}}^2 \right) 
        + \sum_{\substack{\lambda \in \Z^D \setminus \{0\} \\ \tau_\lambda(t) > 1}} \langle \lambda \rangle^{2s} ( \varphi_{\infty} )_{\lambda}^2.
    \end{equation}
    Since $\varphi_{\infty}$ also lies in $H^s$, the right hand side converges to $0$ as $t \to 0$. This concludes the proof of the proposition.
\end{proof}

\subsubsection{From asymptotic data to Cauchy data} \label{subsub:wavescat_01}


\begin{proposition} \label{prop:wavescat_01}
    Let $\psi_{\infty}$ and $\varphi_{\infty}$ be smooth functions on $\mathbb{T}^D$. Then there exists a unique smooth solution $(\phi(t), \psi(t))$ to the system \eqref{eq:wave_phi_evol}--\eqref{eq:wave_psi_evol} for $t > 0$, which achieves the asymptotic data in the sense that for any $s \in \R$, one has the (strong) convergence:
    \begin{equation} \label{eq:wavescat_01_asymp}
        \psi(t) \to \psi_{\infty} \text{ in } H^s(\mathbb{T}^D), \quad \varphi(t) \to \varphi_{\infty} \text{ in } H^s(\mathbb{T}^D) \quad \text{ as } t \to 0.
    \end{equation}

    Moreover, for any $s \in \R$ one has the following bound between Sobolev spaces:
    \begin{equation} \label{eq:wavescat_01_bounded}
        \| \phi (1, \cdot) \|_{H^{s+1}}^2 + \| \psi (1, \cdot) \|_{H^s}^2 \lesssim \| \psi_{\infty} \|_{H^{s+\frac{1}{2}}}^2 + \| \varphi_{\infty} + (\log \mathcal{T}_*) \psi_{\infty} \|_{H^{s + \frac{1}{2}}}^2.
    \end{equation}
\end{proposition}

\begin{proof}
    For existence, we first decompose $(\psi_{\infty}, \varphi_{\infty})$ into their Fourier modes, and define also for $\lambda \neq 0$,
    \[
        (\tilde{\varphi}_{\infty})_{\lambda} = (\varphi_{\infty})_{\lambda} + (\psi_{\infty})_{\lambda} \log t_{\lambda*}.
    \]
    By definition of the Sobolev norm and the operator $\log \mathcal{T}_*$, one has
    \begin{equation} \label{eq:wavescat_01_data}
        (\psi_{\infty})_0^2 + (\varphi_{\infty})_0^2 + \sum_{\lambda \in \Z^D \setminus \{ 0 \}} \langle \lambda \rangle^{2s + 1} \left( (\psi_{\infty})_{\lambda}^2 + (\tilde{\varphi}_{\infty})_{\lambda}^2 \right) = 
        \| \psi_{\infty} \|_{H^{s+\frac{1}{2}}}^2 + \| \varphi_{\infty} + (\log \mathcal{T}_*) \psi_{\infty} \|_{H^{s + \frac{1}{2}}}^2.
    \end{equation}

    We now appeal to Proposition~\ref{prop:wave_low_freq_scat}(ii), which states that for all $\lambda \in \Z^D \setminus \{0\}$, we can find a unique solution to the ODE system \eqref{eq:wave_psi_evol_2}--\eqref{eq:wave_phi_evol_2} which achieves the limits $(\psi_{\infty})_{\lambda}$ and $(\tilde{\varphi}_{\infty})_{\lambda}$ as $t \to 0$, at the rate \eqref{eq:wave_low_freq_scat}. 
    Transforming back to $(\phi_{\lambda}(t), \psi_{\lambda}(t))$, we can find a solution to \eqref{eq:wave_phi_l_evol}--\eqref{eq:wave_psi_l_evol} in the whole domain $t > 0$. Furthermore, Propositions~\ref{prop:wave_low_freq} and \ref{prop:wave_low_freq_scat} imply that for $\lambda \in \Z^D \setminus \{ 0 \}$ and $0 < t \leq t_{\lambda*}$, one has
    \begin{equation} \label{eq:wavescat_01_low}
        \mathcal{E}_{\lambda, low}(t) \lesssim (\psi_{\infty})_{\lambda}^2 + (\tilde{\varphi}_{\infty})_{\lambda}^2.
    \end{equation}
        
    Combining this with Proposition~\ref{prop:wave_high_freq}, Corollary~\ref{cor:wave_high_coercive} and that $\mathcal{E}_{\lambda, high}(t_{\lambda*}) \asymp \mathcal{E}_{\lambda, low}(t_{\lambda*})$, one has for $t_{\lambda*} \leq t \leq 1$ the estimate
    \begin{equation} \label{eq:wavescat_01_high}
        \frac{\psi_{\lambda}^2(t)}{\tau_{\lambda}(t)} + \tau_{\lambda}(t) \phi_{\lambda}^2(t) \lesssim (\psi_{\infty})_{\lambda}^2 + (\tilde{\varphi}_{\infty})_{\lambda}^2.
    \end{equation}

    We now define our solution $(\phi(t), \psi(t))$ to the equations \eqref{eq:wave_phi_evol}--\eqref{eq:wave_psi_evol} as the sum.
    \begin{equation}
        \phi(t, x) = (\varphi_{\infty})_0 + \sum_{\lambda \in \Z^D \setminus \{ 0 \}} \phi_{\lambda}(t) e^{i \lambda \cdot x}, \qquad 
        \psi(t, x) = (\psi_{\infty})_0 + \sum_{\lambda \in \Z^D \setminus \{ 0 \}} \psi_{\lambda}(t) e^{i \lambda \cdot x}.
    \end{equation}
    Let us verify the smoothness of $(\phi(t), \psi(t))$ at $t = 1$. Using \eqref{eq:wavescat_01_high},\eqref{eq:wavescat_01_data} and $\tau_{\lambda}(1) \asymp \langle \lambda \rangle$, we get
    \[
        (\varphi_{\infty})_0^2 + \sum_{\lambda \in \Z^D \setminus \{0\}} \langle \lambda \rangle^{2s+2} \phi_{\lambda}^2(1) + 
        (\psi_{\infty})_0^2 + \sum_{\lambda \in \Z^D \setminus \{0\}} \langle \lambda \rangle^{2s} \psi_{\lambda}^2(1) \lesssim 
        \| \psi_{\infty} \|_{H^{s+\frac{1}{2}}}^2 + \| \varphi_{\infty} + (\log \mathcal{T}_*) \psi_{\infty} \|_{H^{s + \frac{1}{2}}}^2.
    \]
    Therefore $\phi(1, \cdot) \in H^{s+1}$, $\psi(1, \cdot) \in H^s$, and the estimate \eqref{eq:wavescat_01_bounded} holds. Since this is true for all $s \in \R$, these functions are smooth. Furthermore, since the Fourier modes $(\phi_{\lambda}(t), \psi_{\lambda}(t))$ satisfy the ODEs \eqref{eq:wave_phi_l_evol}--\eqref{eq:wave_psi_l_evol} for $\lambda \neq 0$, and we may appeal to Lemma~\ref{lem:wave_zerofreq} for the zero mode, $(\phi(t), \psi(t))$ is a smooth solution to the first order system \eqref{eq:wave_phi_evol}--\eqref{eq:wave_psi_evol}, as required.

    By the construction of $(\phi(t), \psi(t))$ and the proof of Proposition~\ref{prop:wavescat_01}, $\psi(t)$ and $\varphi(t)$ also obey the asymptotics \eqref{eq:wavescat_01_asymp}. Finally, the fact that $(\phi(t), \psi(t))$ is the \textit{unique} smooth solution to \eqref{eq:wave_phi_evol}--\eqref{eq:wave_psi_evol} obeying these asymptotics follows by studying each Fourier mode for $\phi_{\lambda}(t), \psi_{\lambda}(t)$ and appealing to the uniqueness statement in Proposition~\ref{prop:wave_low_freq_scat} for $\lambda \in \Z^D \setminus \{ 0 \}$, and Lemma~\ref{lem:wave_zerofreq} for $\lambda = 0$.
\end{proof}

\subsubsection{Completion of the proof of Theorem~\ref{thm:wave_scat}} \label{sub:wavescat_density}

All that remains is to extend Proposition~\ref{prop:wavescat_10} and Proposition~\ref{prop:wavescat_01} to the setting where the functions $\phi(t)$ and $\psi(t)$ have finite regularity. For the sake of completeness, we provide the standard density argument here:
\begin{enumerate}[(i)]
    \item (Existence of the scattering operator)
        By Proposition~\ref{prop:wavescat_10}, we are already done in the case that $(\phi_C, \psi_C)$ are smooth. Now, for $s > \frac{D}{2} + 1$ suppose instead that $(\phi_C, \psi_C) \in H^{s+1} \times H^s$. By the density of smooth functions in Sobolev spaces, there exist smooth functions $(\phi^{(n)}_C, \psi^{(n)}_C)$ such that
        \[
            (\phi^{(n)}_C, \psi^{(n)}_C) \to (\phi_C, \psi_C) \text{ in } H^{s + 1} \times H^s \text{ as } n \to \infty.
        \]

        Applying Proposition~\ref{prop:wavescat_10} to each of $(\phi^{(n)}_C, \psi^{(n)}_C)$, we construct a sequence of solutions $(\phi^{(n)}(t), \phi^{(n)}(t))$ to the system \eqref{eq:wave_phi_evol}--\eqref{eq:wave_psi_evol} and a sequence of functions $(\psi_{\infty}^{(n)}, \varphi_{\infty}^{(n)})$ such that
        \[
            \| \psi_{\infty}^{(n)} \|_{H^{s + \frac{1}{2}}}^2 + \| \varphi_{\infty}^{(n)} - (\log \mathcal{T}_*) \psi_{\infty}^{(n)} \|_{H^{s + \frac{1}{2}}}^2 \lesssim \| \phi_C^{(n)} \|_{H^{s + 1}}^2 + \| \psi_C^{(n)} \|_{H^{s}}^2,
        \]
        and we have the (strong convergence)
        \[
            \psi^{(n)}(t) \to \psi_{\infty}^{(n)} \text{ in } H^s, \quad \varphi^{(n)}(t) \to \varphi_{\infty}^{(n)} \text{ in } H^s \quad \text{ as } t \to 0.
        \]

        Furthermore, since \eqref{eq:wave_phi_evol}--\eqref{eq:wave_psi_evol} is linear we also have the difference estimates of the type
        \begin{multline*}
            \| \psi_{\infty}^{(n)} - \psi_{\infty}^{(m)} \|_{H^{s + \frac{1}{2}}}^2 + \| (\varphi_{\infty}^{(n)} - (\log \mathcal{T}_*) \psi_{\infty}^{(n)}) - (\varphi_{\infty}^{(m)} - (\log \mathcal{T}_*) \psi_{\infty}^{(m)} ) \|_{H^{s + \frac{1}{2}}}^2 \\[0.4em]
            \lesssim \| \phi_C^{(n)} - \phi_C^{(m)} \|_{H^{s + 1}}^2 + \| \psi_C^{(n)} - \psi_C^{(m)} \|_{H^{s}}^2.
        \end{multline*}
        Therefore, by construction of the sequence $(\phi_C^{(n)}, \psi_C^{(n)})$, it is clear we can obtain $(\psi_{\infty}, \varphi_{\infty})$ such that $(\psi_{\infty}^{(n)}, \varphi_{\infty}^{(n)}) \to (\psi_{\infty}, \varphi_{\infty})$ in the space $\mathcal{H}^s_{\infty}$ defined in \eqref{eq:wave_hilbert2}, and by similar difference estimates applied to the proof of Proposition~\ref{prop:wavescat_10} we obtain a limiting solution 
        \[
            (\phi(t), \psi(t)) \in C^0((0, + \infty), H^{s + 1} \times H^s) \cap C^1((0, + \infty), H^s \times H^{s - 1}),
        \]
        which attain the asymptotics \eqref{eq:asymp_psivarphi}. Since $s > \frac{D}{2} + 1$, Sobolev embedding implies that $H^{s - 1} \subset C^0$ and therefore this solution is classical.

    \item (Asymptotic completeness)
        By Proposition~\ref{prop:wavescat_01}, we are done in the case that $(\psi_{\infty}, \varphi_{\infty})$ are smooth. If instead we only have $(\psi_{\infty}, \varphi_{\infty})$ with regularity as in \eqref{eq:asymp_reg}, i.e.~we have $(\psi_{\infty}, \varphi_{\infty}) \in \mathcal{H}_{\infty}^s$, then first pick a sequence $(\psi_{\infty}^{(n)}, \varphi_{\infty}^{(n)})$ of smooth functions such that
        \[
            (\psi_{\infty}^{(n)}, \varphi_{\infty}^{(n)}) \to (\psi_{\infty}, \varphi_{\infty}) \text{ in } \mathcal{H}_{\infty}^s \text{ as } n \to \infty.
        \]

        Now apply Proposition~\ref{prop:wavescat_01} to each of $(\psi_{\infty}^{(n)}, \varphi_{\infty}^{(n)})$. We obtain a sequence $(\phi^{(n)}(t), \psi^{(n)}(t))$ of solutions of \eqref{eq:wave_phi_evol}--\eqref{eq:wave_psi_evol} which moreover obey the difference estimate
        \[
            \| \phi^{(n)}(1, \cdot) - \phi^{(m)}(1, \cdot) \|_{H^{s+1}}^2 + \| \psi^{(n)}(1, \cdot) - \psi^{(m)}(1, \cdot) \|_{H^s}^2 \lesssim \| (\psi_{\infty}^{(n)} - \psi_{\infty}^{(m)}, \varphi_{\infty}^{(n)} - \varphi_{\infty}^{(m)} ) \|_{\mathcal{H}^s_{\infty}}^2.
        \]
        Therefore, by construction of the sequence $(\phi_{\infty}^{(n)}, \psi_{\infty}^{(n)})$, we construct firstly a limit $(\phi^{(n)}(1, \cdot), \psi^{(n)}(1, \cdot)) \to (\phi(1, \cdot), \psi(1, \cdot))$ in the space $\mathcal{H}^s_C = H^{s+1} \times H^s$, and then by applying (i) again, we get a limiting solution $(\phi(t), \psi(t))$ to the system \eqref{eq:wave_phi_evol}--\eqref{eq:wave_psi_evol} as desired.

    \item (Scattering isomorphism)
        By Proposition~\ref{prop:wavescat_10} and Proposition~\ref{prop:wavescat_01}, we know that the two maps
        \begin{gather*}
            \mathcal{S}_{\downarrow}: (\phi_C, \psi_C) \mapsto (\psi_{\infty}, \varphi_{\infty}) \text{ from (i)}, \\[0.4em]
            \mathcal{S}_{\uparrow}: (\psi_{\infty}, \varphi_{\infty}) \mapsto (\phi(1, \cdot), \psi(1, \cdot)) \text{ from (ii)},
        \end{gather*}
        upon restriction to the space of smooth functions, are well-defined and are such that both $\mathcal{S}_{\downarrow} \circ \mathcal{S}_{\uparrow}$ and $\mathcal{S}_{\uparrow} \circ \mathcal{S}_{\downarrow}$ are the identity map. Furthermore, the bounds \eqref{eq:wavescat_10_bounded} and \eqref{eq:wavescat_01_bounded}, together with the density of $C^{\infty}(\mathbb{T}^D)^2$ in both the Hilbert spaces $\mathcal{H}_C^{\infty}$ and $\mathcal{H}_{\infty}^s$, imply that these maps may indeed be extended to be Hilbert space isomorphisms, as required. \qed
\end{enumerate}


\section{The Linearized Einstein--scalar field system around Kasner} \label{sec:lingrav}

In this section, we derive the linearized Einstein--scalar field system around Kasner, to which Theorem~\ref{thm:einstein_scat} applies. 
This system, to be presented in full detail in Proposition~\ref{prop:adm_linear}, will include a linearized shift vector, and is thus more general than the setup discussed in Section~\ref{sub:intro_scattering2}. 
After stating the equations, we then comment upon several features of this system, including gauge invariance (see Lemma~\ref{lem:diffeo}) and well posedness (see Proposition~\ref{prop:adm_lwp}), and remark upon how to interpret the system near $t = 0$ in Section~\ref{sub:adm_scat}.


After finishing this discussion, we project this system in frequency space using Fourier series (see Appendix~\ref{app:fourier}), in preparation for the energy estimates of Section~\ref{sec:high_freq} and Section~\ref{sec:low_freq}. Finally, we explain how different gauges and energy estimates will be applied in our proof, which will be completed in Section~\ref{sec:einstein_scat}.

\subsection{The linearized Einstein equations} \label{sub:lineinstein}

Our aim is to present the linearized Einstein--scalar field system \eqref{eq:einsteinlinear}--\eqref{eq:wavelinear} around a fixed Kasner background $(\mathcal{M}_{Kas}, g_{Kas}, \phi_{Kas})$. To make this concrete, we first define several variables which we deem \emph{linearly small} in Definition~\ref{def:linsmall}. We then derive the linearized system by Taylor expanding (to first order) the system \eqref{eq:h_evol}--\eqref{eq:momentum} around a fixed Kasner background with 
\[
    (\bar{g}_{ij}, k_{ij}, \phi, \mathbf{N}\phi, n, X^i) = (\mathring{\bar{g}}_{ij}, \mathring{k}_{ij}, p_{\phi} \log t, p_{\phi}, 1, 0),
\]
and removing all terms which are quadratic or higher order in the linearly small quantities.

Since the Kasner solution is itself a solution to the nonlinear system \eqref{eq:h_evol}--\eqref{eq:momentum}, the outcome of the procedure described above is a homogeneous linear system of PDEs where the unknowns are the linearly small quantities and the coefficients are related to the background Kasner variables. This system of PDEs is what is presented in Proposition~\ref{prop:adm_linear}.

For a further justification of the linearization procedure, we refer the reader to Section 3 of \cite{RodnianskiSpeck0}. We now define our linearly small variables, which differ slightly from those of \cite{RodnianskiSpeck0}. 

\begin{definition} \label{def:linsmall}
    We define the following \emph{linearly small} quantities (see also \eqref{eq:kasner_evol}):
    \begin{gather} \label{eq:hk}
        \matr{\upeta}{i}{j} \coloneqq - \frac{1}{2} (\mathring{\bar{g}}^{-1})^{jk} (\bar{g}_{ik} - \mathring{\bar{g}}_{ik}), \quad 
        \matr{\upkappa}{i}{j} \coloneqq t ({\bar{g}}^{-1})^{jk} k_{ik} - t (\mathring{\bar{g}}^{-1})^{jk}\mathring{k}_{ik}
        = t (\bar{g}^{-1})^{jk} k_{ik} + {p_{\underline{i}}} \matr{\delta}{\underline{i}}{j}.
        \\[0.5em] \label{eq:pp}
        \upphi \coloneqq \phi - \mathring{\phi}, \quad \uppsi \coloneqq t ( \mathbf{N} \phi - \mathring{\mathbf{N}} \mathring{\phi}) = t \, \mathbf{N}\phi - p_{\phi},
        \\[0.5em] \label{eq:nx}
        \upnu \coloneqq n - \mathring{n} = n - 1, \quad \upchi^i \coloneqq t X^i.
    \end{gather}
    We also let $\tr \upeta = \matr{\upeta}{i}{i}$ and $\tr \upkappa = \matr{\upkappa}{i}{i}$ be the usual traces. 
\end{definition}

\begin{lemma} \label{lem:adm_lin_gauge}
    A consequence of the CMC gauge condition $\tr_{\bar{g}} k = - t^{-1}$ is that:
    \begin{equation} \label{eq:cmc_linear}
        \tr \upkappa = 0.
    \end{equation}
    On the other hand, the linearized version of the spatially harmonic condition \eqref{eq:spatiallyharmonic}, which we often denote the linearized CMCSH gauge\footnote{We always demarcate when we are using a CMCSH gauge by writing the $(1, 1)$ tensors $\matr{\hat{\upeta}}{i}{j}$ and $\matr{\hat{\upkappa}}{i}{j}$ with a hat.}, is the following:
    \begin{equation} \label{eq:spatiallyharmonic_linear}
        2 \partial_j \matr{\hat{\upeta}}{i}{j} = \partial_i \tr \hat{\upeta}.
    \end{equation}
\end{lemma}

\begin{proof}
    The equation \eqref{eq:cmc_linear} follows immediately by taking the trace of \eqref{eq:hk} and using that $\tr_{\bar{g}} k = \tr_{\mathring{\bar{g}}} \mathring{k} = - t^{-1}$. In the case of spatially harmonic gauge, note that in the usual Euclidean coordinates, the condition \eqref{eq:spatiallyharmonic_general} may be written as $(\bar{g}^{-1})^{ab} \Gamma^i_{ab}[\bar{g}] = 0$, or as
    \[
        \frac{1}{2} (\bar{g}^{-1})^{ab} (\bar{g}^{-1})^{ij} \left( \partial_a \bar{g}_{jb} + \partial_b \bar{g}_{ja} - \partial_j \bar{g}_{ab} \right) = 0.
    \]

    As the background $\mathring{\bar{g}}_{ij}$ is spatially homogeneous, $\partial_i \bar{g}_{jk} = \partial_i ( \bar{g}_{jk} - \mathring{\bar{g}}_{jk} )$ is linearly small. Writing $\bar{h}_{jk} = \bar{g}_{jk} - \mathring{\bar{g}}_{jk}$, one may then Taylor expand the inverse metric $(\bar{g}^{-1})^{ab}$ to write the above equation as
    \[
        \frac{1}{2} (\mathring{g}^{-1})^{ab} ( \mathring{g}^{-1})^{ij} \left( \partial_a \bar{h}_{jb} + \partial_b \bar{h}_{ja} - \partial_j \bar{h}_{ab} \right) = \mathcal{Q}(h, \partial h),
    \]
    where $\mathcal{Q}(h, \partial h)$ is quadratic and higher order in the linearly small quantity $h$. Ignoring this term, and using $\matr{\upeta}{i}{j} = - \frac{1}{2} (\mathring{g}^{-1})^{jk} \bar{h}_{ik}$, this above equation becomes
    \[
        - 2 (\mathring{g}^{-1})^{ij} \partial_a \matr{\upeta}{j}{a} + (\mathring{g}^{-1})^{ij} \partial_j \matr{\upeta}{a}{a} = 0.
    \]
    Up to multiplication by $\mathring{g}_{ik}$ and relabelling indices, this is exactly \eqref{eq:spatiallyharmonic}.
\end{proof}

\begin{proposition}[Linearized Einstein--scalar field system] \label{prop:adm_linear}
    The linear system of PDEs for the variables $(\matr{\upeta}{i}{j}, \matr{\upkappa}{i}{j}, \upphi, \uppsi, \upnu, \upchi^j)$ obtained by linearizing the ADM system \eqref{eq:h_evol}--\eqref{eq:momentum} around a background Kasner solution, where the linearization procedure is as described at the start of Section~\ref{sub:lineinstein}, is as follows.
    \begin{enumerate}[1.]
        \item
            The metric variables $(\matr{\upeta}{i}{j}, \matr{\upkappa}{i}{j})$ obey the evolution equations
            \begin{equation} \label{eq:upeta_evol}
                t \partial_t \matr{\upeta}{i}{j} = \matr{\upkappa}{i}{j} + (2t \matr{\mathring{k}}{p}{j}) \matr{\upeta}{i}{p} - (2 t \matr{\mathring{k}}{i}{p}) \matr{\upeta}{p}{j} + (t \matr{\mathring{k}}{i}{j}) \upnu - \frac{1}{2} \partial_i \upchi^j - \frac{1}{2} \mathring{g}_{ip} \mathring{g}^{jq} \partial_q \upchi^p,
            \end{equation}
            \begin{equation} \label{eq:upkappa_evol}
                t \partial_t \matr{\upkappa}{i}{j} = t^2 \, \matr{\widecheck{\textup{Ric}}}{i}{j}[\upeta] - t^2 \mathring{g}^{jk} \partial_i \partial_k \upnu - ( t \matr{\mathring{k}}{i}{j} ) \upnu + (t \matr{\mathring{k}}{a}{j}) \partial_i \upchi^a - (t \matr{\mathring{k}}{i}{a}) \partial_a \upchi^j.
            \end{equation}
            Here, the expression $\widecheck{\textup{Ric}}[\upeta]$ is the linearized Ricci tensor, and is defined as
            \begin{equation} \label{eq:ricci_lin}
                \matr{\widecheck{\textup{Ric}}}{i}{j}[\upeta] \coloneqq \mathring{g}^{ab} \partial_a \partial_b \matr{\upeta}{i}{j} + \mathring{g}^{ja} \partial_i \partial_a (\tr \upeta) - \mathring{g}^{ab} \partial_i \partial_b \matr{\upeta}{a}{j} - \mathring{g}^{ja} \partial_a \partial_b \matr{\upeta}{i}{b}.
            \end{equation}
            In particular, if $\matr{\hat{\upeta}}{i}{j}$ moreover satisfies the spatially harmonic condition \eqref{eq:spatiallyharmonic_linear}, then
            \begin{equation} \label{eq:ricci_lin_sh}
                \matr{\widecheck{\textup{Ric}}}{i}{j}[\hat{\upeta}] = \mathring{g}^{ab} \partial_a \partial_b \matr{\hat{\upeta}}{i}{j}.
            \end{equation}
        \item
            The matter variables $(\upphi, \uppsi)$ obey
            \begin{equation} \label{eq:upphi_evol}
                t \partial_t \upphi = \uppsi + p_{\phi} \upnu,
            \end{equation}
            \begin{equation} \label{eq:uppsi_evol}
                t \partial_t \uppsi = t^2 \mathring{g}^{ab} \partial_a \partial_b \upphi - p_{\phi} \upnu.
            \end{equation}
        \item
            The linearized lapse $\upnu$ obeys the elliptic equation
            \begin{equation} \label{eq:upnu_elliptic}
                \left( 1 - t^2 \mathring{g}^{ab} \partial_a \partial_b \right) \upnu = - 2 t^2 \mathring{g}^{ab} \partial_a \partial_b (\tr \upeta) + 2 t^2 \mathring{g}^{ab} \partial_i \partial_a \matr{\upeta}{b}{i}.
            \end{equation}
            If $\matr{\hat{\upeta}}{i}{j}$ moreover satisfies the spatially harmonic condition \eqref{eq:spatiallyharmonic_linear}, this equation simplifies to
            \begin{equation} \label{eq:upnu_elliptic_sh}
                \left( 1 - t^2 \mathring{g}^{ab} \partial_a \partial_b \right) \upnu = - t^2 \mathring{g}^{ab} \partial_a \partial_b (\tr \hat{\upeta}).
            \end{equation}
            Furthermore, if \eqref{eq:spatiallyharmonic_linear} holds we have an additional elliptic equation for the shift:
            \begin{equation} \label{eq:upchi_elliptic}
                - t^2 \mathring{g}^{ab} \partial_a \partial_b \upchi^j 
                = - t^2 \mathring{g}^{ij} \partial_i \upnu - 2 t^2 \mathring{g}^{ij} ( t \matr{\mathring{k}}{i}{k} ) \partial_k \upnu
                - 4 t^2 \mathring{g}^{pq} (t \matr{\mathring{k}}{p}{i}) \partial_q \matr{\hat{\upeta}}{i}{j} 
                + 2 t^2 \mathring{g}^{ij} \partial_i [ t \matr{\mathring{k}}{b}{a} \, \matr{\hat{\upeta}}{a}{b} + 2 p_{\phi} \upphi ].
            \end{equation}
        \item
            As well as the linearized CMC condition \eqref{eq:cmc_linear}, the following constraints are satisfied:
            \begin{gather} \label{eq:hamiltonian_linear}
                - 2 t^2 \mathring{g}^{ab} \partial_a \partial_b (\tr \upeta) + 2 t^2 \mathring{g}^{ab} \partial_i \partial_a \matr{\upeta}{b}{i} + 2 t \matr{\mathring{k}}{b}{a} \, \matr{\upkappa}{a}{b} + 4 p_{\phi} \uppsi = 0.
                \\[0.5em] \label{eq:momentum_linear1}
                \partial_j \matr{\upkappa}{i}{j} + \partial_i [ t \matr{\mathring{k}}{b}{a} \, \matr{\upeta}{a}{b} + 2 p_{\phi} \upphi ] - t \matr{\mathring{k}}{i}{j} \partial_j (\tr \upeta) = 0.
                \\[0.5em] \label{eq:momentum_linear2}
                \mathring{g}^{ab} \partial_a \matr{\upkappa}{b}{c} + \mathring{g}^{cd} \partial_d [ t \matr{\mathring{k}}{b}{a} \, \matr{\upeta}{a}{b} + 2 p_{\phi} \upphi ] - 2 \mathring{g}^{ab} \, t \matr{\mathring{k}}{b}{d} \, \partial_a \matr{\upeta}{d}{c} = 0.
            \end{gather}
            As mentioned the solution is in CMCSH gauge if \eqref{eq:spatiallyharmonic_linear} is also satisfied. In this case, \eqref{eq:hamiltonian_linear} simplifies to
            \begin{gather} \label{eq:hamiltonian_linear2}
                - t^2 \mathring{g}^{ab} \partial_a \partial_b (\tr \hat{\upeta}) + 2 t \matr{\mathring{k}}{b}{a} \, \matr{\hat{\upkappa}}{a}{b} + 4 p_{\phi} \uppsi = 0.
            \end{gather}
        \item
            Finally, there are linearized symmetry conditions for $\matr{\upeta}{i}{j}$ and $\matr{\upkappa}{i}{j}$:
            \begin{gather} \label{eq:upeta_sym}
                \mathring{g}^{ab} \matr{\upeta}{b}{c} = \mathring{g}^{cb} \matr{\upeta}{b}{a},
                \\[0.5em] \label{eq:upkappa_sym}
                \mathring{g}^{ab} \left( \matr{\upkappa}{b}{c} - 2 t \matr{\mathring{k}}{b}{d} \, \matr{\upeta}{d}{c} \right)
                =
                \mathring{g}^{cb} \left( \matr{\upkappa}{b}{a} - 2 t \matr{\mathring{k}}{b}{d} \, \matr{\upeta}{d}{a} \right).
            \end{gather}
    \end{enumerate}
\end{proposition}

\begin{remark}
    We outline some of the main differences from the linearized system derived in \cite[Proposition 3.2]{RodnianskiSpeck0}. The first difference is that we use the linearly small quantity $\matr{\upeta}{i}{j}$, while \cite{RodnianskiSpeck0} directly uses the metric perturbation $\bar{h}_{ij} = \bar{g}_{ij} - \mathring{g}_{ij}$. Secondly, we do not insist on transported coordinates, and therefore allow a linearized shift vector $\upchi^j$. Thirdly, we use $\uppsi = \mathbf{N} \phi - p_{\phi}$ as our linearly small quantity, in place of $\partial_t \phi - \frac{p_{\phi}}{t}$. (\cite{RodnianskiSpeck0} do mention our $\uppsi$ as an `alternative approach'.)
\end{remark}

\begin{proof}
    We shall provide detailed proofs of four of the equations in Proposition~\ref{prop:adm_linear}, namely the evolution equations \eqref{eq:upeta_evol} and \eqref{eq:upkappa_evol}, the elliptic equation \eqref{eq:upchi_elliptic} for $\upchi^j$ in CMCSH gauge, and the symmetry condition \eqref{eq:upkappa_sym}. The remaining equations may be derived by similar methods and are left to the reader. 

    As in \cite[Proposition 3.2]{RodnianskiSpeck0}, the overall strategy will be to consider the equations in Proposition~\ref{prop:adm_evol} and Taylor expand them around the background Kasner solution, eventually discarding terms which are quadratic or higher in the linearly small quantities of Definition~\ref{def:linsmall}. We also make use of $\bar{h}_{ij} = \bar{g}_{ij} - \mathring{g}_{ij}$, and as in \cite{RodnianskiSpeck0}, at least within this proof we shall write $(\mathring{g}^{-1})^{jk}$ and $(\bar{g}^{-1})^{jk}$ to denote the inverse metrics.

    We start with the derivation of \eqref{eq:upeta_evol}. For this, the first step is to multiply \eqref{eq:h_evol} by $t (\mathring{g}^{-1})^{jk}$. Using the fact that $t \partial_t (\mathring{g}^{-1})^{jk} = 2 t \matr{\mathring{k}}{\ell}{j} (\mathring{g}^{-1})^{k \ell}$, and relabelling indices, we deduce that
    \[
        - \frac{1}{2} t \partial_t ( (\mathring{g}^{-1})^{jk} \bar{g}_{ik} ) = - t \matr{\mathring{k}}{\ell}{j} (\mathring{g}^{-1})^{k \ell} \bar{g}_{ik} + n (\mathring{g}^{-1})^{jk} t k_{ik} - \frac{1}{2} (\mathring{g}^{-1})^{jk} \bar{g}_{i\ell} \nabla_k (tX^{\ell}) - \frac{1}{2} (\mathring{g}^{-1})^{jk} \bar{g}_{k \ell} \nabla_i (t X^{\ell}).
    \]
    We now begin to Taylor expand. Writing $\bar{g}_{ab} = \mathring{g}_{ab} + \bar{h}_{ab}$ in the first, third and fourth terms of the right hand side, while using $n = 1 + \upnu$ and $(\mathring{g}^{-1})^{jk} = (\bar{g}^{-1})^{jk} + (\mathring{g}^{-1})^{ja} \bar{h}_{ab} (\bar{g}^{-1})^{bk}$ in the second term, we have
    \begin{align*}
        t \partial_t \left( - \frac{1}{2} (\mathring{g}^{-1})^{jk} \bar{g}_{ik} \right)
        &= - t \matr{\mathring{k}}{i}{j} - t \matr{\mathring{k}}{\ell}{j} (\mathring{g}^{-1})^{k\ell} \bar{h}_{ik} + (1 + \upnu) \left( t (\bar{g}^{-1})^{jk} + (\mathring{g}^{-1})^{ja} \bar{h}_{ab} \, t (\bar{g}^{-1})^{bk} \right) k_{ik} \\[0.3em]
        &\qquad - \frac{1}{2} (\mathring{g}^{-1})^{jk} (\mathring{g}_{i \ell} + \bar{h}_{i \ell}) \nabla_k \upchi^{\ell} - \frac{1}{2} (\mathring{g}^{-1})^{jk} (\mathring{g}_{k \ell} + \bar{h}_{k \ell}) \nabla_i \upchi^{\ell}.
    \end{align*}

    We now simply use that $\matr{\upeta}{i}{j} = - \frac{1}{2} (\mathring{g}^{-1})^{jk} \bar{h}_{ik}$, as well as the fact that modulo quadratic error terms, $\nabla_a$ can be replaced by $\partial_a$ since the Christoffel symbols and the renormalized shift $\upchi^{\ell}$ are both linearly small. Therefore for $\mathcal{Q}_{\upeta} = \mathcal{Q}_{\upeta}(\bar{h}, \partial \bar{h}, \upnu, \upchi, \partial \upchi)$ a quadratic and higher order nonlinearity,
    \begin{align*}
        t \partial_t \matr{\upeta}{i}{j} 
        &= t (\bar{g}^{-1})^{jk} k_{ik} - t \matr{\mathring{k}}{i}{j} + 2 t \matr{\mathring{k}}{\ell}{j} \matr{\upeta}{i}{\ell} - 2 t (\bar{g}^{-1})^{bk} k_{ik} \matr{\upeta}{b}{j} + t (\bar{g}^{-1})^{jk} k_{ik} \upnu \\[0.3em]
        &\qquad \hspace{4cm} - \frac{1}{2} (\mathring{g}^{-1})^{jk} \mathring{g}_{i \ell} \partial_k \upchi^{\ell} - \frac{1}{2} \partial_i \upchi^j + \mathcal{Q}_{\upeta}.
    \end{align*}
    Thereby upon inserting the definition of the linearly small quantity $\matr{\upkappa}{i}{j} = t (\bar{g}^{-1})^{jk} k_{ik} - t \matr{\mathring{k}}{i}{j}$, and removing all quadratic or higher order terms in the linearly small quantities, the equation \eqref{eq:upeta_evol} follows.

    The next equation we derive is \eqref{eq:upkappa_evol}. For this equation, we firstly combine the two equations \eqref{eq:h_evol} and \eqref{eq:k_evol} to deduce
    \begin{align*}
        t \partial_t \left( t (\bar{g}^{-1})^{jk} k_{ik} \right) 
        &= - t^2 (\bar{g}^{-1})^{jk} \nabla_i \nabla_k n + n t^2 (\bar{g}^{-1})^{jk} \mathrm{Ric}_{ik}[\bar{g}] - 2 n t^2 (\bar{g}^{-1})^{jk} \nabla_i \phi \nabla_k \phi - (n - 1) t (\bar{g}^{-1})^{jk} k_{ik} \\[0.3em]
        &\qquad  + t (\bar{g}^{-1})^{jk} (t X^{\ell}) \nabla_{\ell} k_{ik} + t (\bar{g}^{-1})^{jk} k_{k \ell} \nabla_k (t X^{\ell}) - t (\bar{g}^{-1})^{k \ell} k_{i k} \nabla_{\ell} (t X^j).
    \end{align*}
    We now simply insert the definitions of all the linearly small quantities, and replace all $\nabla_a$ derivatives with $\partial_a$ derivatives modulo quadratic error terms, to deduce that for some term $\mathcal{Q}_{\upkappa}$ quadratic in the linearly small quantities $\upeta$, $\partial \upeta$, $\partial^2 \upeta$, $\upnu$, $\partial \upnu$, $\partial^2 \upnu$, $\upkappa$, $\partial \upkappa$, $\upchi$, $\partial \upchi$ and $\partial \upphi$,
    \[
        t \partial_t \matr{\upkappa}{i}{j} = 
        t^2 (\mathring{g}^{-1})^{jk} \mathrm{Ric}_{ik}[\bar{g}] - t^2 (\mathring{g}^{-1})^{jk} \partial_i \partial_k \upnu - t \matr{\mathring{k}}{i}{j} \upnu + t \matr{\mathring{k}}{\ell}{j} \partial_k \upchi^{\ell} - t \matr{\mathring{k}}{i}{\ell} \partial_{\ell} \upchi^j + {\mathcal{Q}}_{\upkappa}.
    \]
    Upon using \eqref{eq:ricci_evol_0} to express $\mathrm{Ric}_{ik}[\bar{g}]$ and removing all quadratic or higher order error terms, the equation \eqref{eq:upkappa_evol} follows. (If instead we used \eqref{eq:ricci_evol} we get the desired result in the case of CMCSH gauge.)

    The next order of business is the elliptic equation \eqref{eq:upchi_elliptic}, assuming the CMCSH gauge condition \eqref{eq:spatiallyharmonic}. We shall derive this equation by using \eqref{eq:x_elliptic}, which we multiply by $t^3$ and rearrange to get
    \begin{multline*}
        t^2 (\bar{g}^{-1})^{ab} \nabla_a \nabla_b (t X^i)                 =
        t^2 (\bar{g}^{-1})^{ij} \left( 2 t\matr{k}{j}{a} \nabla_a n + \nabla_j n - 4 n p_{\phi} \nabla_j \phi \right) - 2 n t^2 (\bar{g}^{-1})^{ab} (t \matr{k}{a}{c}) \Gamma^i_{bc}[\bar{g}]
        \\ - 2 t^2 (\bar{g}^{-1})^{ab} \nabla_a (t X^c) \Gamma^i_{bc}[\bar{g}]  - t^2 (\bar{g}^{-1})^{ij}\textup{Ric}_{jk}[\bar{g}] t X^k - 4 n t^2 (\bar{g}^{-1})^{ij} (t \mathbf{N} \phi - p_{\phi}) \nabla_j \phi.
    \end{multline*}

    Now, we observe that the second line of this expression is already quadratic or higher order in our linearly small variables $(\matr{\upeta}{i}{j}, \matr{\upkappa}{i}{j}, \upphi, \uppsi, \upnu, \upchi^j)$ and their derivatives, and we therefore discard it. Similarly, we may Taylor expand all but the final term on the first line, and we deduce that for some $\mathcal{Q}_{\upchi}$ which is quadratic or higher order in the linearly small quantities, one has:
    \[
        t^2 (\mathring{g}^{-1})^{ab} \partial_a \partial_b \upchi^i = t^2 (\mathring{g}^{-1})^{ij} ( 2 t \matr{\mathring{k}}{j}{a} \partial_a \upnu + \partial_j \upnu - 4 p_{\phi} \partial_j \upphi) - 2 t^2 (\bar{g}^{-1})^{ab} (t \matr{k}{a}{c}) \Gamma_{bc}^i[\bar{g}] + \mathcal{Q}_{\upchi}.
    \]

    It remains to expand the expression $t^2(\bar{g}^{-1})^{ab} (t \matr{k}{a}{c}) \Gamma^i_{bc}[\bar{g}]$. Using the local coordinate form of Christoffel symbols, and the fact that $\partial_i \bar{g}_{jk} = \partial_i \bar{h}_{jk}$, we have
    \[
        - 2 t^2 (\bar{g}^{-1})^{ab} (t \matr{k}{a}{c}) \Gamma^i_{bc}[\bar{g}] = 
        - t^2 (\bar{g}^{-1})^{ab} (\bar{g}^{-1})^{ij} (t \matr{k}{a}{c}) \left( \partial_b \bar{h}_{jc} + \partial_c \bar{h}_{jb} - \partial_j \bar{h}_{cb} \right).
    \]
    Since the final parenthesized expression is linearly small, we may replace $(\bar{g}^{-1})^{ij}$ and $(t \matr{k}{i}{j})$ by their background Kasner analogues up to a quadratic error. Using also that $t \matr{k}{i}{j}$ is self-adjoint with respect to $\bar{g}_{ij}$, we find that up to a quadratic or higher order error we call $\mathcal{Q}_{\Gamma}$, one has
    \begin{align*}
        - 2 t^2 (\bar{g}^{-1})^{ab} (t \matr{k}{a}{c}) \Gamma^i_{bc}[\bar{g}]
        &= - 2 t^2 (\mathring{g}^{-1})^{ab} (t \matr{\mathring{k}}{a}{c}) \partial_b ((\mathring{g}^{-1})^{ij} \bar{h}_{jc}) + \partial_a \bar{h}_{jb} + t^2 (\mathring{g}^{-1})^{ij} (t \matr{\mathring{k}}{a}{c}) \partial_j ((\mathring{g}^{-1})^{ab} \bar{h}_{bc}) + \mathcal{Q}_{\Gamma} \\[0.3em]
        &= 4 t^2 (\mathring{g}^{-1})^{ab} (t \matr{\mathring{k}}{a}{c}) \partial_b \matr{\hat{\upeta}}{c}{i} - 2 t^2 (\mathring{g}^{-1})^{ij} (t \matr{\mathring{k}}{a}{c}) \partial_j \matr{\hat{\upeta}}{c}{a} + \mathcal{Q}_{\Gamma}.
    \end{align*}

    Combining this with the expression for $t^2 (\mathring{g}^{-1})^{ab} \partial_a \partial_b \upchi^i$ above, and discarding the nonlinearities $\mathcal{Q}_{\upchi}$ and $\mathcal{Q}_{\Gamma}$, we get the elliptic equation \eqref{eq:upchi_elliptic}, as required.

    The final equation we derive in detail is the symmetry constraint \eqref{eq:upkappa_sym} for $\matr{\upkappa}{i}{j}$ and $\matr{\upeta}{i}{j}$. This arises from the symmetry of the second fundamental form $k_{ij}$, which we write in the following way (as in the elliptic equation for shift, the $(1, 1)$-tensor $\matr{k}{i}{j}$ means $(\bar{g}^{-1})^{jk} k_{ik}$.)
    \[
        (\bar{g}^{-1})^{ab} (t \matr{k}{b}{c}) = (\bar{g}^{-1})^{cb} (t \matr{k}{b}{a}).
    \]
    Expanding the various terms, we therefore have
    \[
        ((\mathring{g}^{-1})^{ab} - (\mathring{g}^{-1})^{ad} \bar{h}_{de} (\bar{g}^{-1})^{eb}) (t \matr{\mathring{k}}{b}{c} + \matr{\upkappa}{b}{c}) = 
        ((\mathring{g}^{-1})^{cb} - (\mathring{g}^{-1})^{cd} \bar{h}_{de} (\bar{g}^{-1})^{eb}) (t \matr{\mathring{k}}{b}{a} + \matr{\upkappa}{b}{a}).
    \]
    The equation \eqref{eq:upkappa_sym} then follows by inserting the definition of $\matr{\upeta}{i}{j}$, rearranging, cancelling $(\mathring{g}^{-1})^{ab} (t \matr{\mathring{k}}{b}{c}) = (\mathring{g}^{-1})^{cb} (t \matr{\mathring{k}}{b}{a})$ from both sides, and finally removing quadratic and higher order terms.
\end{proof}

\subsection{Spatial gauge and well-posedness} \label{sub:adm_gauge}

Though we have fixed our foliation via the CMC condition $\tr_{\bar{g}} k = - t^{-1}$, there is still gauge freedom in the choice of linearized quantities $(\matr{\upeta}{i}{j}, \matr{\upkappa}{i}{j}, \upchi^j)$ in Definition~\ref{def:linsmall} and in Proposition~\ref{prop:adm_linear} due to freedom in the choice of spatial coordinates, which we may view as diffeomorphisms of the spatial manifold $\mathbb{T}^D$ that we allow to be time-dependent. 

While one way of removing this freedom by imposing some gauge on initial data as well as in evolution, for instance by specifying a shift vector $\upchi^j$ or alternatively by imposing the linearized CMCSH condition \eqref{eq:spatiallyharmonic_linear}, it will be useful in our context of this article to allow all such gauge choices, as well as how to change between them. At the infinitesimal or linearized level, this is achieved via a time-dependent spatial vector field $\upxi^j$ on $\mathbb{T}^D$, as in the following lemma.

\begin{lemma} \label{lem:diffeo}
    Let $(\matr{\upeta}{i}{j}, \matr{\upkappa}{i}{j}, \upphi, \uppsi, \upnu, \upchi^j)$ be a solution to the linearized Einstein--scalar field system in Proposition~\ref{prop:adm_linear} (\underline{not necessarily} obeying the linearized CMCSH condition \eqref{eq:spatiallyharmonic_linear}). Let $\upxi^j$ be a possibly time-dependent vector field on $\mathbb{T}^D$. Then the system \eqref{eq:upeta_evol}--\eqref{eq:upkappa_sym} is invariant under the gauge transformation:
    \begin{gather} \label{eq:upeta_diffeo}
        \matr{\upeta}{i}{j} \mapsto \matr{\upeta}{i}{j} + \frac{1}{2} \partial_i \upxi^j + \frac{1}{2} \mathring{g}_{ip} \mathring{g}^{jq} \partial_q \upxi^p,
        \\[0.5em] \label{eq:upkappa_diffeo}
        \matr{\upkappa}{i}{j} \mapsto \matr{\upkappa}{i}{j} + (t \matr{\mathring{k}}{i}{p}) \partial_p \upxi^j - (t \matr{\mathring{k}}{p}{j}) \partial_i \upxi^p,
        \\[0.5em] \label{eq:upchi_diffeo}
        \upchi^j \mapsto \upchi^j - t \partial_t \upxi^j,
    \end{gather}
    while the scalar variables $\upphi, \uppsi, \upnu$ are left unchanged. In particular,
    \begin{enumerate}[(i)]
        \item
            For a solution $(\matr{\upeta}{i}{j}, \matr{\upkappa}{i}{j}, \upphi, \uppsi, \upnu, \upchi^j)$ as above, then applying the above transformation using the choice of vector field $\upxi^j$ given by inverting the elliptic equation
            \[
                t^2 \mathring{g}^{ab} \partial_a \partial_b \upxi^j = t^2 \mathring{g}^{jk} \left( \partial_k \tr \upeta - 2 \partial_{\ell} \matr{{\upeta}}{k}{\ell} \right)
            \]
            yields a solution in linearized CMCSH gauge, via the transformation \eqref{eq:upeta_diffeo}--\eqref{eq:upchi_diffeo}.
        \item
            Let $\upchi^j_{\circ}$ be a fixed (possibly time-dependent) vector field. For any solution $(\matr{\upeta}{i}{j}, \matr{\upkappa}{i}{j}, \upphi, \uppsi, \upnu, \upchi^j)$ as above, then for any choice of $T > 0$ choose the vector field $\upxi^j$ to satisfy
            \[
                \upxi^j = \int^t_{T} \left( \upchi^j(s) - \upchi^j_{\circ}(s) \right) \frac{ds}{s}.
            \]
            Then after applying the transformation \eqref{eq:upeta_diffeo}--\eqref{eq:upchi_diffeo} to our solution, the transformed solution has $\upchi^j \equiv \upchi^j_{\circ}$. In particular, by choosing $\upchi^j_{\circ} \equiv 0$, we can always transform to a solution with zero shift. We say that such solutions are in \underline{linearized CMCTC gauge}.
    \end{enumerate}
\end{lemma}

\begin{proof}[(Sketch) Proof]
    We leave the proof of invariance of the system \eqref{eq:upeta_evol}--\eqref{eq:upkappa_sym} under the transformation \eqref{eq:upeta_diffeo}--\eqref{eq:upchi_diffeo} to the reader, and only mention at a heuristic level how the transformation \eqref{eq:upeta_diffeo}--\eqref{eq:upchi_diffeo} is derived. Essentially, we consider coordinate transformations of the form $\tilde{x}^j = x^j + \upxi^j$. Then the ADM-type metric \eqref{eq:adm_metric} may be rewritten as
    \begin{align*}
        g 
        &= - n^2 dt^2 + \bar{g}_{ij} (dx^i + X^i dt)(dx^j + X^j dt) \\[0.3em]
        &= - n^2 dt^2 + \bar{g}_{ij} \left( d \tilde{x}^i - \frac{\partial \upxi^i}{\partial x^k} dx^k - \frac{\partial \upxi^i}{\partial t} dt + X^i dt \right) \left( d \tilde{x}^j - \frac{\partial \upxi^j}{\partial x^{\ell}} dx^{\ell} - \frac{\partial \upxi^j}{\partial t} dt + X^i dt \right) \\[0.3em]
        &= -n^2 + \tilde{\bar{g}}_{ij} (d \tilde{x}^i + \tilde{X}^i dt) (d \tilde{x}^j + \tilde{X}^j ) + \mathcal{Q}_{\upxi}(\upxi, \partial \upxi).
    \end{align*}

    In this final line, $\mathcal{Q}_{\upchi}(\upxi, \partial \upxi)$ is a quadratic nonlinearity, while the modified first fundamental form $\tilde{\bar{g}}_{ij}$ and shift vector $\tilde{X}^j$ are given by
    \[
        \tilde{\bar{g}}_{ij} = \bar{g}_{ij} - \bar{g}_{ik} \frac{\partial \upxi^k}{\partial x^j} - \bar{g}_{jk} \frac{\partial \upxi^k}{\partial x^i}, \qquad \tilde{X}^j = X^j - \frac{\partial \upxi^j}{\partial t}.
    \]
    One now inserts $\tilde{\bar{g}}_{ij}$ and $\tilde{X}^j$ in place of $\bar{g}_{ij}$ and $X^j$ in the expressions of the linearly small quantities of Definition~\ref{def:linsmall}. (One can also suitably modify the second fundamental form). Discarding terms which are quadratic or higher in the original linearly small quantities and $\upxi^j$, we are left with \eqref{eq:upeta_diffeo}--\eqref{eq:upchi_diffeo}.

    Leaving the remaining details to the reader, we move on to (i) and (ii). For (i), suppose that the change of gauge vector field $\upxi^j$ satisfies the elliptic equation $t^2 \mathring{g}^{ab} \partial_a \partial_b \upxi^j = t^2 \mathring{g}^{jk} (\partial_k \tr \upeta - 2 \partial_{\ell} \matr{\upeta}{k}{l})$, and consider the transformed version of $\matr{\upeta}{i}{j}$, which from \eqref{eq:upeta_diffeo} we know to be
    \[
        \matr{\hat{\upeta}}{i}{j} = \matr{\upeta}{i}{j} + \frac{1}{2} \partial_i \upxi^j + \frac{1}{2} \mathring{g}_{ip} \mathring{g}^{jq} \partial_q \upxi^p.
    \]
    To check this is now in CMCSH gauge, we want to check that $2 \partial_j \matr{\hat{\upeta}}{i}{j} = \partial_i \tr \hat{\upeta}$. For this we compute:
    \begin{align*}
        t^2 \mathring{g}^{ij} ( 2 \partial_k \matr{\hat{\upeta}}{i}{k}  - \partial_i \tr \hat{\upeta} )
        &= t^2 \mathring{g}^{ij} ( 2 \partial_k \matr{\upeta}{i}{k}  - \partial_i \tr \upeta ) + t^2 \mathring{g}^{ij} \partial_k \partial_i \upxi^k + t^2 \mathring{g}^{ik} \partial_i \partial_k \upxi^j - t^2 \mathring{g}^{ij} \partial_i \partial_k \upxi^k \\[0.3em]
        &= t^2 \mathring{g}^{ij} ( 2 \partial_k \matr{\upeta}{i}{k}  + t^2 \mathring{g}^{ik} \partial_i \partial_k \upxi^j = 0.
    \end{align*}

    The final equality followed exactly due to the choice of $\upxi^j$. This completes the proof of (i). The proof of (ii) is more straightforward; one simply uses the transformation formula for the shift \eqref{eq:upchi_diffeo}.
\end{proof}

\begin{remark}
    In Lemma~\ref{lem:diffeo}, one may in particular choose a vector field $\upxi^j$ which is \emph{time-independent}. This has the effect of transforming $\matr{\hat{\upeta}}{i}{j}$ and $\matr{\hat{\upkappa}}{i}{j}$ but leaving $\upchi^j$ unchanged. In particular, even in the class of solutions with $\upchi^j = 0$ such gauge changes are still permitted. Gauge changes of this type were discussed previously in Section~\ref{sub:intro_scattering2_gauge} and were used in the statements of Theorem~\ref{thm:einstein_scat}.
\end{remark}

Upon fixing the spatial gauge, by either fixing the shift vector $\upchi^j$ or by insisting upon the linearized CMCSH condition \eqref{eq:spatiallyharmonic_linear}, we can appeal to the following well-posedness result.

\begin{proposition}[Well-posedness for the linearized Einstein--scalar field system] \label{prop:adm_lwp}
    Let $(\matr{(\upeta_C)}{i}{j}, \matr{(\upkappa_C)}{i}{j}, \upphi_C, \uppsi_C)$ be Cauchy data (for instance at $t=1$) that verifies the constraint equations \eqref{eq:cmc_linear} and \eqref{eq:hamiltonian_linear}--\eqref{eq:upkappa_sym}, with regularity
    \(
    (\matr{(\upeta_C)}{i}{j}, \matr{(\upkappa_C)}{i}{j}, \upphi_C, \uppsi_C) \in H^{s+1} \times H^s \times H^{s+1} \times H^s.
    \)
    Then both of the following well-posedness statements hold.
    \begin{enumerate}[(i)]
        \item
            Suppose furthermore that $\matr{(\hat{\upeta}_C)}{i}{j}$ verifies the linearized CMCSH condition \eqref{eq:spatiallyharmonic_linear}. Then there exists a unique CMCSH solution $(\matr{\hat{\upeta}}{i}{j}(t), \matr{\hat{\upkappa}}{i}{j}(t), \upphi(t), \uppsi(t), \upnu(t), \upchi^j(t))$ to the system \eqref{eq:upeta_evol}--\eqref{eq:upkappa_sym} for $t \in (0, + \infty)$, where $\upchi^j$ has zero average\footnote{In this paper, the vector fields $\upchi^j$ and $\upxi^j$ will always have zero average, meaning their $\lambda = 0$ Fourier coefficients $\upchi_0^j$ and $\upxi_0^j$ always vanish. Note the zero modes have no affect on how $\matr{\upeta}{i}{j}$ and $\matr{\upkappa}{i}{j}$ transform.}, and with regularity:
            \begin{gather*}
                \left( \matr{\upeta}{i}{j}, \matr{\upkappa}{i}{j} \right) \in C^0((0, + \infty), H^{s+1} \times H^s) \cap C^1((0, + \infty), H^s \times H^{s-1}), \\[0.5em]
                \left( \upphi, \uppsi \right) \in C^0((0, + \infty), H^{s+1} \times H^s) \cap C^1((0, + \infty), H^s \times H^{s-1}), \\[0.5em]
                \upnu \in C^0((0, + \infty), H^{s+2}), \qquad \upchi^j \in C^0((0, +\infty), H^{s+2}).
            \end{gather*}
        \item
            Alternatively, let $\upchi^j_{\circ}$ be any (possibly time-dependent) vector field on $\mathbb{T}^D$, with $\upchi^j_{\circ} \in C^0((0, +\infty), H^{s+2})$. Then there exists a unique solution $(\matr{\upeta}{i}{j}(t), \matr{\upkappa}{i}{j}(t), \upphi(t), \uppsi(t), \upnu(t))$ to the system \eqref{eq:upeta_evol}--\eqref{eq:upkappa_sym} for $t \in (0, + \infty)$, such that $\upchi^j \equiv \upchi^j_{\circ}$ at all times, with the same regularity as above.
    \end{enumerate}
\end{proposition}

\begin{proof}
    Let us start with (i). We first use the equations \eqref{eq:upeta_evol}, \eqref{eq:upkappa_evol}, \eqref{eq:ricci_lin_sh}, \eqref{eq:upphi_evol}, \eqref{eq:uppsi_evol}, \eqref{eq:upnu_elliptic}, \eqref{eq:upchi_elliptic}, which is a linear elliptic-hyperbolic system in the background $(\mathcal{M}_{Kas}, g_{Kas})$. (Note it is crucial that we used \eqref{eq:ricci_lin_sh} in place of \eqref{eq:ricci_lin} since only then does $\widecheck{\textup{Ric}}[\hat{\upeta}]$ appearing in \ref{eq:upkappa_evol} become a top order term of hyperbolic character.) 
    Therefore we may apply standard elliptic-hyperbolic theory, e.g.~that of \cite{AnderssonMoncrief}, to deduce that there is a solution to these equations \eqref{eq:upeta_evol}, \eqref{eq:upkappa_evol}, \eqref{eq:ricci_lin_sh}, \eqref{eq:upphi_evol}, \eqref{eq:uppsi_evol}, \eqref{eq:upnu_elliptic}, \eqref{eq:upchi_elliptic} for $t \in (0, + \infty)$. Moreover, this solution has the regularity claimed in (i).

    It remains to check that the constraint equations \eqref{eq:hamiltonian_linear}--\eqref{eq:upkappa_sym} and the gauge conditions \eqref{eq:cmc_linear} and \eqref{eq:spatiallyharmonic} also remain true for $ t \in (0, + \infty)$. The propagation of constraints is left to Proposition~\ref{prop:constraints_1}. Finally, uniqueness follows from the elliptic-hyperbolic theory, so long as we fix $\upchi^j$ to have zero average in order to have uniqueness for solutions of the elliptic equation \eqref{eq:upchi_elliptic}. 

    We now move on to (ii). In order to show existence, the first step will be to apply a gauge transformation to the Cauchy data $(\matr{(\upeta_C)}{i}{j}, \matr{(\upkappa_C)}{i}{j}, \upphi_C, \uppsi_C)$ using a vector field $\upxi^j_C$ which solves the elliptic equation
    \[
        \mathring{g}^{ab}(1) \partial_a \partial_b \upxi^j_C = \mathring{g}^{jk}(1) (\partial_k \tr \upeta_C - 2 \partial_{\ell} \matr{(\upeta_C)}{k}{\ell} ).
    \]
    As in Lemma~\ref{lem:diffeo}, this has the effect of putting the transformed $\matr{(\hat{\upeta}_C)}{i}{j} = \matr{(\upeta_C)}{i}{j} + \frac{1}{2} \partial_i \upxi^j + \frac{1}{2} \mathring{g}_{ip}(1) \mathring{g}^{jq}(1) \partial_q \upxi^p$ in CMCSH gauge, at least at $t = 1$. This allows us to use the existence result in (i), and we generate a CMCSH solution $(\matr{\hat{\upeta}}{i}{j}(t), \matr{\hat{\upkappa}}{i}{j}(t), \upphi(t), \uppsi(t), \upnu(t), \upchi^j(t))$ for $t > 0$ which achieves the Cauchy data $(\matr{(\hat{\upeta}_C)}{i}{j}, \matr{(\hat{\upkappa}_C)}{i}{j}, \upphi_C, \uppsi_C)$.

    We now perform another change of gauge to recover the solution we desire. Note that the shift vector $\upchi^j(t)$ was generated in evolution in CMCSH gauge, and thus cannot be expected to equal $\upchi^j_{\circ}$, and also that our CMCSH solution does not match our initial choice of Cauchy data $(\matr{(\upeta_C)}{i}{j}, \matr{(\upkappa_C)}{i}{j}, \upphi_C, \uppsi_C)$. We resolve both these issues by using yet another gauge transformation vector field coming from Lemma~\ref{lem:diffeo}(ii), namely 
    \[
        \upxi^j = - \upxi^j_C + \int^t_{1} \left( \upchi^j(s) - \upchi^j_{\circ}(s) \right) \frac{ds}{s}.
    \]
    After applying this gauge transformation to our CMCSH solution $(\matr{\hat{\upeta}}{i}{j}(t), \matr{\hat{\upkappa}}{i}{j}(t), \upphi(t), \uppsi(t), \upnu(t), \upchi^j(t))$, it is immediate to see that we obtain a solution satisfying the requirements of (ii). 

    The regularity is also immediate, and we finally turn to uniqueness. By taking the difference of two such solutions, uniqueness is equivalent to the statement that the only solution to the linearized Einstein--scalar field system with trivial data $(\matr{(\upeta_C)}{i}{j}, \matr{(\upkappa_C)}{i}{j}, \upphi_C, \uppsi_C) = (0, 0, 0, 0)$ and vanishing shift $\upchi^j(t) = 0$ being the zero solution. 
    Suppose, on the contrary, that there exists a non-trivial such solution $(\matr{(\upeta_{\circ})}{i}{j}(t), \matr{(\upkappa_{\circ})}{i}{j}(t), \upphi_{\circ}(t), \uppsi_{\circ}(t))$. 

    By uniqueness of CMCSH solutions, we know that this solution must not be CMCSH, hence there exists some time $T > 0$ such that $2 \partial_j \matr{(\upeta_{\circ})}{i}{j} (T) \neq \partial_i \tr \upeta_{\circ}(T)$. 
    Now we apply Lemma~\ref{lem:diffeo}(i) to this solution -- this generates, a new CMCSH solution $(\matr{(\hat{\upeta}_{\circ})}{i}{j}(t), \matr{(\hat{\upkappa}_{\circ})}{i}{j}(t), \upphi_{\circ}(t), \uppsi_{\circ}(t), \upnu_{\circ}(t), \upchi^j_{\circ}(t))$, via a transformation associated to the (zero average) vector field $\upxi^j_{\circ}$ which solves the elliptic equation
    \[
        t^2 \mathring{g}^{ab} \partial_a \partial_b \upxi^j_{\circ} = t^2 \mathring{g}^{jk} ( \partial_j \tr \upeta_{\circ} - 2 \partial_{\ell} ( \matr{(\upeta_{\circ})}{k}{\ell} ).
    \]
    Since $\matr{(\upeta_{\circ})}{i}{j} = 0$ at $t = 1$, we must have $\upxi^j_{\circ}(1) = 0$. On the other hand, by definition of $T > 0$, we must have $\upxi^j_{\circ}(T) \neq 0$. As a result, there must exist some $\hat{T} > 0$ such that $\partial_t \upxi^j_{\circ}(\hat{T}) \neq 0$, and therefore by the transformation rule \eqref{eq:upchi_diffeo} and the fact that we initially had vanishing shift, we must have $\upchi^j_{\circ} (\hat{T}) \neq 0$.

    However, by $\upxi^j_{\circ}(1) = 0$ the solution $(\matr{(\hat{\upeta}_{\circ})}{i}{j}(t), \matr{(\hat{\upkappa}_{\circ})}{i}{j}(t), \upphi_{\circ}(t), \uppsi_{\circ}(t), \upnu_{\circ}(t), \upchi^j_{\circ}(t))$ is a CMCSH solution with trivial data, and hence by uniqueness in the CMCSH gauge we must have $\upchi^j_{\circ}(t) = 0$ for all $t > 0$. This contradicts the existence of a nontrivial solution with $\upchi^j = 0$, and completes the proof of (ii).
\end{proof}

\subsection{Scattering data at \texorpdfstring{$t=0$}{t=0}} \label{sub:adm_scat}

The linearly small quantities $\matr{\upeta}{i}{j}$ and $\upphi$ generically blow up towards $t=0$, and are thus unsuitable choices for scattering data at $t=0$ itself. Therefore, in Definition~\ref{def:upupsilonupvarphi}, we introduced the renormalized quantities $\matr{\Upupsilon}{i}{j}$ and $\upvarphi$, which we recall to be: 
\begin{equation*}
    \matr{\Upupsilon}{i}{j} \coloneqq \matr{\upeta}{i}{j} + \int^1_t \mathring{g}_{ip}(s) \mathring{g}^{jq}(s) \frac{ds}{s} \cdot \matr{\upkappa}{q}{p}, \qquad
        \upvarphi \coloneqq \upphi - \uppsi \log t.
\end{equation*}

In fact, we often entertain a more general form of $\matr{\Upupsilon}{i}{j}$ and $\upvarphi$, where the time $t=1$ is replaced by $t = T > 0$. That is, we introduce the quantities $\matr{\tilde{\Upupsilon}}{i}{j}$ and $\tilde{\upvarphi}$ as
\begin{equation} \label{eq:upupsilon_tilde_0}
    \matr{\tilde{\Upupsilon}}{i}{j} \coloneqq \matr{\upeta}{i}{j} + \int^T_t \mathring{g}_{ip} (s) \mathring{g}^{jq} (s) \, \frac{ds}{s} \cdot \matr{\upkappa}{q}{p}, \qquad
    \tilde{\upvarphi} \coloneqq \upphi - \uppsi \log( \frac{t}{T} ).
\end{equation}
These more generally defined quantities will play a role in the following two lemmas.

Part of Theorem~\ref{thm:einstein_scat} is that $\matr{\Upupsilon}{i}{j}$ and $\upvarphi$ remain bounded near $t = 0$, and in fact tend to limits $\matr{(\Upupsilon_{\infty})}{i}{j}$ and $\upvarphi_{\infty}$ as $t \to 0$. This will also be true for $\matr{\upkappa}{i}{j}$ and $\uppsi$, which will tend to limits $\matr{(\upkappa_{\infty})}{i}{j}$ and $\uppsi_{\infty}$ as $t \to 0$. In this section, we prove two important properties of these renormalized quantities, or more generally the corresponding quantities $(\matr{(\upkappa_{\infty})}{i}{j}, \matr{(\tilde{\Upupsilon}_{\infty})}{i}{j}, \uppsi_{\infty}, \tilde{\upvarphi}_{\infty})$ for any $T > 0$.

Firstly, in Lemma~\ref{lem:scattering_constraints}, we show that these limiting quantities $(\matr{(\upkappa_{\infty})}{i}{j}, \matr{(\tilde{\Upupsilon}_{\infty})}{i}{j}, \uppsi_{\infty}, \tilde{\upvarphi}_{\infty})$ also obey constraint equations, analogous to those of \eqref{eq:hamiltonian_linear}--\eqref{eq:upkappa_sym}, which we often denote the \emph{asymptotic constraint equations}.
Then in Lemma~\ref{lem:upupsilon_gauge}, we determine how $\matr{\tilde{\Upupsilon}}{i}{j}$ transforms under the (linearized) change of gauge discussed in Section~\ref{sub:adm_gauge}. 

\begin{lemma} \label{lem:scattering_constraints}
    Let $s$ be sufficiently large, and let $(\matr{{\upeta}}{i}{j}, \matr{{\upkappa}}{i}{j}, \upphi, \uppsi, \upnu, \upchi^j)$ be a solution to the linearized Einstein--scalar field system in Proposition~\ref{prop:adm_linear}, with regularity as in Proposition~\ref{prop:adm_lwp}, where the Kasner exponents satisfy the subcriticality condition \eqref{eq:subcritical_delta}. For arbitrary $T > 0$, suppose furthermore that the quantities $\matr{\upkappa}{i}{j}$, $\matr{\tilde{\Upupsilon}}{i}{j}$, $\uppsi$ and $\tilde{\upvarphi}$ remain bounded towards $t = 0$, and moreover tend to a limit:
    \[
        (\matr{\upkappa}{i}{j}, \matr{\tilde{\Upupsilon}}{i}{j}, \uppsi, \tilde{\upvarphi}) \to (\matr{(\upkappa_{\infty})}{i}{j}, \matr{(\tilde{\Upupsilon}_{\infty})}{i}{j}, \uppsi_{\infty}, \tilde{\upvarphi}_{\infty}) \quad \text{ in } C^3 \text{ as } t\to 0.
    \]

    Then the following \underline{asymptotic constraint equations} hold: 
    \begin{gather} \label{eq:hamiltonian_linear_scat}
        t \matr{\mathring{k}}{b}{a} \, \matr{(\upkappa_{\infty})}{a}{b} + 2 p_{\phi} \uppsi_{\infty} = 0,
        \\[0.5em] \label{eq:momentum_linear1_scat}
        \partial_j \matr{(\upkappa_{\infty})}{i}{j} + \partial_i [ t \matr{\mathring{k}}{b}{a} \, \matr{({\tilde{\Upupsilon}}_{\infty})}{a}{b} + 2 p_{\phi} \tilde{\upvarphi}_{\infty} ] - t \matr{\mathring{k}}{i}{j} \partial_j (\tr \tilde{\Upupsilon}_{\infty}) = 0,
        \\[0.5em] \label{eq:momentum_linear2_scat}
        \mathring{g}^{ab}(T) \partial_a \matr{(\upkappa_{\infty})}{b}{c} + \mathring{g}^{cd}(T) \partial_d [ t \matr{\mathring{k}}{b}{a} \, \matr{(\tilde{\Upupsilon}_{\infty})}{a}{b} + 2 p_{\phi} \tilde{\upvarphi}_{\infty} ] - 2 \mathring{g}^{ab}(T) \, t \matr{\mathring{k}}{b}{d} \, \partial_a \matr{(\tilde{\Upupsilon}_{\infty})}{d}{c} = 0,
    \end{gather}
    as well as the following symmetry conditions:
    \begin{gather} \label{eq:upeta_sym_scat}
        \mathring{g}^{ab}(T) \matr{(\tilde{\Upupsilon}_{\infty})}{b}{c} = \mathring{g}^{cb}(T) \matr{(\tilde{\Upupsilon}_{\infty})}{b}{a},
        \\[0.5em] \label{eq:upkappa_sym_scat}
        \mathring{g}^{ab}(T) \left( \matr{(\upkappa_{\infty})}{b}{c} - 2 t \matr{\mathring{k}}{b}{d} \, \matr{(\tilde{\Upupsilon}_{\infty})}{d}{c} \right)
        =
        \mathring{g}^{cb}(T) \left( \matr{(\upkappa_{\infty})}{b}{a} - 2 t \matr{\mathring{k}}{b}{d} \, \matr{(\tilde{\Upupsilon}_{\infty})}{d}{a} \right).
    \end{gather} 
    Finally, we have the linearized CMC condition $\tr \upkappa_{\infty} = 0$. 
\end{lemma}

\begin{proof}
    For \eqref{eq:hamiltonian_linear_scat} and \eqref{eq:momentum_linear1_scat}, the idea is to simply take the usual constraints \eqref{eq:hamiltonian_linear} and \eqref{eq:momentum_linear1}, rewrite them in terms of the renormalized variables $\matr{\tilde{\Upupsilon}}{i}{j}$, $\tilde{\upvarphi}$, and then finally take the limit $t \to 0$. For instance, we rewrite the Hamiltonian constraint \eqref{eq:hamiltonian_linear} using \eqref{eq:upupsilon} as
    \[
        - 2 t^2 \mathring{g}^{ab} \partial_a \partial_b \tr \tilde{\Upupsilon} + 2 t^2 \mathring{g}^{ab} \partial_i \partial_a \matr{\tilde{\Upupsilon}}{b}{i} 
        - 2 t^2 \mathring{g}^{ab} \int^T_t \mathring{g}_{bp}(s) \mathring{g}^{iq}(s)\frac{ds}{s} \cdot \partial_i \partial_a \matr{\upkappa}{q}{p}
        + 2(t \matr{\mathring{k}}{b}{a}) \matr{\upkappa}{a}{b} + 4 p_{\phi} \uppsi = 0.
    \]
    We now take the limit of this equation as $t \to 0$. The last two terms will give us the desired \eqref{eq:hamiltonian_linear_scat}, so it suffices to show that the remaining terms tend to $0$.

    As $\partial_a \partial_b \matr{\tilde{\Upupsilon}}{i}{j}$ and $\partial_a \partial_b \matr{\upkappa}{i}{j}$ are bounded towards $t \to 0$ by assumption, it suffices to show that as $t \to 0$,
    \[
        t^2 \mathring{g}^{ab} \to 0 \quad \text{ and } \quad t^2 \mathring{g}^{ab} \int^1_t \mathring{g}_{bp}(s) \mathring{g}^{iq}(s) \frac{ds}{s} = 0.
    \]
    In fact both of these will follow by the non-degeneracy of the Kasner exponents; using a modification of Lemma~\ref{lem:gbound} (with $t_{\lambda*}$ replaced by $1$ in the lemma) to estimate the integral above, one shows these terms are bounded by $\max_{i} t^{2 - 2p_i} (1 + |\log t|)$ and thus tend to $0$ as $t \to 0$.

    Moving to the momentum constraint, we rewrite \eqref{eq:momentum_linear1} as follows: 
    \begin{align*}
        \partial_i \matr{\tilde{\Upupsilon}}{i}{j} + \partial_i [ (t \matr{\mathring{k}}{b}{a}) \matr{\tilde{\Upupsilon}}{a}{b} + 2 p_{\phi} \tilde{\upvarphi}] - (t \matr{\mathring{k}}{i}{j}) \partial_j \tr \tilde{\Upupsilon} = - \partial_i [ (t \matr{\mathring{k}}{b}{a}) \matr{\upkappa}{a}{b} + 2 p_{\phi} \uppsi] \log t.
    \end{align*}
    Note that we used the fact that $t \matr{\mathring{k}}{b}{a}$ is non-zero only when $a = b$, hence the integral appearing in \eqref{eq:upupsilon} ends up being $\log t$. The right hand side of this rewritten momentum constraint appears exactly in the rewritten Hamiltonian constraint above, and (given that third derivatives of $\matr{\tilde{\Upupsilon}}{i}{j}$ and $\matr{\upkappa}{i}{j}$ are also bounded) is therefore bounded by $\max_i t^{2 - 2p_i} (1 + |\log t|^2)$. Thus taking the limit $t \to 0$ yields \eqref{eq:momentum_linear1_scat}.

    We next move to the symmetry conditions \eqref{eq:upeta_sym_scat} and \eqref{eq:upkappa_sym_scat}. We shall actually prove that for all times $t > 0$, we have the symmetry conditions
    \begin{gather} \label{eq:upeta_sym_scat_0}
        \mathring{g}^{ab}(T) \matr{\tilde{\Upupsilon}}{b}{c}(t) = \mathring{g}^{cb}(T) \matr{\tilde{\Upupsilon}}{b}{a}(t),
        \\[0.5em] \label{eq:upkappa_sym_scat_0}
        \mathring{g}^{ab}(T) \left( \matr{\upkappa}{b}{c}(t) - 2 t \matr{\mathring{k}}{b}{d} \, \matr{\tilde{\Upupsilon}}{d}{c}(t) \right)
        =
        \mathring{g}^{cb}(T) \left( \matr{\upkappa}{b}{a}(t) - 2 t \matr{\mathring{k}}{b}{d} \, \matr{\tilde{\Upupsilon}}{d}{a}(t) \right),
    \end{gather} 
%
    and simply take the limit of these as $t \to 0$.
    Note that \eqref{eq:upeta_sym_scat_0} is equivalent to $\mathring{g}^{ac} (T) \mathring{g}_{bd} (T) \matr{\tilde{\Upupsilon}}{c}{d} = \matr{\tilde{\Upupsilon}}{b}{a}$, or
    \begin{equation} \label{eq:upupsilon_wish}
        \mathring{g}^{ac}(T) \mathring{g}_{bd} (T) \, \matr{\upeta}{c}{d} + \mathring{g}^{ac} (T) \mathring{g}_{bd} (T) \int^T_t \mathring{g}_{cp}(s) \mathring{g}^{dq}(s) \frac{ds}{s} \cdot \matr{\upkappa}{q}{p} = \matr{\upeta}{b}{a} + \int^T_t \mathring{g}_{bp}(s) \mathring{g}^{aq}(s) \frac{ds}{s} \cdot \matr{\upkappa}{q}{p}.
    \end{equation}
    We now make the following observation. Using the explicit form of $\mathring{g}_{ij}$ in \eqref{eq:kasner_metric}, one has the identity:
    \[
        \mathring{g}^{ac} (T) \mathring{g}_{bd} (T) \int^T_t \mathring{g}_{cp}(s) \mathring{g}^{dq}(s) \frac{ds}{s} = \mathring{g}_{cp}(t) \mathring{g}^{dq}(t) \int^T_t \mathring{g}^{ac}(s) \mathring{g}_{bd}(s) \frac{ds}{s}.
    \]
    One obtains this identity using a substitution of the form $s \mapsto \tilde{s} = \frac{Tt}{s}$ in the integral.

    Using this identity, one computes the second term on the left hand side of \eqref{eq:upupsilon_wish} to be
    \begin{align*}
        \mathring{g}^{ac} (T) \mathring{g}_{bd} (T) \int^T_t \mathring{g}_{cp}(s) \mathring{g}^{dq}(s) \frac{ds}{s} \cdot \matr{\upkappa}{q}{p} 
        &=
        \int^T_t \mathring{g}^{ac}(s) \mathring{g}_{bd}(s) \frac{ds}{s} \cdot \mathring{g}_{cp}(t) \mathring{g}^{dq}(t)\, \matr{\upkappa}{q}{p}(t) \\[0.5em]
        &=
        \int^T_t \mathring{g}^{ac}(s) \mathring{g}_{bd}(s) \frac{ds}{s} \cdot \left( \matr{\upkappa}{c}{d} + (2 t \matr{\mathring{k}}{p}{d}) \matr{\upeta}{c}{p} - (2 t \matr{\mathring{k}}{c}{p}) \matr{\upeta}{p}{d} \right),
    \end{align*}
    where we used the symmetry condition \eqref{eq:upkappa_sym} in the second step. We next observe that we can identify the latter two terms on the right hand side in the following way:
    \[
        \mathring{g}^{ac} (T) \mathring{g}_{bd} (T) \int^T_t \mathring{g}_{cp}(s) \mathring{g}^{dq}(s) \frac{ds}{s} \cdot \matr{\upkappa}{q}{p} =
        \int^T_t \mathring{g}^{ac}(s) \mathring{g}_{bd}(s) \frac{ds}{s} \cdot \matr{\upkappa}{c}{d} - \int^T_t s \partial_s ( \mathring{g}^{ac} \mathring{g}_{bd} ) \frac{ds}{s} \cdot \matr{\upeta}{c}{d}.
    \]

    Therefore we obtain \eqref{eq:upupsilon_wish} by using the fundamental theorem of calculus, and simply relabelling indices. This completes the proof of \eqref{eq:upeta_sym_scat_0}. For \eqref{eq:upkappa_sym_scat_0}, we firstly observe that
    \[
        \matr{\upkappa}{b}{c} - (2 t \matr{\mathring{k}}{b}{d}) \matr{\tilde{\Upupsilon}}{d}{c} = 
        \matr{\upkappa}{b}{c} - (2 t \matr{\mathring{k}}{b}{d}) \matr{\upeta}{d}{c} + \int^T_t s \partial_s (\mathring{g}_{bp}(s)) \mathring{g}^{cq} \frac{ds}{s} \cdot \matr{\upkappa}{q}{p}.
    \]
    Therefore integrating by parts and using \eqref{eq:upkappa_sym} in a similar way to before, one eventually finds that
    \[
        \matr{\upkappa}{b}{c} - (2 t \matr{\mathring{k}}{b}{d}) \matr{\tilde{\Upupsilon}}{d}{c} = 
        \mathring{g}_{bp}(T) \mathring{g}^{cq}(T) \, \matr{\upkappa}{q}{p} - (2 t \matr{\mathring{k}}{d}{c}) \matr{\tilde{\Upupsilon}}{b}{d}.
    \]
    From this identity, \eqref{eq:upkappa_sym_scat_0} follows from \eqref{eq:upeta_sym_scat_0}. We also deduce \eqref{eq:upeta_sym_scat} and \eqref{eq:upkappa_sym_scat} by taking $t \to 0$.

    Finally, the second momentum constraint \eqref{eq:momentum_linear2_scat} follows from combining \eqref{eq:momentum_linear1_scat} with \eqref{eq:upeta_sym_scat} and \eqref{eq:upkappa_sym_scat}.
\end{proof}

\begin{lemma} \label{lem:upupsilon_gauge}
    Consider the renormalized quantity $\matr{\tilde{\Upupsilon}}{i}{j}$ defined in \eqref{eq:upupsilon_tilde_0}, for some $T > 0$.
    Then the change of gauge associated the vector field $\upxi^j$, which transforms $\matr{\upeta}{i}{j}$ and $\matr{\upkappa}{i}{j}$ as in \eqref{eq:upeta_diffeo} and \eqref{eq:upkappa_diffeo}, then transforms the expression $\matr{\tilde{\Upupsilon}}{i}{j}$ as follows:
    \begin{equation} \label{eq:upupsilon_diffeo}
        \matr{\tilde{\Upupsilon}}{i}{j} \mapsto \matr{\tilde{\Upupsilon}}{i}{j} + \frac{1}{2} \partial_i \upxi^j + \frac{1}{2} \mathring{g}_{ip}(T) \mathring{g}^{jq}(T) \partial_q \upxi^p.
    \end{equation}

    In particular, choosing $\upxi^j$ to be a solution of the elliptic equation
    \begin{equation} \label{eq:upxi_elliptic}
        \mathring{g}^{ab} (T) \partial_a \partial_b \upxi^j = \mathring{g}^{jk} (T) \left( \partial_k \tr \tilde{\Upupsilon} - 2 \partial_{\ell} \matr{\tilde{\Upupsilon}}{k}{\ell} \right),
    \end{equation}
    one can transform $\matr{\tilde{\Upupsilon}}{i}{j}$ in such a way that the transformed version satisfies the differential relation
    \[
        \partial_i \tr \tilde{\Upupsilon} - 2 \partial_j \matr{\tilde{\Upupsilon}}{i}{j} = 0.
    \]
\end{lemma}

\begin{proof}
    From the definition \eqref{eq:upupsilon_tilde_0}, the transformation associated to $\upxi^j$ transforms $\matr{\tilde{\Upupsilon}}{i}{j}$ as follows:
    \begin{align*}
        \matr{\tilde{\Upupsilon}}{i}{j} 
        &\mapsto \left( \matr{\upeta}{i}{j} + \frac{1}{2} \partial_i \upxi^j + \frac{1}{2} \mathring{g}_{ip}(t) \mathring{g}^{jq}(t) \partial_q \upxi^p \right)
        + \int^T_t \mathring{g}_{ip}(s) \mathring{g}^{jq}(s) \frac{ds}{s} \cdot \left( \matr{\upkappa}{q}{p} + (t \matr{\mathring{k}}{q}{s}) \partial_s \upxi^p - (t \matr{\mathring{k}}{s}{p}) \partial_q \upxi^s \right) \\[0.5em]
        &= \matr{\tilde{\Upupsilon}}{i}{j} + \frac{1}{2} \partial_i \upxi^j + \frac{1}{2} \mathring{g}_{ip}(t) + \mathring{g}^{jq}(t) \partial_q \upxi^p + \int^T_t s \partial_s ( \mathring{g}_{ip}(s) \mathring{g}^{jq}(s) ) \frac{ds}{s} \cdot \partial_q \upxi^p.
    \end{align*}
    The transformation formula \eqref{eq:upupsilon_diffeo} then follows by applying the fundamental theorem of calculus.

    The second part of the lemma, concerning the change of gauge associated to the vector field $\upxi^j$ solving the elliptic equation \eqref{eq:upxi_elliptic}, follows from the same argument as Lemma~\ref{lem:diffeo}(i).
\end{proof}

\subsection{Fourier decomposition}

The proof of the scattering results will involve a Fourier decomposition and detailed ODE analysis for each Fourier mode. For this purpose, we will project the whole linearized Einstein--scalar field system of Proposition~\ref{prop:adm_linear} onto the Fourier mode corresponding to $\lambda \in \Z^D$, as below.

\begin{proposition}[Linearized Einstein--scalar field in Fourier space]\label{prop:einstein_linear_l}
    For $\lambda \in \Z^D$, suppose that \linebreak$(\matr{(\upeta_{\lambda})}{i}{j}, \matr{(\upkappa_{\lambda})}{i}{j}, \upphi_{\lambda}, \uppsi_{\lambda}, \upnu_{\lambda}, \upxi^j_{\lambda})$ is the projection of a solution to the linearized Einstein--scalar field system \eqref{eq:upeta_evol}--\eqref{eq:upkappa_sym} onto the $\lambda$-Fourier mode.

    Our solutions will always be CMC in the sense that $\tr \upkappa = 0$, and we say that the $\lambda$-Fourier mode is (linearly) spatially harmonic if the following holds (again the hat denotes CMCSH):
    \begin{equation} \label{eq:spatiallyharmonic_linear_l}
        2 \lambda_j \matr{(\hat{\upeta}_{\lambda})}{i}{j} = \lambda_i \tr \hat{\upeta}_{\lambda}.
    \end{equation}
    As before, we demarcate the occasions when we assume the CMCSH gauge \eqref{eq:spatiallyharmonic_linear_l} using a hat. Then the projected variables $(\matr{(\upeta_{\lambda})}{i}{j}, \matr{(\upkappa_{\lambda})}{i}{j}, \upphi_{\lambda}, \uppsi_{\lambda}, \upnu_{\lambda}, \upxi^j_{\lambda})$ obey the following equations.

    \begin{enumerate}[1.]
        \item
            The metric variables $(\matr{(\upeta_{\lambda})}{i}{j}, \matr{(\upkappa_{\lambda})}{i}{j})$ obey the evolution equations:
            \begin{equation} \label{eq:upeta_evol_l}
                t \partial_t \matr{(\upeta_{\lambda})}{i}{j} = \matr{(\upkappa_{\lambda})}{i}{j} + (2t \matr{\mathring{k}}{p}{j}) \matr{(\upeta_{\lambda})}{i}{p} - (2 t \matr{\mathring{k}}{i}{p}) \matr{(\upeta_{\lambda})}{p}{j} + (t \matr{\mathring{k}}{i}{j}) \upnu_{\lambda} + \frac{\mathrm{i}}{2} \lambda_i \upchi_{\lambda}^j + \frac{\mathrm{i}}{2} \mathring{g}_{ip} \mathring{g}^{jq} \lambda_q \upchi_{\lambda}^p,
            \end{equation}
            \begin{equation} \label{eq:upkappa_evol_l}
                t \partial_t \matr{(\upkappa_{\lambda})}{i}{j} = t^2 \, \matr{(\widecheck{\textup{Ric}}_{\lambda})}{i}{j}[\upeta_{\lambda}] + t^2 \mathring{g}^{jk} \lambda_i \lambda_k \upnu_{\lambda} - ( t \matr{\mathring{k}}{i}{j} ) \upnu_{\lambda} - \mathrm{i} (t \matr{\mathring{k}}{a}{j}) \lambda_i \upchi_{\lambda}^a + \mathrm{i} (t \matr{\mathring{k}}{i}{a}) \lambda_a \upchi_{\lambda}^j.
            \end{equation}
            Here, the expression $(\widecheck{\textup{Ric}}_{\lambda})[\upeta_{\lambda}]$ is defined by
            \begin{equation} \label{eq:ricci_lin_l}
                t^2 \matr{\widecheck{\textup{Ric}}}{i}{j}[(\upeta_{\lambda})] \coloneqq 
                - \tau^2 \matr{(\upeta_{\lambda})}{i}{j} - t^2 \mathring{g}^{ja} \lambda_i \lambda_a (\tr \upeta_{\lambda}) + t^2 \mathring{g}^{ab} \lambda_i \lambda_b \matr{(\upeta_{\lambda})}{a}{j} + t^2 \mathring{g}^{ja} \lambda_a \lambda_b \matr{(\upeta_{\lambda})}{i}{b}.
            \end{equation}
            In particular, if $\matr{(\hat{\upeta}_{\lambda})}{i}{j}$ satisfies the spatially harmonic condition \eqref{eq:spatiallyharmonic_linear}, then
            \begin{equation} \label{eq:ricci_lin_sh_l}
                t^2 \matr{(\widecheck{\textup{Ric}}_{\lambda})}{i}{j}[\hat{\upeta}_{\lambda}] = - \tau^2 \matr{(\hat{\upeta}_{\lambda})}{i}{j}.
            \end{equation}
        \item
            The matter variables $(\upphi_{\lambda}, \uppsi_{\lambda})$ obey
            \begin{equation} \label{eq:upphi_evol_l}
                t \partial_t \upphi_{\lambda} = \uppsi_{\lambda} + p_{\phi} \upnu_{\lambda},
            \end{equation}
            \begin{equation} \label{eq:uppsi_evol_l}
                t \partial_t \uppsi_{\lambda} = - \tau^2 \upphi_{\lambda} - p_{\phi} \upnu_{\lambda}.
            \end{equation}
        \item
            The linearized lapse $\upnu_{\lambda}$ obeys the elliptic equation
            \begin{equation} \label{eq:upnu_elliptic_l}
                \left( 1 + \tau^2 \right) \upnu_{\lambda} = 2 \tau^2 (\tr \upeta_{\lambda}) - 2 t^2 \tilde{g}^{ab} \lambda_i \lambda_a \matr{(\upeta_{\lambda})}{b}{i}.
            \end{equation}
            If $\matr{(\hat{\upeta}_{\lambda})}{i}{j}$ moreover satisfies the spatially harmonic condition \eqref{eq:spatiallyharmonic_linear}, this equation simplifies to
            \begin{equation} \label{eq:upnu_elliptic_sh_l}
                \left( 1 + \tau^2\right) \upnu_{\lambda} = \tau^2 (\tr \hat{\upeta}_{\lambda}).
            \end{equation}
            Furthermore, if \eqref{eq:spatiallyharmonic_linear_l} holds we have an additional elliptic equation for the shift:
            \begin{equation} \label{eq:upchi_elliptic_l}
                \tau^2 \upchi_{\lambda}^j 
                = \mathrm{i} t^2 \mathring{g}^{ij} \lambda_i \upnu_{\lambda} + 2 \mathrm{i} t^2 \mathring{g}^{ij} ( t \matr{\mathring{k}}{i}{k} ) \lambda_k \upnu_{\lambda}
                + 4 \mathrm{i} t^2 \mathring{g}^{pq} (t \matr{\mathring{k}}{p}{i}) \lambda_q \matr{(\hat{\upeta}_{\lambda})}{i}{j} 
                - 2 \mathrm{i} t^2 \mathring{g}^{ij} \lambda_i [ t \matr{\mathring{k}}{b}{a} \, \matr{(\hat{\upeta}_{\lambda})}{a}{b} + 2 p_{\phi} \upphi_{\lambda} ].
            \end{equation}
        \item
            As well as the linearized CMC condition $\tr \upkappa_{\lambda} = 0$, the following constraints are satisfied:
            \begin{gather} \label{eq:hamiltonian_linear_l}
                2 \tau^2 (\tr \upeta_{\lambda}) - 2 t^2 \mathring{g}^{ab} \lambda_i \lambda_a \matr{(\upeta_{\lambda})}{b}{i} + 2 t \matr{\mathring{k}}{b}{a} \, \matr{(\upkappa_{\lambda})}{a}{b} + 4 p_{\phi} \uppsi_{\lambda} = 0.
                \\[0.5em] \label{eq:momentum_linear1_l}
                \lambda_j \matr{(\upkappa_{\lambda})}{i}{j} + \lambda_i [ t \matr{\mathring{k}}{b}{a} \, \matr{(\upeta_{\lambda})}{a}{b} + 2 p_{\phi} \upphi_{\lambda} ] - t \matr{\mathring{k}}{i}{j} \lambda_j (\tr \upeta_{\lambda}) = 0.
                \\[0.5em] \label{eq:momentum_linear2_l}
                \mathring{g}^{ab} \lambda_a \matr{(\upkappa_{\lambda})}{b}{c} + \mathring{g}^{cd} \lambda_d [ t \matr{\mathring{k}}{b}{a} \, \matr{(\upeta_{\lambda})}{a}{b} + 2 p_{\phi} \upphi_{\lambda} ] - 2 \mathring{g}^{ab} \, t \matr{\mathring{k}}{b}{d} \, \lambda_a \matr{(\upeta_{\lambda})}{d}{c} = 0.
            \end{gather}
            As mentioned the solution is in CMCSH gauge if \eqref{eq:spatiallyharmonic_linear_l} is also satisfied. In this case, \eqref{eq:hamiltonian_linear_l} simplifies to
            \begin{gather} \label{eq:hamiltonian_linear2_l}
                \tau^2 (\tr \hat{\upeta}_{\lambda}) + 2 t \matr{\mathring{k}}{b}{a} \, \matr{(\hat{\upkappa}_{\lambda})}{a}{b} + 4 p_{\phi} \uppsi_{\lambda} = 0.
            \end{gather}
        \item
            Finally, there are linearized symmetry conditions for $\matr{\upeta}{i}{j}$ and $\matr{\upkappa}{i}{j}$, which are also propagated:
            \begin{equation} \label{eq:upeta_sym_l}
                \mathring{g}^{ab} \matr{(\upeta_{\lambda})}{b}{c} = \mathring{g}^{cb} \matr{(\upeta_{\lambda})}{b}{a},
            \end{equation}
            \begin{equation} \label{eq:upkappa_sym_l}
                \mathring{g}^{ab} \left( \matr{(\upkappa_{\lambda})}{b}{c} - 2 t \matr{\mathring{k}}{b}{d} \, \matr{(\upeta_{\lambda})}{d}{c} \right)
                =
                \mathring{g}^{cb} \left( \matr{(\upkappa_{\lambda})}{b}{a} - 2 t \matr{\mathring{k}}{b}{d} \, \matr{(\upeta_{\lambda})}{d}{a} \right).
            \end{equation}
    \end{enumerate}
\end{proposition}

\begin{proof}
    The proof follows immediately simply by replacing any spatial derivative $\partial_i$ appearing in the systen \eqref{eq:upeta_evol}--\eqref{eq:upkappa_sym} by the symbol $- \mathrm{i} \lambda_i$, and simplifying accordingly.
\end{proof}

It is also possible to translate Lemma~\ref{lem:diffeo}, Lemma~\ref{lem:scattering_constraints} and Lemma~\ref{lem:upupsilon_gauge} into frequency space. To keep things brief, of these we only write the Fourier projections of the asymptotic constraints \eqref{eq:hamiltonian_linear_scat}--\eqref{eq:upkappa_sym_scat}.

\begin{lemma} \label{lem:scattering_constraints_l}
    Let $(\matr{(\upkappa_{\infty})}{i}{j}, \matr{(\tilde{\Upupsilon}_{\infty})}{i}{j}, \uppsi_{\infty}, \tilde{\upvarphi}_{\infty})$ be as in Lemma~\ref{lem:scattering_constraints}, for some $T > 0$. Then for $\lambda > 0$, the following asymptotic constraints hold between the $\lambda$-Fourier projections $(\matr{((\upkappa_{\infty})_{\lambda})}{i}{j}, \matr{((\tilde{\Upupsilon}_{\infty})_{\lambda})}{i}{j}, (\uppsi_{\infty})_{\lambda}, (\tilde{\upvarphi}_{\infty})_{\lambda})$:
    \begin{gather} \label{eq:hamiltonian_linear_scat_l}
        t \matr{\mathring{k}}{b}{a} \, \matr{((\upkappa_{\infty})_{\lambda})}{a}{b} + 2 p_{\phi} (\uppsi_{\infty})_{\lambda} = 0,
        \\[0.5em] \label{eq:momentum_linear1_scat_l}
        \lambda_j \matr{((\upkappa_{\infty})_{\lambda})}{i}{j} + \lambda_i [ t \matr{\mathring{k}}{b}{a} \, \matr{(({\tilde{\Upupsilon}}_{\infty})_{\lambda})}{a}{b} + 2 p_{\phi} (\tilde{\upvarphi}_{\infty})_{\lambda} ] - t \matr{\mathring{k}}{i}{j} \lambda_j (\tr \tilde{\Upupsilon}_{\infty})_{\lambda} = 0,
        \\[0.5em] \label{eq:momentum_linear2_scat_l}
        \mathring{g}^{ab}(T) \lambda_a \matr{((\upkappa_{\infty})_{\lambda})}{b}{c} + \mathring{g}^{cd}(T) \lambda_d [ t \matr{\mathring{k}}{b}{a} \, \matr{((\tilde{\Upupsilon}_{\infty})_{\lambda})}{a}{b} + 2 p_{\phi} (\tilde{\upvarphi}_{\infty})_{\lambda} ] - 2 \mathring{g}^{ab}(T) \, t \matr{\mathring{k}}{b}{d} \, \lambda_a \matr{((\tilde{\Upupsilon}_{\infty})_{\lambda})}{d}{c} = 0,
        \\[0.5em] \label{eq:upeta_sym_scat_l}
        \mathring{g}^{ab}(T) \matr{((\tilde{\Upupsilon}_{\infty})_{\lambda})}{b}{c} = \mathring{g}^{cb}(T) \matr{((\tilde{\Upupsilon}_{\infty})_{\lambda})}{b}{a},
        \\[0.5em] \label{eq:upkappa_sym_scat_l}
        \mathring{g}^{ab}(T) \left( \matr{((\upkappa_{\infty})_{\lambda})}{b}{c} - 2 t \matr{\mathring{k}}{b}{d} \, \matr{((\tilde{\Upupsilon}_{\infty})_{\lambda})}{d}{c} \right)
        =
        \mathring{g}^{cb}(T) \left( \matr{((\upkappa_{\infty})_{\lambda})}{b}{a} - 2 t \matr{\mathring{k}}{b}{d} \, \matr{((\tilde{\Upupsilon}_{\infty})_{\lambda})}{d}{a} \right).
    \end{gather} 
\end{lemma}

\begin{proof}
    Again just replace $\partial_i$ by $- \mathrm{i} \lambda_i$ in Lemma~\ref{lem:scattering_constraints}.
\end{proof}



\section{Linearized Einstein--scalar field: the high-frequency energy estimate} \label{sec:high_freq}

In this section, we derive a high-frequency energy estimate similar to that of Proposition~\ref{prop:wave_high_freq}, for the Fourier modes of the linearized Einstein--scalar field system given in Proposition~\ref{prop:einstein_linear_l}. 
The high-frequency energy estimates of this section will \emph{always use the CMCSH gauge}, in particular the condition \eqref{eq:spatiallyharmonic_linear_l}, and we make this clear by denoting $\matr{\hat{\upeta}}{i}{j}$ and $\matr{\hat{\upkappa}}{i}{j}$ with a hat.

\subsection{The high-frequency energy quantity}

\begin{definition} \label{def:high_freq}
    The \emph{geometric high-frequency energy quantity} $\mathcal{E}_{\lambda, \upeta, high}(t)$, used for $\lambda \in \Z^D \setminus \{0 \}$ and $\tau_{\lambda}(t) \geq 1$ large, is given by the following:
    \begin{equation} \label{eq:geometry_energy_high}
        \mathcal{E}_{\lambda, \upeta, high}(t) \coloneqq \frac{\zeta^2}{\tau} \matr{(\hat{\upkappa}_{\lambda})}{i}{j} \matr{(\hat{\upkappa}_{\lambda})}{j}{i} + \frac{\zeta}{\tau} \matr{(\hat{\upkappa}_{\lambda})}{i}{j} \matr{(\hat{\upeta}_{\lambda})}{j}{i} + \left( \frac{1}{2 \tau} + \zeta^2 \tau \right) \matr{(\hat{\upeta}_{\lambda})}{i}{j} \matr{(\hat{\upeta}_{\lambda})}{j}{i}.
    \end{equation}
    We also define the \emph{matter high-frequency energy quantity} $\mathcal{E}_{\lambda, \upphi, high}(t)$ as follows:
    \begin{equation} \label{eq:matter_energy_high}
        \mathcal{E}_{\lambda, \upphi, high}(t) \coloneqq \frac{\zeta^2}{\tau} \uppsi_{\lambda}^2 + \frac{\zeta}{\tau} \uppsi_{\lambda} \upphi_{\lambda} + \left( \frac{1}{2 \tau} + \zeta^2 \tau \right) \upphi_{\lambda}^2.
    \end{equation}
    Finally, define the \emph{total high-frequency energy quantity} as the sum $\mathcal{E}_{\lambda, high}(t) \coloneqq \mathcal{E}_{\lambda, \upeta, high}(t) + \mathcal{E}_{\lambda, \upphi, high}(t)$.
\end{definition}

\begin{remark}
    Since $\matr{\hat{\upkappa}}{i}{j}$ is not self-adjoint with respect to $\mathring{g}_{ij}(t)$, it is not clear a priori that $\mathcal{E}_{\lambda, \hat{\upeta}, high}(t)$ is positive definite. The following lemma shows that $\mathcal{E}_{\lambda, high}(t)$ is still uniformly coercive, as required for the energy estimates, at least for $\tau_{\lambda}(t) \geq \mathring{\tau}$, where $\mathring{\tau} \geq 1$ is a constant independent of $\lambda \in \Z^D \setminus \{ 0 \}$.
\end{remark}

\begin{lemma} \label{lem:einstein_energy_high_coercive}
    For $\mathcal{E}_{\lambda, \upeta, high}(t)$ as defined in \eqref{eq:geometry_energy_high}, and suppose that $\matr{(\hat{\upeta}_{\lambda})}{i}{j}, \matr{(\hat{\upkappa}_{\lambda})}{i}{j}$ satisfy the symmetry properties \eqref{eq:upeta_sym_l}--\eqref{eq:upkappa_sym_l}. Then there exists some $\mathring{\tau} \geq 1$, depending only on the Kasner exponents $p_i$, such that for $\tau(t) = \tau_{\lambda}(t) \geq \mathring{\tau}$, one has the following energy coercivity, uniformly in $\lambda \in \Z^D \setminus \{ 0 \}$:
    \begin{equation} \label{eq:geometry_energy_high_coercive}
        \mathcal{E}_{\lambda, \upeta, high} (t) \asymp \frac{1}{ \langle \tau \rangle } \mathring{g}^{ac}(t) \mathring{g}_{bd}(t) \, \matr{(\hat{\upkappa}_{\lambda})}{a}{b} \matr{(\hat{\upkappa}_{\lambda})}{c}{d} + \langle \tau \rangle \mathring{g}^{ac}(t) \mathring{g}_{bd}(t) \, \matr{(\hat{\upeta}_{\lambda})}{a}{b} \matr{(\hat{\upeta}_{\lambda})}{c}{d}.
    \end{equation}
    Similarly, for $\mathcal{E}_{\lambda, \upphi, high}(t)$, one has the following energy coercivity, uniformly in $\lambda \in \Z^D \setminus \{ 0 \}$ and $\tau_{\lambda}(t) \geq \mathring{\tau}$:
    \begin{equation} \label{eq:matter_energy_high_coercive}
        \mathcal{E}_{\lambda, \upphi, high} (t) \asymp \frac{1}{\langle \tau \rangle} \uppsi_{\lambda}^2 + \langle \tau \rangle \upphi_{\lambda}^2.
    \end{equation}
\end{lemma}

\begin{proof}
    Using the (Fourier projections) of the symmetry constraints \eqref{eq:upeta_sym}--\eqref{eq:upkappa_sym}, one may instead write the geometric high-frequency energy quantity $\mathcal{E}_{\lambda, \upeta, high}$ as:
    \begin{align*}
        \mathcal{E}_{\lambda, \upeta, high}(t) 
        &= \frac{\zeta^2}{\tau} \mathring{g}^{ab} \mathring{g}_{cd} \matr{(\hat{\upkappa}_{\lambda})}{a}{c} \matr{(\hat{\upkappa}_{\lambda})}{b}{d}  + \left( \frac{1}{2 \tau} + \zeta^2 \tau \right) \mathring{g}^{ab} \mathring{g}_{cd} \matr{(\hat{\upeta}_{\lambda})}{a}{c} \matr{(\hat{\upeta}_{\lambda})}{b}{d} \\[0.4em]
        &\qquad + \frac{\zeta}{\tau} \left[ \mathring{g}^{ab} \mathring{g}_{cd} + \zeta \, \mathring{g}^{ab} \mathring{g}_{ce} (2t\matr{\mathring{k}}{d}{e}) - \zeta \, \mathring{g}^{ae} \mathring{g}_{cd} (2t \matr{\mathring{k}}{e}{b})  \right] \matr{(\hat{\upkappa}_{\lambda})}{a}{c} \matr{(\hat{\upeta}_{\lambda})}{b}{d}.
    \end{align*}
    Using that $t \matr{\mathring{k}}{i}{j}$ is simply the diagonal matrix $- p_{\underline{i}} \matr{\delta}{\underline{i}}{j}$, we thus have
    \begin{align*}
        \mathcal{E}_{\lambda, \upeta, high}(t) 
        &= \sum_{i, j=1}^D \mathring{g}^{ii} \mathring{g}_{jj} \left[ \frac{\zeta^2}{\tau}  \left( \matr{(\hat{\upkappa}_{\lambda})}{i}{j} \right)^2 + \left( \frac{1}{2 \tau} + \zeta^2 \tau \right) \left( \matr{(\hat{\upeta}_{\lambda})}{i}{j} \right)^2 
        + \frac{\zeta}{\tau} \big( 1 + (2p_i - 2p_j) \zeta \big) \left( \matr{(\hat{\upkappa}_{\lambda})}{i}{j} \right) \left(  \matr{(\hat{\upeta}_{\lambda})}{i}{j} \right) \right].
    \end{align*}

    Using Young's inequality, and the upper bound for $\zeta$ in \eqref{eq:zeta_upperlower}, we find that there exists some $\mathring{\tau}\geq 1$ such that for $\tau \geq \tau_0$, we have the inequality:
    \begin{equation*}
        \frac{\zeta^2}{2 \tau}  \left( \matr{(\hat{\upkappa}_{\lambda})}{i}{j} \right)^2 + \frac{1}{2} \zeta^2 \tau \left( \matr{(\hat{\upeta}_{\lambda})}{i}{j} \right)^2
        \geq \left| \frac{\zeta}{\tau} \big( 1 + (2p_i - 2p_j) \zeta \big) \left( \matr{(\hat{\upkappa}_{\lambda})}{i}{j} \right) \left(  \matr{(\hat{\upeta}_{\lambda})}{i}{j} \right) \right|.
    \end{equation*}
    Combining all these inequalities as well as \eqref{eq:zeta_upperlower} then gives the lower bound in \eqref{eq:geometry_energy_high_coercive}. (We use also that $\tau \asymp \langle \tau \rangle$ for $\lambda \neq 0$.) The upper bound in \eqref{eq:zeta_upperlower} is more straightforward, as is the coercivity \eqref{eq:matter_energy_high_coercive} for $\mathcal{E}_{\lambda, \upphi, high}(t)$.
\end{proof}

The aim of this section will be to prove the following high-frequency energy estimate. As we see shortly, the energy estimate in Proposition~\ref{prop:einstein_high_freq} will 
follow by summing two derivative estimates for $\mathcal{E}_{\lambda, \upeta, high}(t)$ and $\mathcal{E}_{\lambda, \upphi, high}(t)$ which will be obtained in Proposition~\ref{prop:einstein_high_freq_geometry} and Proposition~\ref{prop:einstein_high_freq_matter} respectively.

\begin{proposition} \label{prop:einstein_high_freq}
    Let $(\matr{(\hat{\upeta}_{\lambda})}{i}{j}, \matr{(\hat{\upkappa}_{\lambda})}{i}{j}, \upphi_{\lambda}, \uppsi_{\lambda}, \upnu_{\lambda}, \upchi_{\lambda}^j)$ be a solution to the linearized Einstein--scalar field system \eqref{eq:upeta_evol_l}--\eqref{eq:upkappa_sym_l} in \underline{CMCSH gauge}. Then the $\partial_t$-derivative of the total high-frequency energy quantity obeys the following bound, uniformly in $\lambda \in \Z^D \setminus \{0\}$ and $\tau_{\lambda}(t) \geq \mathring{\tau}$:
    \begin{equation} \label{eq:einstein_high_freq_der}
        \left| t \frac{d}{dt} \mathcal{E}_{\lambda, high}(t) \right| \lesssim \left( t |\zeta'(t)| + \frac{1}{\tau} \right) \mathcal{E}_{\lambda, high}(t).
    \end{equation}
    There exists some $\mathring{C}_{high}$, depending only on the background Kasner spacetime, such that for $\lambda \in \Z^D \setminus \{ 0 \}$ obeying\footnote{$\tau^2_{\lambda}(1) = \langle \lambda \rangle^2$ is greater than $\mathring{\tau}^2$ for sufficiently large $|\lambda|$. For smaller $\lambda$ see the mid-frequency energy estimate of Section~\ref{sub:mid_freq}.} $\tau_{\lambda}(1) \geq \mathring{\tau}$, it is true for any $t$ with $\tau_{\lambda}(t) \geq \mathring{\tau}$ that
    \begin{equation} \label{eq:einstein_high_freq_fin}
        \mathring{C}_{high}^{-1} \mathcal{E}_{\lambda, high}(t) \leq \mathcal{E}_{\lambda, high}(1) \leq \mathring{C}_{high} \, \mathcal{E}_{\lambda, high}(t).
    \end{equation} 
\end{proposition}

\begin{proof}[Proof (assuming Propositions~\ref{prop:einstein_high_freq_matter} and \ref{prop:einstein_high_freq_geometry})]
    The derivative estimate \eqref{eq:einstein_high_freq_der} will follow from summing \eqref{eq:einstein_high_freq_matter_der} and \eqref{eq:einstein_high_freq_geometry_der} from Propositions~\ref{prop:einstein_high_freq_matter} and \ref{prop:einstein_high_freq_geometry} respectively. We deduce \eqref{eq:einstein_high_freq_fin} using a Gr\"onwall argument similar to that of the proof of Proposition~\ref{eq:wave_high_freq_fin}.
    By Lemmas~\ref{lem:tau} and \ref{lem:tauint} (in particular \eqref{eq:tauintegralup} with $\alpha = 1$), we have
    \[
        \int^1_{t_{\lambda*}} \left ( s |\zeta'(s)| + \frac{1}{\tau(s)} \right) \frac{ds}{s} \lesssim 1.
    \]

    Therefore, applying Gr\"onwall's inequality to \eqref{eq:einstein_high_freq_der}, in both the forwards and backwards directions, yields that there exists some $\mathring{C}_{high} > 0$ such that for any $t_1, t_2$ with $\mathring{\tau} \leq \tau(t_1) \leq \tau(t_2) \leq \tau(1)$, one has
    \begin{equation*}
        \mathring{C}_{high}^{-1} \mathcal{E}_{\lambda, high}(t_1) \leq \mathcal{E}_{\lambda, high}(t_2) \leq \mathring{C}_{high} \, \mathcal{E}_{\lambda, high}(t_1).
    \end{equation*} 
    To deduce \eqref{eq:einstein_high_freq_fin}, simply set $t_2 = 1$.
\end{proof}

The remainder of Section~\ref{sec:high_freq} is dedicated to the Propositions~\ref{prop:einstein_high_freq_matter} and \ref{prop:einstein_high_freq_geometry}.

\subsection{The \texorpdfstring{$\mathring{g}$}{g}-norm and elliptic estimates} \label{sub:gnorm}

For our energy estimates, due to the coercivity in Lemma~\ref{lem:einstein_energy_high_coercive}, it will simplify the exposition somewhat to define an inner product on (constant-valued) tensors adapted to the Cauchy slices $(\Sigma_t, \mathring{g}_{ij}(t))$.

\begin{definition}
    Let $\upalpha = \upalpha_{ij \cdots k}^{\phantom{ij \cdots k}pq \cdots r}$ be a (constant complex valued) tensor on $\Sigma_t$. We define the $\mathring{g}$-norm, $|\cdot|_{\mathring{g}}$, of $\alpha$ as follows:
    \[
        | \upalpha |_{\mathring{g}}^2 = | \upalpha_{ij \cdots k}^{\phantom{ij \cdots k}pq \cdots r}|_{\mathring{g}}^2 \coloneqq
        \mathring{g}^{i_1 i_2} \cdots \mathring{g}^{k_1 k_2} \; \mathring{g}_{p_1 p_2} \cdots \mathring{g}_{r_1 r_2} \;
        \upalpha_{i_1 j_1 \cdots k_1}^{\phantom{i_1 j_1 \cdots k_1} p_1 q_1 \cdots r_1} \overline{\upalpha}_{i_2 j_2 \cdots k_2}^{\phantom{i_2 j_2 \cdots k_2} p_2 q_2 \cdots r_2},
    \]
    where $\mathring{g}_{ij} = \mathring{g}_{ij}(t)$ and its inverse $\mathring{g}^{ij}$ have $t$-dependence as in \eqref{eq:kasner_metric}. If $\upalpha$ is a scalar, then $| \upalpha |_{\mathring{g}} = |\upalpha|$.
\end{definition}

The following lemma is standard from the theory of tensors in inner product spaces.

\begin{lemma} \label{lem:gnorm_tensor}
    Let $\upalpha = \upalpha_{ij \cdots k}^{\phantom{ij \cdots k}pq \cdots r}$ and $\upbeta = \upbeta_{lm \cdots n}^{\phantom{lm \cdots n}st \cdots u}$ be (constant) tensors on $\Sigma_t$. We define the tensor product $\upalpha \otimes \upbeta$ and the contraction $\tr \upalpha$ as the tensors:
    \[
        (\upalpha \otimes \upbeta)_{ij \cdots k l m \cdots n}^{\phantom{ ij \cdots k lm \cdots n} pq \cdots r st \dots u} = \upalpha_{ij \cdots k}^{\phantom{ij \cdots k}pq \cdots r} \upbeta_{lm \cdots n}^{\phantom{lm \cdots n}st \cdots u}, \quad (\tr \upalpha)_{j \cdots k}^{\phantom{j \cdots k} q \cdots r} = \upalpha_{ a j \cdots k}^{\phantom{a j \cdots k} a q \cdots r}.
    \]
    Then the following inequalities hold regarding the $\mathring{g}$-norms:
    \[
        | \upalpha \otimes \upbeta |_{\mathring{g}} \leq | \upalpha |_{\mathring{g}} | \upbeta |_{\mathring{g}}, \quad | \tr \upalpha |_{\mathring{g}} \leq \sqrt{D} | \upalpha |_{\mathring{g}}.
    \]
\end{lemma}

On the other hand, this next lemma follows from the definition \eqref{eq:tau} of $\tau$ and \eqref{eq:kasner_evol}:

\begin{lemma} \label{lem:kasner_gnorm}
    Consider $\lambda_i$ as a covector in $\Sigma_t$. Then the $\mathring{g}$-norms of $\lambda_i$ and $t \matr{\mathring{k}}{i}{j}$ are as follows:
    \[
        | \lambda_i |_{\mathring{g}} = \frac{\tau}{t}, \qquad | t \matr{\mathring{k}}{i}{j} |_{\mathring{g}} = \left( \sum_{i=1}^D p_i^2 \right)^{1/2} \leq 1.
    \]
\end{lemma}

The $\mathring{g}$-norm notation will prove useful in following elliptic estimates for $\tr \hat{\upeta}_{\lambda}$, $\upnu_{\lambda}$ and $\upchi_{\lambda}^j$.

\begin{proposition} \label{prop:einstein_high_elliptic}
    For $\lambda \in \Z^D \setminus \{ 0 \}$ and $\tau(t) \geq \mathring{\tau}$, the following bounds hold for $\tr \hat{\upeta}_{\lambda}$ and $\upnu_{\lambda}$.
    \begin{equation} \label{eq:einstein_high_lapse}
        | \tr \hat{\upeta}_{\lambda} | + |\upnu_{\lambda}| \lesssim \frac{1}{\tau^{\frac{3}{2}}} \mathcal{E}_{\lambda, high}^{1/2} (t).
    \end{equation}
    For the shift $\upchi^j_{\lambda}$, we have the following:
    \begin{equation} \label{eq:einstein_high_shift}
       | \upchi^j_{\lambda} |_{\mathring{g}} \lesssim \frac{t}{\tau^{\frac{3}{2}}} \mathcal{E}_{\lambda, high}^{1/2}(t).
    \end{equation}
\end{proposition}

\begin{proof}
    We start with the $\tr \hat{\upeta_{\lambda}}$ estimate, for which we use the linearized Hamiltonian constraint \eqref{eq:hamiltonian_linear_l}. To illustrate how we use the $\mathring{g}$-norm estimates above, one applies the linearized CMCSH gauge \eqref{eq:spatiallyharmonic_linear_l} and rewrites this equation tensorially as
    \[
        \tau^2 ( \tr \hat{\upeta}_{\lambda} ) = - 2 \tr ( t \mathring{k} \otimes \hat{\upkappa}_{\lambda} ) - 4 p_{\phi} \uppsi_{\lambda}.
    \]

    Therefore, upon application of Lemmas~\ref{lem:gnorm_tensor} and \ref{lem:kasner_gnorm} one has
    \[
        \tau^2 | \tr \hat{\upeta}_{\lambda} | = \tau^2 | \tr \hat{\upeta}_{\lambda} |_{\mathring{g}} \lesssim | \hat{\upkappa}_{\lambda} |_{\mathring{g}} + | \uppsi_{\lambda} |_{\mathring{g}}.
    \]
    By the definition of the $\mathring{g}$-norm and the energy coercivity in Lemma~\ref{lem:einstein_energy_high_coercive}, one has $|\hat{\upkappa}|_{\mathring{g}} + |\uppsi|_{\mathring{g}}| \lesssim \tau^{\frac{1}{2}} \mathcal{E}_{\lambda, high}^{1/2}(t)$ for $\tau \geq \mathring{\tau}$, thus the bound $|\tr \hat{\upeta}_{\lambda} | \lesssim \tau^{-\frac{3}{2}} \mathcal{E}_{\lambda, high}^{1/2} (t)$ follows.

    The same estimate for $|\upnu_{\lambda}|$ follows from the elliptic equation for the lapse \eqref{eq:upnu_elliptic_l}; in the $\tau \geq \mathring{\tau}$ regime this equation tells us that $|\upnu_{\lambda}| \leq |\tr \hat{\upeta}_{\lambda} |$. This completes the proof of \eqref{eq:einstein_high_lapse}.

    For \eqref{eq:einstein_high_shift}, we write \eqref{eq:upchi_elliptic} in the tensorial form.
    \[
        \tau^2 \upchi_{\lambda}^j = t^2 \mathring{g}^{ij} \left [ \mathrm{i} \lambda_i + 2 \mathrm{i} (t \mathring{k} \otimes \lambda)_i \right]  \upnu_{\lambda}
        + 4 t^2 \mathring{g}^{pq} (\mathrm{i}\lambda)_q \matr{( t \mathring{k} \otimes \hat{\upeta}_{\lambda} )}{p}{j}
        - 2 t^2 \mathring{g}^{ij} ( \mathrm{i} \lambda )_i \cdot \left [ \tr ( t \mathring{k} \otimes \hat{\upeta}_{\lambda} ) + 2 p_{\phi} \upphi_{\lambda} \right ].
    \]
    Using Lemmas~\ref{lem:gnorm_tensor} and \ref{lem:kasner_gnorm}, we therefore find that
    \[
        \tau^2 | \upchi_{\lambda}^j |_{\mathring{g}} \lesssim \tau t \left( | \upnu_{\lambda} | + | \hat{\upeta}_{\lambda}|_{\mathring{g}} + |\upphi_{\lambda}| \right).
    \]
    Using \eqref{eq:einstein_high_lapse} to estimate $|\upnu_{\lambda}|$ and Lemma~\ref{lem:einstein_energy_high_coercive} to get $|\hat{\upeta}_{\lambda}|_{\mathring{g}} + |\uppsi_{\lambda}| \lesssim \tau^{- \frac{1}{2}} \mathcal{E}_{\lambda, high}^{1/2}$, the bound \eqref{eq:einstein_high_shift} follows. This completes the high-frequency elliptic estimates.
\end{proof}

\subsection{Estimates for the matter terms}

Here, we prove the high-frequency energy estimate for the matter, which is analogous to the proof of Proposition~\ref{prop:wave_high_freq}, except for an argument to deal with the lapse $\upnu_{\lambda}$.

\begin{proposition} \label{prop:einstein_high_freq_matter}
    For $(\matr{(\hat{\upeta}_{\lambda})}{i}{j}, \matr{(\hat{\upkappa}_{\lambda})}{i}{j}, \upphi_{\lambda}, \uppsi_{\lambda}, \upnu_{\lambda}, \upchi_{\lambda})$ as in Proposition~\ref{prop:einstein_high_freq}, the following derivative estimate holds for $\mathcal{E}_{\lambda, \upphi, high}(t)$, uniformly in $\lambda \in \mathbb{Z}^D$ and $\tau(t) \geq \mathring{\tau}$:
    \begin{equation} \label{eq:einstein_high_freq_matter_der}
        \left| t \frac{d}{dt} \mathcal{E}_{\lambda, \upphi, high}(t) \right| \lesssim \left( t |\zeta'(t)| + \frac{1}{\tau} \right) \mathcal{E}_{\lambda, high}(t).
    \end{equation}
\end{proposition}

\begin{proof}
    Following the proof of Proposition~\ref{prop:wave_high_freq}, we compute $t \frac{d}{dt} \mathcal{E}_{\lambda, \upphi, high}$ using \eqref{eq:upeta_evol_l}--\eqref{eq:upkappa_sym_l} as follows:
    \begin{align}
        t \frac{d}{dt} \mathcal{E}_{\lambda, \upphi, high} \addtocounter{equation}{1}
        &= \zeta^2 \frac{2 \uppsi_{\lambda} ( - \tau^2 \upphi_{\lambda} )}{\tau} - \zeta^2p_{\phi} \frac{2 \uppsi_{\lambda} \upnu_{\lambda}}{\tau} - \zeta \frac{\uppsi_{\lambda}^2}{\tau} 
        \tag{\theequation a} \label{eq:matter_high_freq_a} \\[0.3em]
        &\qquad + \zeta \frac{ \upphi_{\lambda}(- \tau^2 \upphi_{\lambda}) }{\tau} + \zeta \frac{\uppsi_{\lambda}^2}{\tau} + \zeta p_{\phi} \frac{(\uppsi_{\lambda} - \upphi_{\lambda})\upnu_{\lambda}}{\tau} - \frac{\uppsi_{\lambda} \upphi_{\lambda}}{\tau} 
        \tag{\theequation b} \label{eq:matter_high_freq_b} \\[0.3em]
        &\qquad + \frac{\upphi_{\lambda} \uppsi_{\lambda}}{\tau} + p_{\phi} \frac{\upphi_{\lambda} \upnu_{\lambda}}{\tau} - \zeta^{-1} \frac{\upphi_{\lambda}^2}{2 \tau} 
        \tag{\theequation c} \label{eq:matter_high_freq_c} \\[0.3em]
        &\qquad + 2 \zeta^2 \tau \upphi_{\lambda} \uppsi_{\lambda} + 2 \zeta^2 p_{\phi} \tau \upphi_{\lambda} \upnu_{\lambda} + \zeta \tau \upphi_{\lambda}^2 
        \tag{\theequation d} \label{eq:matter_high_freq_d} \\[0.3em]
        &\qquad + t \zeta'(t) \left( 2 \zeta \frac{\uppsi_{\lambda}^2}{\tau} + \frac{\uppsi_{\lambda} \upphi_{\lambda}}{\tau} + 2 \zeta \tau \upphi_{\lambda}^2 \right).
        \tag{\theequation e} \label{eq:matter_high_freq_e} 
    \end{align}

    We now exploit the same cancellations as in the proof of Proposition~\ref{prop:wave_high_freq}. The difference is that we now have some additional terms related to the linearized lapse $\upnu_{\lambda}$. After the cancellations, one has
    \[
        t \frac{d}{dt} \mathcal{E}_{\lambda, high} = - \zeta^{-1} \frac{\upphi_{\lambda}^2}{2 \tau} + \text{\eqref{eq:matter_high_freq_e}} + \mathfrak{E}_{\upnu}, \qquad |\mathfrak{E}_{\upnu}| \lesssim \frac{|\uppsi_{\lambda} \upnu_{\lambda}|}{\tau} + \tau |\upphi_{\lambda} \upnu_{\lambda}|.
    \]

    As before, it is straightforward to bound \eqref{eq:matter_high_freq_e} by $t |\zeta'(t)| \mathcal{E}_{\lambda, high}$, while Proposition~\ref{prop:einstein_high_elliptic} and Lemma~\ref{lem:einstein_energy_high_coercive} allow us to estimate the final term $\mathfrak{E}_{\upnu}$ by
    \[
        | \mathfrak{E}_{\upnu} | \lesssim \frac{1}{\tau} \left( \frac{\uppsi_{\lambda}^2}{\tau} + \tau \upphi_{\lambda}^2 \right) + \tau^2 \upnu_{\lambda}^2 \lesssim \frac{1}{\tau} \mathcal{E}_{\lambda, high}.
    \]
    Combining all of these, we deduce the derivative estimate \eqref{eq:einstein_high_freq_matter_der}.
\end{proof}

\subsection{Estimates for the metric terms}

Next, we prove the high frequency energy estimate for the metric quantities $(\matr{\hat{\upeta}}{i}{j}, \matr{\hat{\upkappa}}{i}{j})$, which is similar to Proposition~\ref{prop:einstein_high_freq} but crucially uses the CMCSH gauge and cancellations due to symmetry via \eqref{eq:upeta_sym_l}--\eqref{eq:upkappa_sym_l}.

\begin{proposition} \label{prop:einstein_high_freq_geometry}
    For $(\matr{(\hat{\upeta}_{\lambda})}{i}{j}, \matr{(\hat{\upkappa}_{\lambda})}{i}{j}, \upphi_{\lambda}, \uppsi_{\lambda}, \upnu_{\lambda}, \upchi_{\lambda})$ as in Proposition~\ref{prop:einstein_high_freq}, the following derivative estimate holds for $\mathcal{E}_{\lambda, \upeta, high}(t)$, uniformly in $\lambda \in \mathbb{Z}^D$ and $\tau(t) \geq \mathring{\tau}$:
    \begin{equation} \label{eq:einstein_high_freq_geometry_der}
        \left| t \frac{d}{dt} \mathcal{E}_{\lambda, \upeta, high}(t) \right| \lesssim \left( t |\zeta'(t)| + \frac{1}{\tau} \right) \mathcal{E}_{\lambda, high}(t).
    \end{equation}
\end{proposition}

\begin{proof}
    Using the system \eqref{eq:upeta_evol_l}--\eqref{eq:upkappa_sym_l}, we calculate the derivative $t \frac{d}{dt} \mathcal{E}_{\lambda, \upeta, high}$ as follows, where we group separately the main terms involving $(\matr{(\hat{\upeta}_{\lambda})}{i}{j}, \matr{(\hat{\upkappa}_{\lambda})}{i}{j})$, the terms involving the lapse $\upnu_{\lambda}$ and the shift $\upchi_{\lambda}^j$, and finally the terms where the derivative falls on $\zeta$.
    \begin{align}
        t \frac{d}{dt} \mathcal{E}_{\lambda, \upeta, high} \addtocounter{equation}{1}
        &= \frac{2 \zeta^2}{\tau} \matr{(\hat{\upkappa}_{\lambda})}{j}{i} ( - \tau^2 \matr{(\hat{\upeta}_{\lambda})}{i}{j}) - \frac{\zeta}{\tau} \matr{(\hat{\upkappa}_{\lambda})}{j}{i} \matr{(\hat{\upkappa}_{\lambda})}{i}{j} 
        \tag{\theequation a} \label{eq:metric_high_freq_a} \\[0.4em]
        &\qquad + \frac{2 \zeta^2}{\tau} \matr{(\hat{\upkappa}_{\lambda})}{j}{i} \left[ t^2 \mathring{g}^{jk} \lambda_i \lambda_k - t \matr{\mathring{k}}{i}{j} \right] \upnu_{\lambda} + \frac{2 \zeta^2}{\tau} \matr{(\hat{\upkappa}_{\lambda})}{j}{i} \left[ - \mathrm{i} (t \matr{\mathring{k}}{a}{j}) \lambda_i \upchi_{\lambda}^a + \mathrm{i} (t \matr{\mathring{k}}{i}{a}) \lambda_a \upchi_{\lambda}^j \right]
        \tag{\theequation b} \label{eq:metric_high_freq_b} \\[0.4em]
        &\qquad + \frac{\zeta}{\tau} \matr{(\hat{\upeta}_{\lambda})}{j}{i} ( - \tau^2 \matr{(\hat{\upeta}_{\lambda})}{i}{j} ) - \frac{1}{\tau} \matr{(\hat{\upkappa}_{\lambda})}{j}{i} \matr{(\hat{\upeta}_{\lambda})}{i}{j} 
        \tag{\theequation c} \label{eq:metric_high_freq_c} \\[0.4em]
        &\qquad + \frac{\zeta}{\tau} \matr{(\hat{\upeta}_{\lambda})}{j}{i} \left[ t^2 \mathring{g}^{jk} \lambda_i \lambda_k - t \matr{\mathring{k}}{i}{j} \right] \upnu_{\lambda} + \frac{\zeta}{\tau} \matr{(\hat{\upeta}_{\lambda})}{j}{i} \left[ - \mathrm{i} (t \matr{\mathring{k}}{a}{j}) \lambda_i \upchi_{\lambda}^a + \mathrm{i} (t \matr{\mathring{k}}{i}{a}) \lambda_a \upchi_{\lambda}^j \right]
        \tag{\theequation d} \label{eq:metric_high_freq_d} \\[0.4em]
        &\qquad + \frac{\zeta}{\tau} \matr{(\hat{\upkappa}_{\lambda})}{j}{i} \matr{(\hat{\upkappa}_{\lambda})}{i}{j} 
        + \frac{\zeta}{\tau} \matr{(\hat{\upkappa}_{\lambda})}{j}{i} \left[ (2t \matr{\mathring{k}}{p}{j}) \matr{(\hat{\upeta}_{\lambda})}{i}{p} - (2 t \matr{\mathring{k}}{i}{p}) \matr{(\hat{\upeta}_{\lambda})}{p}{j}\right]  
        \tag{\theequation e} \label{eq:metric_high_freq_e} \\[0.4em]
        &\qquad + \frac{\zeta}{\tau} \matr{(\hat{\upkappa}_{\lambda})}{j}{i} (t \matr{\mathring{k}}{i}{j}) \upnu_{\lambda} + \frac{\zeta}{\tau} \matr{(\hat{\upkappa}_{\lambda})}{j}{i} \left[ \frac{\mathrm{i}}{2} \lambda_i \upchi_{\lambda}^j + \frac{\mathrm{i}}{2} \mathring{g}_{ip} \mathring{g}^{jq} \lambda_q \upchi_{\lambda}^p \right]
        \tag{\theequation f} \label{eq:metric_high_freq_f} \\[0.4em]
        &\qquad + \left( \frac{1}{\tau} + 2 \zeta^2 \tau \right) \matr{(\hat{\upeta}_{\lambda})}{j}{i} \matr{(\hat{\upkappa}_{\lambda})}{i}{j} + \left( \frac{1}{\tau} + 2 \zeta^2 \tau \right) \matr{(\hat{\upeta}_{\lambda})}{j}{i} \left[ (2t \matr{\mathring{k}}{p}{j}) \matr{(\hat{\upeta}_{\lambda})}{i}{p} - (2 t \matr{\mathring{k}}{i}{p}) \matr{(\hat{\upeta}_{\lambda})}{p}{j} \right]
        \tag{\theequation g} \label{eq:metric_high_freq_g} \\[0.4em]
        &\qquad + \left( - \frac{1}{2 \zeta \tau} + \zeta \tau \right) \matr{(\hat{\upeta}_{\lambda})}{j}{i} \matr{(\hat{\upeta}_{\lambda})}{i}{j}
        \tag{\theequation h} \label{eq:metric_high_freq_h} \\[0.4em]
        &\qquad + \left( \frac{1}{\tau} + 2 \zeta^2 \tau \right) \matr{(\hat{\upeta}_{\lambda})}{j}{i} (t \matr{\mathring{k}}{i}{j}) \upnu_{\lambda} + \left( \frac{1}{\tau} + 2 \zeta^2 \tau \right) \matr{(\hat{\upeta}_{\lambda})}{j}{i} \left[ \frac{\mathrm{i}}{2} \lambda_i \upchi_{\lambda}^j + \frac{\mathrm{i}}{2} \mathring{g}_{ip} \mathring{g}^{jq} \lambda_q \upchi_{\lambda}^p \right]
        \tag{\theequation i} \label{eq:metric_high_freq_i} \\[0.4em]
        &\qquad + t \zeta' \left[ \frac{2 \zeta}{\tau} \matr{(\hat{\upkappa}_{\lambda})}{j}{i} \matr{(\hat{\upkappa}_{\lambda})}{i}{j} + \frac{1}{\tau} \matr{(\hat{\upkappa}_{\lambda})}{j}{i} \matr{(\hat{\upeta}_{\lambda})}{i}{j} + 2 \zeta \tau \matr{(\hat{\upeta}_{\lambda})}{j}{i} \matr{(\hat{\upeta}_{\lambda})}{i}{j} \right]
        \tag{\theequation j} \label{eq:metric_high_freq_j}
    \end{align}

    The terms \eqref{eq:metric_high_freq_a}, \eqref{eq:metric_high_freq_c}, \eqref{eq:metric_high_freq_e}, \eqref{eq:metric_high_freq_g} and \eqref{eq:metric_high_freq_h} are what we deem the \textit{main terms}. They feature the following crucial cancellations:
    \begin{itemize}
        \item
            The first term in \eqref{eq:metric_high_freq_a} cancels the part of the first term in \eqref{eq:metric_high_freq_g} featuring $2 \zeta^2 \tau$.
        \item
            The second term in \eqref{eq:metric_high_freq_a} cancels the first term in \eqref{eq:metric_high_freq_e}.
        \item
            The first term in \eqref{eq:metric_high_freq_c} cancels the part of \eqref{eq:metric_high_freq_h} featuring $2 \zeta^2 \tau$.
        \item
            The second term in \eqref{eq:metric_high_freq_c} cancels the part of the first term in \eqref{eq:metric_high_freq_g} featuring $\frac{1}{\tau}$.
        \item
            We leave the remaining term in \eqref{eq:metric_high_freq_e} alone, while we notice upon relabelling indices, the remaining term in \eqref{eq:metric_high_freq_g} cancels itself. In other words,
            \[
                \matr{(\hat{\upeta}_{\lambda})}{j}{i} \left[ (2t \matr{\mathring{k}}{p}{j}) \matr{(\hat{\upeta}_{\lambda})}{i}{p} - (2 t \matr{\mathring{k}}{i}{p}) \matr{(\hat{\upeta}_{\lambda})}{p}{j} \right] = 
                \tr ( \hat{\upeta}_{\lambda} \cdot 2 t \mathring{k} \cdot \hat{\upeta}_{\lambda} ) - \tr ( 2 t \mathring{k} \cdot \hat{\upeta}_{\lambda} \cdot \hat{\upeta}_{\lambda} ) = 0
            \]
    \end{itemize}

    Upon realizing these cancellations, we rewrite the sum of the main terms as follows:
    \begin{equation*} 
        \text{\eqref{eq:metric_high_freq_a}} + \text{\eqref{eq:metric_high_freq_c}} + \text{\eqref{eq:metric_high_freq_e}} + \text{\eqref{eq:metric_high_freq_g}} + \text{\eqref{eq:metric_high_freq_h}} =
        - \frac{1}{2 \zeta \tau} \matr{(\hat{\upeta}_{\lambda})}{j}{i} \matr{(\hat{\upeta}_{\lambda})}{i}{j} 
        + \frac{\zeta}{\tau} \matr{(\hat{\upkappa}_{\lambda})}{j}{i} \left[ (2t \matr{\mathring{k}}{p}{j}) \matr{(\hat{\upeta}_{\lambda})}{i}{p} - (2 t \matr{\mathring{k}}{i}{p}) \matr{(\hat{\upeta}_{\lambda})}{p}{j}\right]  
    \end{equation*}
    and hence using Lemmas~\ref{lem:gnorm_tensor} and \ref{lem:kasner_gnorm} as well as the energy coercivity in Lemma~\ref{lem:einstein_energy_high_coercive}, one has
    \begin{equation} \label{eq:metric_high_freq_main}
        |\text{\eqref{eq:metric_high_freq_a}} + \text{\eqref{eq:metric_high_freq_c}} + \text{\eqref{eq:metric_high_freq_e}} + \text{\eqref{eq:metric_high_freq_g}} + \text{\eqref{eq:metric_high_freq_h}}| \lesssim
        \frac{1}{\tau} \mathcal{E}_{\lambda, \upeta, high}.
    \end{equation}

    We now move onto the terms \eqref{eq:metric_high_freq_b}, \eqref{eq:metric_high_freq_d}, \eqref{eq:metric_high_freq_f} and \eqref{eq:metric_high_freq_i}, which are the \emph{terms involving the lapse and the shift}. In order to estimate these terms, we will have to use the linearized momentum constraint \eqref{eq:momentum_linear1_l} as well as the linearized CMCSH gauge condition \eqref{eq:spatiallyharmonic_linear_l}.
    \begin{itemize}
        \item
            We start with the first term of \eqref{eq:metric_high_freq_b}. Using the linearized momentum constraint \eqref{eq:momentum_linear1_l}, 
            \[
                \frac{1}{\tau} \matr{(\hat{\upkappa}_{\lambda})}{j}{i} t^2 \mathring{g}^{jk} \lambda_i \lambda_k \upnu_{\lambda} = 
                - \tau [ t \matr{\mathring{k}}{b}{a} \matr{(\hat{\upeta}_{\lambda})}{a}{b} + 2 p_{\phi} \upphi_{\lambda} ] \upnu_{\lambda} + \frac{1}{\tau} t^2 \mathring{g}^{jk} ( t \matr{\mathring{k}}{i}{j}) \lambda_j \lambda_k (\tr \hat{\upeta}_{\lambda}) \upnu_{\lambda}.
            \]
            We can now estimate this using Lemmas~\ref{lem:gnorm_tensor}, \ref{lem:kasner_gnorm} and the elliptic estimates in Proposition~\ref{prop:einstein_high_elliptic}. The result is the following:
            \[
                \left | \frac{2 \zeta^2}{\tau} \matr{(\hat{\upkappa}_{\lambda})}{j}{i} t^2 \mathring{g}^{jk} \lambda_i \lambda_k \upnu_{\lambda} \right | \lesssim
                \tau \cdot ( | \matr{(\hat{\upeta}_{\lambda})}{i}{j} |_{\mathring{g}} + |\upphi| ) \cdot |\upnu_{\lambda}| \lesssim \frac{1}{\tau} \mathcal{E}_{\lambda, high}.
            \]
            The remaining terms of \eqref{eq:metric_high_freq_b} can be estimated immediately using the $\mathring{g}$-norm and the elliptic estimates of Proposition~\ref{prop:einstein_high_elliptic}. As an illustrative example, we have
            \[
                \left| \frac{2 \zeta^2}{\tau} \matr{(\hat{\upkappa}_{\lambda})}{j}{i} \mathrm{i} (t \matr{\mathring{k}}{a}{j} ) \lambda_i \upchi_{\lambda}^a \right| \lesssim
                \frac{1}{\tau} \cdot | \matr{(\hat{\upkappa}_{\lambda})}{i}{j} |_{\mathring{g}} \cdot | \lambda_i |_{\mathring{g}} \cdot | \upchi_{\lambda}^j |_{\mathring{g}} \lesssim
                \frac{1}{\tau} \mathcal{E}_{\lambda, high}.
            \]
            At the end, we find the estimate $|\text{\eqref{eq:metric_high_freq_b}}| \lesssim \frac{1}{\tau} \mathcal{E}_{\lambda, high}$.
            
        \item
            The expressions \eqref{eq:metric_high_freq_d} and \eqref{eq:metric_high_freq_f} can be estimated immediately using the $\mathring{g}$-norm lemmas and Proposition~\ref{prop:einstein_high_elliptic}. Omitting the details, one finds that
            \[
                |\text{\eqref{eq:metric_high_freq_d}}| + |\text{\eqref{eq:metric_high_freq_f}}| \lesssim \frac{1}{\tau} \mathcal{E}_{\lambda, high}.
            \]

        \item
            The first term in \eqref{eq:metric_high_freq_i} can be estimated immediately using
            \[
                \left | \left( \frac{1}{\tau} + 2 \zeta^2 \tau \right)  \matr{(\hat{\upeta}_{\lambda})}{j}{i} (t \matr{\mathring{k}}{i}{j}) \upnu_{\lambda} \right | \lesssim
                \tau \cdot | \matr{(\hat{\upeta}_{\lambda})}{i}{j} |_{\mathring{g}} \cdot |\upnu_{\lambda}| \lesssim \frac{1}{\tau} \mathcal{E}_{\lambda, high}.
            \]
            For the second term in \eqref{eq:metric_high_freq_i}, we use the symmetry property \eqref{eq:upeta_sym_l} and the linearized CMCSH condition \eqref{eq:spatiallyharmonic_linear_l} to rewrite this term as
            \[
                \left( \frac{1}{\tau} + 2 \zeta^2 \tau \right) \matr{(\hat{\upeta}_{\lambda})}{j}{i} \left[ \frac{\mathrm{i}}{2} \lambda_i \upchi_{\lambda}^j + \frac{\mathrm{i}}{2} \mathring{g}_{ip} \mathring{g}^{jq} \lambda_q \upchi_{\lambda}^p \right] = 
                \left( \frac{1}{\tau} + 2 \zeta^2 \tau \right) \frac{\mathrm{i}}{2} \lambda_j ( \tr \hat{\upeta}_{\lambda} ) \upchi_{\lambda}^j.
            \]
            The insight is that from \eqref{eq:einstein_high_lapse}, the trace $\tr \hat{\upeta}_{\lambda}$, behaves better than $|\matr{(\hat{\upeta}_{\lambda})}{i}{j}|_{\mathring{g}}$, indeed
            \[
                \left | \left( \frac{1}{\tau} + 2 \zeta^2 \tau \right) \frac{\mathrm{i}}{2} \lambda_j ( \tr \hat{\upeta}_{\lambda} ) \upchi_{\lambda}^j \right | \lesssim
                \tau \cdot |\lambda_i |_{\mathring{g}} \cdot |\upchi_{\lambda}^j|_{\mathring{g}} \cdot | \tr \hat{\upeta}_{\lambda} | \lesssim \frac{1}{\tau} \mathcal{E}_{\lambda, high}.
            \]
            The conclusion is that $|\text{\eqref{eq:metric_high_freq_i}}| \lesssim \frac{1}{\tau} \mathcal{E}_{\lambda, high}$.
    \end{itemize}
    Combining all of the above, we estimate the lapse and shift terms using
    \begin{equation} \label{eq:metric_high_freq_gauge}
        |\text{\eqref{eq:metric_high_freq_b}} + \text{\eqref{eq:metric_high_freq_d}} + \text{\eqref{eq:metric_high_freq_f}} + \text{\eqref{eq:metric_high_freq_i}} \lesssim
        \frac{1}{\tau} \mathcal{E}_{\lambda, high}.
    \end{equation}

    Finally, the final term \eqref{eq:metric_high_freq_j} is estimated trivially using the bounds \eqref{eq:zeta_upperlower} and Lemma~\ref{lem:einstein_energy_high_coercive}:
    \begin{equation} \label{eq:metric_high_freq_zeta}
        |\text{\eqref{eq:metric_high_freq_i}}| \lesssim t |\zeta'(t)| \mathcal{E}_{\lambda, \upeta, high}.
    \end{equation}
    To complete the proof, simply combine the 3 estimates \eqref{eq:metric_high_freq_main}, \eqref{eq:metric_high_freq_gauge} and \eqref{eq:metric_high_freq_zeta}.
\end{proof}


\section{Linearized Einstein--scalar field: the low-frequency energy estimate} \label{sec:low_freq}

In this section, we prove a low-frequency energy estimate analogous to that derived in Proposition~\ref{prop:wave_low_freq} for the Fourier modes of the linearized Einstein--scalar field system derived in Proposition~\ref{prop:einstein_linear_l}. Since $\matr{\upeta}{i}{j}$ and $\upphi$ generically blow up as $t \to 0$, at low frequency we instead consider (Fourier modes of) renormalized quantities $\tilde{\matr{\Upupsilon}{i}{j}}$ and $\tilde{\upvarphi}$, as defined in \eqref{eq:upupsilon_tilde_0}.

To be more precise, we pick $T = t_{\lambda*}$ in \eqref{eq:upupsilon_tilde_0}, and thus define the renormalized quantities $\matr{(\tilde{\Upupsilon}_{\lambda})}{i}{j}$ and $\tilde{\upvarphi}_{\lambda}$ in the following way:
\begin{gather} \label{eq:upupsilon2}
    \matr{(\tilde{\Upupsilon}_{\lambda})}{i}{j}(t) = \matr{(\upeta_{\lambda})}{i}{j} + G_{ip}^{\phantom{ip}jq} (t; t_{\lambda*})\cdot \matr{(\upkappa_{\lambda})}{q}{p}, \quad G_{ip}^{\phantom{ip}jq}(t; t_{\lambda*}) \coloneqq \int^{t_{\lambda*}}_t \mathring{g}_{ip}(s) \mathring{g}^{jq}(s) \frac{ds}{s},
    \\[0.5em] \label{eq:upvarphi2}
    \tilde{\upvarphi}_{\lambda} = \upphi_{\lambda} + \log( \frac{t_{\lambda*}}{t} ) \cdot \uppsi_{\lambda}.
\end{gather}

Moreover, in contrast to the CMCSH gauge used in the high-frequency energy estimates of Section~\ref{sec:high_freq}, for the low-frequency estimates we shall \emph{use a CMCTC} gauge, i.e.~a gauge where the linearized shift $\upchi^j$ vanishes. We firstly rewrite the evolution equations \eqref{eq:upeta_evol_l}--\eqref{eq:uppsi_evol_l} with respect to our renormalized $\matr{\tilde{\Upupsilon}}{i}{j}$, $\tilde{\upvarphi}$ and assuming the gauge.

\begin{lemma} \label{lem:renormalized}
    In a CMCTC gauge with $\upchi^j = 0$, the $t \partial_t$-derivatives of the quantities $(\matr{(\upkappa_{\lambda})}{i}{j}, \matr{(\tilde{\Upupsilon}_{\lambda})}{i}{j}, \uppsi_{\lambda}, \tilde{\upvarphi}_{\lambda})$ can be rewritten as follows:
    \begin{gather} \label{eq:upkappa_evol_l_2}
        \begin{split}
        t \partial_t \matr{(\upkappa_{\lambda})}{i}{j} = 
        - t^2 \mathring{g}^{ab} \lambda_a \lambda_b \left( \matr{(\tilde{\Upupsilon}_{\lambda})}{i}{j} - G_{ip}^{\phantom{ip}jq} \matr{(\upkappa_{\lambda})}{q}{p} \right) - t^2 \mathring{g}^{ja} \lambda_i \lambda_a (\tr \tilde{\Upupsilon}_{\lambda} )
        + t^2 \mathring{g}^{jk} \lambda_i \lambda_k \upnu_{\lambda} - (t \matr{\mathring{k}}{i}{j}) \upnu_{\lambda} \\
        + t^2 \mathring{g}^{ab} \lambda_i \lambda_b \left( \matr{(\tilde{\Upupsilon}_{\lambda})}{a}{j} - G_{ap}^{\phantom{ap}jq} \matr{(\upkappa_{\lambda})}{q}{p}\right) 
        + t^2 \mathring{g}^{ja} \lambda_a \lambda_b \left( \matr{(\tilde{\Upupsilon}_{\lambda})}{i}{b} - G_{ip}^{\phantom{ip}bq} \matr{(\upkappa_{\lambda})}{q}{p}\right),
        \end{split}
        \\[0.5em] \label{eq:upeta_evol_l_2}
        t \partial_t \matr{(\tilde{\Upupsilon}_{\lambda})}{i}{j}
        = G_{ip}^{\phantom{ip}jq} \cdot t \partial_t \matr{(\upkappa_{\lambda})}{q}{p} + (t \matr{\mathring{k}}{i}{j}) \upnu_{\lambda},
        \\[0.5em] \label{eq:uppsi_evol_l_2}
        t \partial_t \uppsi_{\lambda} = - t^2 \mathring{g}^{ab} \lambda_a \lambda_b \left( \tilde{\upvarphi}_{\lambda} - \log( \frac{t_{\lambda*}}{t} ) \uppsi_{\lambda} \right) - p_{\phi} \upnu_{\lambda},
        \\[0.5em] \label{eq:upphi_evol_l_2}
        t \partial_t \tilde{\upvarphi}_{\lambda} = \log( \frac{t_{\lambda*}}{t} ) \cdot t \partial_t \uppsi_{\lambda} + p_{\phi} \upnu_{\lambda}.
    \end{gather}
\end{lemma}

\begin{proof}
    This follows from direct computation using the equations \eqref{eq:upeta_evol_l}--\eqref{eq:uppsi_evol_l}, as well as the symmetry conditions \eqref{eq:upeta_sym_l}--\eqref{eq:upkappa_sym_l}.
\end{proof}

\subsection{The low-frequency energy quantity}

\begin{definition} \label{def:lowfreq}
    The \emph{geometric low-frequency energy quantity} $\mathcal{E}_{\lambda, \upeta, low}(t)$, used for $\lambda \in \Z^D \setminus \{0 \}$ and $\tau_{\lambda}(t) \leq 1$, is given by the following:
    \begin{equation} \label{eq:geometry_energy_low}
        \mathcal{E}_{\lambda, \upeta, low}(t) \coloneqq \mathring{g}^{ac}(t_{\lambda*}) \mathring{g}_{bd}(t_{\lambda*}) \, \matr{(\upkappa_{\lambda})}{a}{b} \matr{(\upkappa_{\lambda})}{c}{d} + \mathring{g}^{ac}(t_{\lambda*}) \mathring{g}_{bd}(t_{\lambda*}) \, \matr{(\tilde{\Upupsilon}_{\lambda})}{a}{b} \matr{(\tilde{\Upupsilon}_{\lambda})}{c}{d}.
    \end{equation}
    We also define the \emph{matter low-frequency energy quantity} $\mathcal{E}_{\lambda, \upphi, low}(t)$, as follows:
    \begin{equation} \label{eq:matter_energy_low}
        \mathcal{E}_{\lambda, \upphi, low}(t) \coloneqq \uppsi_{\lambda}^2 + \tilde{\upvarphi}_{\lambda}^2.
    \end{equation}
    Finally, define the \emph{total low-frequency energy quantity} as the sum $\mathcal{E}_{\lambda, low}(t) \coloneqq \mathcal{E}_{\lambda, \upeta, low}(t) + \mathcal{E}_{\lambda, \upphi, low}(t)$. 
\end{definition}

\begin{remark}
    Unlike the geometric high-frequency energy quantity defined in \eqref{eq:geometry_energy_high}, the metric $\mathring{g}_{ij}$ used as the inner product on the tensors $\matr{(\upkappa_{\lambda})}{i}{j}$ and $\matr{(\tilde{\Upupsilon}_{\infty})}{i}{j}$ is $\mathring{g}_{ij}(t_{\lambda*})$, rather than $\mathring{g}_{ij}(t)$. The fact that this inner product is no longer $t$-dependent is crucial in order to show that the tensors $\matr{(\upkappa_{\lambda})}{i}{j}$ and $\matr{(\tilde{\Upupsilon}_{\lambda})}{i}{j}$ are indeed bounded towards $t = 0$. This also indicates the operators $\mathcal{T}_*^{-p_i+p_j}$ arise in Theorem~\ref{thm:einstein_scat}.

    Another remark is that though the high-frequency energy is defined for $\tau_{\lambda}(t) \geq 1$, the relevant energy estimates are only valid for $\tau_{\lambda}(t) \geq \mathring{\tau} \geq 1$, see Lemma~\ref{lem:einstein_energy_high_coercive} and Proposition~\ref{prop:einstein_high_freq}. In order to eventually compare the high-frequency energy to the low-frequency energy of Definition~\ref{def:lowfreq}, we will require \emph{mid-frequency energy estimates}, valid for $1 \leq \tau_{\lambda}(t) \leq \mathring{\tau}$, see Section~\ref{sub:mid_freq}.
\end{remark}

The aim of this section will be to prove the following low-frequency energy estimate, see the analogous Proposition~\ref{prop:wave_low_freq} in the case of the wave equation. Unlike the high-frequency energy estimate Proposition~\ref{prop:einstein_high_freq}, the following will depend crucially on the subcriticality condition \eqref{eq:subcritical_delta}.

\begin{proposition} \label{prop:einstein_low_freq}
    Let $(\matr{({\upeta}_{\lambda})}{i}{j}, \matr{({\upkappa}_{\lambda})}{i}{j}, \upphi_{\lambda}, \uppsi_{\lambda}, \upnu_{\lambda})$ be a solution to the linearized Einstein--scalar field system \eqref{eq:upeta_evol_l}--\eqref{eq:upkappa_sym_l} in CMCTC gauge. Suppose the background Kasner exponents $p_i$ obey the subcriticality condition \eqref{eq:subcritical_delta} for some $\updelta > 0$. Then the $\partial_t$-derivative of the total low-frequency energy quantity obeys the following bound, uniformly in $\lambda \in \Z^D \setminus \{0\}$ and $0 < \tau_{\lambda}(t) \leq 1$:
    \begin{equation} \label{eq:einstein_low_freq_der}
        \left| t \frac{d}{dt} \mathcal{E}_{\lambda, low}(t) \right| \lesssim \left( \frac{t}{t_{\lambda*}}\right)^{\updelta} \mathcal{E}_{\lambda, low}(t).
    \end{equation}
    There exists some $\mathring{C}_{low}$, depending only on the background Kasner spacetime, such that for $\lambda \in \Z^D \setminus \{ 0 \}$ and $0 < \tau_{\lambda}(t) \leq 1$, one has
    \begin{equation} \label{eq:einstein_low_freq_fin}
        \mathring{C}_{low}^{-1} \mathcal{E}_{\lambda, low}(t) \leq \mathcal{E}_{\lambda, low}(t_{\lambda*}) \leq \mathring{C}_{low} \, \mathcal{E}_{\lambda, low}(t).
    \end{equation} 
\end{proposition}

\begin{proof}[Proof (assuming Proposition~\ref{prop:einstein_low_der})]
    The derivative estimate \eqref{eq:einstein_low_freq_der} follows immediately from Proposition~\ref{prop:einstein_low_der} and the definition of $\mathcal{E}_{\lambda, low}(t)$ upon insertion of the expressions $\mathring{g}_{ij}(t_{\lambda*}) = t_{\lambda*}^{2 p_{\underline{i}}} \delta_{\underline{i}j}$ and $\mathring{g}^{ij}(t_{\lambda*}) = t_{\lambda*}^{-2 p_{\underline{i}}} \delta^{\underline{i}j}$.

    To deduce \eqref{eq:einstein_low_freq_fin}, we make the simple observation that
    \[
        \int^{t_{\lambda*}}_0 \left( \frac{s}{t_{\lambda*}} \right)^{\updelta} \, \frac{ds}{s} \leq \updelta^{-1},
    \]
    thus applying Gr\"onwall's inequality to \eqref{eq:einstein_low_freq_der} in both the forwards and backwards directions yields that for an $0 < t_1 \leq t_2 \leq t_{\lambda*}$, we have (where $C$ is the implied constant in \eqref{eq:einstein_low_freq_der}):
    \[
        \exp( - C \updelta^{-1} ) \mathcal{E}_{\lambda, low}(t_1) \leq \mathcal{E}_{\lambda, low}(t_2) \leq \exp( C \updelta^{-1} ) \mathcal{E}_{\lambda, low}(t_1).
    \]
    The energy estimate \eqref{eq:einstein_low_freq_fin} then follows upon choosing $t_2 = t_{\lambda*}$ and $\mathring{C}_{low} = C \updelta^{-1}$.
\end{proof}

In Section~\ref{sub:low_freq_subcrit}, we prove Proposition~\ref{prop:einstein_low_der}. We shall also use Proposition~\ref{prop:einstein_low_der} to produce limiting quantities $\matr{((\upkappa_{\infty})_{\lambda})}{i}{j}$, $\matr{((\tilde{\Upupsilon}_{\infty})_{\lambda})}{i}{j}, (\uppsi_{\infty})_{\lambda}, (\tilde{\upvarphi}_{\infty})_{\lambda})$ in Proposition~\ref{prop:einstein_low_freq_scat}.

\subsection{Derivative estimates using the subcriticality condition} \label{sub:low_freq_subcrit}

In order to prove the energy estimate in Proposition~\ref{prop:einstein_low_freq}, we will have to estimate the powers of $t$ arising in the evolution equations of Lemma~\ref{lem:renormalized}. The first step will be to estimate the integral quantity $G_{ip}^{\phantom{ip}jq}$.

\begin{lemma} \label{lem:gbound}
    Consider the tensor $G_{ip}^{\phantom{ip}jq}$ defined in \eqref{eq:upupsilon2}. It satisfies the bound
    \begin{align} \label{eq:gbound}
        \left| G_{ip}^{\phantom{ip}jq} (t ; t_{\lambda*}) \right| 
        &\leq 
        \max \left \{ \mathring{g}_{ip}(t) \mathring{g}^{jq}(t), \mathring{g}_{ip}(t_{\lambda*}) \mathring{g}^{jq}(t_{\lambda*}) \right \} \cdot \log( \frac{t_{\lambda*}}{t} ) 
        \nonumber \\[0.5em]
        &= \max \left \{ \left( \frac{t}{t_{\lambda*}} \right)^{2p_{\underline{i}} - 2p_{\underline{j}}}, 1 \right \} \cdot t_{\lambda*}^{2p_{\underline{i}}- 2 p_{\underline{j}}} \cdot \log( \frac{t_{\lambda*}}{t} ) \cdot \delta_{\underline{i}p} \, \delta^{\underline{j}q}.
    \end{align}
\end{lemma}

\begin{proof}
    One simply notes that for any $i, p, j, q$, the quantity $\mathring{g}_{ip}(s) \mathring{g}^{jq}(s)$ is either monotonically increasing or monotonically decreasing in $s \in [t, t_{\lambda*}]$, depending on whether $p_i \leq p_j$ or $p_i \geq p_j$. Therefore
    \[
        G_{ip}^{\phantom{ip}jq} = \int^{t_{\lambda*}}_{t} \mathring{g}_{ip}(s) \mathring{g}^{jq}(s) \frac{ds}{s} \leq
        \max \left \{ \mathring{g}_{ip}(t) \mathring{g}^{jq}(t), \mathring{g}_{ip}(t_{\lambda*}) \mathring{g}^{jq}(t_{\lambda*}) \right \} \cdot \int^{t_{\lambda*}}_{t} \frac{ds}{s},
    \]
    from which the first line in \eqref{eq:gbound} follows immediately. The second line follows upon inserting \eqref{eq:kasner_metric}.
\end{proof}


We now use this lemma, along with the evolution equations in Lemma~\ref{lem:renormalized} to show that the time derivatives of our matter and metric quantities are integrable towards $t=0$. The subcriticality condition \eqref{eq:subcritical_delta} will be crucial in order to produce a power of $t$ that is integrable.

\begin{proposition} \label{prop:einstein_low_der}
    For $(\matr{({\upeta}_{\lambda})}{i}{j}, \matr{({\upkappa}_{\lambda})}{i}{j}, \upphi_{\lambda}, \uppsi_{\lambda}, \upnu_{\lambda})$ as in Proposition~\ref{prop:einstein_low_freq}, define $\matr{(\tilde{\Upupsilon}_{\lambda})}{i}{j}$ and $\tilde{\upvarphi}_{\lambda}$ as in \eqref{eq:upupsilon2}, \eqref{eq:upvarphi2}. Then given the subcriticality condition \eqref{eq:subcritical_delta} on the background Kasner exponents, we have the following derivative estimates uniformly in $\lambda \in \Z^D\setminus \{ 0 \}$ and $0 < \tau_{\lambda}(t) \leq 1$.
    \begin{equation} \label{eq:einstein_low_der}
        t_{\lambda *}^{-p_i + p_j} | t \partial_t \matr{(\upkappa_{\lambda})}{i}{j} | + t_{\lambda*}^{-p_i + p_j} | t \partial_t \matr{(\tilde{\Upupsilon}_{\lambda})}{i}{j} | + | t \partial_t \uppsi_{\lambda} | + | t \partial_t \tilde{\upvarphi}_{\lambda} | \lesssim \left( \frac{t}{t_{\lambda*}} \right)^{\updelta} \cdot \mathcal{E}_{\lambda, low}^{1/2} (t).
    \end{equation}
\end{proposition}

\begin{proof}
    Let us start with $|t \partial_t \matr{(\upkappa_{\lambda})}{i}{j}|$. We will bound this by grouping the terms in  \eqref{eq:upkappa_evol_l_2} as follows:
    \begin{align}
        | t \partial_t \matr{(\upkappa_{\lambda})}{i}{j} | \addtocounter{equation}{1}
        &= - t^2 \mathring{g}^{ab} \lambda_a \lambda_b \matr{(\tilde{\Upupsilon}_{\lambda})}{i}{j} + t^2 \mathring{g}^{ab} \lambda_i \lambda_b \matr{(\tilde{\Upupsilon}_{\lambda})}{a}{j} + t^2 \mathring{g}^{ja} \lambda_a \lambda_b \matr{(\tilde{\Upupsilon}_{\lambda})}{i}{b}
        \tag{\theequation a} \label{eq:upkappa_low_a} \\[0.3em]
        &\quad + G_{ip}^{\phantom{ip}jq} t^2 \mathring{g}^{ab} \lambda_a \lambda_b \matr{(\upkappa_{\lambda})}{q}{p} - G_{ap}^{\phantom{ap}jq} t^2 \mathring{g}^{ab} \lambda_i \lambda_b \matr{(\upkappa_{\lambda})}{q}{p} - G_{ip}^{\phantom{ip}bq}t^2 \mathring{g}^{ja} \lambda_a \lambda_b \matr{(\upkappa_{\lambda})}{q}{p} 
        \tag{\theequation b} \label{eq:upkappa_low_b} \\[0.3em]
        &\quad - t^2 \mathring{g}^{ja} \lambda_i \lambda_a (\tr \tilde{\Upsilon}) + t^2 \mathring{g}^{jk} \lambda_i \lambda_k \upnu_{\lambda} - (t \matr{\mathring{k}}{i}{j}) \upnu_{\lambda}.
        \tag{\theequation c} \label{eq:upkappa_low_c}
    \end{align}

    We shall estimate the terms occuring in \eqref{eq:upkappa_low_a}--\eqref{eq:upkappa_low_c} in terms of the low-frequency energy $\mathcal{E}_{\lambda, low}(t)$. By the definition of the low-frequency energy, we have that
    \begin{equation} \label{eq:kappa_upsilon_energy}
        \left| \matr{(\upkappa_{\lambda})}{i}{j} \right| \lesssim t_{\lambda*}^{p_i - p_j} \mathcal{E}_{\lambda, low}^{1/2}, \qquad
        \left| \matr{(\tilde{\Upupsilon}_{\lambda})}{i}{j} \right| \lesssim t_{\lambda*}^{p_i - p_j} \mathcal{E}_{\lambda, low}^{1/2}, \qquad
        \left| \tr \tilde{\Upupsilon} \right| \lesssim \mathcal{E}_{\lambda, low}^{1/2}.
    \end{equation}
    On the other hand, we estimate $\lambda_i$ using \eqref{eq:tau} and $0 < \tau^2 = \sum_i t^{2 - 2p_i} \lambda_i^2 \leq 1$ as follows
    \begin{equation} \label{eq:lambda_energy}
        | \lambda_i | \leq t_{\lambda*}^{p_i - 1}.
    \end{equation}

    Using these estimates and \eqref{eq:kasner_metric}, it is straightforward to then estimate \eqref{eq:upkappa_low_a} in the following manner.
    \begin{align}
        t_{\lambda*}^{-p_i + p_j} \cdot |\text{\eqref{eq:upkappa_low_a}}| 
        &\lesssim \sum_{a = 1}^D \left( \frac{t}{t_{\lambda*}} \right)^{2- 2p_a} \cdot \left[ t_{\lambda*}^{-p_i + p_j} \left| \matr{(\tilde{\Upupsilon}_{\lambda})}{i}{j} \right| + t_{\lambda*}^{-p_a + p_j} \left| \matr{(\tilde{\Upupsilon}_{\lambda})}{a}{j} \right| + t_{\lambda*}^{-p_i + p_a} \left| \matr{(\tilde{\Upupsilon}_{\lambda})}{i}{a} \right| \right] \nonumber \\[0.5em]
        &\lesssim \max_{a = 1, \ldots, D} \left \{ \left( \frac{t}{t_{\lambda*}} \right)^{2 - 2p_a} \right \} \cdot \mathcal{E}_{\lambda, low}^{1/2}. \label{eq:upkappa_low_a_est}
    \end{align}
    For the second line \eqref{eq:upkappa_low_b}, we first exploit a crucial cancellation in the first and third terms of \eqref{eq:upkappa_low_b} when $b = j$. In particular, from the cancellation and \eqref{eq:gbound} one has
    \[ 
        | G_{ip}^{\phantom{ip}jq} t^2 \mathring{g}^{ab} - G_{ip}^{\phantom{ip}bq} t^2 \mathring{g}^{ja} | \lesssim \max_{i, j \neq b} \left \{ \left( \frac{t}{t_{\lambda*}} \right)^{2 + 2p_i - 2p_j - 2p_{b}} \right\} \log( \frac{t_{\lambda*}}{t} ) \cdot t_{\lambda*}^{2 - p_a - p_{b} - p_q + p_p}.
    \]

    To get this estimate, there are really two cases depending on whether $p_i > p_j$ or $p_i \leq p_j$; mnote that the case where $p_i > p_j$ in \eqref{eq:gbound} is covered by allowing $i = j$ in the above maximum.
    Estimating the second term in \eqref{eq:upkappa_low_b} using \eqref{eq:gbound} is more straightforward -- in that the power of $t$ ends up being $t^{2 - 2p_a}$ or $t^{2 - 2p_j}$ rather than featuring more complicated linear combinations of Kasner exponents -- and one eventually yields
    \begin{equation} \label{eq:upkappa_low_b_est}
        t_{\lambda *}^{- p_i + p_j} \cdot | \text{\eqref{eq:upkappa_low_b}} | \lesssim
        \max_{j \neq b} \left \{ \left( \frac{t}{t_{\lambda*}} \right)^{2 + 2p_i - 2p_j - 2p_{\underline{b}}} \right\} \log( \frac{t_{\lambda*}}{t} ) \cdot \mathcal{E}_{\lambda, low}^{1/2}.
    \end{equation}

    To deal with \eqref{eq:upkappa_low_c}, the $\tr \tilde{\Upupsilon}_{\lambda}$ term is straightforward, while for the $\upnu_{\lambda}$ term we may use the elliptic equation \eqref{eq:upnu_elliptic_l} to get the following preliminary estimate, recalling that now $\tau^2 \leq 1$:
    \begin{align}         
        |\upnu_{\lambda}| 
        &\lesssim t^2 \mathring{g}^{ab} \lambda_a \lambda_b \left| \tr \tilde{\Upupsilon}_{\lambda} \right | + t^2 \mathring{g}^{ab} \lambda_i \lambda_a \left | \matr{(\tilde{\Upupsilon_{\lambda}})}{b}{i} \right| + G_{bp}^{\phantom{bp}iq} t^2 \mathring{g}^{ab} \lambda_i \lambda_a \left| \matr{(\upkappa_{\lambda})}{q}{p} \right| \nonumber \\[0.5em]
        \label{eq:upnu_energy}
        &\lesssim \max_{a = 1, \ldots, D} \left \{ \left( \frac{t}{t_{\lambda *}} \right)^{2 - 2p_a} \right \} \left( 1 + \log \left( \frac{t_{\lambda*}}{t} \right) \right) \cdot \mathcal{E}_{\lambda, low}^{1/2}.
    \end{align}
    As the final term of \eqref{eq:upkappa_low_c} contributes only when $i = j$, and $t \matr{\mathring{k}}{i}{j}$ is bounded by $\max p_i$, we therefore have
    \begin{equation} \label{eq:upkappa_low_c_est}
        t_{\lambda*}^{- p_i + p_j} \cdot | \text{\eqref{eq:upkappa_low_c}} |
        \lesssim \max_{a = 1, \ldots, D} \left \{ \left( \frac{t}{t_{\lambda *}} \right)^{2 - 2p_a} \right \} \left( 1 + \log \left( \frac{t_{\lambda*}}{t} \right) \right) \cdot \mathcal{E}_{\lambda, low}^{1/2}.
    \end{equation}
    
    Combining all of \eqref{eq:upkappa_low_a_est}, \eqref{eq:upkappa_low_b_est} and \eqref{eq:upkappa_low_c_est}, we therefore have
    \begin{equation} \label{eq:upkappa_low_est}
        t_{\lambda*}^{-p_i + p_j} | t \partial_t \matr{(\upkappa_{\lambda})}{i}{j} | 
        \lesssim \max_{b \neq c} \left \{ \left( \frac{t}{t_{\lambda*}} \right)^{2 + 2p_a - 2 p_b - 2 p_c} \right\} \left( 1 + \log( \frac{t_{\lambda*}}{t} ) \right) \cdot \mathcal{E}_{\lambda, low}^{1/2}.
    \end{equation}

    We next move onto $|t \partial_t \matr{(\tilde{\Upupsilon}_{\lambda})}{i}{j}|$. Using \eqref{eq:upeta_evol_l_2} and the above, we write this as follows:
    \begin{align}
        t \partial_t \matr{(\tilde{\Upupsilon}_{\lambda})}{i}{j} \addtocounter{equation}{1}
        &= - G_{ir}^{\phantom{ir}js} t^2 \mathring{g}^{ab} \lambda_a \lambda_b \matr{(\tilde{\Upupsilon}_{\lambda})}{s}{r} + G_{ir}^{\phantom{ir}ja} t^2 \mathring{g}^{sb} \lambda_a \lambda_b \matr{(\tilde{\Upupsilon}_{\lambda})}{s}{r} + G_{ir}^{\phantom{ir}js} t^2 \mathring{g}^{ra} \lambda_a \lambda_b \matr{(\tilde{\Upupsilon}_{\lambda})}{s}{b}
        \tag{\theequation a} \label{eq:upupsilon_low_a} \\[0.3em]
        &\; + G_{ir}^{\phantom{ir}js} G_{sp}^{\phantom{sp}rq} t^2 \mathring{g}^{ab} \lambda_a \lambda_b \matr{(\upkappa_{\lambda})}{q}{p} - G_{ir}^{\phantom{ir}ja} G_{sp}^{\phantom{sp}rq} t^2 \mathring{g}^{sb} \lambda_a \lambda_b \matr{(\upkappa_{\lambda})}{q}{p} - G_{ir}^{\phantom{ir}js} G_{sp}^{\phantom{sp}bq}t^2 \mathring{g}^{ra} \lambda_a \lambda_b \matr{(\upkappa_{\lambda})}{q}{p} 
        \tag{\theequation b} \label{eq:upupsilon_low_b} \\[0.3em]
        &\; - G_{ir}^{\phantom{ir}js} t^2 \mathring{g}^{ra} \lambda_s \lambda_a (\tr \tilde{\Upsilon}) + G_{ir}^{\phantom{ir}js} t^2 \mathring{g}^{rk} \lambda_s \lambda_k \upnu_{\lambda} - G_{ip}^{\phantom{ip}jq}(t \matr{\mathring{k}}{q}{p}) \upnu_{\lambda} + (t \matr{\mathring{k}}{i}{j}) \upnu_{\lambda}. 
        \tag{\theequation c} \label{eq:upupsilon_low_c}
    \end{align}
    To estimate \eqref{eq:upupsilon_low_a}, note that there is a cancellation between the first two terms when $a = s$, in a similar fashion to the cancellation in the estimate of \eqref{eq:upkappa_low_c} above. Tracking the powers of $t$ and $t_{\lambda*}$ carefully, and using \eqref{eq:kappa_upsilon_energy} and \eqref{eq:lambda_energy} as before, one yields
    \begin{equation} \label{eq:upupsilon_low_a_est}
        t_{\lambda*}^{-p_i + p_j} \cdot | \text{\eqref{eq:upupsilon_low_a}} |
        \lesssim \max_{a, b \neq c} \left \{ \left( \frac{t}{t_{\lambda*}} \right)^{2 + 2p_a - 2 p_b - 2 p_c} \right\} \left( 1 + \log( \frac{t_{\lambda*}}{t} ) \right) \cdot \mathcal{E}_{\lambda, low}^{1/2}.
    \end{equation}

    For \eqref{eq:upupsilon_low_b}, there are two cases depending on whether $p_i \leq p_j$ or $p_i \geq p_j$. If $p_i \leq p_j$, there is a cancellation between the first and the second expression when $a = s$ while the contribution of the third expression is fine. Alternatively if $p_i \geq p_j$, there is a cancellation between the first and the third when $b = r$ while the contribution of the second is fine. Omitting the details, we end up at the following estimate
    \begin{equation} \label{eq:upupsilon_low_b_est}
        t_{\lambda*}^{-p_i + p_j} \cdot | \text{\eqref{eq:upupsilon_low_b}} |
        \lesssim \max_{a, b \neq c} \left \{ \left( \frac{t}{t_{\lambda*}} \right)^{2 + 2p_a - 2 p_b - 2 p_c} \right\} \left( 1 + \log( \frac{t_{\lambda*}}{t} ) \right)^{2} \cdot \mathcal{E}_{\lambda, low}^{1/2}.
    \end{equation}
    The term involving $\log( \frac{t_{\lambda*}}{t})$ appears with a square since \eqref{eq:upupsilon_low_b} is quadratic in the $G_{ab}^{\phantom{ab}cd}$ tensor.

    Finally, for \eqref{eq:upupsilon_low_c}, we use again \eqref{eq:upnu_energy} and the fact that the third and fourth terms are nonzero if and only if $i = j$. Therefore
    \begin{equation} \label{eq:upupsilon_low_c_est}
        t_{\lambda*}^{-p_i + p_j} \cdot | \text{\eqref{eq:upupsilon_low_c}} |
        \lesssim \max_{a = 1, \ldots, D} \left \{ \left( \frac{t}{t_{\lambda*}} \right)^{2 - 2p_a} \right\} \left( 1 + \log( \frac{t_{\lambda*}}{t} ) \right)^2 \cdot \mathcal{E}_{\lambda, low}^{1/2}.
    \end{equation}
    The square appears here because we must use \eqref{eq:gbound} on to of \eqref{eq:upnu_energy} to estimate the third term $G_{ip}^{\phantom{ip}jq} (t \matr{\mathring{k}}{q}{p} ) \upnu_{\lambda}$. Combining the three estimates \eqref{eq:upupsilon_low_a}, \eqref{eq:upupsilon_low_b}, \eqref{eq:upupsilon_low_c} then yields
    \begin{equation} \label{eq:upupsilon_low_est}
        t_{\lambda*}^{-p_i + p_j} | t \partial_t \matr{(\tilde{\Upupsilon}_{\lambda})}{i}{j} |
        \lesssim \max_{a, b \neq c} \left \{ \left( \frac{t}{t_{\lambda*}} \right)^{2 + 2p_a - 2 p_b - 2 p_c} \right\} \left( 1 + \log( \frac{t_{\lambda*}}{t} ) \right)^{2} \cdot \mathcal{E}_{\lambda, low}^{1/2}.
    \end{equation}
    (We reiterate that powers of $t$ such as $\left( \frac{t}{t_{\lambda*}} \right)^{2- 2p_c}$ are included in \eqref{eq:upupsilon_low_est} upon taking $a = b$ in the $\max$.)

    For the matter quantities $\uppsi_{\lambda}$ and $\tilde{\upvarphi}_{\lambda}$, we use \eqref{eq:uppsi_evol_l_2} and \eqref{eq:upphi_evol_l_2}. Here, it is straightforward to use $\eqref{eq:upnu_energy}$ as well as the trivial bounds from the definition of the low-frequency energy:
    \[
        |\uppsi_{\lambda}| + |\tilde{\upvarphi}_{\lambda}| \lesssim \mathcal{E}_{\lambda, low}^{1/2}.
    \]
    Inserting all these bounds into \eqref{eq:uppsi_evol_l_2} and \eqref{eq:upphi_evol_l_2}, one deduces
    \begin{equation} \label{eq:matter_low_est}
        | t \partial_t \uppsi_{\lambda} | + | t \partial_t \tilde{\upvarphi}_{\lambda} | 
        \lesssim \max_{a = 1, \ldots, D} \left \{ \left( \frac{t}{t_{\lambda*}} \right)^{2 - 2p_a} \right\} \left( 1 + \log( \frac{t_{\lambda*}}{t} ) \right)^2 \cdot \mathcal{E}_{\lambda, low}^{1/2}. 
    \end{equation}
    
    Combining all of \eqref{eq:upkappa_low_est}, \eqref{eq:upupsilon_low_est} and \eqref{eq:matter_low_est}, one deduces the inequality.
    \[
        t_{\lambda *}^{-p_i + p_j} \left( | t \partial_t \matr{(\upkappa_{\lambda})}{i}{j} | + | t \partial_t \matr{(\tilde{\Upupsilon}_{\lambda})}{i}{j} | \right) + | t \partial_t \uppsi_{\lambda} | + | t \partial_t \tilde{\upvarphi}_{\lambda} | 
        \lesssim \max_{a, b \neq c} \left \{ \left( \frac{t}{t_{\lambda*}} \right)^{2 + 2p_a - 2 p_b - 2 p_c} \right\} \left( 1 + \log( \frac{t_{\lambda*}}{t} ) \right)^{2} \mathcal{E}_{\lambda, low}^{1/2}.
    \]
    At this point that we apply the subcriticality condition \eqref{eq:subcritical_delta}. Since, for $b \neq c$, this condition tells us $2 + 2p_a - 2p_b - 2p_c \geq 2 \updelta$, we may apply the simple inequality $x^{\updelta} (1 + \log x^{-1} )^2 \lesssim 1$ for $x = \frac{t}{t_{\lambda*}}$, to yield the desired derivative bound \eqref{eq:einstein_low_der}.
\end{proof}

\begin{remark}
    A more careful study of the proof shows that the only terms which give a contribution of $\left( \frac{t}{t_{\lambda*}} \right)^{2 + 2 p_a - 2 p_b - 2 p_c}$ rather than just $\left( \frac{t}{t_{\lambda*}} \right)^{2 - 2p_c}$ are terms that involve the expression
    \begin{equation} \label{eq:connection}
        \left( G_{ap}^{\phantom{ap}bq} t^2 \mathring{g}^{cd} - G_{ap}^{\phantom{ap}cq} t^2 \mathring{g}^{bd} \right) \lambda_c \matr{(\upkappa_{\lambda})}{q}{p}.
    \end{equation}
    Upon relabelling indices, this arises in \eqref{eq:upkappa_low_b}, and in the case that $p_i \geq p_j$, we also see it in \eqref{eq:upupsilon_low_b}. (One may use the symmetry constraint $\mathring{g}^{ab}(t_{\lambda*}) \matr{\tilde{\Upupsilon}}{b}{c} = \mathring{g}^{cb}(t_{\lambda*}) \matr{\tilde{\Upupsilon}}{b}{a}$ to reduce to this case.)

    Therefore, even if the subcriticality condition \eqref{eq:subcritical_delta} fails, it is possible to close the argument if the expression \eqref{eq:connection} can be shown to converge to $0$ sufficiently fast as $t \to 0$. Upon expanding out $G_{ip}^{\phantom{ip}jq}$ in \eqref{eq:connection}, one may show that \eqref{eq:connection} converging to $0$ requires \eqref{eq:subcrit_lin_1}; note however to close the argument one must find the rate of convergence in \eqref{eq:subcrit_lin_1}.
\end{remark}

\subsection{Convergence to a limit}

To find the scattering states $(\matr{(\upkappa_{\infty})}{i}{j}, \matr{(\tilde{\Upupsilon}_{\infty})}{i}{j}, \uppsi_{\infty}, \tilde{\upvarphi}_{\infty})$, one must show that $(\matr{\upkappa}{i}{j}(t), \matr{\tilde{\Upupsilon}}{i}{j}(t), \uppsi(t), \tilde{\upvarphi}(t))$ converges as $t \to 0$. Conversely, for asymptotic completeness we also need to show that one can prescribe these limits and launch a solution to the system \eqref{eq:upeta_evol_l}--\eqref{eq:upkappa_sym_l}. We achieve this in the following proposition.
\begin{proposition} \label{prop:einstein_low_freq_scat}
    For fixed $\lambda \in \Z^D \setminus \{ 0 \}$, and $\mathring{C}_{low}>0$ as in Proposition~\ref{prop:einstein_low_freq}, the following hold.
    \begin{enumerate}[(i)]
        \item 
            Let $(\matr{({\upeta}_{\lambda})}{i}{j}, \matr{({\upkappa}_{\lambda})}{i}{j}, \upphi_{\lambda}, \uppsi_{\lambda}, \upnu_{\lambda})$ be a solution to the system \eqref{eq:upeta_evol_l}--\eqref{eq:upkappa_sym_l} for $0 < \tau(t) \leq 1$ \underline{in CMCTC gauge}. Then there exist constant $(1, 1)$ tensors $\matr{((\upkappa_{\infty})_{\lambda})}{i}{j}$, $\matr{((\tilde{\Upupsilon}_{\infty})_{\lambda})}{i}{j}$ and real numbers $(\uppsi_{\infty})_{\lambda}$, $(\tilde{\upvarphi}_{\infty})_{\lambda}$ such that $\matr{({\upkappa}_{\lambda})}{i}{j}(t)$, $\matr{({\tilde{\Upupsilon}}_{\lambda})}{i}{j}(t)$, $\uppsi_{\lambda}(t)$ and $\tilde{\upvarphi}_{\lambda}(t)$ attain these as limits as $t \to 0$, at the following rate
            \begin{gather} \label{eq:einstein_low_freq_scat_metric}
                | \matr{(\upkappa_{\lambda})}{i}{j} (t) - \matr{((\upkappa_{\infty})_{\lambda})}{i}{j} | + 
                | \matr{(\tilde{\Upupsilon}_{\lambda})}{i}{j} (t) - \matr{((\tilde{\Upupsilon}_{\infty})_{\lambda})}{i}{j} |
                \lesssim t_{\lambda *}^{p_i - p_j} \cdot \left( \frac{t}{t_{\lambda*}} \right)^{\updelta} \mathcal{E}_{\lambda, low}^{1/2} (t_{\lambda*}),
                \\[0.5em] \label{eq:einstein_low_freq_scat_matter}
                | \uppsi_{\lambda}(t) - (\uppsi_{\infty})_{\lambda} | + | \tilde{\upvarphi}_{\lambda}(t) - (\tilde{\upvarphi}_{\infty})_{\lambda} |
                \lesssim \left( \frac{t}{t_{\lambda*}} \right)^{\updelta} \mathcal{E}_{\lambda, low}^{1/2} (t_{\lambda*}).
            \end{gather}
            Furthermore, one has the following estimate on the limiting quantities:
            \[
                \mathring{g}^{ac}(t_{\lambda*}) \mathring{g}_{bd}(t_{\lambda*}) \left [ \matr{((\upkappa_{\infty})_{\lambda})}{a}{b} \matr{((\upkappa_{\infty})_{\lambda})}{c}{d} + \matr{((\tilde{\Upupsilon}_{\infty})_{\lambda})}{a}{b} \matr{((\tilde{\Upupsilon}_{\infty})_{\lambda})}{c}{d} \right ]
                + ( \uppsi_{\infty})_{\lambda}^2 + (\tilde{\upvarphi}_{\infty})_{\lambda}^2 \leq \mathring{C}_{low} \mathcal{E}_{\lambda, low} (t_{\lambda*}).
            \]
        \item
            Conversely, let $\matr{((\upkappa_{\infty})_{\lambda})}{i}{j}$, $\matr{((\tilde{\Upupsilon}_{\infty})_{\lambda})}{i}{j}$ be constant $(1,1)$ tensors and $(\uppsi_{\infty})_{\lambda}$ and $(\tilde{\upvarphi}_{\infty})_{\lambda}$ be real numbers, satisfying the $\lambda$-Fourier projections of the asymptotic constraints \eqref{eq:hamiltonian_linear_scat}--\eqref{eq:upkappa_sym_scat}, for $T = t_{\lambda*}$. Then there exists a unique solution $(\matr{({\upeta}_{\lambda})}{i}{j}, \matr{({\upkappa}_{\lambda})}{i}{j}, \upphi_{\lambda}, \uppsi_{\lambda}, \upnu_{\lambda})$ to the system \eqref{eq:upeta_evol_l}--\eqref{eq:upkappa_sym_l} for $0 < \tau(t) \leq 1$ \underline{in CMCTC gauge} such that the solution obeys the asymptotics \eqref{eq:einstein_low_freq_scat_metric}--\eqref{eq:einstein_low_freq_scat_matter}.

            For this solution, one has that 
            \[
                \mathcal{E}_{\lambda, low}(t_{\lambda*}) \leq \mathring{C}_{low} \left( 
                    \mathring{g}^{ac}(t_{\lambda*}) \mathring{g}_{bd}(t_{\lambda*}) \left [ \matr{((\upkappa_{\infty})_{\lambda})}{a}{b} \matr{((\upkappa_{\infty})_{\lambda})}{c}{d} + \matr{((\tilde{\Upupsilon}_{\infty})_{\lambda})}{a}{b} \matr{((\tilde{\Upupsilon}_{\infty})_{\lambda})}{c}{d} \right ]
                + ( \uppsi_{\infty})_{\lambda}^2 + (\tilde{\upvarphi}_{\infty})_{\lambda}^2
                \right).
            \]
    \end{enumerate}
\end{proposition}

\begin{proof}
    We first prove (i). This follows straightforwardly upon combining Proposition~\ref{prop:einstein_low_freq} with Proposition~\ref{prop:einstein_low_der}. More specifically, combining \eqref{eq:einstein_low_freq_fin} with \eqref{eq:einstein_low_der} yields that for a solution $(\matr{({\upeta}_{\lambda})}{i}{j}, \matr{({\upkappa}_{\lambda})}{i}{j}, \upphi_{\lambda}, \uppsi_{\lambda}, \upnu_{\lambda})$ to the system \eqref{eq:upeta_evol_l}--\eqref{eq:upkappa_sym_l} in the range $0 < \tau(t) \leq 1$,
    \[
        t_{\lambda *}^{-p_i + p_j} | t \partial_t \matr{(\upkappa_{\lambda})}{i}{j} | + t_{\lambda*}^{-p_i + p_j} | t \partial_t \matr{(\tilde{\Upupsilon}_{\lambda})}{i}{j} | + | t \partial_t \uppsi_{\lambda} | + | t \partial_t \tilde{\upvarphi}_{\lambda} | \lesssim \left( \frac{t}{t_{\lambda*}} \right)^{\updelta} \cdot \mathcal{E}_{\lambda, low}^{1/2} (t_{\lambda*}).
    \]

    For $\updelta > 0$, $t^{\updelta - 1}$ is integrable towards $t = 0$, thus there is a limit $(\matr{(\upkappa_{\lambda})}{i}{j}, \matr{(\tilde{\Upupsilon}_{\lambda})}{i}{j}, \uppsi_{\lambda}, \tilde{\upvarphi}_{\lambda}) \to (\matr{((\upkappa_{\infty})_{\lambda})}{i}{j}, \matr{((\tilde{\Upupsilon}_{\infty})_{\lambda})}{i}{j}, (\uppsi_{\infty})_{\lambda}, (\tilde{\upvarphi}_{\infty})_{\lambda})$ as $t \to 0$, and moreover that the estimates \eqref{eq:einstein_low_freq_scat_metric} and \eqref{eq:einstein_low_freq_scat_matter} are satisfied. The final estimate in (i) then follows straightforwardly from Proposition~\ref{prop:einstein_low_freq}.

    Moving onto (ii), we apply the Fuchsian theory of first-order ODEs. In order to do so, we firstly reinterpret Proposition~\ref{prop:einstein_low_der} as follows. This proposition tells us that solutions $\mathbf{V} = ( t_{\lambda*}^{-p_{\underline{i}} + p_{\underline{j}}} \matr{(\upkappa_{\lambda})}{\underline{i}}{\underline{j}}, t_{\lambda*}^{-p_{\underline{i}} + p_{\underline{j}}} \matr{(\tilde{\Upupsilon}_{\lambda})}{\underline{i}}{\underline{j}}, \uppsi_{\lambda}, \tilde{\upvarphi}_{\lambda}) $  to the first-order hyperbolic system \eqref{eq:upkappa_evol_l_2}--\eqref{eq:upphi_evol_l_2} coupled to the elliptic equation \eqref{eq:upnu_elliptic_l} satisfy the following Fuchsian ODE:
    \[
        t \frac{d \mathbf{V}}{dt} = \left( \frac{t}{t_{\lambda*}} \right)^{\updelta} F \left[ \frac{t}{t_{\lambda*}}, \mathbf{V} \right],
    \]
    where $F$ is bounded and continuous in both its variables and linear in $\mathbf{V}$. 

    Therefore, we may apply Lemma~\ref{lem:fuchsian}(b) to the above ODE, with $A = 0$ and initial data
    \[
        \mathbf{v} = ( t_{\lambda*}^{-p_{\underline{i}} + p_{\underline{j}}} \matr{((\upkappa_{\infty})_{\lambda})}{\underline{i}}{\underline{j}}, t_{\lambda*}^{-p_{\underline{i}} + p_{\underline{j}}} \matr{((\tilde{\Upupsilon}_{\infty})_{\lambda})}{\underline{i}}{\underline{j}}, (\uppsi_{\infty})_{\lambda}, (\tilde{\upvarphi}_{\infty})_{\lambda}).
    \]
    The lemma generates a solution $(\matr{({\upeta}_{\lambda})}{i}{j}, \matr{({\upkappa}_{\lambda})}{i}{j}, \upphi_{\lambda}, \uppsi_{\lambda}, \upnu_{\lambda})$ to the first-order elliptic hyperbolic system above, which moreover attains the asymptotics \eqref{eq:einstein_low_freq_scat_metric} and \eqref{eq:einstein_low_freq_scat_matter}. Moreover, the solution is unique in the whole domain $0 < t \leq t_{\lambda*}$. (Note Lemma~\ref{lem:fuchsian}(b) only gives local existence and uniqueness, but since the system \eqref{eq:upkappa_evol_l_2}--\eqref{eq:upphi_evol_l_2} and \eqref{eq:upnu_elliptic_l} are linear one may extend to the whole domain $0 < t \leq t_{\lambda*}$.)

    So far, the argument (including Proposition~\ref{prop:einstein_low_der}) has \emph{not} used the constraint equations \eqref{eq:hamiltonian_linear_l}--\eqref{eq:upkappa_sym_l}. It remains to recover these constraints for the solution we generated; we leave this to Proposition~\ref{prop:constraints_2}. It is only here that we require that $(\matr{((\upkappa_{\infty})_{\lambda})}{i}{j}, \matr{((\tilde{\Upupsilon}_{\infty})_{\lambda})}{i}{j}, (\uppsi_{\infty})_{\lambda}, (\tilde{\upvarphi}_{\infty})_{\lambda})$ satisfy the asymptotic constraints \eqref{eq:hamiltonian_linear_scat_l}--\eqref{eq:upkappa_sym_scat_l}.
    This completes the proof of the proposition.
\end{proof}

\subsection{Connecting the high and low frequency energies} \label{sub:mid_freq}
 
We shall eventually show that in a well-chosen gauge, the high-frequency energy quantity $\mathcal{E}_{\lambda, high}(t)$ of Definition~\ref{def:high_freq} and the low-frequency energy quantity $\mathcal{E}_{\lambda, low}(t)$ of Definition~\ref{def:lowfreq} are comparable. One subtlety is that $\mathcal{E}_{\lambda, high}(t)$ only has good coercivity properties for $\tau_{\lambda}(t) \geq \mathring{\tau} \geq 1$, while $\mathcal{E}_{\lambda, low}(t)$ is defined only for $\tau_{\lambda}(t) \leq 1$. To fill in the gap $1 \leq \tau_{\lambda}(t) \leq \mathring{\tau}$, we rely on the following \emph{mid-frequency energy estimate}. Like the high-frequency energy estimate, this will \emph{always use the CMCSH gauge} \eqref{eq:spatiallyharmonic_linear_l}.

\begin{definition} \label{def:mid_freq}
    The \emph{geometric mid-frequency energy quantity} $\mathcal{E}_{\lambda, \upeta, mid}(t)$, used for $\lambda \in \Z^D \setminus \{0\}$ and $1 \leq \tau_{\lambda}(t) \leq \mathring{\tau}$ and solutions in the CMCSH gauge, is given by
    \begin{equation} \label{eq:geometry_energy_mid}
        \mathcal{E}_{\lambda, \upeta, mid} (t) \coloneqq \mathring{g}^{ac}(t) \mathring{g}_{bd}(t) \, \matr{(\hat{\upkappa}_{\lambda})}{a}{b} \matr{(\hat{\upkappa}_{\lambda})}{c}{d} + \mathring{g}^{ac}(t) \mathring{g}_{bd}(t) \, \matr{(\hat{\upeta}_{\lambda})}{a}{b} \matr{(\hat{\upeta}_{\lambda})}{c}{d}.
    \end{equation}
    Similarly, the \emph{matter mid-frequency energy quantity} is defined in the same regime, and is given by
    \begin{equation} \label{eq:matter_energy_mid}
        \mathcal{E}_{\lambda, \upphi, mid}(t) \coloneqq \uppsi_{\lambda}^2 + \upphi_{\lambda}^2.
    \end{equation}
    Finally, define the \emph{total mid-frequency energy quantity} as the sum $\mathcal{E}_{\lambda, low}(t) \coloneqq \mathcal{E}_{\lambda, \upeta, low}(t) + \mathcal{E}_{\lambda, \upphi, low}(t)$.
\end{definition}

\begin{proposition} \label{prop:mid_freq}
    Let $(\matr{(\hat{\upeta}_{\lambda})}{i}{j}, \matr{(\hat{\upkappa}_{\lambda})}{i}{j}, \upphi_{\lambda}, \uppsi_{\lambda}, \upnu_{\lambda}, \upchi^j_{\lambda})$ be a solution to the linearized Einstein--scalar field system \eqref{eq:upeta_evol_l}--\eqref{eq:upkappa_sym_l} in CMCSH gauge. Then there exists some $\mathring{C}_{mid}$, depending only on the background Kasner spacetime, such that for $\lambda \in \Z^D \setminus \{ 0 \}$ and $1 \leq \tau_{\lambda}(t) \leq \mathring{\tau}$, one has 
    \begin{equation} \label{eq:einstein_mid_freq_fin}
        \mathring{C}_{mid}^{-1}\, \mathcal{E}_{\lambda, mid}(t) \leq \mathcal{E}_{\lambda, mid}(\tau_{\lambda}^{-1}(\mathring{\tau})) \leq \mathring{C}_{mid} \, \mathcal{E}_{\lambda, mid}(t).
    \end{equation}
    Enlarging the constant $\mathring{C}_{mid}$ if necessary, one has the following equivalence between the high-frequency energy $\mathcal{E}_{\lambda, high}(t)$ at $t = \tau_{\lambda}^{-1}(\mathring{\tau})$ and the low-frequency energy $\mathcal{E}_{\lambda, low}(t)$ at $t = t_{\lambda*} = \tau_{\lambda}^{-1}(1)$.
    \begin{equation} \label{eq:einstein_high_low_freq}
        \mathring{C}_{mid}^{-1} \, \mathcal{E}_{\lambda, low}(t_{\lambda*}) \leq \mathcal{E}_{\lambda, high}(\tau_{\lambda}^{-1}(\mathring{\tau})) \leq \mathring{C}_{mid} \, \mathcal{E}_{\lambda, low}(t_{\lambda*}).
    \end{equation}
\end{proposition}

\begin{proof}
    The mid-frequency energy estimate \eqref{eq:einstein_mid_freq_fin} will be much more straightforward to show than either of the high- or low-frequency energy estimates, since we are working in a bounded domain $\log \tau \in [0, \log \mathring{\tau}]$, and therefore, by \eqref{eq:log_upperlower}, a bounded domain of $\log( \frac{t}{t_{\lambda*}} ) \subset [0, C \log \mathring{\tau}]$. As a result, it suffices to show a derivative estimate of the form:
    \begin{equation} \label{eq:einstein_mid_freq_der}
        \left| t \frac{d}{dt} \mathcal{E}_{\lambda, mid}(t) \right| \lesssim \mathcal{E}_{\lambda, mid}(t).
    \end{equation}

    To prove \eqref{eq:einstein_mid_freq_der}, we use the same $\mathring{g}$-norm that was introduced in Section~\ref{sub:gnorm}; then by definition of $\mathcal{E}_{\lambda, mid}$,
    \[
        \mathcal{E}_{\lambda, mid}(t) = | \hat{\upeta}_{\lambda} |_{\mathring{g}}^2  + | \hat{\upkappa}_{\lambda} |_{\mathring{g}}^2 + \upphi_{\lambda}^2 + \uppsi_{\lambda}^2.
    \]
    Further, using the same proof as Proposition~\ref{prop:einstein_high_elliptic} it is straightforward to show that $ |\upnu_{\lambda}|^2 \lesssim \mathcal{E}_{\lambda, mid}(t)$ and $ |\upchi_{\lambda}|_{\mathring{g}}^2 \lesssim t^2 \mathcal{E}_{\lambda, mid}(t)$. From this point, \eqref{eq:einstein_mid_freq_der} follows immediately upon differentiating $\mathcal{E}_{\lambda, mid}(t)$ term-by-term using \eqref{eq:upeta_evol_l}--\eqref{eq:uppsi_evol_l}, and applying Cauchy-Schwarz. 

    Using the fact that $\mathcal{E}_{\lambda, mid}(t_{\lambda*}) = \mathcal{E}_{\lambda, low}(t_{\lambda*})$, while $\mathcal{E}_{\lambda, mid}(\tau^{-1}(\mathring{\tau})) \asymp \mathcal{E}_{\lambda, high}(\tau^{-1}(\mathring{\tau}))$ by Lemma~\ref{lem:einstein_energy_high_coercive}, the final estimate \eqref{eq:einstein_high_low_freq} follows immediately from \eqref{eq:einstein_mid_freq_fin}.
\end{proof}

\section{Scattering for the linearized Einstein--scalar field equations} \label{sec:einstein_scat}

In Section~\ref{sec:einstein_scat}, we prove our scattering result for the linearized Einstein--scalar field system around Kasner spacetimes satisfying the subcriticality condition, Theorem~\ref{thm:einstein_scat}. Recall that the gauge used in Theorem~\ref{thm:einstein_scat} is a CMCTC gauge with vanishing shift, $\upchi^j = 0$, though we do permit changes of gauge associated to a \underline{time-independent} vector field $\upxi^j$ using \eqref{eq:upeta_gauge}--\eqref{eq:upupsilon_gauge}.

On the other hand, our energy estimates use various different gauges; though the low-frequency energy estimate of Section~\ref{sec:low_freq} used a CMCTC gauge, our high-frequency energy estimate of Section~\ref{sec:high_freq} crucially used a CMCSH gauge (as did the mid-frequency estimate of Section~\ref{sub:mid_freq}, though this was not as crucial). 

As a result, we instead first prove a version of Theorem~\ref{thm:einstein_scat} not in the CMCTC gauge, but in a gauge called the \emph{Constant Mean Curvature Frequency Adapted} (CMCFA) gauge which we introduce in Section~\ref{sub:adm_setup}. This reinterpreted scattering result, Theorem~\ref{thm:einstein_scat_v2}, is stated in Section~\ref{sub:einstein_scat_v2} and its proof will be in Section~\ref{sub:einstein_scat_proof_v2}.

In Section~\ref{sub:scat_conclusion}, we understand how to translate back to CMCTC gauge, and thereby complete the proofs of the remaining theorems regarding the linearized Einstein--scalar field system, that is the scattering result Theorem~\ref{thm:einstein_scat} and also the asymptotics of Theorem~\ref{thm:asymp_einstein}.

\subsection{Frequency adapted norms and gauges} \label{sub:adm_setup}



As well as the gauge issue mentioned above, we make the following comments about Theorem~\ref{thm:einstein_scat}; the idea being that these comments will inform how our second version, Theorem~\ref{thm:einstein_scat_v2}, will differ:
%
\begin{enumerate}[(I)]
    \item
        The variables $(\matr{\upeta}{i}{j}, \matr{\upkappa}{i}{j}, \upphi, \uppsi)$ used to describe Cauchy data at $t=1$ differ from the variables $(\matr{\upkappa}{i}{j}, \matr{\Upupsilon}{i}{j}, \uppsi, \upvarphi)$ used to describe the asymptotic data at $t = 0$.
    \item
        In retrieving the optimal regularity of \eqref{eq:asymp_metric_reg}, we require a gauge transformation associated to a vector field $\upxi^j$. (Furthermore, the regularity of $\upxi^j$ cannot be controlled particularly well.)
    \item
        The regularity statements \eqref{eq:asymp_metric_reg}--\eqref{eq:asymp_scalar_reg}, and the spaces $\mathcal{H}^s_{\infty}$, feature a correction by a term involving the symbol $\mathcal{T}_*$; in particular $\mathcal{T}_*^{-p_i+p_j}$ often acts on tensorial quantities such as $\matr{\upkappa}{i}{j}$.
\end{enumerate}
These issues are inevitable, in light of the energy estimates presented in Section~\ref{sec:high_freq} and Section~\ref{sec:low_freq}. This is since both the relevant gauge and the metrics $\mathring{g}_{ij}(t)$ and $\mathring{g}_{ij}(t_{\lambda*})$ used to define the energies differ between the high-frequency and the low-frequency regimes.

Thus, in Theorem~\ref{thm:einstein_scat_v2}, we first present a version of Theorem~\ref{thm:einstein_scat} that will follow more cleanly from our energy estimates. In preparation, we make a series of definitions corresponding to (I)--(III) above. Each of theses definitions will involve the symbol $\mathcal{T}_*$, see Definition~\ref{def:tstar}. 
We start by defining new renormalized quantities in order to deal with (I).

\begin{definition} \label{def:time_freq_renormalized}
    We define the following \emph{frequency adapted renormalized quantity} ${}^{(F)} \matr{\Upupsilon}{i}{j}$:
    \begin{equation} \label{eq:upupsilon_tilde}
        {}^{(F)} \matr{\Upupsilon}{i}{j} \coloneqq \matr{\upeta}{i}{j} + \int^{\mathcal{T}_*}_{\min\{t, \mathcal{T}_*\}} \mathring{g}_{ip} (s) \mathring{g}^{jq} (s) \, \frac{ds}{s} \cdot \matr{\upkappa}{q}{p} = \matr{\Upupsilon}{i}{j} - \int^1_{\max\{t,\mathcal{T}_*\}} \mathring{g}_{ip}(s) \mathring{g}^{jq}(s) \, \frac{ds}{s} \cdot \matr{\upkappa}{q}{p},
    \end{equation}
    as well as the following \emph{frequency adapted scalar field quantity} ${}^{(F)} \upvarphi$:
    \begin{equation} \label{eq:upvarphi_tilde}
        {}^{(F)} \upvarphi \coloneqq \upphi + \max \left \{ 0, \log\left( \frac{\mathcal{T}_*}{t} \right) \right\} \cdot \uppsi = \upvarphi + \max \{ \log( \mathcal{T}_* ), \log t \} \cdot \uppsi.
    \end{equation}
\end{definition}

\begin{remark}
    The point is that the Fourier coefficients $\matr{({}^{(F)}\Upupsilon_{\lambda})}{i}{j}$ coincide with $\matr{(\upeta_{\lambda})}{i}{j}$ in the high-frequency regime $t \geq t_{\lambda*}$, while $\matr{({}^{(F)}\Upupsilon_{\lambda})}{i}{j}$ coincides with $\matr{(\tilde{\Upupsilon}_{\lambda})}{i}{j}$ in the low-frequency regime $t \leq t_{\lambda*}$; here $\matr{(\tilde{\Upupsilon}_{\lambda})}{i}{j}$ is defined in \eqref{eq:upupsilon_tilde}. Similarly, ${}^{(F)} \upvarphi_{\lambda}$ coincides with $\upphi_{\lambda}$ at high frequency and with $\tilde{\upvarphi}_{\lambda}$ at low frequency.
\end{remark}

Next, we define a new (linearized) gauge condition used to deal with (II).

\begin{definition} \label{def:time_freq_gauge}
    We say that a solution $(\matr{\upeta}{i}{j}, \matr{\upkappa}{i}{j}, \upphi, \uppsi, \upnu, \upchi^j)$ solving the linearized Einstein--scalar field system \eqref{eq:upeta_evol}--\eqref{eq:upkappa_sym} is in a linearized \emph{Constant Mean Curvature Frequency Adapted} (CMCFA) gauge if frequency adapted renormalized quantity ${}^{(F)} \matr{{\Upupsilon}}{i}{j}$ satisfies the differential relation
    \begin{equation} \label{eq:time_freq_harmonic}
        \partial_i \tr \,\! {}^{(F)} \Upupsilon - 2 \partial_j {}^{(F)} \matr{{\Upupsilon}}{i}{j} = 0.
    \end{equation}
\end{definition}

Finally, we define a norm for $(1, 1)$ tensor fields to deal with (III). We shall often apply this to the $(1, 1)$ tensor fields ${}^{(F)} \matr{\Upupsilon}{i}{j}$, $\matr{\upkappa}{i}{j}$ and also the derivative of the shift, $\partial_i \upchi^j$, viewed as a $(1, 1)$ tensor field.

\begin{definition} \label{def:time_freq_norm}
    Let $A = \matr{A}{i}{j}$ be a $(1, 1)$ tensor field on $\mathbb{T}^D$. Define the \emph{frequency adapted $L^2$ norm} of $A$, thought to live on the boundary $\{ t = 0 \}$ of $\mathcal{M}_{Kas}$, via the following:
    \begin{equation*}
        \| \matr{A}{i}{j} \|_{L^{2, F}(\Sigma_0)}^2 \coloneqq \sum_{i, j = 1}^D \| \mathcal{T}_*^{-p_i + p_j} \matr{A}{i}{j} \|_{L^2}^2 = \mathring{g}^{ac}(\mathcal{T}_*) \mathring{g}_{bd}(\mathcal{T}_*) \cdot \matr{A}{a}{b} \matr{A}{c}{d}.
    \end{equation*}
    Equivalently, one has the frequency space definition:
    \begin{equation*}
        \| \matr{A}{i}{j} \|_{L^{2, F}(\Sigma_0)}^2 \coloneqq \sum_{\lambda \in \Z^D} \mathring{g}^{ac}(t_{\lambda*}) \mathring{g}_{bd}(t_{\lambda*}) \matr{(A_{\lambda})}{a}{b} \matr{(A_{\lambda})}{c}{d}.
    \end{equation*}
    One further defines the \emph{frequency adapted $H^s$ norm} of $A$ using:
    \begin{equation}
        \| \matr{A}{i}{j} \|_{H^{s, F}(\Sigma_0)}^2 \coloneqq \sum_{\lambda \in \Z^D} \langle \lambda \rangle^{2s} \mathring{g}^{ac}(t_{\lambda*}) \mathring{g}_{bd}(t_{\lambda*}) \matr{(A_{\lambda})}{a}{b} \matr{(A_{\lambda})}{c}{d}.
    \end{equation}
\end{definition}

\subsection{Statement of scattering for the linearized Einstein--scalar field system, v2} \label{sub:einstein_scat_v2}

\begin{theorem}[Scattering for the linearized Einstein--scalar field system in subcritical Kasner spacetimes, v2] \label{thm:einstein_scat_v2}
    Let $(\mathcal{M}_{Kas}, g_{Kas}, \phi_{Kas})$ be a (generalized) Kasner spacetime whose exponents satisfy the subcriticality condition \eqref{eq:subcritical_delta} with $\updelta > 0$. Consider solutions $(\matr{\upeta}{i}{j}, \matr{\upkappa}{i}{j}, \upphi, \uppsi, \upnu, \upchi^j)$ to the linearized Einstein--scalar field system of Proposition~\ref{prop:adm_linear} around $(\mathcal{M}_{Kas}, g_{Kas}, \phi_{Kas})$, with the solution such that the CMCFA gauge condition \eqref{eq:time_freq_harmonic} holds. Then the scattering theory is as follows:
    \begin{enumerate}[(i)]
        \item \label{item:einstein_v2_scatop} (Existence of the scattering operator)
            Let $s$ be sufficiently large. Then given Cauchy data \linebreak$(\matr{(\upeta_C)}{i}{j}, \matr{(\upkappa_C)}{i}{j}, \upphi_C, \uppsi_C) \in H^{s+1} \times H^s \times H^{s+1} \times H^s$ at $t=1$ satisfying the constraint equations \eqref{eq:hamiltonian_linear}--\eqref{eq:upkappa_sym} and gauge conditions \eqref{eq:time_freq_harmonic} and $\tr \upkappa_C = 0$, there exists a unique (classical) solution $(\matr{\upeta}{i}{j}, \matr{\upkappa}{i}{j}, \upphi, \uppsi)$ to the system \eqref{eq:upeta_evol}--\eqref{eq:upkappa_sym} satisfying $\tr \upkappa = 0$ and the CMCFA condition \eqref{eq:time_freq_harmonic}, which attains the Cauchy data:
            \[
                \matr{{}^{(F)}\Upupsilon}{i}{j}(1, \cdot) = \matr{\upeta}{i}{j}(1, \cdot) = \matr{(\upeta_C)}{i}{j}(\cdot), \quad
                \matr{\upkappa}{i}{j}(1, \cdot) = \matr{(\upkappa_C)}{i}{j}(\cdot), \quad
                {}^{(F)}{\upvarphi}(1, \cdot) = \upphi(1, \cdot) = \upphi_C (\cdot), \quad
                \uppsi(1, \cdot) = \uppsi_C (\cdot).
            \]

            Furthermore, for some $\beta > 0$ depending only on the background $(\mathcal{M}_{Kas}, g_{Kas})$, there exist unique $(1, 1)$ tensors $\matr{(\upkappa_{\infty})}{i}{j}$ and $\matr{({}^{(F)} \Upupsilon_{\infty})}{i}{j}$, and functions $\uppsi_{\infty}$ and ${}^{(F)}\upvarphi_{\infty}$ on $\mathbb{T}^D$, such that
            \begin{gather} \label{eq:asymp_upkappaupupsilon2}
                \matr{\upkappa}{i}{j}(t, \cdot) \to \matr{(\upkappa_{\infty})}{i}{j}(\cdot), \quad \matr{{}^{(F)}\Upupsilon}{i}{j}(t, \cdot) \to \matr{({}^{(F)}\Upupsilon_{\infty})}{i}{j}(\cdot) \quad \text{ in } H^{s - \beta} \text{ strongly as } t \to 0,
                \\[0.5em] \label{eq:asymp_uppsiupphi2}
                \uppsi(t, \cdot) \to \uppsi_{\infty}(\cdot), \quad {}^{(F)} \upvarphi(t, \cdot) \to {}^{(F)} \upvarphi_{\infty}(\cdot) \quad \text{ in } H^s \text{ strongly as } t \to 0.
            \end{gather}
            These limiting objects $\matr{(\upkappa_{\infty})}{i}{j}$, $\matr{(\tilde{\Upupsilon}_{\infty})}{i}{j}$, $\uppsi_{\infty}$ and $\upphi_{\infty}$ obey the asymptotic constraints \eqref{eq:hamiltonian_linear_scat}--\eqref{eq:upkappa_sym_scat} (with $\matr{\tilde{\Upupsilon}}{i}{j}$ and $\tilde{\upvarphi}$ replaced by $\matr{{}^{(F)}\Upupsilon}{i}{j}$ and ${}^{(F)}\upvarphi$ respectively, and $T$ replaced by the symbol $\mathcal{T}_*$) as well as the limiting frequency adapted gauge condition:
            \begin{equation} \label{eq:asymp_gauge2}
                \partial_i \tr \,\!  {}^{(F)} \Upupsilon_{\infty} - 2 \partial_j \matr{({}^{(F)}\Upupsilon_{\infty})}{i}{j} = 0.
            \end{equation}

            Finally, 
            for the frequency adapted Sobolev spaces $H^{s, F}(\Sigma_0)$ in Definition~\ref{def:time_freq_norm}, one has
            \begin{equation} \label{eq:asymp_reg2}
                \matr{(\upkappa_{\infty})}{i}{j}, \matr{({}^{(F)} \Upupsilon_{\infty})}{i}{j} \in H^{s + \frac{1}{2}, F}(\Sigma_0), \qquad \uppsi_{\infty}, {}^{(F)} \upvarphi_{\infty} \in H^{s + \frac{1}{2}}.
            \end{equation}

        \item \label{item:einstein_v2_asympcomp} (Asymptotic completeness)
            Let $s$ be sufficiently large, and let $\matr{(\upkappa_{\infty})}{i}{j}$, $\matr{({}^{(F)}\Upupsilon_{\infty})}{i}{j}$ be $(1, 1)$ tensor fields and $\uppsi_{\infty}$, ${}^{(F)} \upvarphi_{\infty}$ be functions on $\mathbb{T}^D$ verifying the constraints \eqref{eq:hamiltonian_linear_scat}--\eqref{eq:upkappa_sym_scat} (modified as in (i)), with regularity as in \eqref{eq:asymp_reg2}, and obeying the gauge condition \eqref{eq:asymp_gauge2}. Then there exists a unique (classical) solution $(\matr{\upeta}{i}{j}, \matr{\upkappa}{i}{j}, \upphi, \uppsi, \upnu, \upchi^j)$ to the system \eqref{eq:upeta_evol}--\eqref{eq:upkappa_sym} satisfying $\tr \upkappa = 0$ and the CMCFA condition \eqref{eq:time_freq_harmonic}, such that as $t \to 0$, the asymptotics \eqref{eq:asymp_upkappaupupsilon2}--\eqref{eq:asymp_uppsiupphi2} hold.

            The regularity of the solution will be as follows:
            \begin{gather*}
                (\matr{\upeta}{i}{j}, \matr{\upkappa}{i}{j}) \in C^0((0, + \infty), H^{s+1} \times H^s) \cap C^1((0, + \infty), H^s \times H^{s-1}), \\[0.4em]
                ( \upphi, \uppsi ) \in C^0((0, + \infty), H^{s+1} \times H^s) \cap C^1((0, + \infty), H^s \times H^{s-1}), \\[0.4em]
                ( \upnu, \upchi^j ) \in C^0((0, + \infty), H^{s+2} \times H^{s+2}).
            \end{gather*}

        \item \label{item:einstein_v2_scatiso} (Scattering isomorphism)
            For $s \in \R$, we define the following Hilbert spaces together with their norms:
            \begin{gather} 
                \mathcal{H}^s_C = \left \{ \left( \matr{\upeta}{i}{j}, \matr{\upkappa}{i}{j}, \upphi, \uppsi \right): \matr{\upeta}{i}{j} \in H^{s+1}, \matr{\upkappa}{i}{j} \in H^s, \upphi \in H^{s+1}, \uppsi \in H^s \right \},
                \\[0.5em]
                \left \| \left( \matr{\upeta}{i}{j}, \matr{\upkappa}{i}{j}, \upphi, \uppsi \right) \right \|_{\mathcal{H}^s_C}^2 \coloneqq \| \matr{\upeta}{i}{j} \|_{H^{s+1}}^2 + \| \matr{\upkappa}{i}{j} \|_{H^s}^2 + \| \upphi \|_{H^{s+1}}^2 + \| \uppsi \|_{H^s}^2,
                \\[0.5em]
                {}^{(F)} \mathcal{H}^s_{\infty} = \left \{ \left( \matr{\upkappa}{i}{j}, \matr{\Upupsilon}{i}{j}, \uppsi, \upvarphi \right): \matr{\upkappa}{i}{j} \in H^{s+\frac{1}{2}, F}(\Sigma_0), \matr{\Upupsilon}{i}{j} \in H^{s + \frac{1}{2}, F}(\Sigma_0), 
                \uppsi \in H^{s+\frac{1}{2}}, \upvarphi \in H^{s + \frac{1}{2}} \right \},
                \\[0.5em]
                \left \| \left( \matr{\upkappa}{i}{j}, \matr{\Upupsilon}{i}{j}, \uppsi, \upvarphi \right) \right \|_{{}^{(F)}\mathcal{H}^s_{\infty}}^2 = 
                \| \matr{\upkappa}{i}{j} \|_{H^{s + \frac{1}{2}, F}(\Sigma_0)}^2 + \| \matr{\Upupsilon}{i}{j} \|_{H^{s + \frac{1}{2}, F}(\Sigma_0)}^2
                + \| \uppsi \|_{H^{s+\frac{1}{2}}}^2 + \| \upvarphi \|_{H^{s+ \frac{1}{2}}}^2.
            \end{gather}
            Further, let $\mathcal{H}_{C, c}^s$ be the subspace of $\mathcal{H}_C$ obeying the constraint equations \eqref{eq:hamiltonian_linear}--\eqref{eq:upkappa_sym} as well as the gauge conditions $\tr \upkappa = 0$ and \eqref{eq:time_freq_harmonic}, and ${}^{(F)} \mathcal{H}_{\infty, c}^s$ be the subspace of ${}^{(F)}\mathcal{H}^s_{\infty}$ obeying the asymptotic constraints \eqref{eq:hamiltonian_linear_scat_l}--\eqref{eq:upkappa_sym_scat_l}, modified as in (i), and the gauge conditions $\tr \upkappa = 0$ and \eqref{eq:asymp_gauge2}.

            Then the map ${}^{(F)} \mathcal{S}_{\downarrow}$, which is defined to map $(\matr{(\upeta_C)}{i}{j}, \matr{(\upkappa_C)}{i}{j}, \upphi_C, \uppsi_C)$ from (\ref{item:einstein_v2_scatop}) to the limiting quantities $(\matr{(\upkappa_{\infty})}{i}{j}, \matr{({}^{(F)}\Upupsilon_{\infty})}{i}{j}, \uppsi_{\infty}, {}^{(F)} \upvarphi_{\infty})$, may be extended to a Hilbert space isomorphism ${}^{(F)} \mathcal{S}_{\downarrow}: \mathcal{H}_{C, c}^s \to {}^{(F)} \mathcal{H}_{\infty, c}^s$, whose inverse ${}^{(F)} \mathcal{S}_{\uparrow} = {}^{(F)} {\mathcal{S}}_{\downarrow}^{-1}$ is exactly given by (a suitable extension of) the map taking $(\matr{(\upkappa_{\infty})}{i}{j}, \matr{({}^{(F)} {\Upupsilon}_{\infty})}{i}{j}, \uppsi_{\infty}, {}^{(F)} {\upvarphi}_{\infty})$ in (\ref{item:einstein_v2_asympcomp}) to $(\matr{\upeta}{i}{j}(1, \cdot), \matr{\upkappa}{i}{j}(1, \cdot), \upphi(1, \cdot), \uppsi(1, \cdot))$.
%
    \end{enumerate}
\end{theorem}

\subsection{Proof of scattering for the linearized Einstein--scalar field system, v2} \label{sub:einstein_scat_proof_v2}

As in Section~\ref{sub:wavescat_construction} for the case of the wave equation, we first assume smoothness for all our variables and derive a priori estimates between Sobolev spaces, then use the density of smooth functions to recover the full scattering theorem. As in Section~\ref{sub:wavescat_construction}, we have two propositions, providing the a priori estimates in both directions of the scattering theory. Before proceeding, we determine what happens at zero frequency.

\begin{lemma} \label{lem:einstein_zerofreq}
    Let $(\matr{\upeta}{i}{j}, \matr{\upkappa}{i}{j}, \upphi, \uppsi, \upnu, \upchi^j)$ be a solution of the linearized Einstein--scalar field system of Proposition~\ref{prop:adm_linear}, in any gauge. Then with $\matr{\Upupsilon}{i}{j}$ and $\upvarphi$ as in Definition~\ref{def:upupsilonupvarphi}, the $0$-frequency Fourier coefficients \newline $(\matr{(\upkappa_0)}{i}{j}, \matr{(\Upupsilon_0)}{i}{j}, \uppsi_0, \upvarphi_0)$ are constant, in particular for all $t > 0$:
    \begin{equation}
        \matr{(\upkappa_0)}{i}{j}(t) = \matr{(\upkappa_0)}{i}{j}(0), \quad
        \matr{(\Upupsilon_0)}{i}{j}(t) = \matr{(\Upupsilon_0)}{i}{j}(0), \quad
        \uppsi_0 (t) = \uppsi_0(1), \qquad \upvarphi_0 (t) = \upvarphi_0 (1).
    \end{equation}
\end{lemma}

\begin{proof}
    This follows immediately from Proposition~\ref{prop:einstein_linear_l} in the case that $\lambda = 0$.
\end{proof}

\subsubsection{From Cauchy data to asymptotic data}

\begin{proposition} \label{prop:einsteinscat_10}
    Let $\matr{(\upeta_C)}{i}{j}$ and $\matr{(\upkappa_C)}{i}{j}$ be smooth $(1, 1)$ tensors on $\mathbb{T}^D$, and $\upphi_C$ and $\uppsi_C$ be smooth functions on $\mathbb{T}^D$, obeying the constraint equations \eqref{eq:hamiltonian_linear}--\eqref{eq:upkappa_sym} with $t = 1$, as well as \eqref{eq:time_freq_harmonic} and $\tr \upkappa_C = 0$. Then there exists a unique\footnote{For uniqueness, one requires as always that $\upchi^j$ has zero average.} smooth solution $(\matr{\upeta}{i}{j}, \matr{\upkappa}{i}{j}, \upphi, \uppsi, \upnu, \upchi^j)$ to the system \eqref{eq:upeta_evol}--\eqref{eq:upkappa_sym} for $t > 0$, satisfying $\tr \upkappa = 0$ and the CMCFA gauge condition \eqref{eq:time_freq_harmonic}, that achieves the initial data. 

    Moreover, for some $\beta > 0$, there exist smooth $(1, 1)$-tensors $\matr{(\upkappa_{\infty})}{i}{j}$ and $\matr{({}^{(F)}\Upupsilon_{\infty})}{i}{j}$, and smooth functions $\uppsi_{\infty}$ and ${}^{(F)} \upvarphi_{\infty}$, such that for $s \in \R$ and the solution as above, the following strong convergence holds as $t \to 0$, where the frequency adapted quantities $\matr{{}^{(F)}\Upupsilon}{i}{j}$ and ${}^{(F)} \upvarphi$ are given in Definition~\ref{def:time_freq_renormalized}:
    \begin{equation} \label{eq:einsteinscat_10_conv}
        \matr{\upkappa}{i}{j}(t) \to \matr{(\upkappa_{\infty})}{i}{j}, \quad
        \matr{{}^{(F)} \Upupsilon}{i}{j}(t) \to \matr{({}^{(F)} \Upupsilon_{\infty})}{i}{j} \text{ in } H^{s- \beta}, \quad
        \uppsi(t) \to \uppsi_{\infty}, \quad {}^{(F)}\upvarphi(t) \to {}^{(F)} \upvarphi_{\infty} \text{ in } H^s.
    \end{equation}

    For the $H^{s, F}(\Sigma_0)$ norms of Definition~\ref{def:time_freq_norm}, one has
    \begin{multline} \label{eq:einsteinscat_10_bounded}
        \| \matr{(\upkappa_{\infty})}{i}{j} \|_{H^{s + \frac{1}{2}, F}(\Sigma_0)}^2 + 
        \| \matr{({}^{(F)} \Upupsilon_{\infty})}{i}{j} \|_{H^{s + \frac{1}{2}, F}(\Sigma_0)}^2 +
        \| \upphi_{\infty} \|_{H^{s + \frac{1}{2}}}^2 + \| {}^{(F)} \upvarphi_{\infty} \|_{H^{s + \frac{1}{2}}}^2 \\[0.4em]
        \lesssim \| \matr{(\upeta_C)}{i}{j} \|_{H^{s+1}}^2 
        + \| \matr{(\upkappa_C)}{i}{j} \|_{H^{s}}^2 + \| \upphi_C \|_{H^{s+1}}^2 + \| \uppsi_C \|_{H^s}^2.
    \end{multline}
    Finally, the quantities $\matr{(\upkappa_{\infty})}{i}{j}$, $\matr{({}^{(F)}\Upupsilon_{\infty})}{i}{j}$, $\uppsi_{\infty}$ and ${}^{(F)} \upvarphi_{\infty}$ satisfy the constraints \eqref{eq:hamiltonian_linear_scat}--\eqref{eq:upkappa_sym_scat}, with $T = \mathcal{T}_*$.
\end{proposition}

\begin{proof}
    We start with existence and uniqueness of a smooth solution for $t > 0$. For this purpose, we firstly apply Proposition~\ref{prop:adm_lwp}(ii) with $\upchi^j_{\circ} =0$, which gives us a unique smooth solution $(\matr{\tilde{\upeta}}{i}{j}, \matr{\tilde{\upkappa}}{i}{j}, \upphi, \uppsi, \upnu)$ in the CMCTC gauge (i.e.~with vanishing $\upchi^j$). 

    However, we would like existence and uniqueness in the CMCFA gauge instead. In light of Lemma~\ref{lem:diffeo} and Lemma~\ref{lem:upupsilon_gauge}, this may be achieved via the change of gauge associated to the vector field $\upxi^j$, with
    \begin{equation} \label{eq:upxi_tstar}
        \mathcal{T}_M^2 \mathring{g}^{ab} (\mathcal{T}_M) \partial_a \partial_b \upxi^j = \mathcal{T}_M^2 \mathring{g}^{jk} (\mathcal{T}_M) ( \partial_k \tr \tilde{\Upupsilon} - 2 \partial_{\ell} \matr{\tilde{\Upupsilon}}{k}{\ell} ), \quad \text{ where } \mathcal{T}_M = \max \{t, \mathcal{T}_* \}.
    \end{equation}
    The symbol $\mathcal{T}_M = \max \{t, \mathcal{T}_*\}$ arises due to the definition of ${}^{(F)}\matr{\Upupsilon}{i}{j}$. 
    Using this change of gauge, one has existence in the CMCFA gauge. Uniqueness also follows by a similar argument to uniqueness in Proposition~\ref{prop:adm_lwp}(b) -- essentially reverse this change of gauge and use uniqueness in CMCTC.

    Given our solution, we now take the Fourier decomposition. From this point, for $\lambda \in \Z^D \setminus \{ 0 \}$, we shall always work in the variables $(\matr{(\upkappa_{\lambda})}{i}{j}, \matr{({}^{(F)}\Upupsilon_{\lambda})}{i}{j}, \uppsi_{\lambda}, {}^{(F)} \upvarphi_{\lambda}, \upnu_{\lambda}, \upchi^j_{\lambda})$, and that by definition of the CMCFA gauge, we have the following description in the high-frequency regime $t_{\lambda*} \geq 1$ and the low-frequency regime $0 < t_{\lambda*} \leq 1$: for the high-frequency (and mid-frequency) regime $t_{\lambda*} \geq 1$,
    \[
        (\matr{(\upkappa_{\lambda})}{i}{j}, \matr{({}^{(F)}\Upupsilon_{\lambda})}{i}{j}, \uppsi_{\lambda}, {}^{(F)} \upvarphi_{\lambda}, \upnu_{\lambda}, \upchi^j_{\lambda}) = 
        (\matr{(\hat{\upkappa}_{\lambda})}{i}{j}, \matr{(\hat{\upeta}_{\lambda})}{i}{j}, \uppsi_{\lambda}, \upphi_{\lambda}, \upnu_{\lambda}, \upchi^j_{\lambda}) 
    \]
    and the $\lambda$-Fourier projections in this range are in the CMCSH gauge, i.e.~they satisfy \eqref{eq:spatiallyharmonic_linear_l}. This means that we can apply our high-frequency energy estimates of Section~\ref{sec:high_freq}.

    On the other hand, in the low-frequency regime $0 < t_{\lambda*} \leq 1$, we have
    \[
        (\matr{(\upkappa_{\lambda})}{i}{j}, \matr{({}^{(F)}\Upupsilon_{\lambda})}{i}{j}, \uppsi_{\lambda}, {}^{(F)} \upvarphi_{\lambda}, \upnu_{\lambda}, \upchi^j_{\lambda}) = 
        (\matr{(\upkappa_{\lambda})}{i}{j}, \matr{(\tilde{\Upupsilon}_{\lambda})}{i}{j}, \uppsi_{\lambda}, \tilde{\upvarphi}_{\lambda}, \upnu_{\lambda}, \upchi^j_{\lambda}),
    \]
    with $\matr{(\tilde{\Upupsilon}_{\lambda})}{i}{j}$ and $\tilde{\upvarphi}_{\lambda}$ as in \eqref{eq:upupsilon2} and \eqref{eq:upvarphi2}. Here, we do not quite have $\upchi^j_{\lambda} = 0$ as in the CMCTC gauge used in Section~\ref{sec:low_freq}, and instead satisfy the CMCFA gauge condition $2 \lambda_j {}^{(F)}\matr{\Upupsilon}{i}{j} = \lambda_i \tr \,\! {}^{(F)} \Upupsilon$. Nevertheless, as we see shortly it is not a major difficulty to go back and forth from CMCFA gauge to CMCTC gauge in the low-frequency regime, and one may still apply the low-frequency energy estimates of Section~\ref{sec:low_freq}.

    We first relate our energies to the initial data $(\matr{(\upeta_C)}{i}{j}, \matr{(\upkappa_C)}{i}{j}, \uppsi_C, \upphi_C)$. Splitting into high-frequencies and mid-frequences depending on $\tau_{\lambda}(1)$, and using Lemma~\ref{lem:einstein_energy_high_coercive} in the high-frequency case, we relate the energies to Sobolev norms of initial data:
    \begin{multline} \label{eq:einsteinscat_10_data}
        |\matr{(\upkappa_0)}{i}{j}|^2 + |\matr{({}^{(F)} \Upupsilon_0)}{i}{j}|^2 + \uppsi_0^2 + {}^{(F)} \upvarphi_0^2 + \sum_{\substack{\lambda \in \mathbb{Z}^D \setminus \{0\} \\ \tau_{\lambda}(1) < \mathring{\tau}}} \langle \lambda \rangle^{2s + 1} \mathcal{E}_{\lambda, mid}(1) + \sum_{\substack{\lambda \in \mathbb{Z}^D \setminus \{0\} \\ \tau_{\lambda}(1) \geq \mathring{\tau}} } \langle \lambda \rangle^{2s+1} \mathcal{E}_{\lambda, high}(1) \\[0.3em]
        \asymp \| \matr{(\upeta_C)}{i}{j} \|_{H^{s + 1}}^2 + \| \matr{(\upkappa_C)}{i}{j} \|_{H^s}^2 
        + \| \upphi_C \|_{H^{s+1}}^2 + \| \uppsi_C \|_{H^s}^2.
    \end{multline}
    Applying the high-frequency energy estimate Proposition~\ref{prop:einstein_high_freq}, as well as the mid-frequency energy estimate Proposition~\ref{prop:einstein_high_freq} (which both use the CMCSH as in our case for $t \geq t_{\lambda*}$), we deduce that
    \begin{multline} \label{eq:einstein_low_freq_scat_tstar}
        |\matr{(\upkappa_0)}{i}{j}|^2 + |\matr{({}^{(F)} \Upupsilon_0)}{i}{j}|^2 + \uppsi_0^2 + {}^{(F)} \upvarphi_0^2 + \sum_{\lambda \in \mathbb{Z}^D \setminus \{0\}} \langle \lambda \rangle^{2s+1} \mathcal{E}_{\lambda, low}(t_{\lambda*}) \\[0.4em]
        \lesssim \| \matr{(\upeta_C)}{i}{j} \|_{H^{s + 1}}^2 + \| \matr{(\upkappa_C)}{i}{j} \|_{H^s}^2 
        + \| \upphi_C \|_{H^{s+1}}^2 + \| \uppsi_C \|_{H^s}^2.
    \end{multline}

    We next apply the low-frequency energy estimate Proposition~\ref{prop:einstein_low_freq}. However, since this proposition assumes the CMCTC gauge, we need to find a way to translate the energy estimate. Our method is as follows: we first generate another solution $(\matr{({}^{(T)}\upkappa_{\lambda})}{i}{j}, \matr{({}^{(T)} \tilde{\Upupsilon}_{\lambda})}{i}{j}, \uppsi_{\lambda}, \tilde{\upvarphi}_{\lambda}, \upnu_{\lambda})$ which is in the CMCTC gauge and which \emph{agrees with our solution} at $t = t_{\lambda*}$, and thus to which Proposition~\ref{prop:einstein_low_freq} and Proposition~\ref{prop:einstein_low_der} apply, then transform into CMCFA gauge using some change of gauge vector field $\upxi_{\lambda}^j$.

    This change of gauge vector field agrees exactly with \eqref{eq:upxi_tstar} above. Furthermore, projecting \eqref{eq:upxi_tstar} onto low frequencies with $0 < t \leq t_{\lambda*}$ where by definition, $\tau^2(t_{\lambda*}) = t_{\lambda*}^2 \mathring{g}^{ab}(t_{\lambda*}) \lambda_a \lambda_b = 1$, $\upxi_{\lambda}^j$ is exactly given by
    \begin{equation} \label{eq:upxi_low}
        \upxi^j_{\lambda} = t_{\lambda*}^2 \mathring{g}^{jk}(t_{\lambda*}) [ \lambda_k \tr \,\! {}^{(T)} \tilde{\Upupsilon}_{\lambda} - 2 \lambda_{\ell} \matr{({}^{(T)}\tilde{\Upupsilon}_{\lambda})}{k}{\ell} ].
    \end{equation}
    Upon checking the $t_{\lambda*}$ weights carefully, Propositions~\ref{prop:einstein_low_freq} and \ref{prop:einstein_low_der} tell us that
    \begin{equation} \label{eq:low_freq_upxi}
        | t_{\lambda*}^{p_j - 1} \upxi^j_{\lambda} | \lesssim \mathcal{E}_{\lambda, low}^{1/2}(t_{\lambda*}), \qquad \left| t \frac{t}{dt} (t_{\lambda*}^{p_j - 1} \upxi^j_{\lambda}) \right| \lesssim \left( \frac{t}{t_{\lambda*}} \right)^{\updelta} \mathcal{E}_{\lambda, low}^{1/2}(t_{\lambda*}).
    \end{equation}
    
    Next, the change of gauge vector field $\upxi^j_{\lambda}$ relates the two gauges by (see \eqref{eq:upkappa_diffeo}, \eqref{eq:upupsilon_diffeo} and \eqref{eq:upchi_diffeo}):
    \begin{gather*}
        \matr{(\upkappa_{\lambda})}{i}{j} = \matr{({}^{(T)}\upkappa_{\lambda})}{i}{j} - \mathrm{i} (t \matr{\mathring{k}}{i}{p}) \lambda_p \upxi^j_{\lambda} + \mathrm{i} (t \matr{\mathring{k}}{p}{j}) \lambda_i \upxi^p_{\lambda},
        \\[0.5em]
        \matr{({}^{(F)} \Upupsilon_{\lambda})}{i}{j} = \matr{({}^{(T)} \tilde{\Upupsilon}_{\lambda})}{i}{j} - \frac{\mathrm{i}}{2} \lambda_i \upxi_{\lambda}^j - \frac{\mathrm{i}}{2} \mathring{g}_{ip}(t_{\lambda*}) \mathring{g}^{jq}(t_{\lambda*}) \lambda_q \upxi_{\lambda}^p,
        \\[0.5em]
        \upchi^j_{\lambda} = - t \frac{t}{dt} \upxi^j_{\lambda}.
    \end{gather*}
    By \eqref{eq:low_freq_upxi}, $\lambda_i \upxi_{\lambda}^j$ satisfies the same estimates as either of $\matr{({}^{(T)}\upkappa_{\lambda})}{i}{j}$ or $\matr{({}^{(T)} \tilde{\Upupsilon}_{\lambda})}{i}{j}$. So all the low-frequency energy estimates of Section~\ref{sec:low_freq} apply equally well to $(\matr{(\upkappa_{\lambda})}{i}{j}, \matr{({}^{((F)} \Upupsilon_{\lambda})}{i}{j})$ in our CMCFA gauge. 

    In particular, from the second estimate in \eqref{eq:low_freq_upxi} we deduce an analogue of Proposition~\ref{prop:einstein_low_freq_scat}(i), in other words we have unique tensors $\matr{((\upkappa_{\infty})_{\lambda})}{i}{j}$ and $\matr{(({}^{(F)}\Upupsilon_{\infty})_{\lambda})}{i}{j}$, and real numbers $(\uppsi_{\infty})_{\lambda}$ and $({}^{(F)} \upvarphi_{\infty})_{\lambda}$ with the CMCFA analogue of the estimates \eqref{eq:einstein_low_freq_scat_metric}--\eqref{eq:einstein_low_freq_scat_matter}. We now simply define 
    \begin{equation*}
        \matr{(\upkappa_{\infty})}{i}{j} (x) = \matr{(\upkappa_0)}{i}{j} + \sum_{\lambda \in \Z^D \setminus \{0\}} \matr{((\upkappa_{\infty})_{\lambda})}{i}{j} \, e^{i \lambda \cdot x}, 
        \qquad
        \matr{({}^{(F)}\Upupsilon_{\infty})}{i}{j} (x) = \matr{({}^{(F)}\Upupsilon_0)}{i}{j} + \sum_{\lambda \in \Z^D \setminus \{0\}} \matr{(({}^{(F)}\Upupsilon_{\infty})_{\lambda})}{i}{j} \, e^{i \lambda \cdot x}, 
    \end{equation*}
    \begin{equation*}
        \uppsi_{\infty}(x) = \uppsi_0 + \sum_{\lambda \in \Z^D \setminus \{0\}} (\uppsi_{\infty})_{\lambda} \, e^{i \lambda \cdot x}, 
        \qquad
        {}^{(F)} \upvarphi_{\infty}(x) = {}^{(F)} \upvarphi_0 + \sum_{\lambda \in \Z^D \setminus \{0\}} ({}^{(F)} \upvarphi_{\infty})_{\lambda} \, e^{i \lambda \cdot x}.
    \end{equation*}
    The final estimate in Proposition~\ref{prop:einstein_low_freq_scat}(i), combined with \eqref{eq:einstein_low_freq_scat_tstar} and the definition of $H^{s, F}(\Sigma_0)$ then yields \eqref{eq:einsteinscat_10_bounded}. Since this is true for all $s \in \R$, and $H^{s, F}(\Sigma_0) \subset H^{s - \beta}$ for some $\beta > 0$ independent of $s$, this tells us that our limiting tensors and functions are smooth.

    It remains to show the convergence \eqref{eq:einsteinscat_10_conv}. Let us first take some $r \in \R$, and estimate
    \[
        \| \matr{\upkappa}{i}{j} - \matr{(\upkappa_{\infty})}{i}{j} \|_{H^r}^2 = \sum_{\lambda \in \Z^d} \langle \lambda \rangle^{2r} | \matr{(\upkappa_{\lambda})}{i}{j} - \matr{((\upkappa_{\infty})_{\lambda})}{i}{j} |^2.
    \]
    We estimate this by splitting into low and high frequencies, and then applying the energy estimates. At low frequencies, we shall use the convergence \eqref{eq:einstein_low_freq_scat_metric}, while at high and mid frequencies we only use boundedness. For convenience, let us define:
    \begin{equation} \label{eq:hm_freq}
        \mathcal{E}_{\lambda, hm}(1) = \begin{cases}
            \mathcal{E}_{\lambda, high}(1) & \text{ if } \tau_{\lambda}(1) \geq \mathring{\tau},\\
            \mathcal{E}_{\lambda, mid}(1) & \text{ if } 1 \leq \tau_{\lambda}(1) < \mathring{\tau}.
        \end{cases}
    \end{equation}

    Using the energy estimates, one then finds:
    \begin{align*}
        \| \matr{\upkappa}{i}{j}(t) - \matr{(\upkappa_{\infty})}{i}{j} \|_{H^r}^2 
        &\leq \sum_{\substack{\lambda \in \Z^D,\\ \tau_{\lambda}(t) \leq 1}} \langle \lambda \rangle^{2r} | \matr{(\upkappa_{\lambda})}{i}{j}(t) - \matr{((\upkappa_{\infty})_{\lambda})}{i}{j} |^2 + \sum_{\substack{\lambda \in \Z^D, \\ \tau_{\lambda}(t) > 1}} \langle \lambda \rangle^{2r} ( |\matr{(\upkappa_{\lambda})}{i}{j}|^2 + |\matr{((\upkappa_{\infty})_{\lambda})}{i}{j}|^2 ) \\[0.4em]
        &\lesssim \sum_{\substack{\lambda \in \Z^D, \\ \tau_{\lambda}(t) \leq 1}} \langle \lambda \rangle^{2r } t_{\lambda*}^{2 p_i - 2p_j} \left( \frac{t}{t_{\lambda*}} \right)^{2 \updelta} \mathcal{E}_{\lambda, hm}(1) +
        \sum_{\substack{\lambda \in \Z^D \\ \tau_{\lambda}(t) > 1}} \langle \lambda \rangle^{2r} ( \tau t^{2p_i- 2p_j} + t_{\lambda*}^{2p_i - 2p_j} ) \mathcal{E}_{\lambda, hm}(1).
    \end{align*}
    
    To finally show convergence, we select an $N > 0$ such that the following quantities are bounded, uniformly in $\lambda \in \Z^D \setminus \{ 0 \}$:
    \[
        \langle \lambda \rangle^{-N} t_{\lambda *}^{2p_i - 2p_j - 2 \updelta} + \langle \lambda \rangle^{-N} \tau \max_{t \geq t_{\lambda*}} \{ t^{2 p_i - 2p_j} \} \lesssim 1.
    \]
    This is always possible, since $\tau \leq \langle \lambda \rangle$, while $t_{\lambda *}^{-1}$ grows polynomially in $\langle \lambda \rangle$. Substituting this into the above, we therefore find that
    \[
        \| \matr{\upkappa}{i}{j} - \matr{(\upkappa_{\infty})}{i}{j} \|_{H^r}^2 =
        \sum_{\substack{\lambda \in \Z^D, \\ \tau_{\lambda}(t) \leq 1}} t^{2 \updelta} \langle \lambda \rangle^{2r - N} \mathcal{E}_{\lambda, hm}(1) + \sum_{\substack{\lambda \in \Z^D, \\ \tau_{\lambda}(t) > 1}} \langle \lambda \rangle^{2r - N} \mathcal{E}_{\lambda, hm}(1).
    \]
    Selecting $2r - N = 2s + 1$, we therefore deduce the $\matr{\upkappa}{i}{j}$ convergence in \eqref{eq:einsteinscat_10_conv} for $\beta = \frac{1+N}{2}$. The $\matr{{}^{(F)}\Upupsilon}{i}{j}$ convergence follows similarly (the additional difficulty being that we need to incorporate the bound for $G_{ip}^{\phantom{ip}jq}$ in Lemma~\ref{lem:gbound}, and thus may have to increase $\beta$ further), while the $\uppsi$ and ${}^{(F)} \upvarphi$ convergence follow from similar considerations to Proposition~\ref{prop:wavescat_10}. 

    Since $H^{s - \beta} \subset C^3$ for $s$ sufficiently large, the fact that our limiting quantities satisfy the relevant constraints follows from a version of Lemma~\ref{lem:scattering_constraints} with $T = \mathcal{T}_*$. This completes the proof of the proposition.
\end{proof}

\subsubsection{From asymptotic data to Cauchy data}

\begin{proposition} \label{prop:einsteinscat_01}
    Let $\matr{(\upkappa_{\infty})}{i}{j}, \matr{({}^{(F)}\Upupsilon_{\infty})}{i}{j}$ be smooth $(1, 1)$ tensors and $\uppsi_{\infty}, {}^{(F)}\upvarphi_{\infty}$ be smooth functions on $\mathbb{T}^D$ satisfying the asymptotic constraints \eqref{eq:hamiltonian_linear_scat}--\eqref{eq:upkappa_sym_scat} with $T = \mathcal{T}_*$. Then there exists a unique smooth solution $(\matr{\upeta}{i}{j}, \matr{\upkappa}{i}{j}, \upphi, \uppsi, \upnu, \upchi^j)$ to the system \eqref{eq:upeta_evol}--\eqref{eq:upkappa_sym} for $t > 0$ in CMCFA gauge, which achieves the asymptotic data in the sense that for any $s \in \R$ and $\beta$ as in Propostion~\ref{prop:einsteinscat_10}, one has the (strong) convergence:
    \begin{equation} \label{eq:einsteinscat_01_conv}
        \matr{\upkappa}{i}{j}(t) \to \matr{(\upkappa_{\infty})}{i}{j}, \quad
        \matr{{}^{(F)} \Upupsilon}{i}{j}(t) \to \matr{({}^{(F)} \Upupsilon_{\infty})}{i}{j} \text{ in } H^{s- \beta}, \quad
        \uppsi(t) \to \uppsi_{\infty}, \quad {}^{(F)}\upvarphi(t) \to {}^{(F)} \upvarphi_{\infty} \text{ in } H^s.
    \end{equation}

    Moreover, for any $s \in \R$ one has the following bound between Sobolev spaces:
    \begin{multline} \label{eq:einsteinscat_01_bounded}
        \| \matr{\upeta}{i}{j}(1, \cdot ) \|_{H^{s+1}} + \| \matr{\upkappa}{i}{j}(1, \cdot) \|_{H^{s}}^2 + \| \upphi(1, \cdot) \|_{H^{s+1}}^2 + \| \uppsi(1, \cdot) \|_{H^s}^2
        \\[0.4em] \lesssim 
        \| \matr{(\upkappa_{\infty})}{i}{j} \|_{H^{s + \frac{1}{2}, F}(\Sigma_0)}^2 + 
        \| \matr{({}^{(F)} \Upupsilon_{\infty})}{i}{j} \|_{H^{s + \frac{1}{2}, F}(\Sigma_0)}^2 +
        \| \upphi_{\infty} \|_{H^{s + \frac{1}{2}}}^2 + \| {}^{(F)} \upvarphi_{\infty} \|_{H^{s + \frac{1}{2}}}^2.
    \end{multline}
\end{proposition}

\begin{proof}
    For existence, we decompose the scattering data $(\matr{(\upkappa_{\infty})}{i}{j}, \matr{({}^{(F)} \Upupsilon_{\infty})}{i}{j}, \uppsi_{\infty}, {}^{(F)}\upphi_{\infty})$ into their Fourier modes. For any $\lambda \neq 0$, we now appeal to Proposition~\ref{prop:einstein_low_freq_scat}(ii), which tells us that there exists a solution to the ODE system \eqref{eq:upkappa_evol_l_2}--\eqref{eq:upphi_evol_l_2} in the range $t \in (0, t_{\lambda*})$, which achieves the relevant limits as $t \to 0$, at the rate \eqref{eq:einstein_low_freq_scat_metric}--\eqref{eq:einstein_low_freq_scat_matter}.

    However, the ODE system \eqref{eq:upkappa_evol_l_2}--\eqref{eq:upphi_evol_l_2} is in the CMCTC gauge, as opposed to the CMCFA gauge as required. Nevertheless, as in the proof of Proposition~\ref{prop:einsteinscat_10}, we may use a change of gauge associated to the vector field $\upxi_{\lambda}^j$ in \eqref{eq:upxi_low} to generate another solution in the CMCFA gauge. As in that proof, we still have the estimates \eqref{eq:low_freq_upxi}, and as a result the energy estimates of Section~\ref{sec:low_freq} will all remain valid in this gauge. We denote this solution in CMCFA gauge to be $(\matr{(\upkappa_{\lambda})}{i}{j}, \matr{({}^{(F)}\Upupsilon_{\lambda})}{i}{j}, \uppsi_{\lambda}, {}^{(F)} \upvarphi_{\lambda}, \upnu_{\lambda}, \upchi^j_{\lambda})$ as before.

    We extend this solution to the range $t > t_{\lambda *}$ by using instead the CMCSH gauge evolution; this is valid since the CMCSH gauge and the CMCFA gauge coincide for $t \geq t_{\lambda*}$. Then combining all of the low-frequency estimates in Proposition~\ref{prop:einstein_low_freq_scat}, the mid-frequency estimate in Proposition~\ref{prop:mid_freq} and the high-frequency estimate of Proposition~\ref{prop:einstein_high_freq}, we find that for $\lambda \in \Z^D$ and $\mathcal{E}_{\lambda, hm}(1)$ as in \eqref{eq:hm_freq}, our solution obeys:
    \begin{equation} \label{eq:einsteinscat_01_low}
        \mathcal{E}_{\lambda, hm}(1) \lesssim 
        \mathring{g}^{ac}(t_{\lambda*}) \mathring{g}_{bd}(t_{\lambda*}) \left [ \matr{((\upkappa_{\infty})_{\lambda})}{a}{b} \matr{((\upkappa_{\infty})_{\lambda})}{c}{d} + \matr{(({}^{(F)} \Upupsilon_{\infty})_{\lambda})}{a}{b} \matr{(({}^{(F)}\Upupsilon_{\infty})_{\lambda})}{c}{d} \right ]
        + ( \uppsi_{\infty})_{\lambda}^2 + ({}^{(F)}\upvarphi_{\infty})_{\lambda}^2
        .
    \end{equation}

    Using Lemma~\ref{lem:einstein_zerofreq} to deal with $\lambda = 0$, and using the coercivity of the high-frequency energy in Lemma~\ref{lem:einstein_energy_high_coercive}, we therefore have the estimate
    \begin{multline*}
        \langle \lambda \rangle^{2s} | \matr{(\upkappa_{\lambda})}{i}{j}(1) |^2 + 
        \langle \lambda \rangle^{2s+2} | \matr{({}^{(F)}\Upupsilon_{\lambda})}{i}{j}(1) |^2 + 
        \langle \lambda \rangle^{2s} | \uppsi_{\lambda} (1)|^2 + 
        \langle \lambda \rangle^{2s+2} | {}^{(F)} \upvarphi_{\lambda} (1)|^2 \\[0.4em]
        \lesssim 
        \langle \lambda \rangle^{2s + 1} \left [ \mathring{g}^{ac}(t_{\lambda*}) \mathring{g}_{bd}(t_{\lambda*}) \left [ \matr{((\upkappa_{\infty})_{\lambda})}{a}{b} \matr{((\upkappa_{\infty})_{\lambda})}{c}{d} + \matr{(({}^{(F)} \Upupsilon_{\infty})_{\lambda})}{a}{b} \matr{(({}^{(F)}\Upupsilon_{\infty})_{\lambda})}{c}{d} \right ] 
        + ( \uppsi_{\infty})_{\lambda}^2 + ({}^{(F)}\upvarphi_{\infty})_{\lambda}^2
    \right].
    \end{multline*}
    By smoothness of the scattering data, the right-hand side is summable in $\lambda \in \Z^D$ for all $s$, and doing this sum yields the estimate \eqref{eq:einsteinscat_01_bounded}.

    Now, upon applying Fourier inversion we recover a smooth solution $(\matr{\upeta}{i}{j}, \matr{\upkappa}{i}{j}, \upphi, \uppsi, \upnu, \upchi^j)$ to the system \eqref{eq:upeta_evol}--\eqref{eq:upkappa_sym} for $t > 0$ in CMCFA gauge, which moreover achieves the asymptotics \eqref{eq:einsteinscat_01_conv} using the same proof as in Proposition~\ref{prop:einsteinscat_10}. We leave the details to the reader.

    All that remains is to show uniqueness. Since the system is linear it suffices to show uniqueness in the case of trivial scattering data $\matr{(\upkappa_{\infty})}{i}{j} = 0$, $\matr{({}^{(F)}\Upupsilon_{\infty})}{i}{j} = 0$, $\uppsi_{\infty} = 0$ and ${}^{(F)} \upvarphi_{\infty} = 0$. Suppose that we have a smooth solution $(\matr{\upeta}{i}{j}, \matr{\upkappa}{i}{j}, \upphi, \uppsi, \upnu, \upxi^j)$ to \eqref{eq:upeta_evol}--\eqref{eq:upkappa_sym} in the CMCFA gauge, achieving these trivial asymptotics in the sense of \eqref{eq:einsteinscat_01_conv}. To show that this solution vanishes, it will suffice to show each Fourier mode vanishes.

    The $0$-Fourier mode must vanish by Lemma~\ref{lem:einstein_zerofreq}, so now take any $\lambda \in \Z^D \setminus \{0\}$, and focus on the low-frequency regime with $0 < t \leq t_{\lambda*}$. Recall that Proposition~\ref{prop:einstein_low_freq_scat}(ii) gives us uniqueness in CMCTC gauge, and for uniqueness in CMCFA gauge we look for a gauge transformation that retains the trivial asymptotic data but so that the transformed solution is in CMCTC gauge.

    The idea is that any $\lambda$-mode projected solution $(\matr{(\upkappa_{\lambda})}{i}{j}, \matr{({}^{(F)}\Upupsilon_{\lambda})}{i}{j}, \uppsi_{\lambda}, {}^{(F)} \upvarphi_{\lambda}, \upnu_{\lambda}, \upchi^j_{\lambda})$ in CMCFA gauge must obey $|\upchi^j_{\lambda}(t)| \leq D t^{\updelta}$ for some constant $D$ that may depend on data but does not depend on $t$. This follows from \eqref{eq:upxi_low} and the description of $\upchi^j_{\lambda}$ in the proof of Proposition~\ref{prop:einsteinscat_10}. (In other words, that any CMCFA solution may be constructed from another CMCTC solution, though this latter CMCTC solution will not necessarily have trivial asymptotic data.)

    We find a CMCTC solution with trivial asymptotic data via the gauge transformation associated to the vector field ${}^{(T)}\upxi^j_{\lambda}(t) = \int^t_0 \upchi^j_{\lambda} \frac{ds}{s}$. Since $|\upchi^j_{\lambda}(t)| \leq D t^{\updelta}$, ${}^{(T)} \upxi^j_{\lambda}(t)$ is well-defined and vanishes at $t=0$. Further, by Lemma~\ref{lem:diffeo}(ii) with $T = 0$, it transforms our CMCFA solution to a CMCTC solution which still has trivial asymptotic data. Thus by uniqueness of CMCTC solutions from asymptotic data, the transformed solution must vanish. By \eqref{eq:upupsilon_diffeo} with $T = t_{\lambda*}$, this tells us that
    \[
        \matr{({}^{(F)} \Upupsilon_{\lambda})}{i}{j} = \frac{\mathrm{i}}{2} \lambda_i {}^{(T)}\upxi^j_{\lambda} + \frac{\mathrm{i}}{2} \mathring{g}_{ip}(t_{\lambda*}) \mathring{g}^{jq}(t_{\lambda*}) \lambda_q {}^{(T)} \upxi^q_{\lambda}.
    \]

    We now simply insert this into our frequency adapted gauge condition \eqref{eq:time_freq_harmonic}. Eventually we obtain $\mathring{g}^{ab}(t_{\lambda*}) \lambda_a \lambda_b {}^{(T)} \upxi_{\lambda}^j = 0$. In other words, the gauge transformation vector field has ${}^{(T)}\upxi^j_{\lambda} =  0$, and our solution must have been trivial in the first place. This concludes the proof of the proposition.
\end{proof}

\subsubsection{Completion of the proof of Theorem~\ref{thm:einstein_scat_v2}}

All the remains is to extend Proposition~\ref{prop:einsteinscat_10} and Proposition~\ref{prop:einsteinscat_01} to the setting where $(\matr{\upkappa}{i}{j}, \matr{{}^{(F)}\Upupsilon}{i}{j}, \uppsi, {}^{(F)} \upvarphi)$ have finite regularity, by using the density of smooth functions in $H^{s}$ and in $H^{s,F}(\Sigma_0)$. This is a slightly more complicated version of what was done for the wave equation in Section~\ref{sub:wavescat_density}, and therefore we leave the details to the reader, with only 2 comments:
\begin{itemize}
    \item[--]
        In both directions, the data $(\matr{(\upeta_C)}{i}{j}, \matr{(\upkappa_C)}{i}{j}, \upphi_C, \uppsi_C)$ or $(\matr{(\upkappa_{\infty})}{i}{j}, \matr{({}^{(F)}\Upupsilon_{\infty})}{i}{j}, \uppsi_{\infty}, {}^{(F)} \upvarphi_{\infty})$ is not completely free, and must be arranged to satisfy various constraint equations and gauge conditions. Thus in approximating finite regularity data by smooth data, the approximating data must also satisfy these equations.

        However, since in our case the constraint equations and gauge conditions are linear, it is easy to smoothly approximate data while still satisfying these equations, for instance by convolution with standard mollifiers.

    \item[--]
        In order to insist both that our solutions are classical and to make sure that the convergence \eqref{eq:asymp_upkappaupupsilon2} is also in $C^3$ in order to satisfy the assumptions of Lemma~\ref{lem:scattering_constraints}, we must choose $s > \frac{D}{2} + 3 + \beta$. This is what we mean by requiring $s$ to be sufficiently large in (i) and (ii). For (iii), the Hilbert space isomorphism remains true for all $s \in \R$.
\end{itemize}
\noindent
This concludes our proof of Theorem~\ref{thm:einstein_scat_v2}. \qed

\subsection{Proof of scattering in transported coordinates} \label{sub:scat_conclusion}

In this section, we prove the original version of the scattering result, Theorem~\ref{thm:einstein_scat}, in a CMCTC gauge, and finally derive the asymptotics of Theorem~\ref{thm:asymp_einstein} concening the linearized metric and second fundamental form. To go from CMCFA gauge to CMCTC gauge, we use Lemma~\ref{lem:diffeo}(ii); thus it remains to understand $\upchi^j$ in CMCFA gauge, or more precisely the related change of gauge vector field
\begin{equation} \label{eq:upxi_cmcfa}
    \upxi^j(t) \coloneqq - \int^1_t \upchi^j(s) \, \frac{ds}{s}.
\end{equation}

\begin{lemma} \label{lem:upxi_cmcfa}
    Let $(\matr{\upeta}{i}{j}, \matr{\upkappa}{i}{j}, \upphi, \uppsi, \upnu, \upchi^j)$ be a solution to the linearized Einstein--scalar field system \eqref{eq:upeta_evol}--\eqref{eq:upkappa_sym} in CMCFA gauge, with regularity as in either of (i) or (ii) of Theorem~\ref{thm:einstein_scat_v2}. Then the change of gauge vector field $\upxi^j$ defined in \eqref{eq:upxi_cmcfa} satisfies
    \[
        \upxi^j \in C^0((0, + \infty): H^{s + \frac{1}{2}} ).
    \]
    Moreover there exists a unique vector field $\upxi^j_{\infty} \in H^{s + \frac{1}{2}}$ such that $\upxi^j(t) \to \upxi_{\infty}^j$ in $H^{s + \frac{1}{2}}$ as $t \to 0$.
\end{lemma}

\begin{proof}
    Let us first work at the level of each Fourier mode $(\matr{(\upeta_{\lambda})}{i}{j}, \matr{(\upkappa_{\lambda})}{i}{j}, \upphi_{\lambda}, \uppsi_{\lambda}, \upnu_{\lambda}, \upchi^j_{\lambda})$, where $\lambda \in \Z^D \setminus \{ 0 \}$ (recall that $\upchi^j$ and thus $\upxi^j$ have zero average). Using \eqref{eq:einstein_high_shift} from Lemma~\ref{prop:einstein_high_elliptic}, and unpacking the $\mathring{g}$-norm, we find that for $\tau_{\lambda}(t) \geq \mathring{\tau}$,
    \begin{equation} \label{eq:upxi_high_cmctc}
        | \upchi^j_{\lambda} (t) | \lesssim \frac{t^{1 - p_j}}{\tau^{\frac{3}{2}}} \mathcal{E}_{\lambda, hm}^{1/2}(1).
    \end{equation}

    Here we have used the high-frequency energy estimate Proposition~\ref{prop:einstein_high_freq} to allow the right-hand side to be $\mathcal{E}_{\lambda, hm}(1)$, as defined in \eqref{eq:hm_freq}, in place of $\mathcal{E}_{\lambda, high}(t)$. Using also the mid-frequency energy estimate Proposition~\ref{prop:mid_freq} and the associated elliptic estimates, \eqref{eq:upxi_high_cmctc} also holds in the mid-frequency range, and thus for the whole range $\tau_{\lambda}(t) \geq 1$.

    In the low-frequency regime, we estimate $|\upchi^j_{\lambda}|$ by referring to the proof of Proposition~\ref{prop:einsteinscat_10}, in particular \eqref{eq:low_freq_upxi} and the description of $\upchi^j_{\lambda}$ immediately below it. Combining all the energy estimates, one deduces that in the low-frequency range $0 < \tau_{\lambda}(t) \leq 1$, one has
    \begin{equation} \label{eq:upxi_low_cmctc}
        | \upchi^j_{\lambda} (t) | \lesssim t_{\lambda*}^{1- p_j} \cdot \left( \frac{t}{t_{\lambda*}} \right)^{\updelta} \mathcal{E}_{\lambda, hm}^{1/2}(1).
    \end{equation}
    Upon using that $t \leq 1$ and $t_{\lambda*} \leq 1$ in our setting, we combine \eqref{eq:upxi_high_cmctc} and \eqref{eq:upxi_low_cmctc} to get that for all $t > 0$,
    \begin{equation} \label{eq:upxi_comb_cmctc}
        | \upchi^j_{\lambda}(t) | \lesssim \min \left \{ \tau^{- \frac{3}{2}}, \left( \frac{t}{t_{\lambda*}} \right)^{\updelta} \right\} \cdot \mathcal{E}_{\lambda, hm}^{1/2}(1).
    \end{equation}

    \eqref{eq:upxi_comb_cmctc} implies that $|\upchi^j_{\lambda}(t)|$ is integrable (with respect to the differential $\frac{dt}{t}$) both towards $\tau \to \infty$ and towards $t \to 0$. Indeed, using also Lemma~\ref{lem:tauint}, it is straightforward to see that $\upxi^j_{\lambda}(t)$, defined in \eqref{eq:upxi_cmcfa}, has a limit $(\upxi_{\infty})_{\lambda}^j$ as $t \to 0$, and we have
    \begin{equation} \label{eq:upxi_reg}
        | (\upxi_{\infty})_{\lambda}^j | \lesssim \mathcal{E}^{1/2}_{\lambda, high}(1), \qquad | \upxi^j_{\lambda}(t) -(\upxi_{\infty})_{\lambda}^j | \lesssim \max \left \{ \left( \frac{t}{t_{\lambda*}} \right)^{\updelta}, 1\right\} \cdot \mathcal{E}_{\lambda, high}^{1/2}(1).
    \end{equation}

    To complete the proof, we recall that in Theorem~\ref{thm:einstein_scat_v2}(i) and (ii), we have $\sum_{\lambda \in \Z^D \setminus \{0\}} \langle \lambda \rangle^{2s + 1} \mathcal{E}_{\lambda, hm}(1) < + \infty$. The lemma then follows straightforwardly from \eqref{eq:upxi_reg}.
\end{proof}

With Lemma~\ref{lem:upxi_cmcfa}, we now deduce Theorem~\ref{thm:einstein_scat} and Theorem~\ref{thm:asymp_einstein} from Theorem~\ref{thm:einstein_scat_v2}.

\begin{proof}[Proof of Theorem~\ref{thm:einstein_scat}] We prove (\ref{item:einstein_scatop}), (\ref{item:einstein_asympcomp}) and (\ref{item:einstein_scatiso}) as follows:
    \begin{enumerate}[(i)]
        \item
            Let $(\matr{(\upeta_C)}{i}{j}, \matr{(\upkappa_C)}{i}{j}, \upphi_C, \uppsi_C)$ be Cauchy data with regularity as given and satisfying the constraints \eqref{eq:hamiltonian_linear}--\eqref{eq:upkappa_sym}. We now apply Theorem~\ref{thm:einstein_scat_v2}, which gives us a classical solution $(\matr{\upeta}{i}{j}, \matr{\upkappa}{i}{j}, \upphi, \uppsi, \upnu, \upchi^j)$ of \eqref{eq:upeta_evol}--\eqref{eq:upkappa_sym} in the CMCFA gauge. There are two steps to get to the setup of Theorem~\ref{thm:einstein_scat}.

            Firstly, even before changing gauge, we upgrade the convergence of the frequency adapted quantities $\matr{{}^{(F)}\Upupsilon}{i}{j}$ and ${}^{(F)} \upvarphi$ to convergence of $\matr{\Upupsilon}{i}{j}$ and $\upvarphi$ in Definition~\ref{def:upupsilonupvarphi}. From Definition~\ref{def:time_freq_renormalized}, we have
            \begin{equation} \label{eq:time_freq_back} 
                \matr{\Upupsilon}{i}{j} = \matr{{}^{(F)} \Upupsilon}{i}{j} + \int^1_{\max\{t,\mathcal{T}_*\}} \mathring{g}_{ip}(s) \mathring{g}^{jq}(s) \, \frac{ds}{s} \cdot \matr{\upkappa}{q}{p}, \quad
                \upvarphi = {}^{(F)} \upvarphi - \max \{ \log( \mathcal{T}_* ), \log t \} \cdot \uppsi.
            \end{equation}
            Here, the object involving integral in \eqref{eq:time_freq_back} is a time-dependent symbol which, in some symbol class $S^{\beta_0}$ for $\beta > 0$, converges (strongly) as follows:
            \[
                \int^1_{\max\{t,\mathcal{T}_*\}} \mathring{g}_{ip}(s) \mathring{g}^{jq}(s) \, \frac{ds}{s} \to
                \int^1_{\mathcal{T}_*} \mathring{g}_{ip}(s) \mathring{g}^{jq}(s) \, \frac{ds}{s} 
            \]
            Therefore, for $s$ sufficiently large, the convergence \eqref{eq:asymp_upkappaupupsilon2} together with Sobolev embedding implies the existence of $\matr{(\Upupsilon_{\infty})}{i}{j}$ such that the convergence \eqref{eq:asymp_upkappaupupsilon} also holds, at least in CMCFA gauge. A similar result holds for $\upvarphi$, and here we can get convergence in $H^s$ by the same proof as Proposition~\ref{prop:wavescat_10}.

            To transform back to CMCTC gauge, we use the change of gauge associated to the vector field $\upxi^j$ defined in \eqref{eq:upxi_cmcfa}. By \eqref{eq:upkappa_diffeo} and \eqref{eq:upupsilon_diffeo}, $\matr{\upkappa}{i}{j}$ and $\matr{\Upupsilon}{i}{j}$ transform as:
            \[
                \matr{\upkappa}{i}{j} \mapsto \matr{\upkappa}{i}{j} + (t \matr{\mathring{k}}{i}{p}) \partial_p \upxi^j - (t \matr{\mathring{k}}{p}{j}) \partial_i \upxi^p, \quad
                \matr{\Upupsilon}{i}{j} \mapsto \matr{\Upupsilon}{i}{j} + \frac{1}{2} \partial_i \upxi^j + \frac{1}{2} \mathring{g}_{ip}(1) \mathring{g}^{jq}(1) \partial_q \upxi^p.
            \]
            Since, by Lemma~\ref{lem:upxi_cmcfa}, the vector field $\upxi^j$ also converges to $\upxi_{\infty}^j$ in $H^{s + \frac{1}{2}}$ as $t \to 0$, Sobolev embedding implies that we also have the desired convergence in CMCTC gauge. 

            Finally, the regularity statement \eqref{eq:asymp_metric_reg} follows from \eqref{eq:asymp_reg2}, upon using \eqref{eq:time_freq_back} and unraveling the $H^{s, F}(\Sigma_0)$ norm from Definition~\ref{def:time_freq_norm}. The remainder of Theorem~\ref{thm:einstein_scat}(i) is straightforward.

        \item
            The proof of Theorem~\ref{thm:einstein_scat}(ii) follows from an identical manner to the above; we first apply Theorem~\ref{thm:einstein_scat_v2}(ii) to find the solution in CMCFA gauge and then transform back using $\upxi^j$ from \eqref{eq:upxi_cmcfa}.

        \item
            For the scattering isomorphism in Theorem~\ref{thm:einstein_scat}(iii), one simply notes that the spaces $\mathcal{H}^s_{C,c}$ in Theorem~\ref{thm:einstein_scat} and Theorem~\ref{thm:einstein_scat_v2} are identical, while the difference between $\mathcal{H}^s_{\infty, c}$ and ${}^{(F)} \mathcal{H}^s_{\infty, c}$ exactly encodes going from the frequency adapted quantities $\matr{{}^{(F)} \Upupsilon}{i}{j}$ and ${}^{(F)} \upvarphi$ back to $\matr{\Upupsilon}{i}{j}$ to $\upvarphi$. Thus Theorem~\ref{thm:einstein_scat}(iii) follows immediately from Theorem~\ref{thm:einstein_scat_v2}(iii).  \qedhere
    \end{enumerate}
\end{proof}

\begin{proof}[Proof of Theorem~\ref{thm:asymp_einstein}]
    The asymptotics of Theorem~\ref{thm:asymp_einstein} follow straightforwardly from Theorem~\ref{thm:einstein_scat}(i) and the definition of the linearly small quantities in Definition~\ref{def:linsmall}. The scalar field asymptotics \eqref{eq:einstein_asymp_3} and \eqref{eq:einstein_asymp_4} may be read off immediately, while the metric perturbations $h_{ij} = \bar{g}_{ij} - \mathring{\bar{g}}_{ij}$ and $\kappa_{ij} = t(k_{ij} - \mathring{k}_{ij})$ are related to $\matr{\Upupsilon}{i}{j}$ and $\matr{\upkappa}{i}{j}$ (to linear order) as follows:
    \begin{gather*}
        h_{ij} = - 2 \mathring{g}_{jk} \matr{\upeta}{i}{k} = - 2 \mathring{g}_{jk} \matr{\Upupsilon}{i}{k} + 2 \mathring{g}_{jk} \int^1_t \mathring{g}_{ip}(s) \mathring{g}^{kq}(s) \frac{ds}{s} \cdot \matr{\upkappa}{q}{p}, \\[0.4em]
        k_{ij} = \mathring{g}_{jk} \matr{\upkappa}{i}{k} - 2 \mathring{g}_{jk} (t \matr{\mathring{k}}{\ell}{k}) \matr{\upeta}{i}{\ell} = \mathring{g}_{jk} ( \matr{\upkappa}{i}{k} - (2 t \matr{\mathring{k}}{\ell}{k}) \matr{\Upupsilon}{i}{\ell} ) + \mathring{g}_{jk} (2 t \matr{\mathring{k}}{\ell}{k}) \int^1_t \mathring{g}_{ip} (s) \mathring{g}^{\ell q} (s) \frac{ds}{s} \cdot \matr{\upkappa}{q}{p}.
    \end{gather*}

    The asymptotics \eqref{eq:einstein_asymp_1} and \eqref{eq:einstein_asymp_2} then follow upon insertion of the convergence $\matr{\upkappa}{i}{j}(t) \to \matr{(\upkappa_{\infty})}{i}{j}$ and $\matr{\Upupsilon}{i}{j}(t) \to \matr{(\Upupsilon_{\infty})}{i}{j}$ as $t \to 0$ from Theorem~\ref{thm:einstein_scat}(i), as well as inserting the explicit form of $\mathring{g}_{ij}$ and $t \matr{\mathring{k}}{i}{j}$. 

    To be completely precise, there is a caveat in that Theorem~\ref{thm:einstein_scat}(i) assumes the spatially harmonic condition \eqref{eq:spatiallyharmonic} at $t = 1$, but one may simply proceed via a time-independent vector field $\upxi^j$ which brings us into this gauge, e.g.~in a similar fashion to Lemma~\ref{lem:diffeo}(b) we take $\upxi^j$ time-independent solving 
    \[
        t^2 \mathring{g}^{ab}(1) \partial_a \partial_b \upxi^j = t^2 \mathring{g}^{jk}(1)  \left(  \partial_k \tr \upeta(1) - 2 \partial_{\ell} \matr{\upeta}{k}{\ell}(1) \right).
    \]
    One then transforms back, again using Lemma~\ref{lem:diffeo}; since $\upxi^j$ is time-independent this does not change the structure of the desired asymptotics.
\end{proof}

\appendix

\section{Fourier decomposition and \texorpdfstring{$\Psi$DOs}{PDOs}} \label{app:fourier}

In Appendix~\ref{app:fourier}, we review the Fourier decomposition and Sobolev spaces $H^s$ on $\mathbb{T}^D = [- \pi, \pi]^D$. We also review pseudodifferential operators on $\mathbb{T}^D$ and their symbol calculus, and then mention some properties of the symbol $\mathcal{T}_*$ defined in Definition~\ref{def:tstar} in Lemma~\ref{lem:tstar}.

\begin{definition} \label{def:fourier}
    Let $f: \mathbb{T}^D \to \R$ be a sufficiently regular\footnote{From the standard Fourier theory, $f \in L^1(\mathbb{T}^D)$ will be sufficient, however are eventually more interested in the Sobolev spaces $H^s$. As we also entertain $s < 0$ in some cases, one may even consider the case where $f$ is simply a distribution.} function. The \emph{standard Fourier decomposition}, or \emph{Fourier series},  $(f_{\lambda})_{\lambda \in \Z^d}$, is defined as follows. For $\lambda = (\lambda_1, \ldots, \lambda_D) \in \Z^D$, we have
    \begin{equation} \label{eq:fourier}
        f_{\lambda} = \frac{1}{(2 \pi)^D} \int_{\mathbb{T}^D} f(x) e^{- i \lambda \cdot x} \, dx = \frac{1}{(2 \pi)^D} \int_{x \in [- \pi, \pi]^D} f(x) e^{-i \lambda_k x^k} \,dx^1 \cdots dx^D.
    \end{equation}
    A similar decomposition holds for tensors in $\mathbb{T}^D$, e.g.~a $(1, 1)$ tensor $A = \matr{A}{i}{j}$ on $\mathbb{T}^D$ has Fourier decomposition:
    \begin{equation*}
        \matr{\left( A_{\lambda} \right)}{i}{j} = \frac{1}{(2\pi)^D} \int_{\mathbb{T}^D} \matr{A}{i}{j}(x) e^{-i \lambda \cdot x} \, dx.
    \end{equation*}
\end{definition}

\begin{remark}
For sufficiently regular functions $f$, the Fourier inversion formula holds:
\begin{equation}
    f(x) = \sum_{\lambda \in \Z^D} f_{\lambda} \, e^{i \lambda \cdot x},
\end{equation}
hence one is able to recover a function $f$ on $\mathbb{T}^D$ from its Fourier series, and similarly for tensors on $\mathbb{T}^D$. 
\end{remark}

\begin{definition}
    For $s \in \R$, a function $f: \mathbb{T}^D \to \R$ lies in the \emph{Sobolev space} $H^s = H^s(\mathbb{T}^D)$ if the following squared norm is finite, where $\langle \cdot \rangle$ is the usual Japanese bracket; $\langle x \rangle^2 \coloneqq 1 + |x|^2$.
    \begin{equation}
        \| f \|_{H^s}^2 \coloneqq \sum_{\lambda \in \Z^D} \langle \lambda \rangle^{2s} |f_{\lambda}|^2.
    \end{equation}
     Similarly, a $(1, 1)$ tensor $A = \matr{A}{i}{j}$ lies in the Sobolev space $H^s$ if 
    \begin{equation}
        \| \matr{A}{i}{j} \|_{H^s}^2 \coloneqq \sum_{i, j = 1}^D \sum_{\lambda \in \Z^D} \langle \lambda \rangle^{2s} \left| \matr{\left(A_{\lambda}\right)}{i}{j} \right|^2 < + \infty.
    \end{equation}
\end{definition}

\begin{definition} \label{def:symbolcalc}
    Let $L: \mathbb{Z}^D \to \R$ be a function that grows no faster than polynomially. We define the \emph{symbol} $\mathcal{L}$ associated to $L$, as follows:
    For $f \in C^{\infty}(\mathbb{T}^D)$ smooth, let $\mathcal{L} f$ be the (smooth) function whose Fourier coefficients are given by the following expression
    \begin{singlespace}
        \[
            \left( \mathcal{L} f \right)_{\lambda} = \begin{cases}
                L(\lambda) f_{\lambda} & \text{ if } \lambda \neq 0,\\[0.3em]
                0 & \text { otherwise. }
            \end{cases}
        \]
    \end{singlespace} \noindent
    Next, following the symbol calculus of \cite{HormanderPseudo}, for $F: \R \to \R$ such that $F \circ T: \mathbb{Z}^D \to \R$ grows no faster than polynomially, define the operator $F(\mathcal{L}): C^{\infty}(\mathbb{T}^D) \to C^{\infty}(\mathbb{T}^D)$ to be such that 
    \begin{singlespace}
    \[
        \left( F(\mathcal{L}) f \right)_{\lambda} = \begin{cases}
            F(\mathcal{L}) \, (F \circ L (\lambda)) f_{\lambda} & \text{ if } \lambda \neq 0,\\[0.3em]
            0 & \text { otherwise. }
        \end{cases}
    \]
    \end{singlespace} \noindent
    For $m \in \R$, we say that $\mathcal{L}$ is in the \emph{symbol class} $S^m$ if $L(\lambda)$ is bounded by a constant multiple of $\langle \lambda \rangle^m$. If $\mathcal{L}$ is in the symbol class $S^m$, then for any $s \in \R$, $\mathcal{L}$ may be extended as a bounded operator from $H^s$ to $H^{s - m}$.
\end{definition}

As promised, we end this section with a brief discussion of the properties of the symbol $\mathcal{T}_*$, which is defined in Definition~\ref{def:tstar}.

\begin{lemma} \label{lem:tstar}
    Let $\mathcal{T}_*$ be as in Definition~\ref{def:tstar}. Then:
    \begin{enumerate}[(1)]
        \item
            The symbol $\log(\mathcal{T}_*)$ lies in the symbol class $S^{\varepsilon}$ if and only if $\varepsilon > 0$.
        \item
            For $p = \max p_i$, the symbol $\mathcal{T}_*^{-1}$ lies in the symbol class $S^N$ if and only if $N \geq \frac{1}{1-p}$.
        \item
            For $P = \min p_i$, the symbol $\mathcal{T}_*$ lies in the symbol class $S^N$ if and only if $N \geq - \frac{1}{1-P}$.
    \end{enumerate}
\end{lemma}

\begin{proof}
    Using Definition~\ref{def:tstar}, the lemma will follow from an understanding of the growth rates of the function $t_*: \Z^D \to (0, + \infty)$ such that $t_*(\lambda) = t_{\lambda*}$ satisfies $t_*(0) = 1$ and
    \[
        \sum_{i=1}^D \lambda_i^2 t_{\lambda*}^{2 - 2p_i} = 1, \text{ for } \lambda \neq 0.
    \]

    For each $\lambda \neq 0$, it is immediate from this definition that for all $i$, one has $t_{\lambda*} \leq \lambda_i^{- \frac{1}{1-p_i}}$, and that this is inequality is sharp if $\lambda_i$ is the only nonvanishing component of $\lambda$. This proves (3) immediately.

    To show (2), note that for $\lambda \neq 0$, there must be some $i \in \{1, \ldots D\}$ such that $\lambda_i^2 t_{\lambda*}^{2 - 2p_i} \geq \frac{1}{D}$. Therefore for this $i$, we have that $t_{\lambda*}^{-1} \lesssim \lambda_i^{\frac{1}{1-p_i}} \leq \langle \lambda \rangle^{\frac{1}{1-p}}$. Thus if $N \geq \frac{1}{1-p}$, the symbol $\mathcal{T}_*$ lies in the symbol class $S^N$. Conversely, if without loss of generality $p = p_1$, then considering $\lambda = (\lambda_1, 0, \ldots, 0)$ proves that we cannot have $t_{\lambda*}^{-1} \lesssim \langle \lambda \rangle^{N}$ for any smaller $N$.

    Finally, since we have shown both that $t_{\lambda*} \leq \langle \lambda \rangle^{-\frac{1}{1-P}}$ and that $t_{\lambda*}^{-1} \leq \langle \lambda \rangle^{\frac{1}{1-p}}$, it is clear that $\log t_{\lambda*}^{-1} \asymp \log \langle \lambda \rangle$. From this, the claim (1) is also straightforward.
\end{proof}

\section{Fuchsian ODE systems} \label{app:fuchsian}

A \emph{Fuchsian PDE} is an equation of the following form, where $\mathbf{V}: (0, T) \times \R^k \to \R^N$, $A$ is a suitable matrix-valued operator, and $f$ vanishes suitably as $t \to 0$.
\begin{equation*}
    \left( t \frac{\partial}{\partial t} + A \right) \mathbf{V} = f [ t, \mathbf{V} ].
\end{equation*}
Such equations originally arose in the ``Fuchs-Frobenius'' study of special functions and their singularities in the complex plane.
We refer the reader to the introduction to \cite{KichenassamyFuchsian} for details.

The above definition of Fuchsian PDEs is somewhat imprecise, and in this article we only make concrete definitions in the case of ODEs. In this case, the definition will be as follows.

\begin{definition}[Fuchsian ODEs] \label{def:fuchsian}
    Let $\mathbf{V}: [0, T) \subset \R \to \R^N$ be a continuous vector-valued function, and let $\varepsilon > 0$. We say that $\mathbf{V}$ satisfies a \emph{Fuchsian ODE with $t^{\varepsilon}$ weight} if for $t \in (0, t)$, the following equation holds:
    \begin{equation} \label{eq:fuchsian}
        \left( t \frac{d}{dt} + A \right) \mathbf{V} = t^{\varepsilon} F [ t, \mathbf{V} ]
    \end{equation}
    where $A$ is a constant $N \times N$ matrix, and $F: [0, T) \times \R^N \to \R^N$ is continuous in both arguments and locally Lipschitz in its second argument.
\end{definition}

\begin{remark}
    One can make more general definitions where the RHS of \eqref{eq:fuchsian} tends to $0$ at a slower rate than $t^{\epsilon}$, or where the regularity assumptions are somewhat relaxed. However we do not need to do this here, and we refer the interested reader to \cite[Chapter 5]{KichenassamyFuchsian}.
\end{remark}

We would like a local existence and uniqueness theory for $\mathbf{V}$ satisfying Fuchsian ODEs of the form \eqref{eq:fuchsian}, with initial data given at the `singularity' at $t = 0$. The result, which parallels the usual Picard--Lindel\"of theorem for regular ODEs, is as follows:

\begin{lemma} \label{lem:fuchsian}
    Consider a Fuchsian ODE of the form \eqref{eq:fuchsian}, and suppose moreover that the matrix $A$ is such that (the operator norm of) the matrix $\sigma^A = \exp( A \log \sigma )$ is uniformly bounded for $\sigma \in (0, 1)$. Then
    \begin{enumerate}[(a)]
        \item \label{item:fuchsian_a}
            Consider the initial data $\mathbf{V}(0) = 0$. Then for some $T_* > 0$, there exists a unique continuous $\mathbf{V}: [0, T_*) \to \R^N$ which has $V(0) = 0$ and satisfies \eqref{eq:fuchsian} for $t \in (0, T_*)$. Further, the solution satisfies the bound $|\mathbf{V}(t)| \lesssim t^{\varepsilon}$.

        \item \label{item:fuchsian_b}
            Now let $\mathbf{v} \in \R^N$ be in the kernel of $A$, and consider initial data $\mathbf{V}(0) = \mathbf{v}$. Then for some $T_* > 0$, there exists a unique continuous $\mathbf{V}: [0, T_*) \to \R^N$ which has $V(0) = \mathbf{V}$ and satisfies \eqref{eq:fuchsian} for $t \in (0, T_*)$. Further, the solution satisfies the bound $|\mathbf{V}(t) - \mathbf{v}| \lesssim t^{\epsilon}$.
    \end{enumerate}
\end{lemma}

\begin{proof}
    For part (a), we employ a contraction mapping argument in the following subset of $C([0, T_*), \R^N)$:
    \[
        X = \{ \mathbf{W} \in C([0, T_*), \R^N): \mathbf{W}(0) = 0, \,|\mathbf{W}(t)| \leq M \},
    \]
    where $T_* > 0$ and $M > 0$ will be fixed later. The idea is to use the following map $S: X \to X$.
    \begin{equation} \label{eq:iteration}
        S(\mathbf{W})(0) = 0, \quad S(\mathbf{W}) (t) = t^{\varepsilon} \int^1_0 \sigma^{A - 1 + \varepsilon} F[ \sigma t, \mathbf{W}(\sigma t) ] \, d \sigma \qquad \forall t \in (0, T_*).
    \end{equation}

    Firstly, we note that $\mathbf{V}$ satisfies the Fuchsian ODE \eqref{eq:fuchsian} with $\mathbf{V}(0) = 0$ if and only if $S(\mathbf{V}) = \mathbf{V}$. This is because the integral in \eqref{eq:iteration} can be rewritten as:
    \begin{align*}
        t^{\varepsilon} \int^1_0 \sigma^{A - 1 + \varepsilon} F[ \sigma t, \mathbf{W}(\sigma t) ] \, d \sigma 
        &= t^{- A} \int^1_0 (\sigma t)^{A + \varepsilon} F[ \sigma t, \mathbf{W}(\sigma t) ] \, \frac{d \sigma}{\sigma} \\[0.4em]
        &= t^{- A} \int^t_0 \tilde{\sigma}^{A + \varepsilon} F[ \tilde{\sigma}, \mathbf{W}(\tilde{\sigma})] \, \frac{ d \tilde{\sigma}}{\tilde{\sigma}}.
    \end{align*}
    Therefore $S(\mathbf{V}) = \mathbf{V}$ if and only if we have that
    \[
        t^A \mathbf{V} (t) = \int^t_0 \tilde{\sigma}^{A + \epsilon} F[ \tilde{\sigma}, \mathbf{V}(\tilde{\sigma})] \, \frac{d \tilde{\sigma}}{\tilde{\sigma}}, \qquad \forall t \in (0, T_*)
    \]
    which is obviously equivalent to $\mathbf{V}$ satisfying \eqref{eq:fuchsian} for $t \in (0, T_*)$.

    We next show that for correctly chosen $M$ and $T_*$, the map $S$ sends the domain $X$ to itself. The fact that $S(\mathbf{W})$ is continuous on $[0, T_*)$ is easily checked, while for the boundedness let us first suppose that $K$ is the Lipschitz constant of $F$ with respect to its second argument (since this lemma is local in nature, we assume without loss of generality that $K$ is a global Lipschitz constant). Then we have
    \begin{align*}
        | S (\mathbf{W})(t) | 
        &\leq t^{\varepsilon} \int^1_0 \| \sigma^{A} \|_{op} \, \sigma^{-1 + \epsilon} | F[ \sigma t, \mathbf{W}(\sigma t)] | \, d \sigma \\[0.4em]
        &\leq t^{\epsilon} \sup_{\sigma \in (0, 1)} \| \sigma^{A} \|_{op} \cdot \left( \sup_{s \in (0, t)}F[s, 0] + KM \right) \cdot \int^1_0 \sigma^{-1 + \epsilon} \, d \sigma \\[0.4em]
        &= \epsilon^{-1} t^{\epsilon} \sup_{\sigma \in (0, 1)} \| \sigma^{A} \|_{op} \cdot \left( \sup_{s \in (0, t)}F[s, 0] + KM \right).
    \end{align*}
    Therefore for any $M>0$, for $T_*$ chosen sufficiently small we have $|S (\mathbf{W})(t)| \leq M$, hence $S: X \to X$ as required.
    A similar computation shows that $S$ is a contraction mapping with respect to the $C^0$ norm. Indeed,
    \begin{align*}
        | S (\mathbf{W}_1) (t) - S (\mathbf{W}_2) (t) |
        &\leq t^{\epsilon} \int^1_0 \| \sigma^A \|_{op} \, \sigma^{-1 + \epsilon} | F[ \sigma t, \mathbf{W}_1(\sigma t) ] - F[ \sigma t, \mathbf{W}_2(\sigma t) ] | \, d \sigma \\[0.4em]
        &\leq t^{\epsilon} \sup_{\sigma \in (0, 1)} \| \sigma^A \|_{op} \cdot \int^1_0 \sigma^{-1 + \epsilon} \cdot K | \mathbf{W}_1(\sigma t) - \mathbf{W}_2(\sigma t) | \, d \sigma \\[0.4em]
        &\leq K \epsilon^{-1} t^{\epsilon} \sup_{\sigma \in (0, 1)} \| \sigma^A \|_{op} \cdot \sup_{s \in (0, T_*)} | \mathbf{W}_1 (s) - \mathbf{W}_2 (s)|.
    \end{align*}
    Therefore $S$ is a contraction mapping so long as $T_*$ is chosen such that $K \epsilon^{-1} T_*^{\epsilon} \sup_{\sigma \in (0, 1)} \| \sigma^A \|_{op} < 1$.

    Hence by the contraction mapping theorem, there is a unique fixed point $\mathbf{V}$ of $S$ in the set $X$. This completes the existence part of the proof. For uniqueness, first observe that the contraction argument shows that two solutions $\mathbf{V}_1$ and $\mathbf{V}_2$ to \eqref{eq:fuchsian} agree on an interval $[0, T_-] \subset [0, T_*)$ so long as they remain bounded by $M$. Moreover, by the above bounds they must have $|\mathbf{V}_1(t)|, |\mathbf{V}_2(t)| \lesssim t^{\epsilon}$ for $t \in [0, T_-]$.

    Therefore, by a straightforward continuity argument, they must agree in the whole interval $[0, T_*)$, and satisfy the bound $|\mathbf{V}_1(t)| = |\mathbf{V}_2(t)| \lesssim t^{\epsilon}$ everywhere in this interval of existence. This concludes the proof of (a).

    Finally, part (b) follows immediately from part (a) simply by taking $\tilde{\mathbf{V}} = \mathbf{V} - \mathbf{v}$ and, upon using the fact that $A \mathbf{v} = 0$, rewriting the Fuchsian ODE \eqref{eq:fuchsian} as
    \[
        \left( t \frac{d}{dt} \tilde{\mathbf{V}} + A \right) = t^{\epsilon} \tilde{F}[t, \tilde{\mathbf{V}}], \qquad \tilde{F}[t, \mathbf{W}] = F[t, \mathbf{W} + \mathbf{v}].
    \]
    Applying part (a) to the above equation, this concludes the proof of (b).
\end{proof}

\section{Propagation of constraints} \label{app:constraints}

In Appendix~\ref{app:constraints}, we state and prove two propositions explaining how the constraint equations (and gauge conditions) are propagated in evolution. The first, Proposition~\ref{prop:constraints_1}, shows that in the CMCSH evolution of Proposition~\ref{prop:adm_lwp}(i), the constraints \eqref{eq:hamiltonian_linear}--\eqref{eq:upkappa_sym} and the gauge conditions \eqref{eq:cmc_linear} and \eqref{eq:spatiallyharmonic}, remain true on all Cauchy slices $\Sigma_t$. The second, Proposition~\ref{prop:constraints_2}, starts with the asymptotic constraint equations of Lemma~\ref{lem:scattering_constraints} and shows that the solution produced by Proposition~\ref{prop:einsteinscat_01} solves the actual constraint equations \eqref{eq:hamiltonian_linear}--\eqref{eq:upkappa_sym} as well as the CMC condition $\tr \upkappa = 0$.

\begin{proposition} \label{prop:constraints_1}
    Let $(\matr{\hat{\upeta}}{i}{j}(t), \matr{\hat{\upkappa}}{i}{j}(t), \upphi(t), \uppsi(t), \upnu(t), \upchi^j(t))$ be a solution to the elliptic-hyperbolic system \eqref{eq:upeta_evol}, \eqref{eq:upkappa_evol}, \eqref{eq:ricci_lin_sh}, \eqref{eq:upphi_evol}, \eqref{eq:uppsi_evol}, \eqref{eq:upnu_elliptic}, \eqref{eq:upchi_elliptic} for $t > 0$. Suppose furthermore that for some time $T > 0$, the constraint equations \eqref{eq:hamiltonian_linear}, \eqref{eq:momentum_linear1}, \eqref{eq:upeta_sym} and \eqref{eq:upkappa_sym} hold at $t = T$, and also that the CMCSH gauge constraints $\tr \hat{\upkappa} = 0$ and \eqref{eq:spatiallyharmonic} hold at $t = T$. Then the whole system \eqref{eq:upeta_evol}--\eqref{eq:upkappa_sym} holds for $t > 0$.
\end{proposition}

\begin{proof}
    Let us start with the symmetry constraints \eqref{eq:upeta_sym} and \eqref{eq:upkappa_sym}. For this purpose, we define
    \[
        \mathcal{A}^{ac} \coloneqq \mathring{g}^{ab} \matr{\hat{\upeta}}{b}{c} - \mathring{g}^{cb} \matr{\hat{\upeta}}{b}{a}, \quad 
        \mathcal{B}^{ac} \coloneqq \mathring{g}^{ab} (\matr{\hat{\upkappa}}{b}{c} - (2 t \matr{\mathring{k}}{b}{d}) \matr{\hat{\upeta}}{d}{c})
        - \mathring{g}^{cb} (\matr{\hat{\upkappa}}{b}{a} - (2 t \matr{\mathring{k}}{b}{d}) \matr{\hat{\upeta}}{d}{a}).
    \]
    Using only \eqref{eq:upeta_evol} and \eqref{eq:upkappa_evol}, one may derive evolution equations for $\mathcal{A}^{ac}$ and $\mathcal{B}^{ac}$:
    \[
        t \partial_t \mathcal{A}^{ac} = \mathcal{B}^{ac} + (2 t \matr{\mathring{k}}{b}{a}) \mathcal{A}^{bc} + (2 t \matr{\mathring{k}}{b}{c}) \mathcal{A}^{ab}, \quad
        t \partial_t \mathcal{B}^{ac} = t^2 \mathring{g}^{ij} \partial_i \partial_j \mathcal{A}^{ac} - t^2 \mathring{g}^{ab} \partial_b \partial_d \mathcal{A}^{dc} - t^2 \mathring{g}^{cb} \partial_b \partial_d \mathcal{A}^{ad}.
    \]

    Now, assuming the spatially harmonic gauge condition \eqref{eq:spatiallyharmonic} holds, the expression $t^2 \mathring{g}^{ab} \partial_b \partial_d \mathcal{A}^{dc} + t^2 \mathring{g}^{cb} \partial_b \partial_d \mathcal{A}^{ad}$ will vanish. Unfortunately we have not yet shown that \eqref{eq:spatiallyharmonic} is true for all times, but our strategy will be first apply a gauge transformation, as in Proposition~\ref{prop:adm_lwp}(ii), to create a new solution to the elliptic-hyperbolic system \eqref{eq:upeta_evol}, \eqref{eq:upkappa_evol}, \eqref{eq:ricci_lin_sh}, \eqref{eq:upphi_evol}, \eqref{eq:uppsi_evol}, \eqref{eq:upnu_elliptic}, \eqref{eq:upchi_elliptic} that is spatially harmonic.

    Then since $\mathcal{A}^{ac}$ and $\mathcal{B}^{ac}$ are gauge invariant under the transformations \eqref{eq:upeta_diffeo} and \eqref{eq:upkappa_diffeo},
    it still suffices to deal with the case where we assume \eqref{eq:spatiallyharmonic}, in which case our equations reduce to
    \[
        t \partial_t \mathcal{A}^{ac} = \mathcal{B}^{ac} + (2 t \matr{\mathring{k}}{b}{a}) \mathcal{A}^{bc} + (2 t \matr{\mathring{k}}{b}{c}) \mathcal{A}^{ab}, \quad
        t \partial_t \mathcal{B}^{ac} = t^2 \mathring{g}^{ij} \partial_i \partial_j \mathcal{A}^{ac}.
    \]
    Hence $\mathcal{A}^{ac}$ and $\mathcal{B}^{ac}$ obey a linear and homogeneous first order hyperbolic system, and therefore their vanishing at $t = T$ implies their vanishing everywhere. To be explicit, one can use the following energy estimate:
    \[
        \text{ if } \; \mathcal{E}_{\mathcal{A} \mathcal{B}} (t) = \int_{\mathbb{T}^D} \mathring{g}_{ab} \mathring{g}_{cd} ( t^2 \mathring{g}^{ij} \partial_i \mathcal{A}^{ac} \partial_j \mathcal{A}^{bd} + \mathcal{B}^{ac} \mathcal{B}^{bd} + \mathcal{A}^{ac} \mathcal{A}^{bd}) \, dx, \qquad \text{ then } \; t \partial_t \mathcal{E}_{\mathcal{A} \mathcal{B}}(t) \lesssim \mathcal{E}_{\mathcal{A} \mathcal{B}}(t).
    \]
    Gr\"onwall's inequality therefore shows that $\mathcal{A}^{ac} = \mathcal{B}^{ac} = 0$ for all $t > 0$, proving \eqref{eq:upeta_sym} and \eqref{eq:upkappa_sym}.

    We next move onto the Hamiltonian and momentum constraints \eqref{eq:hamiltonian_linear} and \eqref{eq:momentum_linear1}. Instead of dealing with these equations directly, we instead define:
    \begin{gather*}
        \mathcal{H} = - t^2 \mathring{g}^{ab} \partial_a \partial_b \tr \hat{\upeta} + t^2 \mathring{g}^{ab} \partial_i \partial_a \matr{\hat{\upeta}}{b}{i} + (t \matr{\mathring{k}}{b}{a}) \matr{\hat{\upkappa}}{a}{b} + 2 p_{\phi} \uppsi, \\[0.3em]
        \mathcal{P}_i = \partial_j \matr{\hat{\upkappa}}{i}{j} + \partial_i ((t \matr{\mathring{k}}{b}{a}) \matr{\hat{\upeta}}{a}{b} + 2 p_{\phi} \upphi) - (t \matr{\mathring{k}}{i}{j}) \partial_j \tr \hat{\upeta}.
    \end{gather*}
    Then using \eqref{eq:upeta_evol}, \eqref{eq:upkappa_evol}, \eqref{eq:ricci_lin_sh} and the elliptic equation \eqref{eq:upnu_elliptic}, one derives the evolution equations:
    \[
        t \partial_t \mathcal{H} = t^2 \mathring{g}^{ij} \partial_i \mathcal{P}_j, \quad 
        t \partial_t \mathcal{P}_i = \partial_i \mathcal{H}.
    \]
    
    Note that the constraints being satisfied at $t = T$ imply that $\mathcal{H} = 0$ and $\mathcal{P}_i = 0$ at $t = T$. Now
    \[
        \text{ if } \; \mathcal{E}_{\mathcal{H} \mathcal{P}} (t) = \int_{\mathbb{T}^D} ( t^2 \mathring{g}^{ij} \mathcal{P}_i \mathcal{P}_j + \mathcal{H}^2 ) \, dx, \qquad \text{ then } \; t \partial_t \mathcal{E}_{\mathcal{H} \mathcal{P}}(t) \lesssim \mathcal{E}_{\mathcal{H} \mathcal{P}}(t).
    \]
    Therefore, $\mathcal{H} = 0$ and $\mathcal{P}_i = 0$ for all times $t > 0$. Furthermore, using also the elliptic equations \eqref{eq:upnu_elliptic} and \eqref{eq:upchi_elliptic} one derives that
    \[
        t \partial_t \tr \hat{\upkappa} = 0, \quad t \partial_t ( 2 \partial_j \matr{\hat{\upeta}}{i}{j} - \partial_i \tr \hat{\upeta}) = 2 \mathcal{P}_i.
    \]
    Since we already have $\mathcal{P}_i = 0$, this implies the gauge conditions \eqref{eq:cmc_linear} and \eqref{eq:spatiallyharmonic} remain true in evolution. Combining this with $\mathcal{H} = 0$ and $\mathcal{P}_i = 0$, we also deduce \eqref{eq:hamiltonian_linear} and \eqref{eq:momentum_linear1}. Finally, \eqref{eq:momentum_linear2} may be derived as a consequence of \eqref{eq:momentum_linear1}, \eqref{eq:upeta_sym} and \eqref{eq:upkappa_sym}.
\end{proof}

\begin{proposition} \label{prop:constraints_2}
As in Proposition~\ref{prop:einstein_low_freq_scat}, let $(\matr{(\upkappa_{\lambda})}{i}{j}, \matr{(\tilde{\Upupsilon}_{\lambda})}{i}{j}, \uppsi_{\lambda}, \tilde{\upvarphi}_{\lambda}, \upnu_{\lambda})$ be a solution to the ODE system \eqref{eq:upkappa_evol_l_2}--\eqref{eq:upphi_evol_l_2} together with \eqref{eq:upnu_elliptic_l}, is such that $(\matr{(\upkappa_{\lambda})}{i}{j}, \matr{(\tilde{\Upupsilon}_{\lambda})}{i}{j}, \uppsi_{\lambda}, \tilde{\upvarphi}_{\lambda}) \to (\matr{((\upkappa_{\infty})_{\lambda})}{i}{j}, \matr{((\tilde{\Upupsilon}_{\infty})_{\lambda})}{i}{j}, (\uppsi_{\infty})_{\lambda}, (\tilde{\upvarphi}_{\infty})_{\lambda}) $ as $t \to 0$.

Suppose moreover that these limiting quantities satisfy the asymptotic constraints \eqref{eq:hamiltonian_linear_scat_l}--\eqref{eq:upkappa_sym_scat_l}, with $T = \mathcal{T}_*$. Then the solution $(\matr{(\upkappa_{\lambda})}{i}{j}, \matr{(\tilde{\Upupsilon}_{\lambda})}{i}{j}, \uppsi_{\lambda}, \tilde{\upvarphi}_{\lambda}, \upnu_{\lambda})$ obeys the constraint equations \eqref{eq:hamiltonian_linear_l}--\eqref{eq:upkappa_sym_l}. (Here $\matr{(\tilde{\Upupsilon}_{\lambda})}{i}{j}$ and $\tilde{\upvarphi}_{\lambda}$ are related to $\matr{(\upeta_{\lambda})}{i}{j}$ and $\upphi_{\lambda}$ by \eqref{eq:upupsilon2} and \eqref{eq:upvarphi2}.)
\end{proposition}

\begin{proof}
    The strategy is to combine the constraint propagation method of Proposition~\ref{prop:constraints_1} with some Fuchsian analysis. We start with the symmetry constraints \eqref{eq:upeta_sym_l} and \eqref{eq:upkappa_sym_l}. Instead of dealing with $\matr{\upeta}{i}{j}$ directly, we first define
    \begin{gather*}
        (\mathcal{C}_{\lambda})^{ac} \coloneqq \mathring{g}^{ab}(t_{\lambda*}) \matr{(\tilde{\Upupsilon}_{\lambda})}{b}{c} - \mathring{g}^{cb} (t_{\lambda*}) \matr{(\tilde{\Upupsilon}_{\lambda})}{b}{a}, \\[0.3em] 
        (\mathcal{D}_{\lambda})^{ac} \coloneqq \mathring{g}^{ab}(t_{\lambda*}) (\matr{(\upkappa_{\lambda})}{b}{c} - (2 t \matr{\mathring{k}}{b}{d}) \matr{(\tilde{\Upupsilon}_{\lambda})}{d}{c})
        - \mathring{g}^{cb}(t_{\lambda*}) (\matr{(\upkappa_{\lambda})}{b}{a} - (2 t \matr{\mathring{k}}{b}{d}) \matr{(\tilde{\Upupsilon}_{\lambda})}{d}{a}).
    \end{gather*}
    We use that by the proof of Lemma~\ref{lem:scattering_constraints}, $((\mathcal{C}_{\lambda})^{ac}, (\mathcal{D}_{\lambda})^{ac}) = (0, 0)$ if and only $((\mathcal{A}_{\lambda})^{ac}, (\mathcal{B}_{\lambda})^{ac}) = (0, 0)$, where $\mathcal{A}^{ac}$ and $\mathcal{B}^{ac}$ were defined in the proof of Proposition~\ref{prop:constraints_1}.

    By the assumptions of Proposition~\ref{prop:constraints_2}, $((\mathcal{C}_{\lambda})^{ac}, (\mathcal{D}_{\lambda})^{ac})$ are bounded and tend to $(0, 0)$ as $t \to 0$. Thus by the Fuchsian analysis of Lemma~\ref{lem:fuchsian}, particularly the uniqueness statement, one may deduce that $((\mathcal{C}_{\lambda})^{ac}, (\mathcal{D}_{\lambda})^{ac})$, and therefore also $((\mathcal{A}_{\lambda})^{ac}, (\mathcal{B}_{\lambda})^{ac})$, vanish identically if we can prove that for the vector $\mathbf{C} = ((\mathcal{C}_{\lambda})^{ac}, (\mathcal{D}_{\lambda})^{ac}: a,c = 1, \ldots, D)$, there is some bounded (time-dependent) matrix $A_{\mathcal{C}\mathcal{D}}$ such that
    \begin{equation} \label{eq:fuchsianCD}
        t \frac{d}{dt} \mathbf{C} = t^{\updelta} A_{\mathcal{C}\mathcal{D}} \cdot \mathbf{C}.
    \end{equation}

    One may verify \eqref{eq:fuchsianCD} directly by using the system \eqref{eq:upkappa_evol_l_2}, \eqref{eq:upeta_evol_l_2}, \eqref{eq:upnu_elliptic_l} and going through a lengthy calculation. We propose an alternative approach using linear algebra. Firstly, by Proposition~\ref{prop:einstein_low_der} (ignoring all the appearances of $t_{\lambda*}$), for a solution of the system \eqref{eq:upkappa_evol_l_2}--\eqref{eq:upphi_evol_l_2} and \eqref{eq:upnu_elliptic_l} there does exist a bounded (time-dependent) matrix $B_{\mathcal{C} \mathcal{D}}$ such that for $\mathbf{V} = (\matr{(\upkappa_{\lambda})}{i}{j}, \matr{(\tilde{\Upupsilon}_{\lambda})}{i}{j}, \uppsi_{\lambda}, \tilde{\upvarphi}_{\lambda})$, such that we verify the Fuchsian ODE:
    \begin{equation*}
        t \frac{d}{dt} \mathbf{C} = t^{\updelta} B_{\mathcal{C}\mathcal{D}} \cdot \mathbf{V}.
    \end{equation*}
    On the other hand there is the natural linear map $L: (\matr{(\upkappa_{\lambda})}{i}{j}, \matr{(\tilde{\Upupsilon}_{\lambda})}{i}{j}, \uppsi_{\lambda}, \tilde{\upvarphi}_{\lambda}) \mapsto ((\mathcal{C}_{\lambda})^{ac}, (\mathcal{D}_{\lambda})^{ac})$, which is surjective onto the $D(D-1)$-dimensional space $(\Skew_{D \times D})^2$, meaning two copies of the space of skew-symmetric matrices. 

    Thus \eqref{eq:fuchsianCD} holds if $B_{\mathcal{C} \mathcal{D}}$ factorizes as $B_{\mathcal{C} \mathcal{D}} = A_{\mathcal{C} \mathcal{D}} \circ L$ for some bounded $A_{\mathcal{C} \mathcal{D}}$. The key observation is that $\ker L \subset \ker B_{\mathcal{C} \mathcal{D}}$; this follows since by Proposition~\ref{prop:constraints_1}, for $t > 0$ we know that $t \frac{d}{dt} ( (\mathcal{A}_{\lambda})^{ac}, (\mathcal{B}_{\lambda})^{ac} ) = 0$ if $((\mathcal{A}_{\lambda})^{ac}, (\mathcal{B}_{\lambda})^{ac}) = 0$ and thus the same holds true for $((\mathcal{C}_{\lambda})^{ac}, \mathcal{D}_{\lambda})^{ac})$. Therefore, one may apply the linear algebra fact that since $\ker L \subset \ker B_{\mathcal{C} \mathcal{D}}$ and $L$ is surjective onto $(\Skew_{D \times D})^2$, there must exist some map $A_{\mathcal{C} \mathcal{D}}: (\Skew_{D \times D})^2 \to (\Skew_{D \times D})^2$ such that $B_{\mathcal{C} \mathcal{D}}$ factorizes as $B_{\mathcal{C} \mathcal{D}} = A_{\mathcal{C} \mathcal{D}} \circ L$.

    In fact, one may apply the linear algebra fact in such a way that $A_{\mathcal{C} \mathcal{D}}$ remains uniformly bounded in time. Therefore the Fuchsian ODE \eqref{eq:fuchsianCD}, and thus the propagation of the symmetry constraints follows. We summarize the strategy as suitably renormalized constraints (i.e.~$((\mathcal{C}_{\lambda})^{ac}, (\mathcal{D}_{\lambda})^{ac})$) + good derivative bounds (i.e.~Proposition~\ref{prop:einstein_low_der}) + constraint propagation at $t > 0$ (i.e.~Proposition~\ref{prop:constraints_1}) implies constraint propagation from asymptotic data.

    One can play the same game for the Hamiltonian and momentum constraints. Define
    \begin{gather*}
        \mathcal{H}_{\lambda} = t^2 \mathring{g}^{ab} \lambda_a \lambda_b \tr \tilde{\Upupsilon}_{\lambda} 
        - t^2 \mathring{g}^{ab} \lambda_i \lambda_a \matr{(\tilde{\Upupsilon}_{\lambda})}{b}{i} + t^2 \mathring{g}^{ab} \lambda_i \lambda_a \int^{t_{\lambda*}}_t \mathring{g}_{bp}(s) \mathring{g}^{iq}(s) \frac{ds}{s} \cdot \matr{(\upkappa_{\lambda})}{q}{p} + (t \matr{\mathring{k}}{b}{a}) \matr{(\upkappa_{\lambda})}{a}{b} + 2 p_{\phi} \upphi_{\lambda}, \\[0.3em]
        (\mathcal{Q}_{\lambda})_i = - \mathrm{i} \lambda_j \matr{\upkappa}{i}{j} - \mathrm{i} \lambda_i ((t \matr{\mathring{k}}{b}{a}) \matr{(\tilde{\Upupsilon}_{\lambda})}{a}{b} + 2 p_{\phi} \tilde{\upvarphi}_{\lambda}) + \mathrm{i} (t \matr{\mathring{k}}{i}{j}) \lambda_j \tr \tilde{\Upupsilon} \hspace{6cm}
        \\[0.2em] \hspace{2cm} + \mathrm{i} \lambda_i \log( \frac{t}{t_{\lambda*}} ) \left(t^2 \mathring{g}^{ab} \lambda_a \lambda_b \tr \tilde{\Upupsilon}_{\lambda} 
        - t^2 \mathring{g}^{ab} \lambda_i \lambda_a \matr{(\tilde{\Upupsilon}_{\lambda})}{b}{i} + t^2 \mathring{g}^{ab} \lambda_i \lambda_a \int^{t_{\lambda*}}_t \mathring{g}_{bp}(s) \mathring{g}^{iq}(s) \frac{ds}{s} \cdot \matr{(\upkappa_{\lambda})}{q}{p} \right).
    \end{gather*}

    It is straightforward to check that given the assumptions of the proposition, $(\mathcal{H}_{\lambda}, (\mathcal{Q}_{\lambda})_i)$ are bounded and tend to $(0, 0)$ as $t \to 0$. Since $(\mathcal{Q}_{\lambda})_i = (\mathcal{P}_{\lambda})_i + \mathrm{i} \lambda_i \log( \frac{t}{t_{\lambda*}}) \mathcal{H}_{\lambda}$, for $\mathcal{P}_i$ as in the proof of Proposition~\ref{prop:constraints_1}, we also have constraint propagation at $t > 0$. Finally, from Proposition~\ref{prop:einstein_low_der} one may prove good derivative bounds in the sense that for some bounded (time-dependent) matrix $B_{\mathcal{H} \mathcal{Q}}$ and $\mathbf{V}$ as before,
    \[
        t \frac{d}{dt} \begin{pmatrix} \mathcal{H}_{\lambda} \\[0.1em] (\mathcal{Q}_{\lambda})_i \end{pmatrix} = t^{\updelta} B_{\mathcal{H} \mathcal{Q}} \cdot \mathbf{V}.
    \]
    Therefore following the same program as before we deduce that $\mathcal{H}_{\lambda}$ and $(\mathcal{Q}_{\lambda})_i$ vanish for all $t > 0$, and therefore the constraints \eqref{eq:hamiltonian_linear_l} and \eqref{eq:momentum_linear1_l} are satisfied. 
    
    As always, \eqref{eq:momentum_linear2_l} follows from \eqref{eq:momentum_linear1_l} and \eqref{eq:upeta_sym_l}--\eqref{eq:upkappa_sym_l}. This concludes the proof of the proposition.
\end{proof}

\bibliography{bibliography_master.bib}
\bibliographystyle{abbrvnat_mod}

\end{document}